\documentclass[a4paper,11pt]{article}

\usepackage{amsmath,amsfonts,amssymb,graphicx,color,enumerate,subfigure,wasysym,enumitem}
\usepackage{fullpage}

\makeatletter

\@addtoreset{equation}{section}
\makeatother

\usepackage{hyperref}
\newtheorem{theorem}{Theorem}[section]

\newtheorem{lemma}[theorem]{Lemma}
\newtheorem{proposition}[theorem]{Proposition}

\newtheorem{definition}[theorem]{Definition}
\newtheorem{remark}[theorem]{Remark}
\newtheorem{corollary}[theorem]{Corollary}
\newtheorem{conjecture}[theorem]{Conjecture}

\newenvironment{proof}{\begin{trivlist}
 \item[]\hspace{0cm}{\bf Proof: }
 \hspace{0cm} }{\hfill $\blacksquare$
 \end{trivlist}}

\newcommand \Inte{{\rm{ Int\,}}}
\newcommand \ar {\rightarrow}

\newcommand\beq{\begin{equation}}
\newcommand\eeq{\end{equation}}
\newcommand \Ab{{\bf A}}

\newcommand \Cb{{\bf P}}
\newcommand \Pb{{\bf P}}
\newcommand \cb {{\bf p}}
\newcommand \pb{{\bf p}}
\newcommand \xb {{\bf x}}
\newcommand \yb {{\bf y}}

\newcommand \jb {{\bf j}}
\newcommand \agb{{\boldsymbol{\alpha}}}
\newcommand \curl{{\rm Curl\,}}

\definecolor{gr}{rgb} {0., 0.55, 0.23 }
\definecolor{bl}{rgb} {0., 0.5, 1. }
\definecolor{cy}{rgb} {0., 0.57, 0.67 }
\definecolor{mg}{rgb} {0.85, 0., 0.85}
\definecolor{marron}{rgb} {0.6, 0.40, 0.1} 
\definecolor{or}{rgb} {0.9, 0.5, 0.}

\newcommand{\Rd}{\color{red}}
\newcommand{\Gr}{\color{gr}}

\newcommand{\Bl}{\color{blue}}
\newcommand{\Mg}{\color{mg}}

\begin{document}

\title{Nodal and spectral minimal partitions\\
-- The state of the art in 2015 --}
\author{V. Bonnaillie-No\"el\footnote{D\'epartement de Math\'ematiques et Applications (DMA - UMR 8553), PSL, CNRS, ENS Paris, 45 rue d'Ulm, F-75230 Paris cedex 05, France,
\texttt{bonnaillie@math.cnrs.fr}}, and B. Helffer\footnote{Laboratoire de Math\'ematiques (UMR 8628), Universit\'e Paris Sud, B\^at 425, F-91405 Orsay cedex, France and Laboratoire de Math\'ematique Jean Leray  (UMR 6629), Universit\'e de Nantes, 2 rue de la Houssini\`ere, BP 92208, F-44322 Nantes cedex 3, France, \texttt{bernard.helffer@math.u-psud.fr}}}
\date{\today}
\maketitle

\begin{abstract}
In this article, we propose a state of the art concerning the nodal and spectral minimal partitions. 
First we focus on the nodal partitions and give some examples of Courant sharp cases. 
Then we are interested in minimal spectral partitions. Using the link with the Courant sharp situation, we can determine the minimal $k$-partitions for some particular domains. \\
We also recall some results about the topology of regular partitions and Aharonov-Bohm approach. 
The last section deals with the asymptotic behavior of minimal $k$-partition. 
\end{abstract}

\tableofcontents
%\newpage 

\paragraph{Acknowledgements.}\ \\
{\small We would like to thank particularly our first collaborators T. Hoffmann-Ostenhof and S. Terracini, and also P. B\'erard, C. L\'ena, B. Noris, M. Nys, M. Persson Sundqvist and G. Vial  who join us for the continuation of this programme devoted to minimal partitions.\\
We would also like to thank P. Charron, B. Bogosel, D. Bucur, T. Deheuvels,  A. Henrot, D.~Jakobson,  J. Lamboley, J. Leydold,  \'E. Oudet, M. Pierre, I. Polterovich, T. Ranner\dots for their interest, their help and for useful discussions. \\
The authors are supported by the ANR (Agence Nationale de la Recherche), project OPTIFORM n$^{\rm o}$ ANR-12-BS01-0007-02 and by the Centre Henri Lebesgue (program ``Investissements d'avenir'' -- n$^{\rm o}$ ANR-11-LABX-0020-01). During the writing of this work, the second author was Simons foundation visiting fellow at the Isaac Newton Institute in Cambridge.}

\section{Introduction}
We consider mainly the Dirichlet realization of the Laplacian operator  in $\Omega$, when $\Omega$ is a bounded domain in $\mathbb R^2$ with piecewise-$C^1$ boundary (corners or cracks permitted). This operator will be denoted by $H(\Omega)$. We would like to analyze the relations between the nodal domains of the eigenfunctions of $H(\Omega)$ and the partitions of $\Omega$ by $ k$ open sets $ D_i$ which are minimal in the sense that the maximum over the $D_i$'s of the groundstate energy of the Dirichlet realization of the Laplacian $H(D_i)$ is minimal. This problem can be seen as a strong competition limit of segregating species in population dynamics (see \cite{CTV0,CTV2} and references therein). 
 
\begin{definition}\label{OPSa}
%Let $ 1 \le k\in \mathbb N$. 
A {\bf partition} (or $k$-partition for indicating the cardinal of the partition) of $\Omega$ is a family $\mathcal D=\{D_i\}_{1\leq i\leq k}$ of $k$ mutually disjoint sets in $\Omega$ (with $k\geq1$ an integer).
\end{definition} 
We denote by $\mathfrak O_k =\mathfrak O_k(\Omega)$ the set of partitions of $\Omega$ where the $D_i$'s are domains (i.e. open and connected). We now introduce the notion of the energy of a partition. 
\begin{definition}\label{regOma}
For any integer $ k\ge 1$, and for $\mathcal D=\{D_i\}_{1\leq i\leq k}$ in $\mathfrak O_k(\Omega)$, we introduce the energy of the partition:
\begin{equation}\label{LaD}
\Lambda(\mathcal D)=\max_{1\leq i\leq k}\lambda(D_i).
\end{equation} 
\end{definition}
The optimal problem we are interested in is to determine for any integer $k\geq1$
\begin{equation}\label{eq.Lkdef}
\mathfrak L_{k}=\mathfrak L_{k}(\Omega) = \inf_{\mathcal D\in\mathfrak O_k(\Omega)} \Lambda(\mathcal D).
\end{equation}
We can also consider the case of a two-dimensional Riemannian manifold and the Laplacian is then the Laplace Beltrami operator. We denote by $\{\lambda_j(\Omega),j\geq 1\}$ (or more simply $\lambda_j$ if there is no ambiguity) the non decreasing sequence of its eigenvalues and by $\{u_j,j\geq 1\}$ some associated orthonormal basis of eigenfunctions. For shortness, we often write $\lambda(\Omega)$ instead of $\lambda_{1}(\Omega)$. The groundstate $ u_1$ can be chosen to be strictly positive in $\Omega$, but the other excited eigenfunctions $u_k$ must have zerosets. Here we recall that for $ u\in C^0(\overline\Omega)$, the nodal set (or zeroset) of $u$ is defined by~:
\begin{equation}
N(u)=\overline{\{\xb \in \Omega\:\big|\: u(\xb )=0\}}\,.
\end{equation}
In the case when $u$ is an eigenfunction of the Laplacian, the $\mu(u)$ components of $\Omega\setminus N(u)$ are called the nodal domains of $ u$ and define naturally a partition of $\Omega$ by $\mu(u)$ open sets, which will be called a {\bf nodal partition}.\\ 
Our main goal is to discuss the links between the partitions of $\Omega$ associated with these eigenfunctions and the minimal partitions of $\Omega$.
 
 %%%%%%%%%%%%
 \section{Nodal partitions}\label{s2}
 \subsection{Minimax characterization}\label{ss2.1}
 
\subsubsection*{Flexible criterion}
We first give a flexible criterion for the determination of the bottom of the spectrum. 
\begin{theorem} \label{Theorem2.1}
Let $\mathcal H$ be an Hilbert space of infinite dimension and $P$ be a self-adjoint semibounded operator\footnote{The operator is associated with a coercive continuous symmetric sesquilinear form via Lax-Milgram's theorem. See for example \cite{HeCambridge}.} of form domain $Q(P)\subset \mathcal H$ with compact injection. 
Let us introduce
 \begin{equation}\label{def11}
\mu_1(P) = \inf_{\phi\in Q( P)\setminus\{0\}} \frac{\langle P\phi\;|\;\phi\rangle_\mathcal H}{\|\phi\|^2}\;,
\end{equation}
and, for $n\geq 2$
\begin{equation}\label{def1}
\mu_n(P) = \sup_{\psi_1,\psi_2,\dots,\psi_{n-1}\in Q(P)}\;
\inf_{\substack{\phi\in [{\rm span\, }(\psi_1,\dots,\psi_{n-1})]^{\perp};\\ \phi\in Q(P)\setminus\{0\}}} \frac{\langle P\phi\;|\;\phi\rangle_\mathcal H}{\|\phi\|^2}\;.
\end{equation}
Then $\mu_n(P)$ is the $n$-th eigenvalue when ordering the eigenvalues in increasing order (and counting the multiplicity).
\end{theorem}
Note that the proof involves the following proposition
\begin{proposition}\label{Proposition2.2}
Under the conditions of Theorem \ref{Theorem2.1}, suppose that there exist a constant $a$ and a $n$-dimensional subspace $V\subset Q(P)$ such that 
$$\langle P \phi\,,\, \phi\rangle_{\mathcal H} \leq a\|\phi\|^2,\qquad\forall \phi\in V.$$ 
Then $\mu_n(P) \leq a\;.$
\end{proposition}
This could be applied when $P$ is the Dirichlet Laplacian (form domain $H_0^1(\Omega)$), the Neumann Laplacian (form domain $H^1(\Omega)$) and the Harmonic oscillator (form domain $B^1(\mathbb R^n):=\{u\in L^2 (\mathbb R^n):\ x_j u\in L^2(\mathbb R^n), \partial_{x_j} u \in L^2 (\mathbb R^n)\}$).

\subsubsection*{An alternative characterization of $\lambda_2$}
$\mathfrak L_2(\Omega)$ was introduced in \eqref{eq.Lkdef}. We now introduce another spectral sequence associated with the Dirichlet Laplacian.
\begin{definition}\label{def.Lk}
For any $k\geq1$, we denote by $L_k(\Omega)$ (or $L_{k}$ if there is no confusion) the smallest eigenvalue (if any) for which there exists an eigenfunction with $k$ nodal domains. We set $L_k(\Omega) =+\infty$ if there is no eigenfunction with $k$ nodal domains.
\end{definition}
\begin{proposition}\label{L=L=L2}
  $\mathfrak L_2(\Omega) =\lambda_2(\Omega) =L_2(\Omega)$.
\end{proposition}
\begin{proof}
By definition of $L_{k}$, we have $\mathfrak L_2\leq L_{2}$.\\
The equality $\lambda_2 = L_2$ is a standard consequence of the fact that a second eigenfunction has exactly two nodal domains: the upper bound is a consequence of Courant and the lower bound is by orthogonality.\\
It remains to show that $\lambda_2 \leq \mathfrak L_2$. This is a consequence of the min-max principle. For any $\varepsilon>0$, there exists a $2$-partition $\mathcal D = \{D_1,D_2\}$ of energy $\Lambda$ such that $\Lambda<\mathfrak L_{2}+\varepsilon$. We can construct a $2$-dimensional space generated by the two ground states (extended by $0$) $u_1$ and $u_2$ of energy less than $\Lambda$. This implies: 
$$\lambda_2 \leq \Lambda < \mathfrak L_{2}+\varepsilon \,.$$
It is sufficient to take the limit $\varepsilon\to0$ to conclude.
\end{proof}

\subsection{On the local structure of nodal sets}
We refer for this section to the survey of P. B\'erard \cite{Be-1} or the book by I. Chavel \cite{Chavel}. We first mention a proposition (see \cite[Lemma 1, p. 21-23]{Chavel}) which is implicitly used in many proofs and was overlooked in \cite{CH}.
\begin{proposition}
If $u$ is an eigenfunction associated with $\lambda$ and $D$ is one of its nodal domains then the restriction of $u$ to $D$ belongs to $H_0^1(D)$ and is an eigenfunction of the Dirichlet realization of the Laplacian in $D$. Moreover $\lambda$ is the ground state energy in $D$.
\end{proposition}

\begin{proposition}\label{thm:nodinfo}
Let $f$ be a real valued eigenfunction of the Dirichlet-Laplacian on a two dimensional locally flat Riemannian manifold $\Omega$ with smooth boundary. Then $f\in C^\infty(\overline\Omega)$. Furthermore, $f$ has the following properties:
 \begin{enumerate}
\item\label{item:taylor} If $f$ has a zero of order $\ell $ at a point $\xb_0\in\overline\Omega$ then the Taylor expansion of $f$ is
\begin{equation}\label{eqn:harmpoly}
 f(\xb )=p_\ell (\xb -\xb_0)+O(|\xb -\xb_0|^{\ell +1}),
\end{equation}
where $p_\ell $ is a real valued, non-zero, harmonic, homogeneous polynomial function of degree $\ell$.\\
Moreover if $\xb_0\in\partial\Omega$, the Dirichlet boundary conditions imply that
\begin{equation}\label{eqn:sinexp}
f(\xb)=a\, r^\ell \sin \ell \omega+O(r^{\ell +1}), 
\end{equation}
for some non-zero $a\in\mathbb R $, where $(r,\omega)$ are polar coordinates of $\xb $ around $\xb_0$. The angle $\omega$ is chosen so that the tangent to the boundary at $\xb_0$ is given by the equation $\sin\omega=0$.
\item \label{item:nod}The nodal set $\mathcal N (f)$ is the union of finitely many, smoothly immersed circles in $\Omega$, and smoothly immersed lines, with possible self-intersections, which connect points of $\partial\Omega$. Each of these immersions is called a \textit{nodal line}. The connected components of $\Omega\setminus\mathcal N (f)$ are called \textit{nodal domains}.
\item \label{item:order} If $f$ has a zero of order $\ell $ at $\xb_0\in \Omega$ then exactly $\ell $ segments of nodal lines pass through $\xb_0$. The tangents to the nodal lines at $\xb_0$ dissect the disk into $2\ell $ equal angles.\\
If $f$ has a zero of order $\ell $ at $\xb_{0}\in\partial\Omega$ then exactly $\ell $ segments of nodal lines meet the boundary at $\xb_0$. The tangents to the nodal lines at $\xb_0$ are given by the equation $\sin \ell \omega=0$, where $\omega$ is chosen as in~\eqref{eqn:sinexp}.
\end{enumerate}
\end{proposition}

\begin{proof}
The proof that $f\in C^\infty(\overline\Omega)$ can be found in~\cite[Theorem 20.4]{Wlok}. The function $f$ is actually analytic in $\Omega$ (property of the Laplacian). Hence, $f$ being non identically $0$, Part~{\it\ref{item:taylor}} becomes trivial. See~\cite{Bers,Chen} for the proof of the other parts (no problem in dimension 2).
\end{proof}

\begin{remark}
\begin{itemize}[leftmargin=*]
\item In the case of Neumann condition, Proposition~\ref{thm:nodinfo} remains true if the Taylor expansion~\eqref{eqn:sinexp} for a zero of order $\ell $ at a point $\xb_0\in\partial \Omega$ is replaced by
\[ f(\xb)=a\, r^\ell \cos \ell \omega+O(r^{\ell +1}),\] 
where we have used the same polar coordinates $(r,\omega)$ centered at $\xb_0$.
\item Proposition~\ref{thm:nodinfo} remains true for polygonal domains, see \cite{Dau88} and for more general domains \cite{HHOT1} (and references therein).
\end{itemize}
\end{remark}

From the above, we should remember that nodal sets are regular in the sense:
\begin{itemize}[label=--,itemsep=-5pt]
\item The singular points on the nodal lines are isolated.
\item At the singular points, an even number of half-lines meet with equal angle.
\item At the boundary, this is the same adding the tangent line in the picture.
\end{itemize}
This will be made more precise later for more general partitions in Subsection~\ref{ss4.2}.

 \subsection{Weyl's theorem}\label{ss2.3}
If nothing else is written, we consider the Dirichlet realization of the Laplacian in a bounded regular set $\Omega \subset \mathbb R^n$,  which will be denoted by  $ H(\Omega)$. For $\lambda\in \mathbb R$, we introduce the counting function $N(\lambda)$ by:
\begin{equation}\label{defNlambda}
N(\lambda):= \sharp \{j\,:\, \lambda_j < \lambda\}\,.
\end{equation}
We write $N(\lambda,\Omega)$ if we want to recall in which open set the realization is considered.\\
Weyl's theorem (established by H. Weyl in 1911) gives the asymptotic behavior of $N(\lambda)$ as $\lambda \ar +\infty$.
\begin{theorem}[Weyl]
As $\lambda \ar +\infty$,
\begin{equation}\label{wform1}
 N(\lambda) \sim \frac{\omega_n}{(2\pi)^{n}} |\Omega| \lambda^{\frac n2},
\end{equation}
where $\omega_n$ denotes the volume of a ball of radius 1 in $\mathbb R^n$ and $|\Omega|$ the volume of $\Omega$.
\end{theorem}
In dimension $n=2$, we find:
\begin{equation}\label{wform2}
 N(\lambda) \sim \frac{|\Omega|}{4\pi} \lambda\,.
\end{equation}
\begin{proof} 
The proof of Weyl's theorem can be found in \cite{W}, \cite[p.\,42]{CH} or in \cite[p.\,30-32]{Chavel}. We sketch here the so called Dirichlet-Neumann bracketing technique, which goes roughly in the following way and is already presented in \cite{CH}.\\
The idea is to use a suitable partition $\{D_{i}\}_{i}$ to prove lower and upper bound according to $\sum_{i}N(\lambda,D_{i})$. If the domains $D_{i}$ are cubes, the eigenvalues of the Laplacian (with Dirichlet or Neumann conditions) are known explicitly and this gives explicit bounds for $N(\lambda,D_i)$ (see \eqref{Nlambdarec} for the case of the square). Let us provide details for the lower bound which is the most important for us. For any partition $\mathcal D=\{D_{i}\}_{i}$ in $\Omega$, we have
\begin{equation}\label{compbelow}
\sum_{i} N(\lambda,D_i) \leq N(\lambda, \Omega)\,.
\end{equation}
Given $\varepsilon >0$, we can find a partition $\{D_i\}_{i}$ of $\Omega$ by cubes such that $| \Omega \setminus \cup_{i} D_i| \leq \varepsilon |\Omega|$. Summing up in \eqref{compbelow}, and using Weyl's formula (lower bound) for each $D_i$, we obtain:
$$
N(\lambda, \Omega) \geq (1-\varepsilon) |\Omega| \lambda^\frac n 2 + o_\varepsilon (\lambda^\frac n2)\,.
$$
Let us deal now with the upper bound. For any partition $\mathcal D=\{D_{i}\}_{i}$ in $\mathbb R^n$ such that $\Omega \subset \cup \overline{D_i}$, we have the upper bound
\begin{equation}\label{compabove}
N(\lambda,\Omega) \leq \sum_{i} N(\lambda,-\Delta^{\sf Neu} _{D_i})\,,
\end{equation}
where $N(\lambda,\Delta^{\sf Neu} _{D_i})$ denotes the number of eigenvalues, below $\lambda$, of the Neumann realization of the Laplacian in $D_i$. Then we choose a partition with suitable cubes for which the eigenvalues are known explicitly.
\end{proof}
\begin{remark}
To improve the lower bound with more explicit remainder, we can use more explicit lower bounds for the counting function for the cube (see Subsection \ref{ss3.3} in the 2D case) and also consider cubes of size depending on $\lambda$. This will be also useful in Subsection~\ref{ss9.3}.
\end{remark}

We do not need here improved Weyl's formulas with control of the remainder (see however in the analysis of Courant sharp cases \eqref{Nlambdarec} and \eqref{Nlambdatore}).  We nevertheless mention a formula due to V. Ivrii in 1980 (cf \cite[Chapter XXIX, Theorem 29.3.3 and Corollary 29.3.4]{Hor4}) which reads:
\begin{equation}\label{eq.Weyl2terms}
N(\lambda) = \frac{\omega_n}{(2\pi)^n} |\Omega| \lambda^\frac n 2 - \frac 14 \frac{\omega_{n-1}}{(2\pi)^{n-1}} |\partial \Omega| \lambda^{\frac{n-1}{2}} + r(\lambda),
\end{equation}
where $r(\lambda)= \mathcal O (\lambda^{\frac{n-1}{2}} )$ in general but can also be shown to be $o (\lambda^{\frac{n-1}{2}} )$ under some rather generic conditions about the geodesic billiards (the measure of periodic trajectories should be zero) and $C^\infty$  boundary. This is only in this case that the second term is meaningful. \\
Formula \eqref{eq.Weyl2terms} can also be established in the case of irrational rectangles as a very special case in \cite{Iv}, but more explicitly in \cite{Ku} 
without any assumption of irrationality. This has also been extended in particular to specific triangles of interest (equilateral, right angled isosceles, hemiequilateral) by P. B\'erard %\footnote{We thank him for this information} 
(see \cite{Be83} and references therein).

\begin{remark}
\begin{enumerate}
\item The same asymptotics \eqref{wform1} is true for the Neumann realization of the Laplacian. The two-terms asymptotics \eqref{eq.Weyl2terms} also holds but with the sign $+$ before the second term.
\item For the harmonic oscillator (particular case of a Schr\"odinger operator $-\Delta + V$, with $V(\xb)\ar +\infty$ as $|\xb |\ar + \infty$) the situation is different. One can use either the fact that the spectrum is explicit or a pseudodifferential calculus. For the isotropic harmonic oscillator $-\Delta + |\xb|^2$ in $\mathbb R^n$, the formula reads
\begin{equation}\label{wform3a}
 N(\lambda)\sim \frac{\omega_{2n-1}}{(2\pi)^{n}}\, \frac{\lambda^n}{2n} \,.
\end{equation}
Note that the power of $\lambda$ appearing in  the asymptotics for the harmonic oscillator in $\mathbb R^n$ is, for a given $n$, the double of the one obtained for the Laplacian. \end{enumerate}
\end{remark}

%%%
\subsection{Courant's theorem and Courant sharp eigenvalues}
This theorem was established by R. Courant \cite{Cou} in 1923 for the Laplacian with Dirichlet or Neumann conditions.
\begin{theorem}[Courant]\label{thm.Courant}
The number of nodal components of the $k$-th eigenfunction is not greater than $k$.
\end{theorem}
 \begin{proof}
The main arguments of the proof are already present in Courant-Hilbert \cite[p.\,453-454]{CH}. Suppose that $u_k$ has $(k+1)$ nodal domains $\{D_{i}\}_{1\leq i\leq k+1}$. We also assume $\lambda_{k-1} < \lambda_k$. Considering $k$ of these nodal domains and looking at $\Phi_a:= \sum_{i=1}^k a_i \phi_i$ where $\phi_i$ is the ground state in each $D_i$, we can determine $a_i$ such that $\Phi_a$ is orthogonal to the $(k-1)$ first eigenfunctions. On the other hand $\Phi_a$ is of energy $\leq \lambda_k$. Hence it should be an eigenfunction for $\lambda_k$. But $\Phi_a$ vanishes in the open set $D_{k+1}$ in contradiction with the property of an eigenfunction which cannot be flat at a point. 
\end{proof}

\subsubsection*{On Courant's theorem with symmetry}
Suppose that there exists an isometry $g$ such that $g (\Omega) = \Omega$ and $g^2 = {\sf Id}$. Then $g$ acts naturally on $L^2(\Omega)$ by $gu (\xb) = u(g^{-1} \xb)\,,\, \forall \xb \in \Omega\,,$ and one can naturally define an orthogonal decomposition of $L^2(\Omega)$
$$ L^2(\Omega)= L^2_{\sf odd} \oplus L^2_{\sf even}\,,$$
where by definition $L^2_{\sf odd}= \{u\in L^2 \,,\, gu =-u\}$, resp. $L^2_{\sf even}= \{u\in L^2 \,,\, gu =u\}$. These two spaces are left invariant by the Laplacian and one can consider separately the spectrum of the two restrictions. Let us explain for the ``odd case'' what could be a Courant theorem with symmetry. If $u$ is an eigenfunction in $L^2_{\sf odd}$ associated with $\lambda$, we see immediately that the nodal domains appear by pairs (exchanged by $g$) and following the proof of the standard Courant theorem we see that if $\lambda = \lambda^{\sf odd}_j$ for some $j$ (that is the $j$-th eigenvalue in the odd space), then the number $\mu(u)$ of nodal domains of $u$ satisfies $\mu(u) \leq j$.\\ 
We get a similar result for the "even" case (but in this case a nodal domain $D$ is either $g$-invariant or $g(D)$ is a distinct nodal domain).\\ 
These remarks lead to improvement when each eigenspace has a specific symmetry. This will be the case for the sphere, the harmonic oscillator, the square (see \eqref{antisym}), the Aharonov-Bohm operator, \ldots where $g$ can be the antipodal map, the map $(x,y)\mapsto (-x,-y)$, the map $(x,y)\mapsto (\pi-x,\pi -y)$, the deck map (as in Subsection \ref{ss8.5}), \dots

\begin{definition}
We say that $(u,\lambda)$ is a spectral pair for $H(\Omega)$ if $\lambda$ is an eigenvalue of the Dirichlet-Laplacian $H(\Omega)$ on $\Omega$ and $u\in E(\lambda)\setminus \{0\}$, where $E(\lambda)$ denotes the eigenspace attached to $\lambda$.
\end{definition}
\begin{definition}\label{DefCourantsharp}
We say that a spectral pair $(u,\lambda)$ is Courant sharp if $\lambda=\lambda_k$ and $u$ has $k$ nodal domains. We say that an eigenvalue $\lambda_k$ is Courant sharp if there exists an eigenfunction $u$ associated with $\lambda_{k}$ such that $(u,\lambda_k)$ is a Courant sharp spectral pair.
\end{definition}
If the Sturm-Liouville theory shows that in dimension $1$ all the spectral pairs are Courant sharp, we will see below that when the dimension is $\geq 2$, the Courant sharp situation can only occur for a finite number of eigenvalues.\\

The following property of transmission of the Courant sharp property to sub-partitions will be useful in the context of minimal partitions. Its proof can be found in \cite{AHH}.
\begin{proposition}\label{submu}\ 
\begin{enumerate}%[i)]
\item Let $(u,\lambda)$ be a {\bf Courant sharp} spectral pair for $H(\Omega)$ with $\lambda=\lambda_k$ and $\mu(u)=k$. Let $\mathcal D^{(k)}=\{D_i\}_{1\leq i\leq k}$ be the family of the nodal domains associated with $u$. Let $L$ be a subset of $\{1,\ldots,k\}$ with $\sharp L =\ell$ and let $\mathcal D_{L}$ be the subfamily $\{D_i\}_{i\in L}$. Let $\Omega_L=\Inte (\overline{\cup_{i\in L} D_i})\setminus\partial\Omega$. Then
\begin{equation}\label{lk=ll}
\lambda_\ell(\Omega_L)=\lambda_k,
\end{equation}
where $\{\lambda_j(\Omega_L)\}_{j}$ are the eigenvalues of $H(\Omega_L)$.
\item Moreover, when $\Omega_L$ is connected, $u\big |_{\Omega_L}$ is {\bf Courant sharp} and $\lambda_\ell(\Omega_L)$ is simple.
\end{enumerate}
\end{proposition}

 \subsection{Pleijel's theorem}\label{ss2.5}
Motivated by Courant's Theorem, Pleijel's theorem (1956) says
\begin{theorem}[Weak Pleijel's theorem]\label{Pleijelweakform}
If the dimension is $\geq 2$, there is only a finite number of Courant sharp eigenvalues of the Dirichlet Laplacian.
\end{theorem}
This theorem is the consequence of a more precise theorem which gives a link between Pleijel's theorem and partitions. For describing this result and its proof, we first recall the Faber-Krahn inequality:
\begin{theorem}[Faber-Krahn inequality]
For any domain $D\subset \mathbb R^2$, we have 
\begin{equation}\label{eq.FK}
|D|\ \lambda(D) \geq \lambda(\Circle )\,,
\end{equation}
 where $|D|$ denotes the area of $D$ and $\Circle$ is the disk of unit area ${\mathcal B}\Big(0, \frac{1}{\sqrt{\pi}}\Big)\,. $
\end{theorem}
\begin{remark}\label{rem.FK}
Note that improvements can be useful when $D$ is "far" from a disk. It is then interesting to have a lower bound for $|D|\ \lambda(D) - \lambda(\Circle )$. We refer for example to \cite{BDPV} and \cite{HaNa}. These ideas are behind recent improvements by Steinerberger \cite{St}, Bourgain \cite{Bo} and Donnelly \cite{Do14} of the strong Pleijel's theorem below. See also Subsection \ref{ss9.1}.
\end{remark}

By summation of Faber-Krahn's inequalities \eqref{eq.FK} applied to each $D_i$ and having in mind Definition~\ref{regOma}, we deduce:
\begin{lemma}\label{lem.FK}
For any open partition $\mathcal D$ in $\Omega$ we have
\begin{equation}\label{fkv2}
 |\Omega|\ \Lambda(\mathcal D) \geq \sharp (\mathcal D) \, \lambda( \Circle )\,,
\end{equation}
where $\sharp (\mathcal D)$ denotes the number of elements of the partition.
\end{lemma}
Note that instead of using summation, we can prove the previous lemma by using the fact that there exists some $D_i$ with $|D_i| \leq \frac{|\Omega|}{k}$ and apply Faber-Krahn's inequality for this $D_i$. There is no gain in our context, but in other contexts (see for example Proposition \ref{prop3.10}), we could have a Faber-Krahn's inequality with constraint on the area, which becomes satisfied for $k$ large enough (see \cite{BeMe}).\\
Let us now give the strong form of Pleijel's theorem.
\begin{theorem}[Strong Pleijel's theorem]\label{PropPleijel1}
Let $\phi_{n}$ be an eigenfunction of $H(\Omega)$ associated with $\lambda_{n}(\Omega)$. Then 
\begin{equation}\label{bourgain3} 
\limsup_{n\ar +\infty} \frac{\mu(\phi_n)}{n}\leq \frac{4 \pi }{ \lambda(\Circle)}\,,%=\left(\frac{2}{{\bf j}} \right)^2 \,,
\end{equation}
where $\mu(\phi_n)$ is the cardinal of the nodal components of $\Omega \setminus N(\phi_n)$.
\end{theorem}
\begin{remark}
Of course, this implies Theorem~\ref{Pleijelweakform}. We have indeed 
$$\lambda(\Circle) =\pi {\bf j}^2\,,$$ 
and ${\bf j}\simeq 2.40$ is the smallest positive zero of the Bessel function of first kind. Hence
$$
\frac{4 \pi }{ \lambda(\Circle)} = \left(\frac{2}{{\bf j}} \right)^2 < 1\,.
$$
\end{remark}
\begin{proof}
We start from the following identity
\begin{equation}\label{idpl}
\frac{\mu(\phi_n)}{n}\, \frac{n}{\lambda_n} \, \frac{\lambda_n}{\mu(\phi_n)} =1\,.
\end{equation}
Applying Lemma~\ref{lem.FK} to the nodal partition of $\phi_n$ (which is associated with $\lambda_{n}$), we have
$$\frac{\lambda_n}{\mu(\phi_n)}\geq \frac{ \lambda(\Circle)}{|\Omega|}.$$
Let us take a subsequence $\phi_{n_i}$ such that $\lim_{i\rightarrow +\infty}\frac{\mu(\phi_{n_i})}{n_i} = \limsup_{n\ar +\infty} \frac{\mu(\phi_n)}{n}\,,$ and implementing in \eqref{idpl}, we deduce:
\begin{equation}\label{ineqpl1}
 \frac{ \lambda(\Circle)}{|\Omega|}\, \limsup_{n\ar +\infty} \frac{\mu(\phi_n)}{n}\, \liminf_{\lambda \ar +\infty} \frac{N(\lambda)}{\lambda} \leq 1.
\end{equation}
Having in mind Weyl's formula \eqref{wform2}, we get \eqref{fkv2}.
\end{proof}

To finish this section, let us mention the particular case of irrational rectangles (see \cite{BGS} and \cite{Pol}). 
\begin{proposition}\label{Prop2.21}
Let us denote by $\mathcal R(a,b)$ the rectangle $(0,a\pi)\times(0,b\pi)$, with $a>0$ and $b>0$. We assume that $b^2/a^2$ is irrational. Then Theorem~\ref{PropPleijel1} is true for the rectangle $\mathcal R(a,b)$ with constant ${4 \pi }/{ \lambda(\Circle)}$ replaced by $4\pi/\lambda(\Square)=2/\pi,$ where $\Square$ is a square of area 1. Moreover we have
\begin{equation}\label{bourgain3Rec} 
\limsup_{n\ar +\infty} \frac{\mu(\phi_n)}{n} =  \frac2 \pi.
\end{equation}
\end{proposition}
\begin{proof}
Since $b^2/a^2$ is irrational, the eigenvalues $\hat \lambda_ {m,n}$ are simple and eigenpairs are given, for $m\geq1, n\geq 1$, by
\begin{equation} \label{defphimn}
\hat \lambda_ {m,n}=\frac{m^2}{a^2} +\frac{n^2}{b^2}\,,\qquad
\phi_{m,n} (x,y)= \sin \frac{mx}{a} \,\sin \frac{ny}{b}\,.
\end{equation}
Without restriction we can assume $a=1$. Thus we have $\mu(\phi_{m,n})=mn\,$. Applying Weyl asymptotics \eqref{wform2} with $\lambda=\hat \lambda_{m,n}\,$ gives
\begin{equation}\label{Weyl}
k(m,n):= \sharp\{(\tilde m,\tilde n):\hat \lambda_ {\tilde m,\tilde n}(b)<\lambda\} =\frac{b\pi}{4}\left(m^2+\frac{n^2}{b^2}\right)+o(\lambda)\,.
\end{equation}
We have $\lambda_{k(m,n) +1} = \hat \lambda_{m,n}\,.$ We observe that $\mu(\phi_{n,m})/ k(n,m)$ is asymptotically given by
\begin{equation}\label{Pmnb}
P(m,n;b):=\frac{4mn}{\pi\left(m^2b+\frac{n^2}{b}\right)}\le \frac{2}{\pi}\,.
\end{equation}
Taking a sequence $(m_k,n_k)$ such that $b=\lim_{k\rightarrow \infty}\frac{n_k}{m_k}$ with $m_k\ar +\infty\,$, we deduce
\begin{equation}\label{Paj}
\lim_{k\ar +\infty} P(m_k,n_k;b)=\frac{2}{\pi}\,,
\end{equation}
which gives the proposition by using this sequence of eigenfunctions $\phi_{m_k,n_k}$.
\end{proof}
\begin{remark} 
There is no hope in general to have a positive lower bound for $\liminf \mu(\phi_{n}) /n$. A. Stern for the square and the sphere (1925), H. Lewy for the sphere (1977), J. Leydold for the harmonic oscillator \cite{Ley1} (see \cite{BeHe1, BeHe2, BeHe3} for the references, analysis of the proofs and new results) have constructed infinite sequences of eigenvalues such that a corresponding eigenfunctions have two or three nodal domains. On the contrary, it is conjectured in \cite{HPS} that for the Neumann problem in the square this $\liminf$ should be strictly positive.\\
Coming back to the previous proof of Proposition \ref{Prop2.21}, one immediately sees that 
$$\lim_{m\ar + \infty} P(m,1;b) =0\,.$$
\end{remark}
\begin{remark}\label{conjPol}
Inspired by computations of \cite{BGS}, it has been conjectured by Polterovich \cite{Pol} that the constant $2/\pi = {4\pi}/{\lambda(\square)}$ is optimal for the validity of a strong Pleijel's theorem with a constant independent of the domain (see the discussion in \cite{HH7}). A less ambitious conjecture is that Pleijel's theorem holds with the constant ${4\pi}/{\lambda(\hexagon)}$, where $\hexagon$ is the regular hexagon of area $1$. This is directly related to the hexagonal conjecture which will be discussed in Section \ref{sec.klarge}.
\end{remark}

\subsection{Notes}\label{ss2.6}
 Pleijel's Theorem extends to bounded domains in $\mathbb{R}^n$, and more generally to compact $n$-manifolds with boundary, with a constant $\gamma(n) <1$ replacing $ {4 \pi} / {\lambda(\Circle)}$ %$ 4/{\bf j} ^2$ 
in the right-hand side of \eqref{bourgain3} (see Peetre \cite{Pe}, B\'{e}rard-Meyer \cite{BeMe}). It is also interesting to note that this constant is independent of the geometry. It is also true for the Neumann Laplacian in a piecewise analytic bounded domain in $\mathbb R^2$ (see \cite{Pol} whose proof is based on a control of the asymptotics of the number of boundary points belonging to the nodal sets of the eigenvalue $\lambda_k$ as $k\ar +\infty$, a difficult result proved by Toth-Zelditch \cite{ToZe}).

%%%%%%%%%%%%%%%%%%%%%%%%%
 \section{Courant sharp cases: examples}\label{s3}
This section is devoted to determine the Courant sharp situation for some examples. Outside its interest in itself, this will be also motivated by the fact that it gives us examples of minimal partitions. This kind of analysis seems to have been initiated by \AA. Pleijel.
First, we recall that according to Theorem~\ref{Pleijelweakform}, there is a finite number of Courant sharp eigenvalues. We will try to quantify this number or to find at least lower bounds or upper bounds for the largest integer $n$ such that $\lambda_{n-1} < \lambda_n$ with $\lambda_n$ Courant sharp.

\subsection{Thin domains}\label{ss3.1}
This subsection is devoted to thin domains for which L\'ena proves in \cite{LenTh} that, under some geometrical assumption, any eigenpair is Courant sharp as soon as the domain is thin enough.\\
Let us fix the framework.
Let $a>0$, $b>0$ and $h\in C^\infty((-a,b),\mathbb R^+)$. We assume that $h$ has a unique maximum at $0$ which is non degenerate. For $\varepsilon>0$, we introduce
\begin{equation*}
 \Omega_{\varepsilon}=\left\{ (x_1,x_2)\in \mathbb R^2 \,,\, -a<x_1<b \mbox{ and } -\varepsilon h(x_1)<x_2<\varepsilon h(x_1) \right\}.
\end{equation*}

\begin{theorem}
\label{chap1.thmf.Nodal}
 For any $k\ge1$, there exists $\varepsilon_k>0$ such that, if $0<\varepsilon\le \varepsilon_k\,$, 
 the first $k$ Dirichlet eigenvalues $\{\lambda_{j}(\Omega_\varepsilon),1\leq j\leq k\}$ are simple and Courant sharp. 
 %the eigenvalues $\lambda_{\ell}(\varepsilon)$ for $\ell=1,\dots,k$ have multiplicity $1$ and the associated eigenfunctions are Courant sharp.
\end{theorem}
 \begin{proof}
The asymptotic behavior of the eigenvalues for a domain whose width is proportional to $\varepsilon$, as $\varepsilon\to0$, was established by L. Friedlander and M.~Solomyak \cite{FriSol09} and the first terms of the expansion are given. An expansion at any order was proved by D. Borisov and P. Freitas for planar domains in \cite{BorFre09}. The proof of L\'ena is based on a semi-classical approximation of the eigenpairs of the Schr\"odinger operator. Then he established some elliptic estimates with a control according to $\varepsilon$ and applies some Sobolev imbeddings to prove the uniform convergence of the quasimodes and their derivative functions. The proof is achieved by adapting some arguments of \cite{FreKre08} to localize the nodal sets. 
\end{proof}

 \begin{remark}\label{rem.recttore}
 \begin{itemize}
\item The rectangle $\Omega_{\varepsilon}=(0,\pi)\times(0, \varepsilon \pi)$ does not fulfill the assumptions of Theorem~\ref{chap1.thmf.Nodal}. Nevertheless (see Subsection \ref{ss3.2}), we have: \\
For any $k\geq1$, there exists $\varepsilon_{k}>0$ such that $\lambda_{j}(\Omega_{\varepsilon})$ is Courant sharp for $1\leq j\leq k$ and $ 0<\varepsilon\leq\varepsilon_{k}$. Furthermore, when $0<\varepsilon<\varepsilon_{k}$,  the nodal partition of the corresponding eigenfunction $u_j$ consists of $j$ similar vertical strips.
\item In the case of the flat torus $\mathbb T(a,b)=(\mathbb R/a\mathbb Z)\times (\mathbb R/b\mathbb Z)$, with $0 < b \leq a$, the first eigenvalue $\lambda_{1}$ is always Courant sharp. If $b \leq 2/k$, the eigenvalues $\lambda_{2j}$ for $1\leq j \leq k$ are Courant sharp and if $b< 2/k$, the nodal partition of any corresponding eigenfunction consists of  $2 j$ similar strips (see Subsection~\ref{ss3.5}).
\end{itemize}
\end{remark}

\subsection{Irrational rectangles}\label{ss3.2}
The detailed analysis of the spectrum of the Dirichlet Laplacian in a rectangle is the example treated as the toy model in \cite{Pl}. Let $\mathcal R(a,b)=(0,a\pi)\times (0,b\pi)$. We recall \eqref{defphimn}. If it is possible to determine the Courant sharp cases when $b^2/a^2$ is irrational (see for example \cite{HHOT1}), it can become very difficult in general situation. If we assume that $b^2/a^2$ is irrational, all the eigenvalues have multiplicity $1$. For a given $\hat \lambda_{m,n}$, we know that the corresponding eigenfunction has $mn$ nodal domains. As a result of a case by case analysis combined with Proposition \ref{submu}, we obtain the following characterization of the Courant sharp cases:
\begin{theorem}\label{th.rect} 
Let $a > b$ and $\frac{b^2}{a^2} \not\in \mathbb Q$. Then the only cases when $\hat \lambda_{m,n}$ is a Courant sharp eigenvalue are the following:
\begin{enumerate}
\item $(m,n) =(2,3)$ if $\frac 8 5 < \frac{a^2}{b^2} < \frac 53$\,;
\item $(m,n)=(2,2)$ if $ 1 < \frac{a^2}{b^2} < \frac 53$\,;
\item\label{item3thRect} $(1,n)$ if $\frac{n^2-1}{3} < \frac{a^2}{b^2}$\,.
\end{enumerate}
\end{theorem}
\begin{remark}
Note that in Case \ref{item3thRect} of Theorem~\ref{th.rect}, we can remove the assumption that $\frac{b^2}{a^2} \not\in \mathbb Q$.
\end{remark}

\subsection{Pleijel's reduction argument for the rectangle}\label{ss3.3}
The analysis of this subsection is independent of the arithmetic properties of $b^2/a^2$. Following (and improving) a remark in a course of R. Laugesen \cite{Lau}, one has a lower bound of $N(\lambda)$ in the case of the rectangle $\mathcal R=\mathcal R(a,b):=(0,a\pi) \times (0,b\pi)$, which can be expressed in terms of area and perimeter. One can indeed observe that the area of the intersection of the ellipse $\{ \frac{(x+1)^2}{a^2} + \frac{(y+1)^2}{b^2} <\lambda\}$ with $\mathbb R^+\times \mathbb R^+$ is a lower bound for $N(\lambda)$.\\
The formula (to compare with the two terms asymptotics \eqref{eq.Weyl2terms}) reads for $\lambda \geq \frac{1}{a^2}+\frac{1}{b^2}$:
\begin{equation}\label{Nlambdarec}
N(\lambda) > \frac{1}{4\pi} |\mathcal R| \lambda - \frac{1}{2\pi}|\partial\mathcal R| \sqrt{\lambda} +1\,.
\end{equation}
We get, in the situation $\lambda_{n-1} < \lambda_n\,$,
\begin{equation}\label{eq.rect2}
n > \frac{\pi ab}{4} \lambda_n - (a+b) \sqrt{\lambda_n} +2 \,.
\end{equation}
On the other hand, if $\lambda_n$ is Courant sharp, Lemma~\ref{lem.FK} gives the necessary condition
\begin{equation}\label{i2} 
\frac{n}{\lambda_n} \leq \pi \frac{ab}{{\bf j}^2}.
\end{equation}
Then, combining this last relation with \eqref{eq.rect2}, we get the inequality $\frac{a+b}{\sqrt{\lambda_n}} > \pi ab \left(\frac{1}{4} - \frac{1}{{\bf j}^2}\right)\,,$ and finally a Courant sharp eigenvalue $\lambda_n$ should satisfy:
\begin{equation}
 \lambda_n < \frac{1}{\pi^2}\left(\frac{1}{4} - \frac{1}{{\bf j}^2}\right)^{-2} \left( \frac{a+b}{ab} \right)^2\,,
\end{equation}
with
$$
\frac{1}{\pi^2}\left(\frac{1}{4} - \frac{1}{{\bf j}^2}\right)^{-2} \simeq 17.36 \,.
$$
Hence, using the expression \eqref{defphimn} of the eigenvalues, we have just to look at the pairs $\ell, m \in \mathbb N^*$ such that
$$
\frac{\ell^2}{a^2} + \frac{m^2}{b^2} < \frac{1}{\pi^2}\left(\frac{1}{4} - \frac{1}{{\bf j}}\right)^{-2} \left( \frac{a+b}{ab} \right)^2 \,.
$$
Suppose that $a \geq b$. We can then normalize by taking $a=1$. We get the condition:
$$
 \ell^2 + \frac{1}{b^2} m^2 
 < \frac{1}{\pi^2}\left(\frac{1}{4} - \frac{1}{{\bf j}^2}\right)^{-2} \left( \frac{1+b}{b} \right)^2 
 \simeq 17.36 \left(\frac{1+b}{b}\right)^2 \leq 69.44 \,.
$$
This is compatible with the observation (see Subsections \ref{ss3.1} and \ref{ss3.2}) that when $b$ is small, the number of Courant sharp cases will increase. In any case, when $a=1$, this number is $\geq [\frac 1 b]$ and using \eqref{i2}, 
$$n\leq \lambda_{n} \pi\frac{ab}{{\bf j}^2} 
\leq\frac{1}{\pi}\left(\frac{\bf j}{4} - \frac{1}{{\bf j}}\right)^{-2} \frac{(1+b)^2}b
\leq 9.27 \frac{(1+b)^2}b.$$
In the next subsection we continue with a complete analysis of the square.

\subsection{The square}\label{ss3.4}
We now take $a=b=1$ and describe the Courant sharp cases.
\begin{theorem}\label{thm.partnodalcarre}
In the case of the square, the Dirichlet eigenvalue $\lambda_k$ is Courant sharp if and only if $ k=1,2,4$.
\end{theorem}
\begin{remark} 
This result was obtained by Pleijel \cite{Pl} who was nevertheless sketchy (see \cite{BeHe1}) in his analysis of the eigenfunctions in the $k$-th eigenspace for $k=5,7,9$ for which he only refers to pictures in Courant-Hilbert \cite{CH}, actually reproduced from Pockel \cite{Poc}. Details can be found in \cite{BeHe1} or below.
\end{remark}
\begin{proof} 
From the previous subsection, we know that it is enough to look at the eigenvalues which are less than 69 (actually $68$ because $69$ is not an eigenvalue). Looking at the necessary condition \eqref{i2} eliminates most of the candidates associated with the remaining eigenvalues and we are left after computation with the analysis of the three cases $ k=5,7,9$. These three eigenvalues correspond respectively to the pairs $ (m,n)= (1,3)$, $(m,n)= (2,3)$ and $(m,n)=(1,4)$ and have multiplicity $2$. Due to multiplicities, we have (at least) to consider the family of eigenfunctions $ (x,y) \mapsto \Phi_{m,n}(x,y,\theta)$ defined by
\begin{equation} \label{eq.def.Phimnt}
 (x,y) \mapsto \Phi_{m,n}(x,y,\theta):= \cos \theta\, \phi_{m,n}(x,y) + \sin \theta\, \phi_{n,m}(x,y)\,,
\end{equation}
for $ m, n \ge 1$, and $\theta \in [0,\pi)$. \\
Let us analyze each of the three cases $k=5,7,9$. For $\lambda_7$ ($(m,n)= (2,3)$) and $\lambda_9$ ($(m,n)= (1,4)$), we can use some antisymmetry argument. We observe that
\begin{equation} \label{antisym}
\phi_{m,n}(\pi-x,\pi -y,\theta) = (-1)^{m+n}\phi (x,y,\theta)\,.
\end{equation}
Hence, when $ m+n$ is odd, any eigenfunction corresponding to $m^2 +n^2$ has necessarily an even number of nodal domains. Hence $\lambda_7$ and $\lambda_9$ cannot be Courant sharp.\\
For the remaining eigenvalue $\lambda_5$ ($(m,n)=(1,3)$), we look at the zeroes of $\Phi_{1,3}(x,y,\theta)$ and consider the $C^\infty$ change of variables $\cos x = u\,,\, \cos y=v\,$, which sends the square $(0,\pi)\times (0,\pi)$ onto $(-1,1)\times (-1,1)$. In these coordinates, the zero set of $\Phi_{1,3}(x,y,\theta)$ inside the square is given by:
$\cos \theta \, ( 4v^2 - 1) + \sin \theta\, (4u^2 - 1) =0\,$.\\
Except the two easy cases when $\cos \theta =0$ or $\sin \theta =0$, which can be analyzed directly (product situation), we immediately get that the only possible singular point is $ (u,v) =(0,0)$, i.e. $(x,y) = ( \frac \pi 2,\frac \pi 2)$, and that this can only occur for $\cos \theta + \sin \theta =0$, i.e. for $\theta = \frac \pi 4$.\\
We can then conclude that the number of nodal domains is $2$, $3$ or $4$. \\
This achieves the analysis of the Courant sharp cases for the square of the Dirichlet-Laplacian.
\end{proof}

\begin{remark}
For an eigenvalue $\lambda$, let $\hat\mu_{max}(\lambda) = \max_{u\in E(\lambda)}\mu(u)$. For a given eigenvalue $\hat \lambda_ {m,n}$ of the square with multiplicity $\geq 2$, a natural question is to determine if 
$$\widehat \mu_{\rm max}(\hat \lambda_{m,n}) = \mu_{max}(m,n)\quad \mbox{ with }\quad
\mu_{\rm max}(m,n)=\sup\{ m_jn_j \,:\,m_j^2 +n_j^2=m^2+n^2\}\,.$$
The problem is not easy because one has to consider, in the case of degenerate eigenvalues, linear combinations of the canonical eigenfunctions associated with the $\hat \lambda_ {m,n}\,$. Actually, as stated above, the answer is negative. As observed by Pleijel \cite{Pl}, the eigenfunction $\Phi_{1,3, \frac{3\pi}{4}}$ defined in \eqref{eq.def.Phimnt} corresponds to the fifth eigenvalue and has four nodal domains delimited by the two diagonals, and $\mu_{max}(1,3)=3$. One could think that this guess could hold for large enough eigenvalues but $u_k:= \Phi_{1,3, \frac{3\pi}{4}}( 2^k x, 2^k y)$ is an eigenfunction associated with the eigenvalue $\lambda_{n(k)}= \hat \lambda_{2^k,3\cdot 2^k}=10 \cdot 4^k$ with $4^{k+1}$ nodal domains. Using Weyl's asymptotics, we get that the corresponding quotient $\frac{\mu(u_k)}{n(k)}$ is asymptotic to $\frac{8}{5 \pi}$. This does not contradict the Polterovich conjecture (see Remark \ref{conjPol}). 
\end{remark}

\subsection{Flat tori}\label{ss3.5}
Let $\mathbb T(a,b)$ be the torus $(\mathbb R/a\mathbb Z)\times (\mathbb R/b\mathbb Z)$, with $0 < b \leq a$. Then the eigenvalues of the Laplace-Beltrami operator on $\mathbb T(a,b)$ are 
\begin{equation}
\lambda_{m,n}(a,b)=4\pi^2\left(\frac{m^2}{a^2}+\frac{n^2}{b^2}\right),\qquad\mbox{ with }m\geq0,\ n\geq0\,,
\end{equation}
and a basis of eigenfunctions is given by $\sin mx \sin ny\,$, $\cos mx \sin ny\,$, $\sin mx \cos ny\,$, $\cos mx \cos ny$, where we should eliminate the identically zero functions when $mn=0$. The multiplicity can be $1$ (when $m=n=0$), $2$ when $m=n$ (and no other pair gives the same eigenvalue), $4$ for $m\neq n$ if no other pair gives the same eigenvalue, which can occur when $b^2/a^2\in \mathbb Q$. Hence the multiplicity can be much higher than in the Dirichlet case.

\subsubsection*{Irrational tori}\label{sss3.5.1}
\begin{theorem}\label{theorem3.8}
Suppose $b^2/a^2$ be irrational. If $\min (m,n) \geq 1$, then the eigenvalue $\lambda_{m,n}(a,b)$ is not Courant sharp.
\end{theorem}
\begin{proof}
The proof given in \cite{HH6} is based on two properties. The first one is to observe that if $\lambda_{m,n}(a,b) =\lambda_{k(m,n)}$ then $k(m,n) \geq 4 mn + 2m + 2n-2\,. $\\
The second one is to prove (which needs some work) that for $m,n >0$ any eigenvalue in $E(\lambda_{m,n}(a,b))$ has either $4 mn$ nodal domains or $2 D(m,n)$ nodal domains where $D(m,n)$ is the greatest common denominator of $m$ and $n$.
\end{proof}
Hence we are reduced to the analysis of the case when $mn=0$.\\
As mentioned in Remark~\ref{rem.recttore}, it is easy to see that, independently of the rationality or irationality of $\frac{b^2}{a^2}$, for $b \leq \frac{2}{k}$, the eigenvalues $\lambda_1=0$, and $\lambda_{2\ell}$ for $1\leq \ell \leq k$ are Courant sharp.

\subsubsection*{The isotropic torus}%\label{sec.isotore}}
In this case we can completely determine the cases where the eigenvalues are Courant sharp. The first eigenvalue has multiplicity $1$ and the second eigenvalue has multiplicity $4$. By the general theory we know that this is Courant sharp. C. L\'ena \cite{Len3} has proven:
\begin{theorem} %[L\'ena]
The only Courant sharp eigenvalues for the Laplacian on $\mathbb T^2:= (\mathbb R/\mathbb Z)^2$ are the first and the second ones.
\end{theorem}
\begin{proof}
The proof is based on a version of the Faber-Krahn inequality for the torus which reads:
\begin{proposition}\label{prop3.10}
If $\Omega$ is an open set in $\mathbb T^2$ of area $\leq \frac 1 \pi$, then the standard Faber-Krahn inequality is true.
\end{proposition}
Combined with an explicit lower bound for the Weyl law, one gets
\begin{equation}\label{Nlambdatore}
N(\lambda) \geq \frac{1}{4\pi} \lambda - \frac{2}{\pi} \sqrt{\lambda} -3\,.
\end{equation}
One can then proceed in a similar way as for the rectangle case with the advantage here that the only remaining cases correspond to the first and second eigenvalues.
\end{proof}

\subsection{The disk}\label{ss3.6}
Although the spectrum is explicitly computable, we are mainly interested in the ordering of the eigenvalues corresponding to different angular momenta. Consider the Dirichlet realization in the unit disk ${\mathcal B}(0,1)\subset\mathbb R^2$ (where ${\mathcal B}(0,r)$ denotes the disk of radius $r$). We have in polar coordinates:
$-\Delta=- \frac{\partial^2}{\partial r^2} -\frac 1r \frac {\partial}{\partial r} - \frac{1}{r^2} \frac{\partial^2}{\partial \theta^2}\,.$\\
The Dirichlet boundary conditions require that any eigenfunction $u$ satisfies $u(1,\theta)=0$, for $\theta\in[0,2\pi)$. We analyze for any $\ell\in \mathbb N$ the eigenvalues $\tilde \lambda_{\ell, j}$ of 
$$ 
\Big(- \frac{d^2}{d r^2} - \frac 1r \frac {d}{dr} + \frac{\ell^2}{r^2}\Big)f_{\ell,j}=\tilde \lambda_{\ell,j}f_{\ell,j}\;,\; \mbox{ in } (0,1)\;.
$$
The operator is self adjoint for the scalar product in $ L^2((0,1),r\,dr)$. The corresponding eigenfunctions of the eigenvalue problem take the form 
\begin{equation}\label{diskf}
u (r,\theta)= c\, f_{\ell,j} (r) \cos (\ell \theta +\theta_0),\qquad\mbox{ with } c \neq 0\,,\, \theta_0\in \mathbb R \;,
\end{equation} 
where the $ f_{\ell,j}$ are suitable Bessel functions. For the corresponding $\tilde \lambda_{\ell, j}$'s, we find the following ordering
\begin{equation}\label{ladisk}
\begin{array}{ll}
\lambda_1=\tilde \lambda_{0,1}&<\lambda_2=\lambda_3=\tilde \lambda_{1,1}
<\lambda_4=\lambda_5 =\tilde \lambda_{2,1}<\lambda_6=\tilde \lambda_{0,2}\\ 
&< \lambda_7 = \lambda_8 = \tilde \lambda_{3,1} < \lambda_9=\tilde \lambda_{10} =\tilde \lambda_{1,2} 
< \lambda_{11}=\lambda_{12} = \tilde \lambda_{4,1} < \dots 
\end{array}
\end{equation}
We recall that the zeros of the Bessel functions are related to the eigenvalues by the relation:
\begin{equation}\label{zeroeig} 
\tilde \lambda_{\ell,k} = \,( {\bf j}_{\ell,k})^2\;.
\end{equation}
Moreover all the ${\bf j}_{\ell,k}$ are distinct (see Watson \cite{Wa}). This comes back to deep results by C.L. Siegel  \cite{Siegel66} proving  in 1929 a conjecture of J. Bourget (1866). The multiplicity is either $ 1$ (if $\ell=0$) or $2$ if $\ell >0$ and we have
$$
\mu(u_1)=1;\qquad 
\mu(u)=2 \,,\, \forall u \in E(\lambda_2)\,; \qquad
\mu(u)=4\,,\,\forall u \in E( \lambda_4)\,; \qquad 
\mu(u_6)=2\,, \,\cdots
$$
Hence $\lambda_1$, $\lambda_2$ and $\lambda_4$ are Courant sharp and   it is proven in \cite{HHOT1} that
we have finally:
\begin{proposition}\label{prop.nodalDisk}
Except the cases $k=1$, $2$ and $4$, the eigenvalue $\lambda_{k}$ of the Dirichlet Laplacian in the disk is never Courant sharp. 
\end{proposition}
Notice that the Neumann case can also be treated (see \cite{HPS1}) and that the result is the same. As observed in \cite{Ashu}
 Siegel's theorem also  holds for the zeroes of the derivative of the above Bessel functions \cite{Shid89,Siegel66}.
 %{\clr [10] A. B. Shidlovskii, Transcendental numbers, vol. 12 of de Gruyter Studies
%in Mathematics, Walter de Gruyter $\&$ Co., Berlin, 1989. Translated from the
%Russian by Neal Koblitz, With a foreword by W. Dale Brownawell.
%[11] C. L. Siegel, \"Uber einige Anwendungen diophantischer Approximationen,
%Abh. Preuss. Akad. Wiss., Phys.-math. Kl., (1929). Reprinted in Gesammelte
%Abhandlungen I, Berlin-Heidelberg-New York: Springer-Verlag, 1966.
% }
\subsection{Angular sectors}\label{ss3.7}
Let $\Sigma_{\omega}$ be an angular sector of opening $\omega$. We are interested in finding the Courant sharp eigenvalues in function of the opening $\omega$. The eigenvalues $(\check \lambda_{m,n}(\omega),u_{m,n}^\omega)$ of the Dirichlet Laplacian on the angular sector $\ensuremath{\Sigma_{\omega}}$ are given by
\begin{equation*}
 \check{\lambda}_{m,n}(\omega)={\bf j}_{m\frac\pi\omega,n}^2
\qquad\mbox{ and }\qquad 
u_{m,n}^{\omega}(\rho,\theta)=J_{m\frac\pi\omega}(\jb_{m\frac\pi\omega,n}\, \rho)\sin(m\pi(\tfrac\theta\omega+\tfrac12)),
\end{equation*}
where $\jb_{m\frac\pi\omega,n}$ is the $n$-th positive zero of the Bessel function of the first kind $J_{m\frac\pi\omega}$. The Courant sharp situation is analyzed in \cite{BL} and can be summed up in the following proposition:
\begin{proposition}\label{prop3.12}
Let us define
\begin{align*}
\omega_{k}^1=\inf\{\omega\in (0,2\pi]:\check{\lambda}_{1,k}(\omega) \ge \check{\lambda}_{2,1}(\omega)\},\qquad& \forall k\geq 2,\\[5pt]
\omega_{k}^2 = \inf\{\omega\in (0,2\pi]: \check{\lambda}_{k,1}(\omega)<\check{\lambda}_{1,2}(\omega)\},\qquad& \mbox{for }2\leq k\leq 5.
\end{align*}
If $2\leq k\leq5$, the eigenvalue $\lambda_{k}$ is Courant sharp if and only if 
$\omega\in ]0,\omega_{k}^1]\cup[\omega_{k}^2,2\pi].$\\
If $k\geq6$, the eigenvalue $\lambda_{k}$ is Courant sharp for $\omega\leq \omega_{k}^1$.
\end{proposition}

\subsection{Notes}\label{ss3.8}
Some other cases have been analyzed:
\begin{itemize}[label=--,leftmargin=*,itemsep=0pt]
\item the square for the Neumann-Laplacian, by Helffer--Persson-Sundqvist \cite{HPS}, 
\item the annulus for the Neumann-Laplacian, by Helffer--Hoffmann-Ostenhof \cite{HH5}, 
\item the sphere by Leydold \cite{LeyTh,Ley3} and Helffer-Hoffmann-Ostenhof--Terracini \cite{HHOT2}, 
\item the irrational torus by Helffer--Hoffmann-Ostenhof \cite{HH6}, 
\item the equilateral torus, the equilateral, hemi-equilateral and right angled isosceles triangles by B\'erard-Helffer \cite{BeHe4}, 
\item the isotropic harmonic oscillator by Leydold \cite{Ley1}, B\'erard-Helffer \cite{BeHe3} and Charron \cite{Char} . 
\end{itemize}
Except for the cube \cite{HeKi}, similar questions in dimension $>2$ have not been considered till now (see however \cite{Char} for Pleijel's theorem).
%{\clr B. Helffer and R. Kiwan. Dirichlet eigenfunctions on the cube, sharpening the Courant nodal inequality. Preprint.\\}
 
\section{Introduction to minimal spectral partitions}\label{s4}
Most of this section comes from the founding paper \cite{HHOT1}.
\subsection{Definition}\label{ss4.1}
We now introduce the notion of spectral minimal partitions.
\begin{definition}[Minimal energy]\label{regOm}
Minimizing the energy over all the $k$-partitions, we introduce:
\begin{equation}\label{frakL} 
\mathfrak L_{k}(\Omega)=\inf_{\mathcal D\in \mathfrak O_k(\Omega)}\:\Lambda(\mathcal D).
\end{equation}
We will say that $\mathcal D\in \mathfrak O_k(\Omega)$ is minimal if  $\mathfrak L_{k}(\Omega)=\Lambda(\mathcal D)$. 
\end{definition}
Sometimes (at least for the proofs) we have to relax this definition by considering quasi-open or measurable sets for the partitions. We will not discuss this point in detail (see \cite {HHOT1}). We recall that if $ k=2$, we have proved in Proposition~\ref{L=L=L2} that $\mathfrak L_2(\Omega) =\lambda_2(\Omega).$\\

 More generally (see \cite{HHOT1}), for any integer $k\geq1$ and $ p\in [1,+\infty[$, we define the $p$-energy of a $k$-partition $\mathcal D=\{D_{i}\}_{1\leq i\leq k}$ by
\begin{equation}\label{LAp} 
\Lambda_{p}(\mathcal D)=\Big(\frac 1k \sum_{i=1}^k \lambda(D_i)^p\Big)^{\frac 1p}\,.
\end{equation}
The notion of $ p$-minimal $k$-partition can be extended accordingly, by minimizing $\Lambda_{p}(\mathcal D)$. Then we can consider the optimization problem
\begin{equation}\label{Lfrakp} 
\mathfrak L_{k,p}(\Omega)=\inf_{\mathcal D\in \mathfrak O_k}\Lambda_{p}(\mathcal D)\,.
\end{equation}
For $p=+\infty$, we write $\Lambda_{\infty}(\mathcal D)=\Lambda(\mathcal D)$ and $\mathfrak L_{k,\infty}(\Omega)=\mathfrak L_k(\Omega)\,$.

\subsection{Strong and regular partitions\label{ss4.2}}
The analysis of the properties of minimal partitions leads us to introduce two notions of regularity that we present briefly.
\begin{definition}\label{Definition3}
A partition $\mathcal D=\{D_i\}_{1\leq i\leq k}$ of $\Omega$ in $\mathfrak O_k$ is called {\bf strong} if 
\begin{equation}\label{defstr}
\Inte(\overline{\cup_i D_i}) \setminus \partial \Omega =\Omega\;.
\end{equation}
We say that $\mathcal D$ is {\bf nice} if $\Inte(\overline{D_i})= D_i$, for any $1\leq i\leq k$.
\end{definition}
For example, in Figure \ref{fig.exvecpABcarre}, only the fourth picture gives a nice partition. Attached to a strong partition, we associate a closed set in $\overline{\Omega}$~:
\begin{definition}[Boundary set]\label{Definition4}
\begin{equation}\label{assclset} 
\partial \mathcal D = \overline{ \cup_i \left( \Omega \cap \partial D_i \right)}\;.
\end{equation}
\end{definition}
$\partial \mathcal D $ plays the role of the nodal set (in the case of a nodal partition). This leads us to introduce the set ${\mathfrak O}^{\mathsf{reg}}_k(\Omega)$ of {\bf regular} partitions, which should satisfy the following properties~:
\begin{enumerate}[label={\rm(\roman*)}]
\item Except  at  finitely many distinct $\xb_i\in\Omega\cap \partial\mathcal D$ in the neigborhood of which $\partial \mathcal D $ is the union of $\nu(\xb_i)$ smooth curves ($\nu(\xb_i)\geq 2$) with one end at $\xb_i$, $\partial \mathcal D $ is locally diffeomorphic to a regular curve.
\item $\partial\Omega\cap \partial \mathcal D $ consists of a (possibly empty) finite set of points $\yb_j$. Moreover $\partial \mathcal D $ is near $\yb_j$ the union of $\rho(\yb_j)$ distinct smooth half-curves which hit $\yb_j$.
\item $\partial \mathcal D $ has the equal angle meeting property, that is the half curves cross with equal angle at each singular interior point of $\partial \mathcal D $ and also at the boundary together with the tangent to the boundary.
\end{enumerate}
We denote by $X(\partial \mathcal D )$ the set corresponding to the points $\xb_i$ introduced in (i) and by $Y(\partial \mathcal D )$ corresponding to the points $\yb_i$ introduced in (ii). 
\begin{remark} 
This notion of regularity for partitions is very close to what we have observed for the nodal partition of an eigenfunction in Proposition~\ref{thm:nodinfo}. The main difference is that in the nodal case there is always an even number of half-lines meeting at an interior singular point. 
\end{remark}
 
\begin{figure}[h!]
\begin{center}
\subfigure[Bipartite partitions.\label{fig.bip}]{\begin{tabular}{c}
\includegraphics[height=2cm]{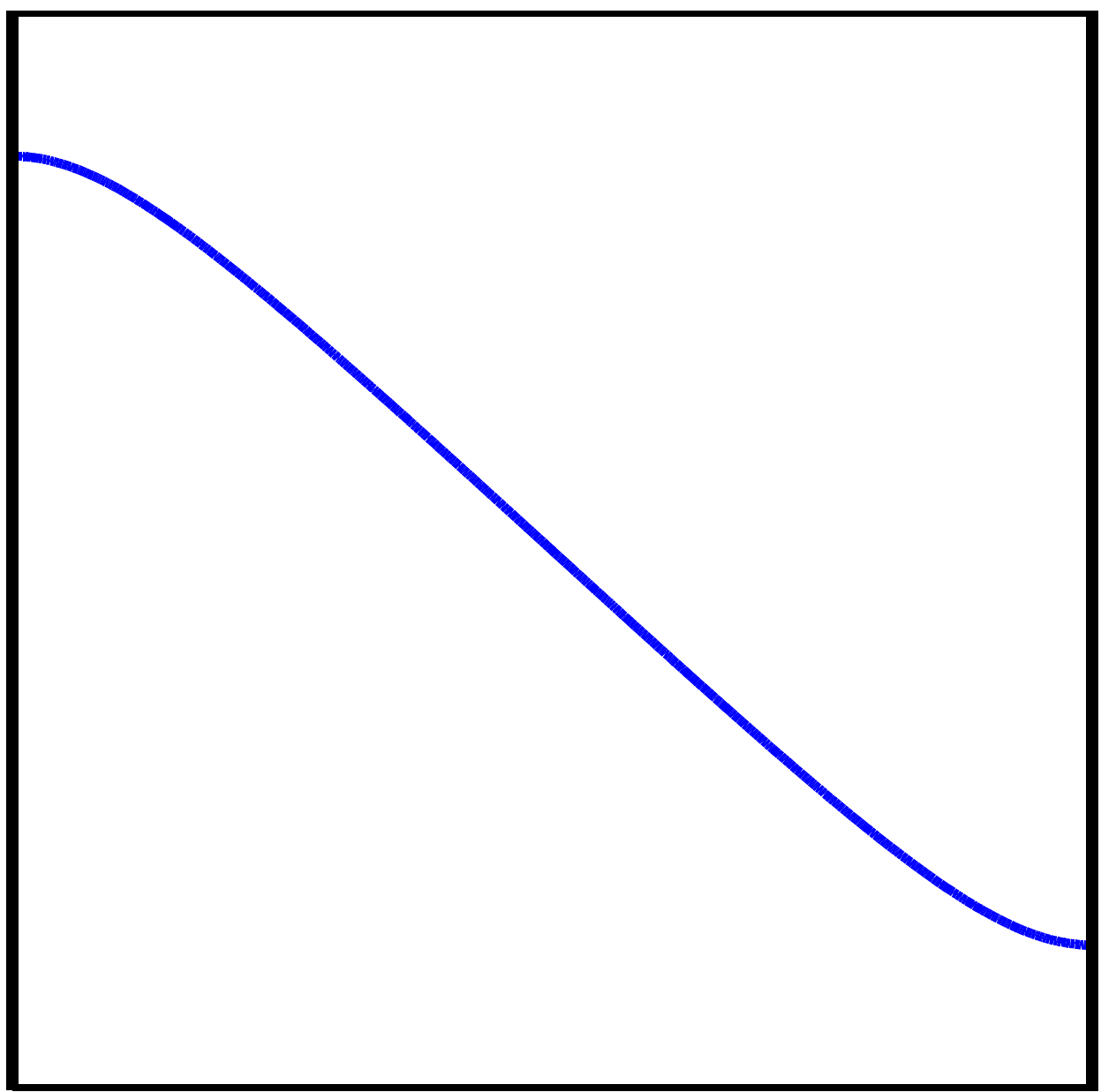}
\includegraphics[height=2cm]{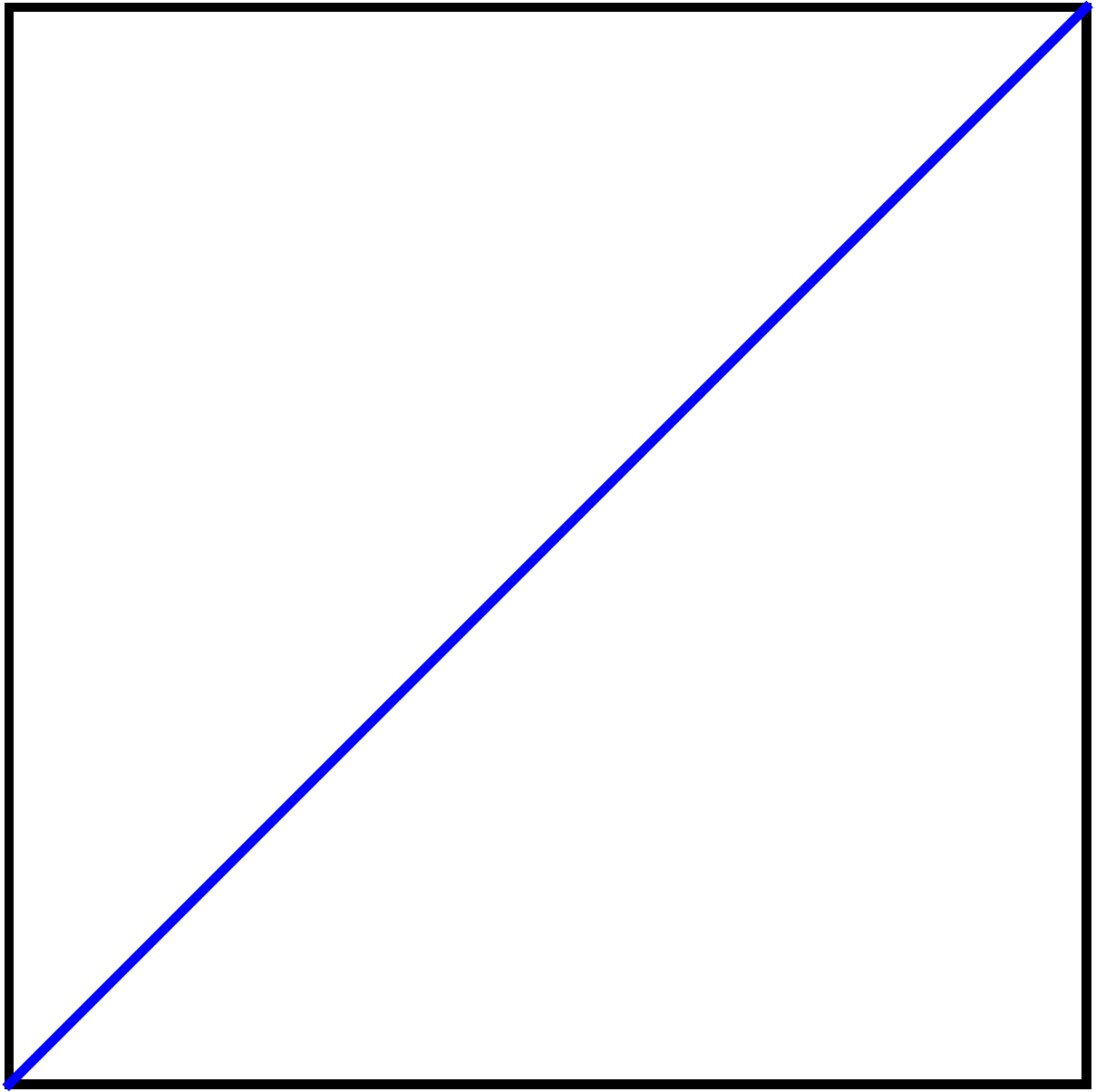}
\includegraphics[height=2cm]{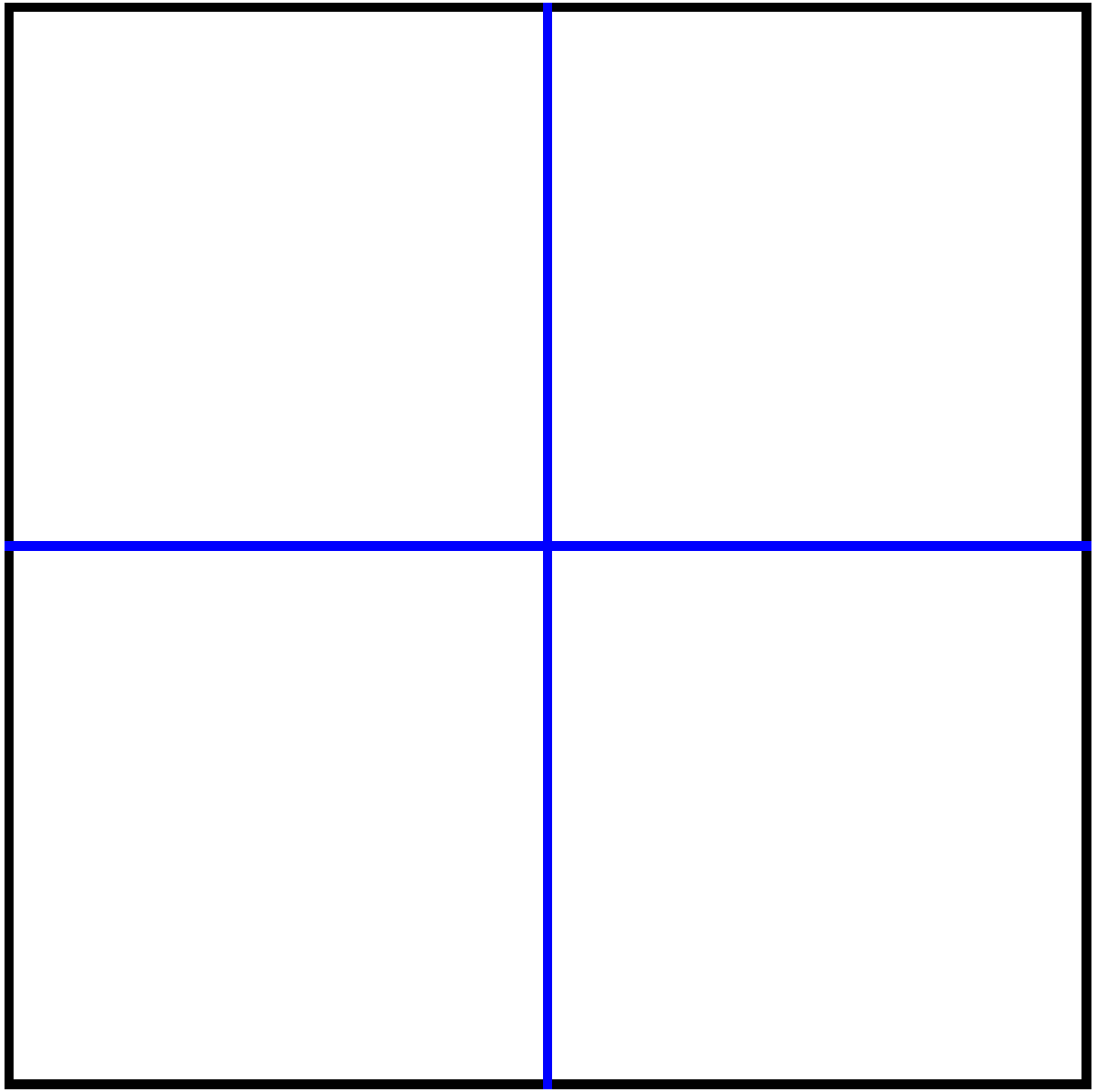}\\
\includegraphics[height=2cm]{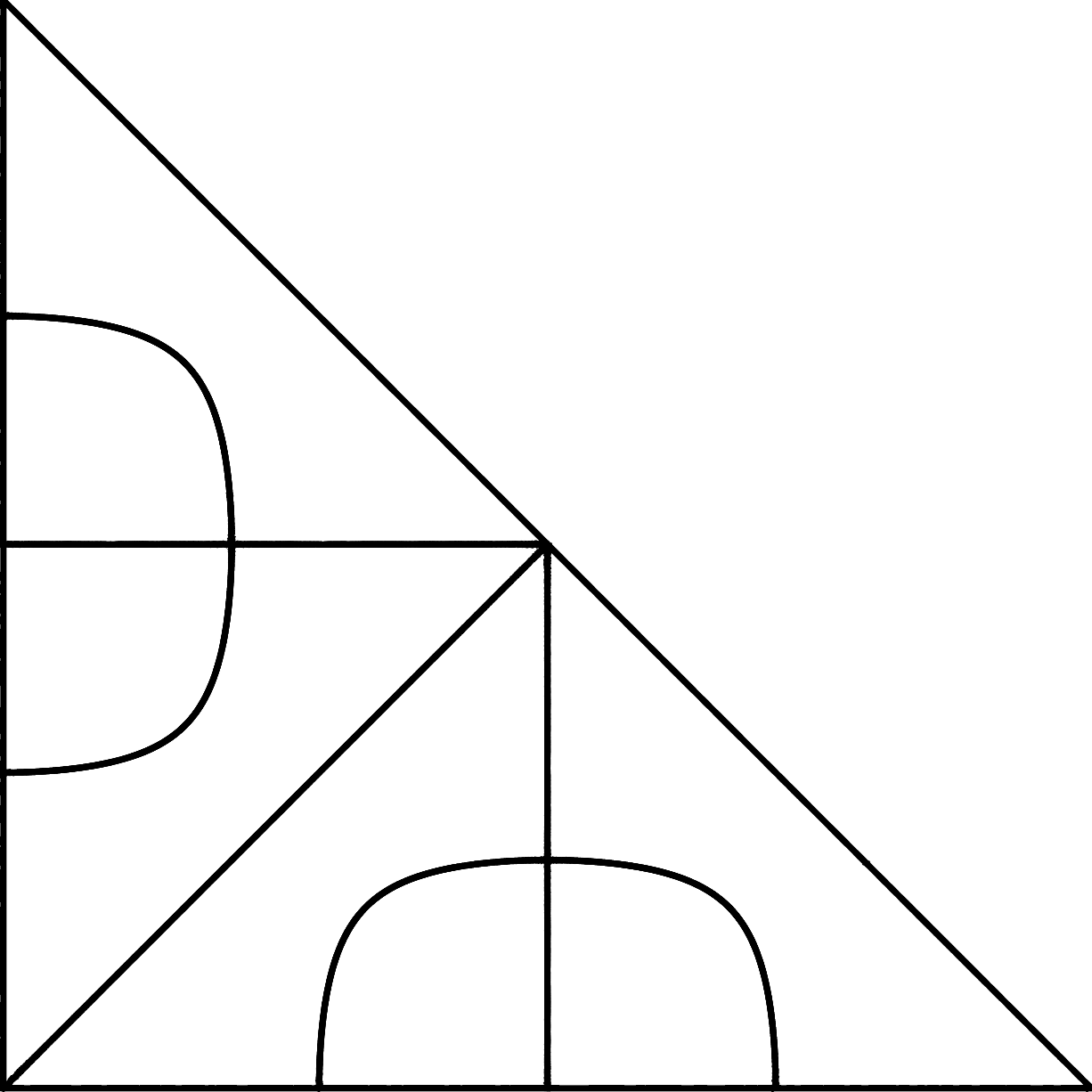}
\includegraphics[width=2cm,angle=90]{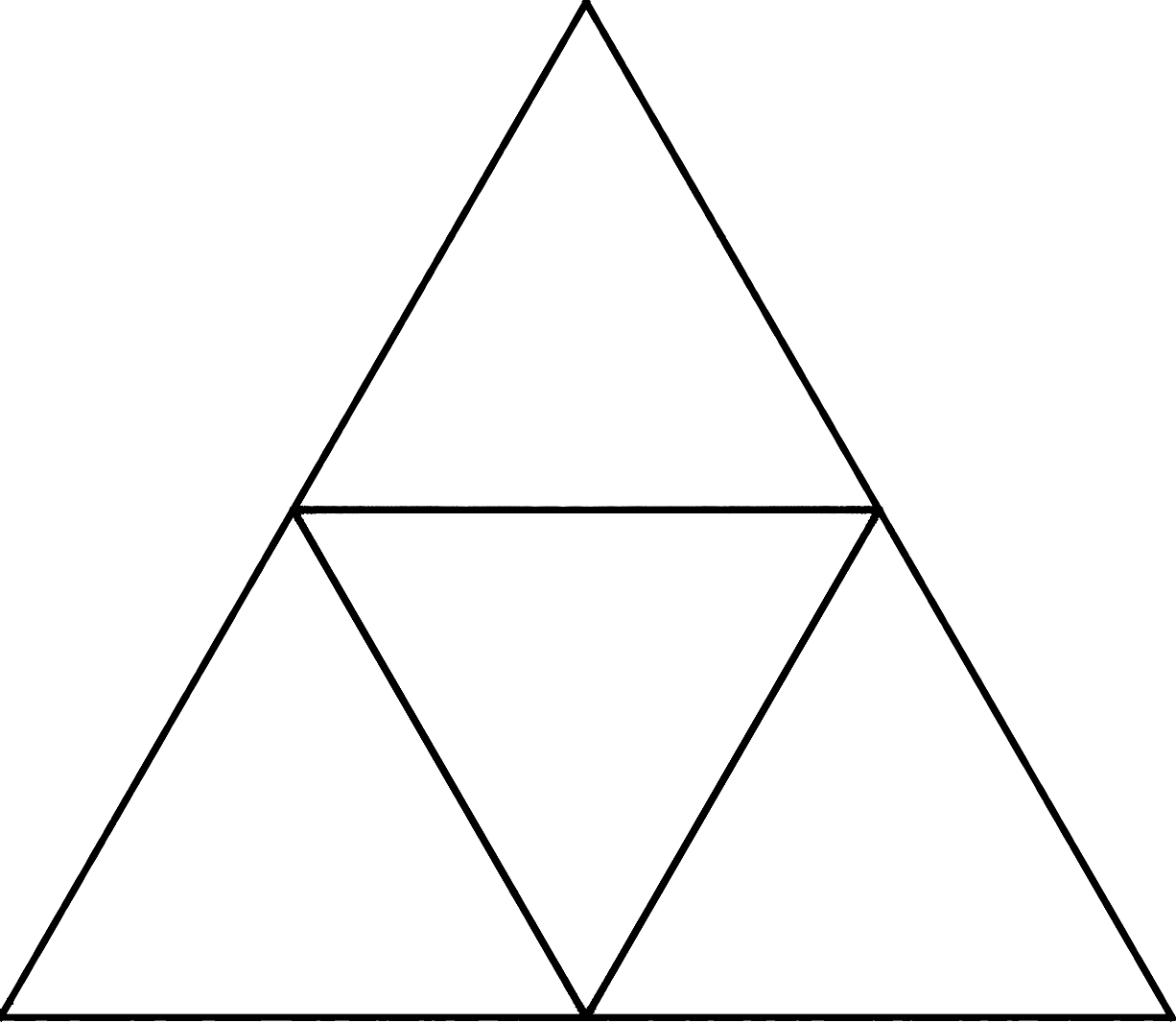}
\includegraphics[width=2cm,angle=90]{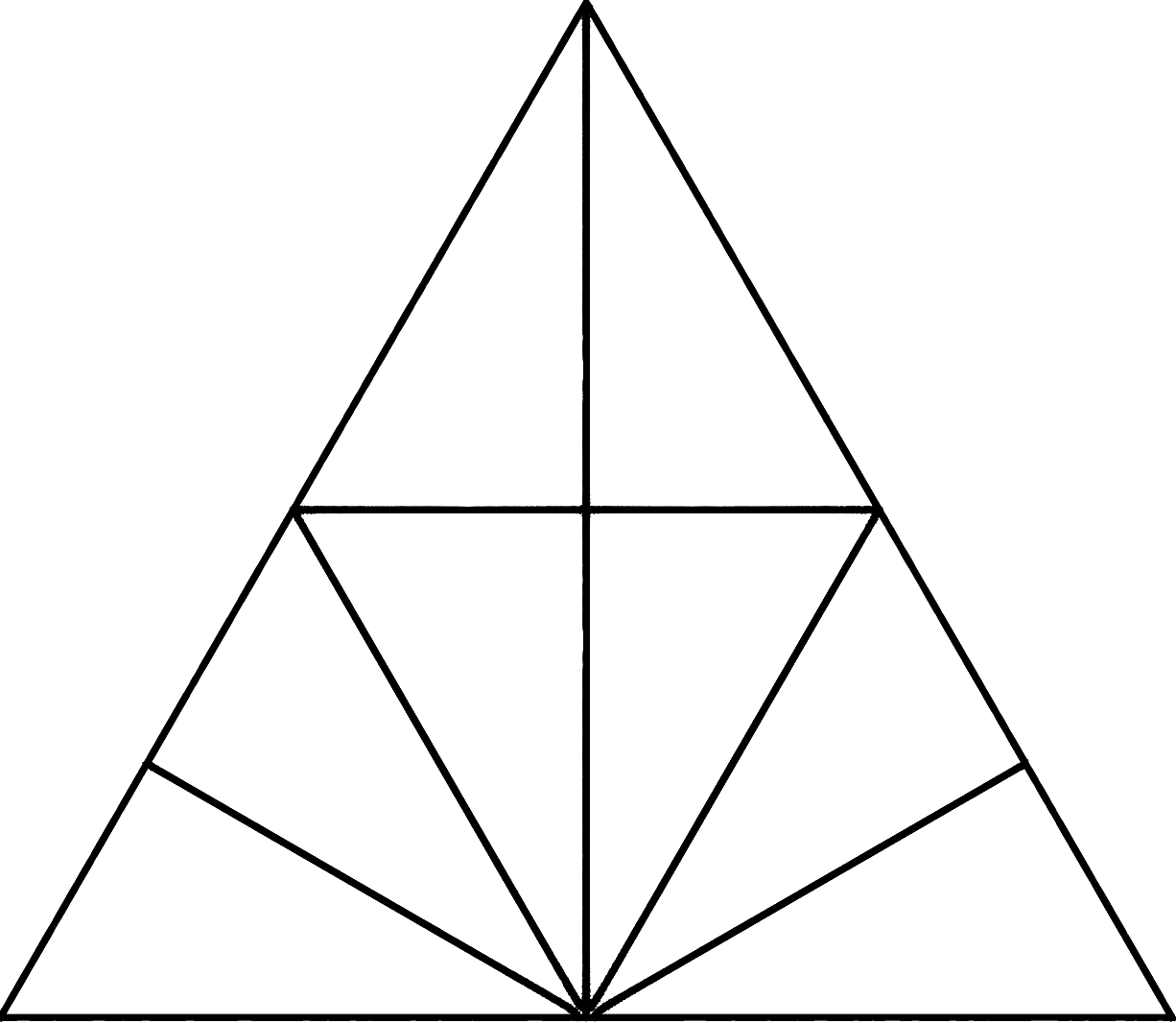}
\includegraphics[width=2cm,angle=90]{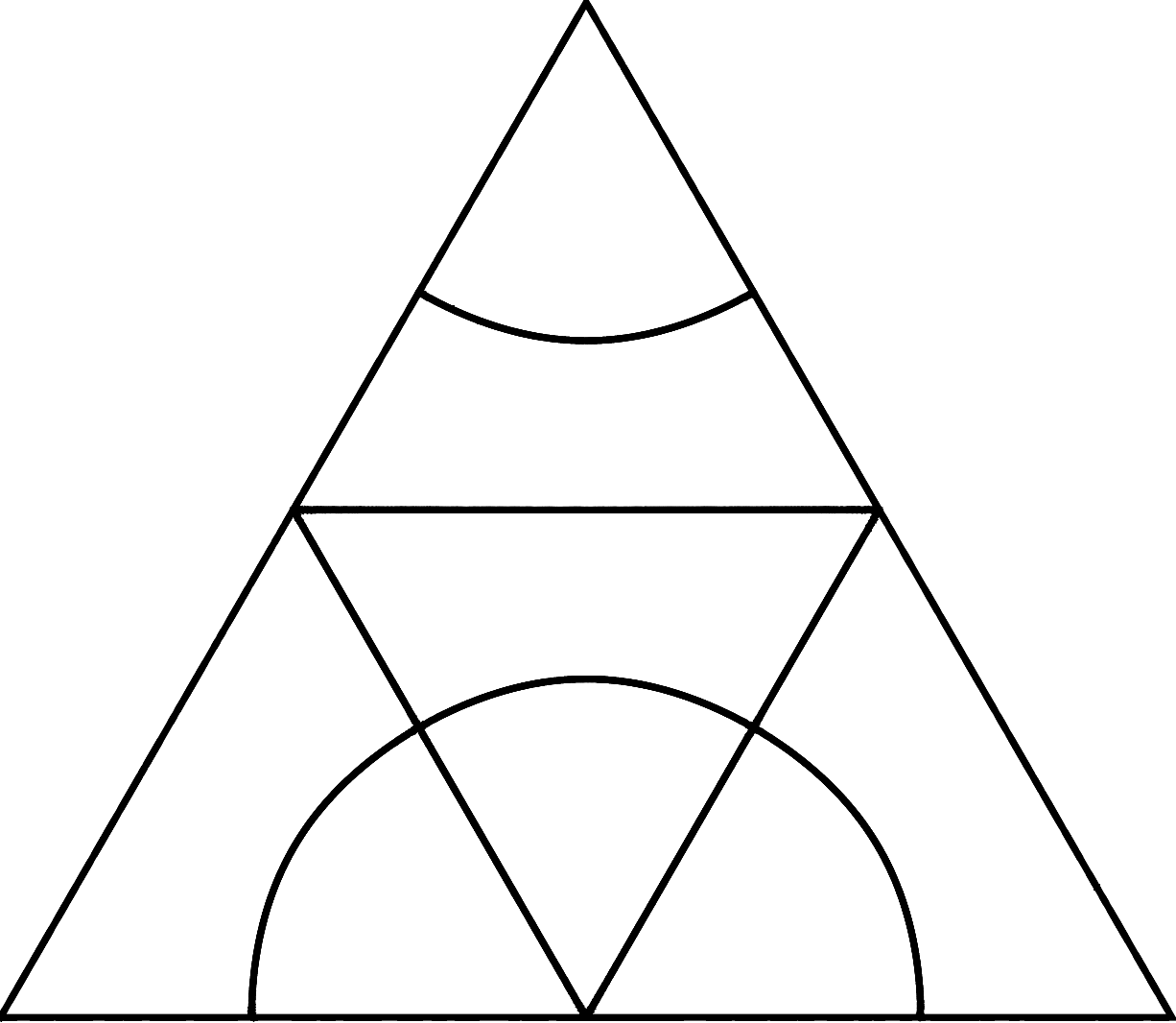}
\end{tabular}}\\
\subfigure[Non bipartite partitions.\label{fig.nonbip}]{\begin{tabular}{c}
\includegraphics[height=2cm]{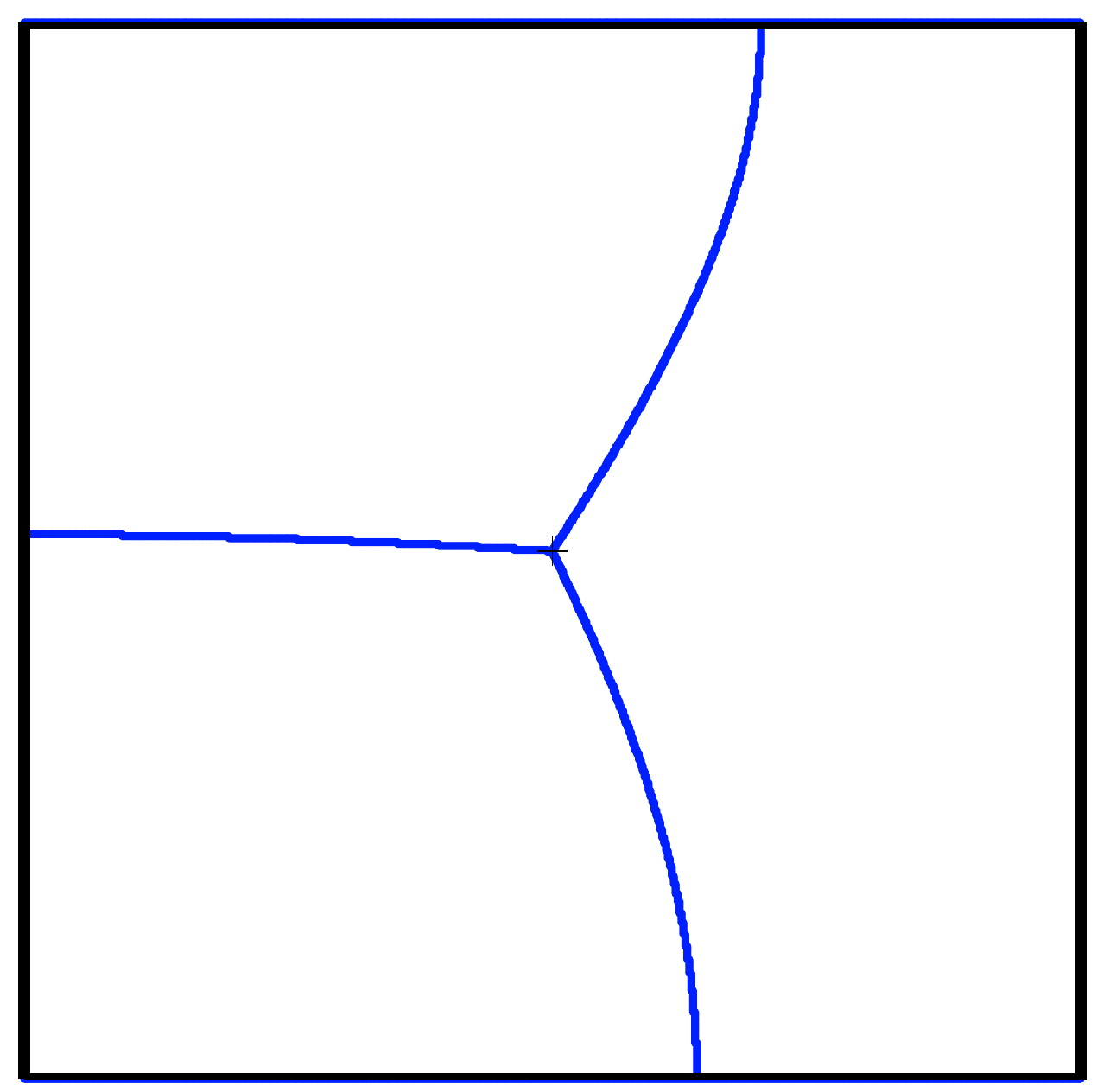}
\includegraphics[height=2cm]{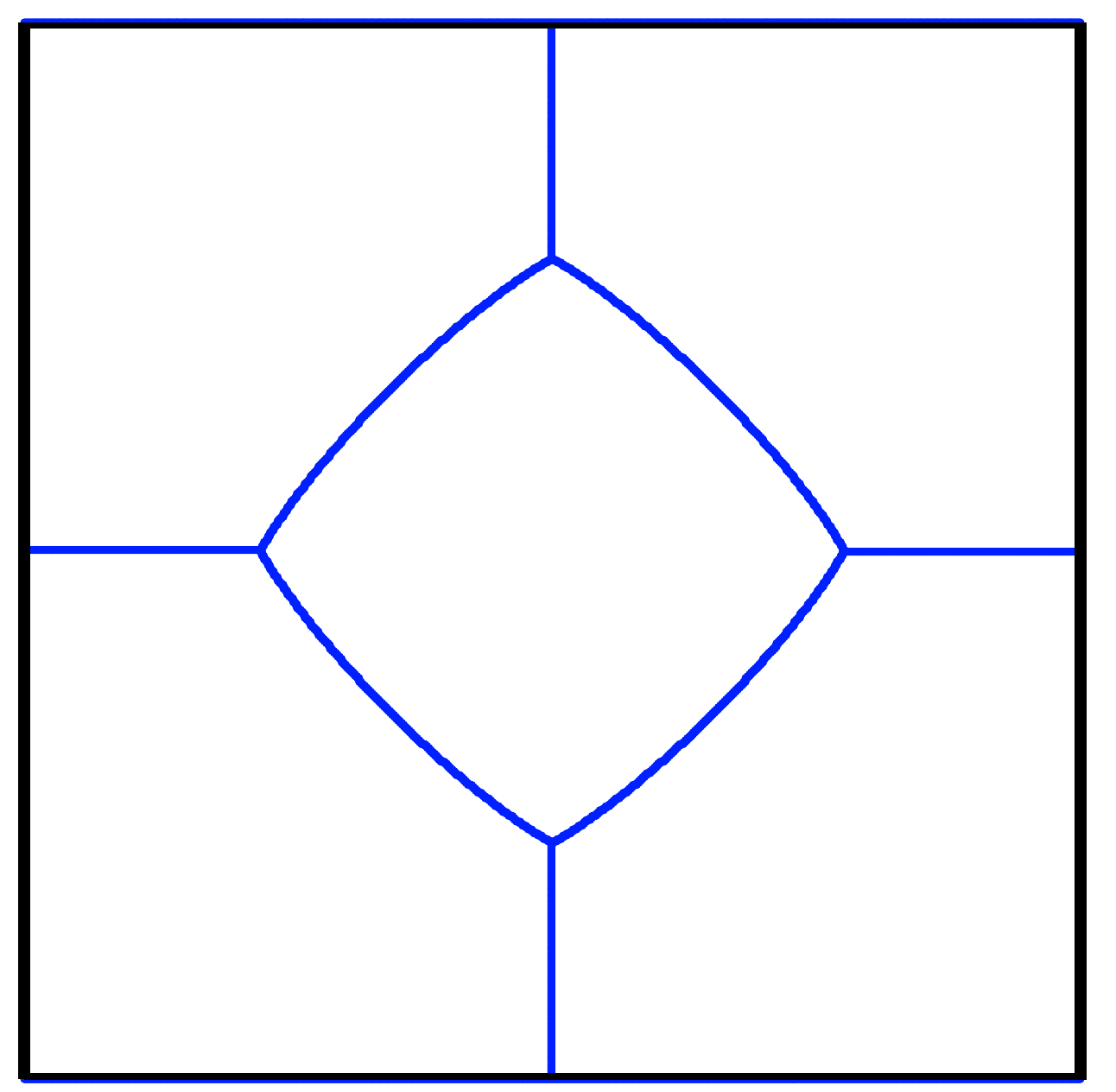}
\includegraphics[height=2cm]{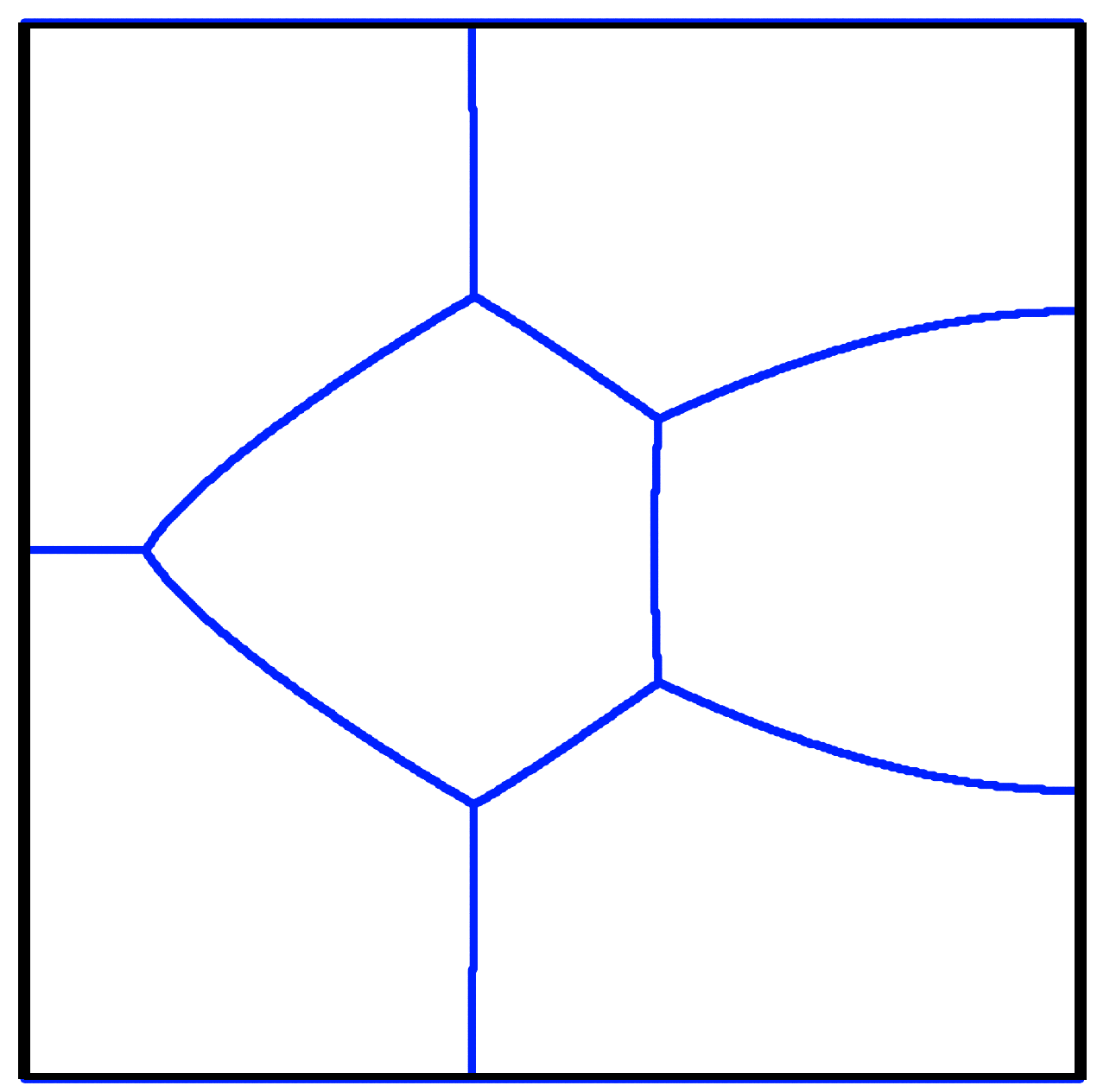}
\includegraphics[height=2cm]{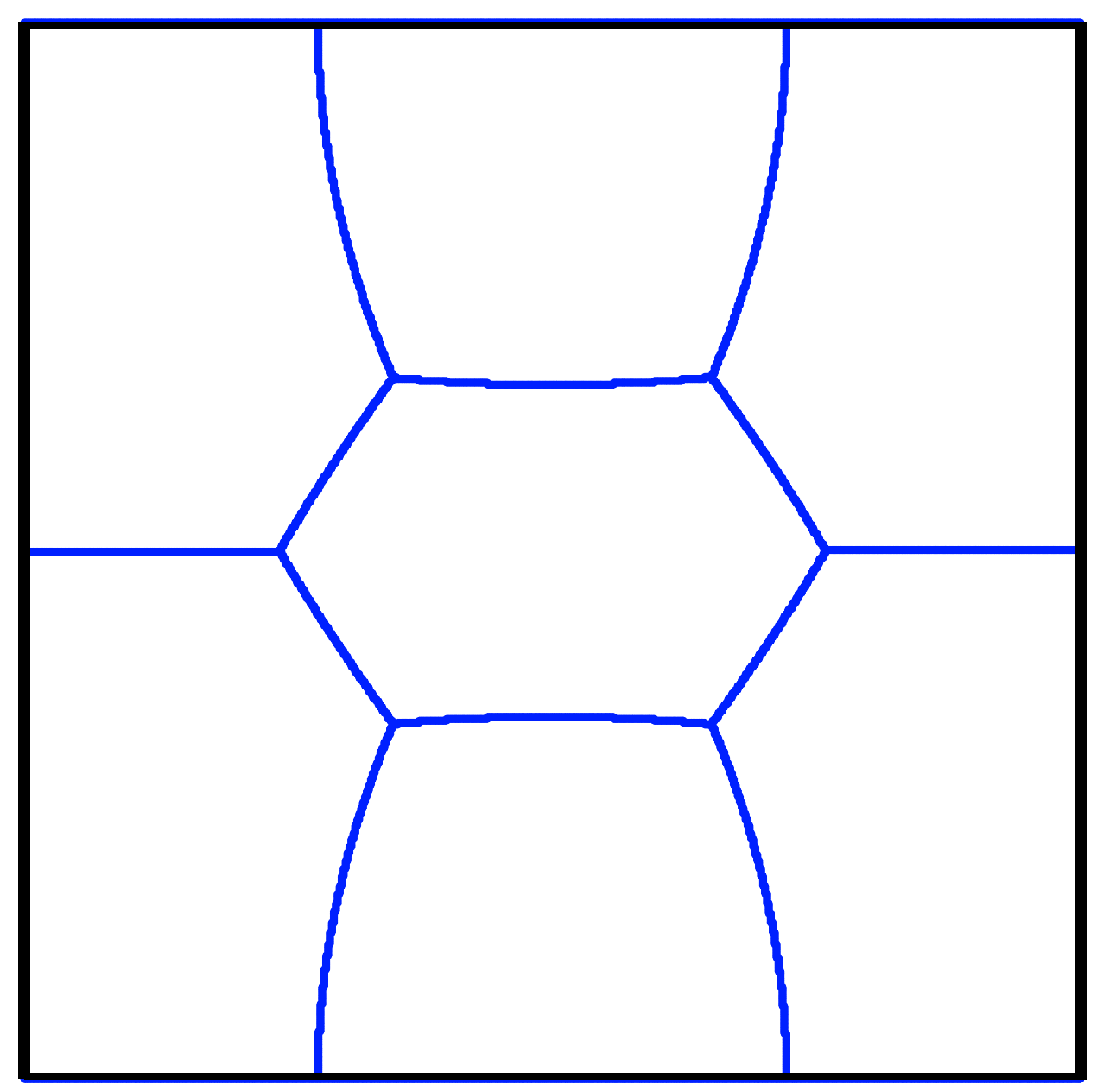}
\includegraphics[height=2cm]{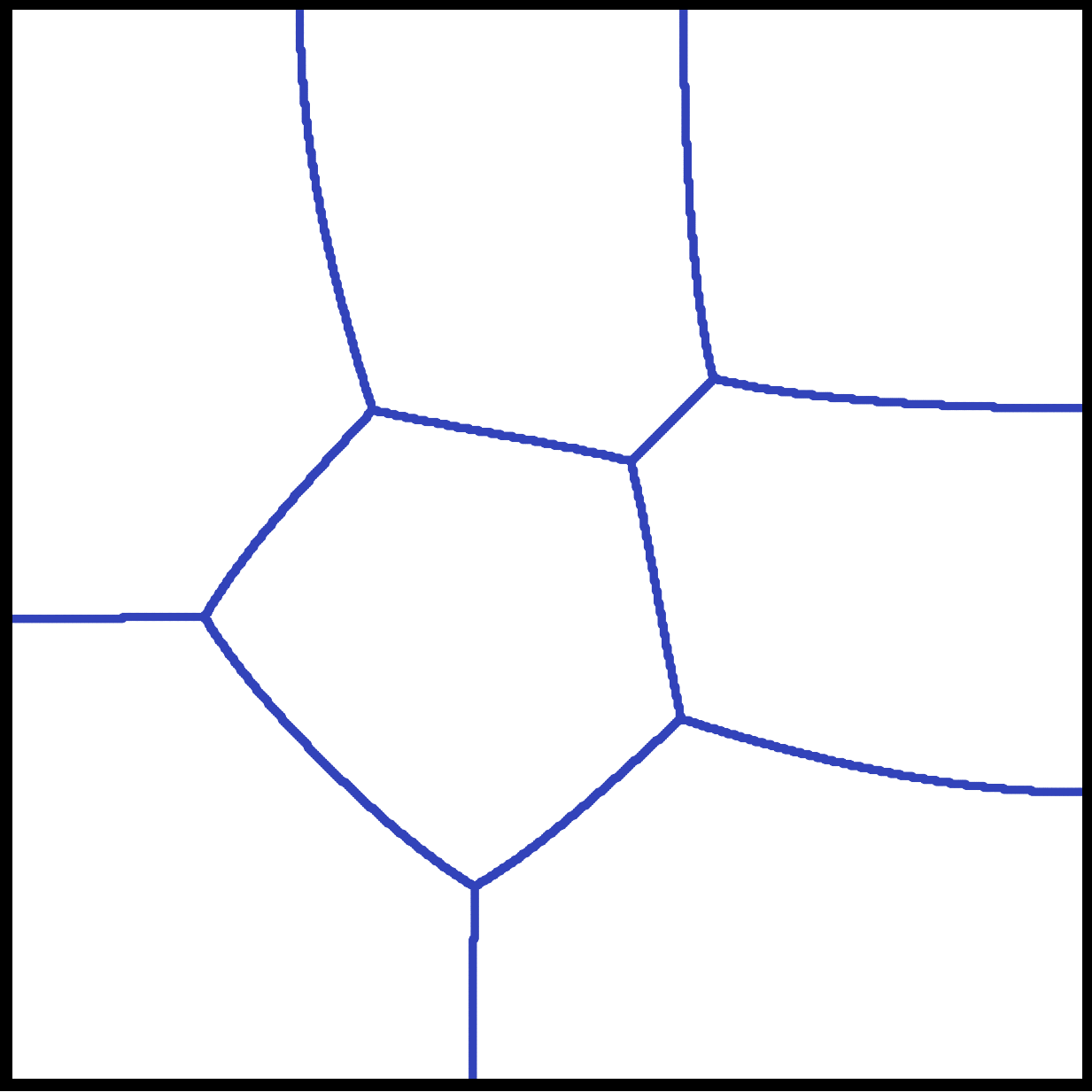}
\includegraphics[height=2cm]{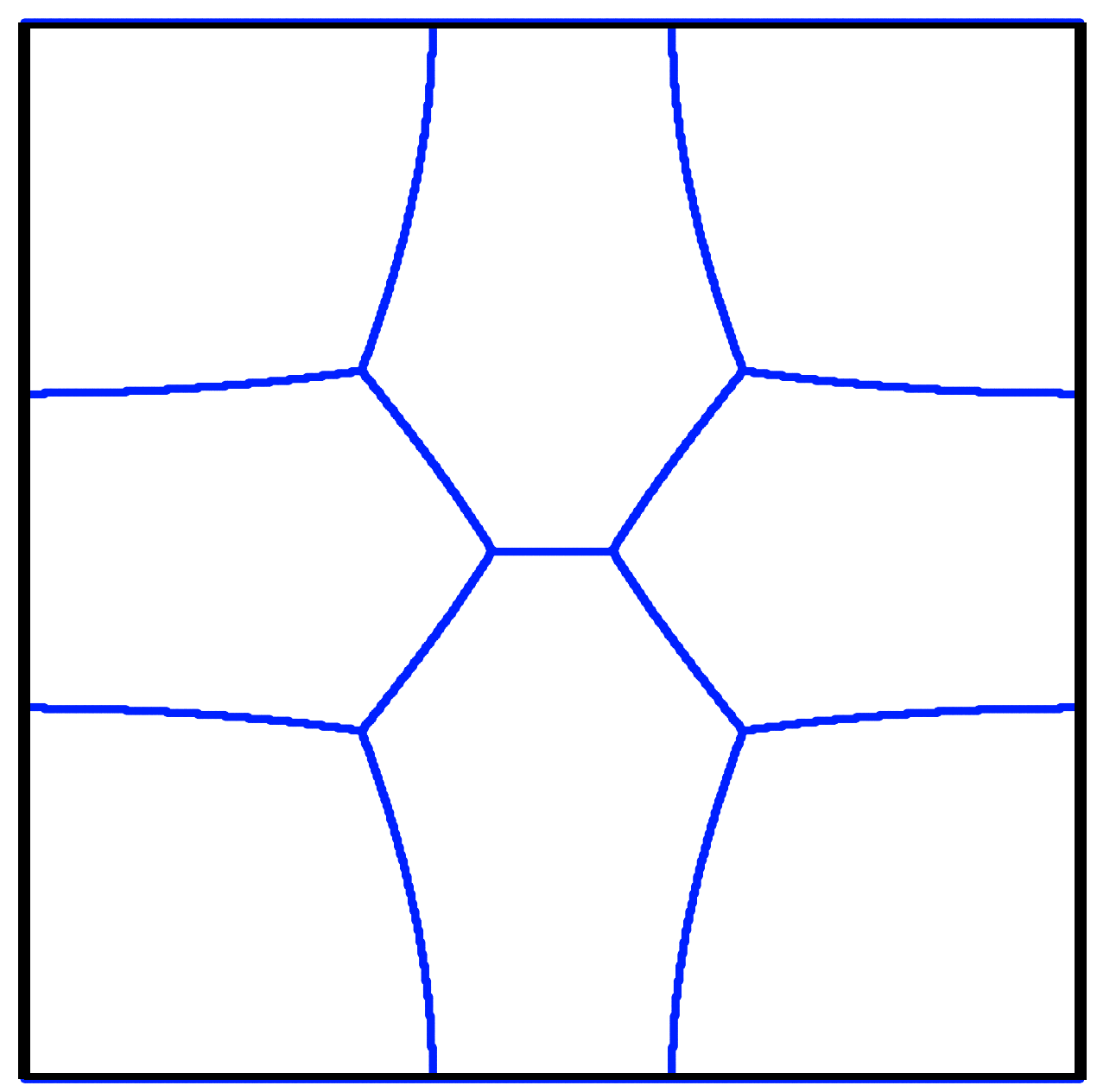}\\
\includegraphics[height=2cm]{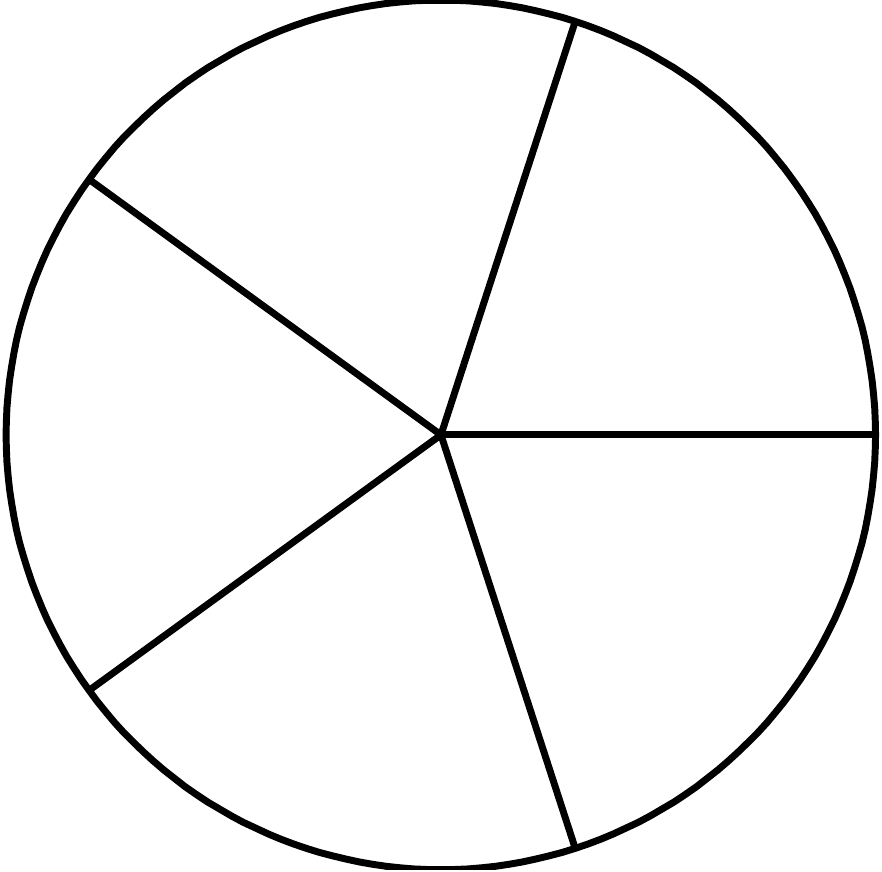}
\includegraphics[height=2cm]{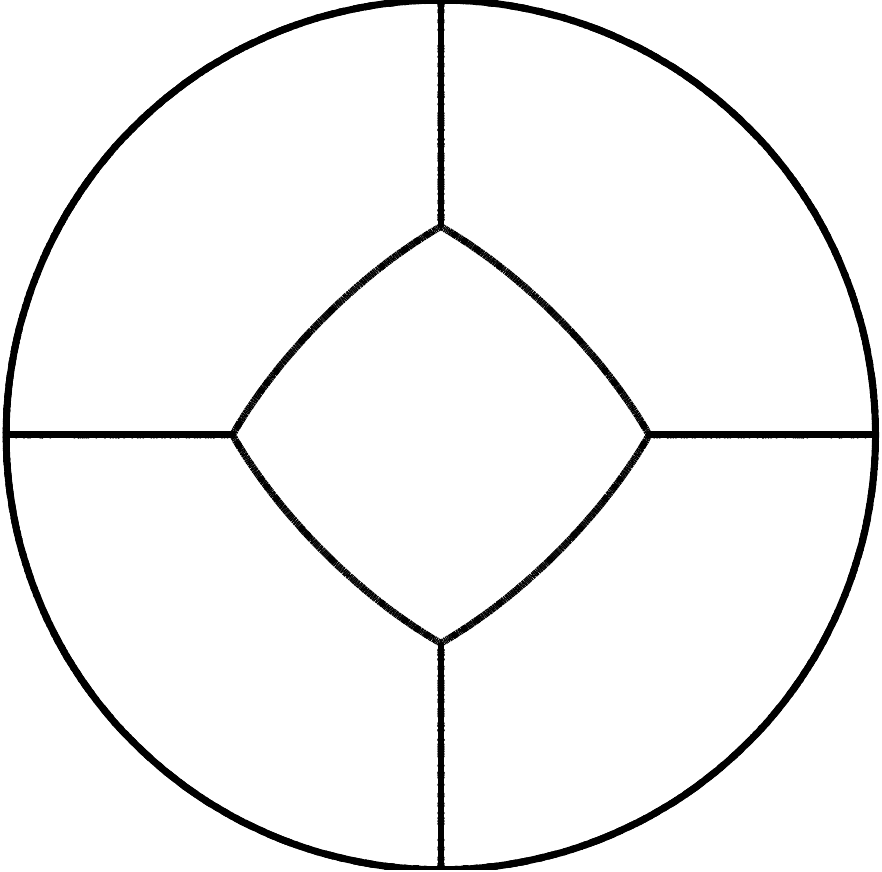}
\includegraphics[height=2cm]{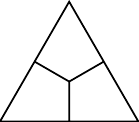}
\includegraphics[height=2cm]{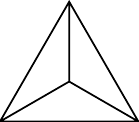}
\end{tabular}}
\caption{Examples of partitions.\label{fig.expart}}
\end{center}
\end{figure}
Examples of regular partitions are given in Figure~\ref{fig.expart}. More precisely, the partitions represented in Figures~\ref{fig.bip} are nodal (we have respectively some nodal partition associated with the double eigenvalue $\lambda_{2}$ and with $\lambda_{4}$ for the square, $\lambda_{15}$ for the  right angled isosceles triangle and the last two pictures are nodal partitions associated with the double eigenvalue $\lambda_{12}$ on the equilateral triangle). On the contrary, all the partitions presented in Figures~\ref{fig.nonbip} are not nodal. 

\subsection{Bipartite partitions}\label{ss4.3}
\begin{definition}
We say that two sets $ D_i,D_j$ of the partition $\mathcal D$ are neighbors and write $ D_i\sim D_j$, if 
$$D_{ij}:=\Inte(\overline {D_i\cup D_j})\setminus \partial \Omega $$
is connected. We say that a regular partition is bipartite if it can be colored by two colors (two neighbors having two different colors). 
\end{definition}
Nodal partitions are the main examples of bipartite partitions (see Figure~\ref{fig.bip}). Figure~\ref{fig.nonbip} gives examples of non bipartite partitions. Some examples can also be found in \cite{CyBaHo}.

Note that in the case of a planar domain we know by graph theory that if for a regular partition all the $\nu(\xb_i)$ are even then the partition is bipartite. This is no more the case on a surface. See for example the third subfigure in Figure \ref{fig.torepart} for an example on $\mathbb T^2$.

 \subsection{Main properties of minimal partitions}\label{ss4.4}
It has been proved by Conti-Terracini-Verzini \cite{CTV0, CTV2, CTV:2005} (existence) and Helffer--Hoffmann-Ostenhof--Terracini \cite{HHOT1} (regularity) the following theorem:
\begin{theorem}\label{thstrreg}
For any $ k$, there exists a minimal $k$-partition which is strong, nice and regular. Moreover any minimal $k$-partition has a strong, nice and regular representative\footnote{possibly after a modification of the open sets of the partition by capacity $0$ subsets.}. The same result holds for  the $p$-minimal $k$-partition problem with $p\in [1,+\infty)$.
\end{theorem}
The proof is too involved to be presented here. We just give one statement used for the existence (see \cite{CTV:2005}) for $p\in [1,+\infty)$. The case $p=+\infty$ is harder.
\begin{theorem}\label{extremality_condition_p}~\\
Let $ p\in[1,+\infty)$ and let $\mathcal D=\{D_i\}_{1\leq i\leq k}\in \mathfrak O_k$ be a $ p$-minimal $ k$-partition associated with $\mathcal L_{k,p}$ and let $ (\phi_i)_i$ be any set of positive eigenfunctions normalized in $ L^2$ corresponding to $(\lambda(D_i))_i$. Then, there exist $a_i>0$, such that the functions $u_i=a_i\phi_i$ verify in $\Omega$ the variational inequalities
\begin{enumerate}[label={(I\alph*)}]
\item[(I1)] $ -\Delta u_i\leq \lambda(D_i)u_i$,
\item[(I2)] $ -\Delta\left( u_i-\sum_{j\neq i}u_j\right)\geq \lambda(D_i)u_i-\sum_{j\neq i} \lambda(D_j)u_j$.
\end{enumerate}
\end{theorem}
These inequalities imply that $U=(u_1,...,u_k)$ is in the class $\mathcal S^*$ as defined in \cite{CTV2} which ensures the Lipschitz continuity of the $u_i$'s in $\Omega$. Therefore we can choose a partition made of open representatives $D_i=\{u_i>0\}$.\\
Other proofs of a somewhat weaker version of the existence statement have been given by Bucur-Buttazzo-Henrot \cite{BBH}, Caffarelli-Lin \cite{CL1}. The minimal partition is shown to exist first in the class of quasi-open sets and it is then proved that a representative of the minimizer is open. Note that in some of these references these minimal partitions are also called {\bf optimal} partitions.\\

When $p=+\infty$, minimal spectral properties have two important properties.
\begin{proposition}\label{prop.minspecpart}
If ${\mathcal D}=\{D_{i}\}_{1\leq i\leq k}$ is a minimal $k$-partition, then
\begin{enumerate}
\item The minimal partition ${\mathcal D}$ is a {\bf spectral equipartition}, that is satisfying:
$$\lambda(D_i)=\lambda(D_j)\,,\qquad \mbox{ for any }\quad1\leq i,j\leq k\,.$$
\item For any pair of neighbors $D_i\sim D_j$,
\begin{equation}\label{pc1}
\lambda_2(D_{ij}) = \mathfrak L_k (\Omega)\,.
\end{equation}
\end{enumerate}
\end{proposition}
 
\begin{proof}
For the first property, this can be understood, once the regularity is obtained by pushing the boundary and using the Hadamard formula \cite{Henrot} (see also Subsection~\ref{sspk}). For the second property, we can observe that $\{D_i,D_j\}$ is necessarily a minimal $2$-partition of $D_{ij}$ and in this case, we know that $\mathfrak L_2(D_{ij})=\lambda_2(D_{ij})$ by Proposition \ref{L=L=L2}. Note that it is a stronger property than the claim in \eqref{pc1}.
\end{proof}
 
\begin{remark}\label{rempaircomp}
In the proof of Theorem~\ref{thstrreg}, one obtains on the way the useful construction. Attached to each $D_i$, there is a distinguished ground state $u_i$ such that $u_i >0$ in $D_i$ and such that for each pair of neighbors $\{D_i,D_j\}$, $u_i-u_j$ is the second eigenfunction of the Dirichlet Laplacian in $D_{ij}$.
\end{remark}
Let us now establish two important properties concerning the monotonicity (according to $k$ or the domain $\Omega$).
\begin{proposition}\label{l<L}
 For any $k\geq1$, we have 
\begin{equation}\label{frakL<}
\mathfrak L_k (\Omega)<\mathfrak L_{k+1}(\Omega).
\end{equation}
\end{proposition}
\begin{proof}
We take indeed a minimal $(k+1)$-partition of $\Omega$. We have proved that this partition is regular. If we take any subpartition by $k$ elements of the previous partitions, this
 cannot be a minimal $k$-partition (it has not the ``strong partition'' property). So the inequality in \eqref{frakL<} is strict.
\end{proof}
The second property concerns the {\bf domain monotonicity}. 
\begin{proposition}
If $\Omega \subset \widetilde \Omega$, then
$$\mathfrak L_k (\widetilde \Omega) \leq \mathfrak L_k(\Omega)\;,\quad \forall k\geq 1\,.$$
\end{proposition}
We observe indeed that each partition of $\Omega$ is a partition of $\widetilde \Omega$.

\subsection{Minimal spectral partitions and Courant sharp property}
A natural question is whether a minimal partition of $\Omega$ is a nodal partition. We have first the following converse theorem (see \cite{HH:2005a,HHOT1}):
\begin{theorem}\label{partnod}
If the minimal partition is bipartite, this is a nodal partition.
\end{theorem}
\begin{proof}
Combining the bipartite assumption and the pair compatibility condition mentioned in Remark \ref{rempaircomp}, it is immediate to construct some $u\in H_0^1(\Omega)$ such that 
$$u_{\vert D_i} = \pm u_i,\quad\forall 1\leq i\leq k,\qquad\mbox{ and }
\qquad-\Delta u =\mathfrak L_k(\Omega) u \quad\mbox{ in }\Omega\setminus X(\partial \mathcal D).$$ 
But $X(\partial \mathcal D )$ consists of a finite set and $-\Delta u -\mathfrak L_k(\Omega) u$ belongs to $H^{-1}(\Omega)$. This implies that $-\Delta u =\mathfrak L_k(\Omega) u$ in $\Omega$ and hence $u$ is an eigenfunction of $H(\Omega)$ whose nodal set is $\partial\mathcal D$.
 \end{proof}

The next question is then to determine how general is the previous situation. Surprisingly this only occurs in the so called Courant sharp situation. For any integer $ k\ge 1$, we recall that $L_{k}(\Omega)$ was introduced in Definition~\ref{def.Lk}. In general, one can show, as an easy consequence of the max-min characterization of the eigenvalues, that 
\begin{equation} \label{compLLL}
\lambda_k(\Omega)\leq\mathfrak L_k(\Omega)\leq L_k(\Omega)\,.
\end{equation}
The last but not least result (due to \cite{HHOT1}) gives the full picture of the equality cases:
\begin{theorem}\label{TheoremL=L}
Suppose $\Omega\subset \mathbb R^2$ is regular. If $\mathfrak L_k(\Omega)=L_k(\Omega)$ or $\mathfrak L_k(\Omega)=\lambda_k(\Omega)$, then 
\begin{equation}\label{LLL}
\lambda_k(\Omega)=\mathfrak L_k(\Omega)=L_k(\Omega)\,.
\end{equation}
In addition, there exists a Courant sharp eigenfunction associated with $\lambda_{k}(\Omega)$.
\end{theorem}
This answers a question informally mentioned in \cite[Section 7]{BHIM}. 
\begin{proof}
It is easy to see using a variation of the proof of Courant's theorem that the equality $\lambda_k=\mathfrak L_k$ implies \eqref{LLL}. Hence the difficult part is to get \eqref{LLL} from the assumption that $L_k=\mathfrak L_k =\lambda_{m(k)}$, that is to prove that $m(k)=k$. This involves a construction of an exhaustive family $\{\Omega(t),\ t \in (0,1)\}$, interpolating between $\Omega(0):=\Omega \setminus \mathcal N( \phi_k)$ and $\Omega(1):=\Omega$, where $\phi_k$ is an eigenfunction corresponding to $L_k$ such that its nodal partition is a minimal $k$-partition. This family is obtained by cutting small intervals in each regular component of $\mathcal N(\phi_k)$. $L_k$ being an eigenvalue common to all  $H(\Omega(t))$, but its labelling changing between $t=0$ and $t=1$, we get by a tricky argument   a contradiction for some $t_0$ where the multiplicity of $L_k$ should increase.
\end{proof}

%%%%%
 \subsection{On subpartitions of minimal partitions}\label{ssubp}
Starting from a given strong $k$-partition, one can consider subpartitions by considering a subfamily of $D_i$'s such that $\Inte(\cup {\overline{D_i}})$ is connected, typically a pair of two neighbors.
Of course a subpartition of a minimal partition should be minimal. If it was not the case, we should be able to decrease the energy by deformation of the boundary.
The next proposition is useful and reminiscent of Proposition \ref{submu}.
\begin{proposition}\label{subpartition}
Let $\mathcal D = \{D_i\}_{1\leq i\leq k}$ be a minimal $k$-partition for $\mathfrak L_k(\Omega)$. Then, for any subset $I\in \{1,\dots,k\}$, the associated subpartition $\mathcal D^I = \{D_i\}_{i\in I}$ satisfies 
\begin{equation}
\mathfrak L_k(\Omega)= \Lambda (\mathcal D^I) = \mathfrak L _{|I|} ( \Omega^I)\;,
\end{equation}
where $\Omega^I:=\Inte (\overline{\cup_{i\in I} D_i})\;$.
\end{proposition}
This is clear from the definition and the previous results that any subpartition $\mathcal D^I$ of a minimal partition $\mathcal D$ should be minimal on $\Omega^I$ and this proves the proposition. One can also observe that if this subpartition is bipartite, then it is nodal and actually Courant sharp. In the same spirit, starting from a minimal regular $k$-partition $\mathcal D$ of a domain $\Omega$, we can extract (in many ways) in $\Omega$ a connected domain $\widetilde\Omega$ such that $\mathcal D$ becomes a minimal bipartite $k$-partition of $\widetilde \Omega$. It is achieved by removing from $\Omega$ a union of a finite number of regular arcs corresponding to pieces of boundaries between two neighbors of the partition. 
\begin{corollary}\label{corext}
If $\mathcal D$ is a minimal regular $k$-partition, then for any extracted
connected open set $\widetilde \Omega$ associated with $\mathcal D$,
we have
\begin{equation}
\lambda_k( \widetilde \Omega) = \mathfrak L_k( \widetilde \Omega)\;.
\end{equation}
\end{corollary}
This last criterion has been analyzed in \cite{BHV} for glueing of triangles, squares
and hexagons as a test of minimality in connexion with the hexagonal conjecture (see Subsection \ref{ss9.1}).

\subsection{Notes}
Similar results hold in the case of compact Riemannian surfaces when considering the Laplace-Beltrami operator (see \cite{Chavel}) (typically for $\mathbb S^2$\cite{HHOT2} and $\mathbb T^2$\cite{Len3}). In the case of dimension $3$, let us mention that Theorem \ref{TheoremL=L} is proved in \cite{HHOT3}. The complete analysis of minimal partitions was not achieved in \cite{HHOT1} but it is announced in the introduction of \cite{RTT} that it can now be obtained.

\section{On $p$-minimal $k$-partitions}\label{s5}
The notion of $p$-minimal $k$-partition has been already defined in Subsection \ref{ss4.4}. We would like in this section to analyze the dependence on $p$ of these minimal partitions.
\subsection{Main properties}
Inequality \eqref{compLLL} is replaced by the  following one (see \cite{HHOT2} for $p=1$ and \cite{HH:2009} for general $p$) 
\begin{equation} 
\left(\frac 1k \sum_{j=1}^k \lambda_j(\Omega)^p\right)^{\frac 1 p} \leq \mathfrak L_{k,p} (\Omega)\,.
\end{equation}
This is optimal for the disjoint union of $k$-disks with different (but close) radius.\\

\subsection{Comparison between different $p$'s\label{sspk}}
\begin{proposition}\label{levels} For any $k\geq1$ and any $p\in[1,+\infty)$, there holds
\begin{align}
&\frac{1}{k^{1/p}}\mathfrak L_k(\Omega) \leq \mathfrak L_{k,p}(\Omega) \leq \mathfrak L_k(\Omega)\;,\label{compa1} \\
&\mathfrak L_{k,p} (\Omega) \leq \mathfrak L_{k,q}(\Omega)\,,\qquad\mbox{ if }p\leq q\,.\label{monoto} 
\end{align} 
\end{proposition}
Let us notice that \eqref{compa1} implies that 
$$\lim_{p\ar +\infty} \mathfrak L_{k,p}(\Omega)=\mathfrak L_k(\Omega)\,,$$
and this can be useful in the numerical approach for the determination of $\mathfrak L_k(\Omega)$.\\
Notice also the inequalities can be strict! It is the case if $\Omega$ is a disjoint union of two disks, possibly related by a thin channel (see \cite{BBH,HHOT2} for details).\\ 
In the case of the disk ${\mathcal B}\subset\mathbb R^2$, we do not know if the equality $\mathfrak L_{2,1}({\mathcal B})= \mathfrak L_{2,\infty} ({\mathcal B})$ is satisfied or not. Other aspects of this question will be discussed in Section \ref{sec.klarge}. \\

Coming back to open sets in $\mathbb R^2$, it was established recently in \cite{HH:2009} that the inequality
\begin{equation} \label{ineq.2part}
\mathfrak L_{2,1} (\Omega) < \mathfrak L_{2,\infty} (\Omega) 
\end{equation}
is ``generically'' satisfied. Moreover, we can give explicit examples (equilateral triangle) of convex domains for which this is true. This answers (by the negative) some question in \cite{BBH}. The proof (see \cite{HH:2009}) is based on the following proposition:
\begin{proposition}\label{Criterion}
Let $\Omega$ be a domain in $\mathbb R^2$ and $ k\geq 2$. Let $\mathcal D$ be a minimal $k$-partition for $\mathfrak L_k(\Omega)$ and suppose that there is a pair of neighbors $\{D_i,D_j\}$ such that for the second eigenfunction $\phi_{ij}$ of $ H(D_{ij})$ having $ D_i$ and $ D_j$ as nodal domains we have
\begin{equation}\label{unequalenergy} 
\int_{D_i} |\phi_{ij}(x,y)|^2dx dy \neq \int_{D_j}|\phi_{ij}(x,y)|^2dx dy\,.
\end{equation}
Then 
\begin{equation} 
\mathfrak L_{k,1}(\Omega) < \mathfrak L_{k,\infty} (\Omega)\,.
\end{equation}
\end{proposition}
The proof involves the Hadamard formula (see \cite{Henrot}) which concerns the variation of some simple eigenvalue of the Dirichlet Laplacian by deformation of the boundary.  Here we can make the deformation in $\partial D_i\cap \partial D_j$.\\

We recall from Proposition~\ref{prop.minspecpart} that the $\infty$-minimal $k$-partition (we write simply minimal $k$-partition in this case) is a spectral equipartition. This is not necessarily the case for a $p$-minimal $k$-partition. Nevertheless we have the following property:
\begin{proposition}
Let $\mathcal D$ be a $p$-minimal $k$-partition. If $\mathcal D$ is a spectral equipartition, then this $k$-partition is $q$-minimal for any $q \in [p,+\infty]$.
\end{proposition}
This leads us to define a real $p(k,\Omega)$ as the infimum over $p\geq1$ such that there exists a $p$-minimal $k$-equipartition. \\

\subsection{Examples}
Bourdin-Bucur-Oudet \cite{BBO} have proposed an iterative method to exhibit numerically candidates for the $1$-minimal $k$-partition. Their algorithm can be generalized to the case of the $p$-norm with $p<+\infty$ and this method has been implemented for several geometries like the square or the torus. For any $k\geq 2$ and $p\geq1$, we denote by $\underline{\mathcal D}^{k,p}$ the partition obtained numerically. Some examples of $\underline{\mathcal D}^{k,p}$ are given in Figures~\ref{fig.partkpcarre} for the square. For each partition $\underline{\mathcal D}^{k,p} = \{ {\underline D}_{i}^{k,p}\}_{1\leq i\leq k}$, we represent in Figures~\ref{fig.Lkpcarre} the eigenvalues $(\lambda({\underline D}_{i}^{k,p}))_{1\leq i\leq k}$, the energies $\Lambda_{p}(\underline{\mathcal D}^{k,p} )$ and $\Lambda_{\infty}(\underline{\mathcal D}^{k,p})$. We observe that for the case of the square, if $k\notin\{1,2,4\}$ (that is to say, if we are not in the Courant sharp situation), then the partitions obtained numerically are not spectral equipartitions for any $p<+\infty$. In the case $k=3$, the first picture of Figure~\ref{fig.partkpcarre} suggests that the triple point (which is not at the center for $p=1$) moves to the center as $p\to+\infty$. Consequently, the $p$-minimal $k$-partition can not be optimal for $p=+ \infty$ and 
$$p(k,\square)=+\infty\quad\mbox{ for }k \in \{3,5,6,7,8\}.$$
Conversely, for any $p$, the algorithm produces a nodal partition when $k=2$ and $k=4$. This suggests
$$p(k,\square)=1\quad\mbox{ for }k\in\{1,2,4\}.$$

\begin{figure}[h!t]
\begin{center}
\includegraphics[height=3cm]{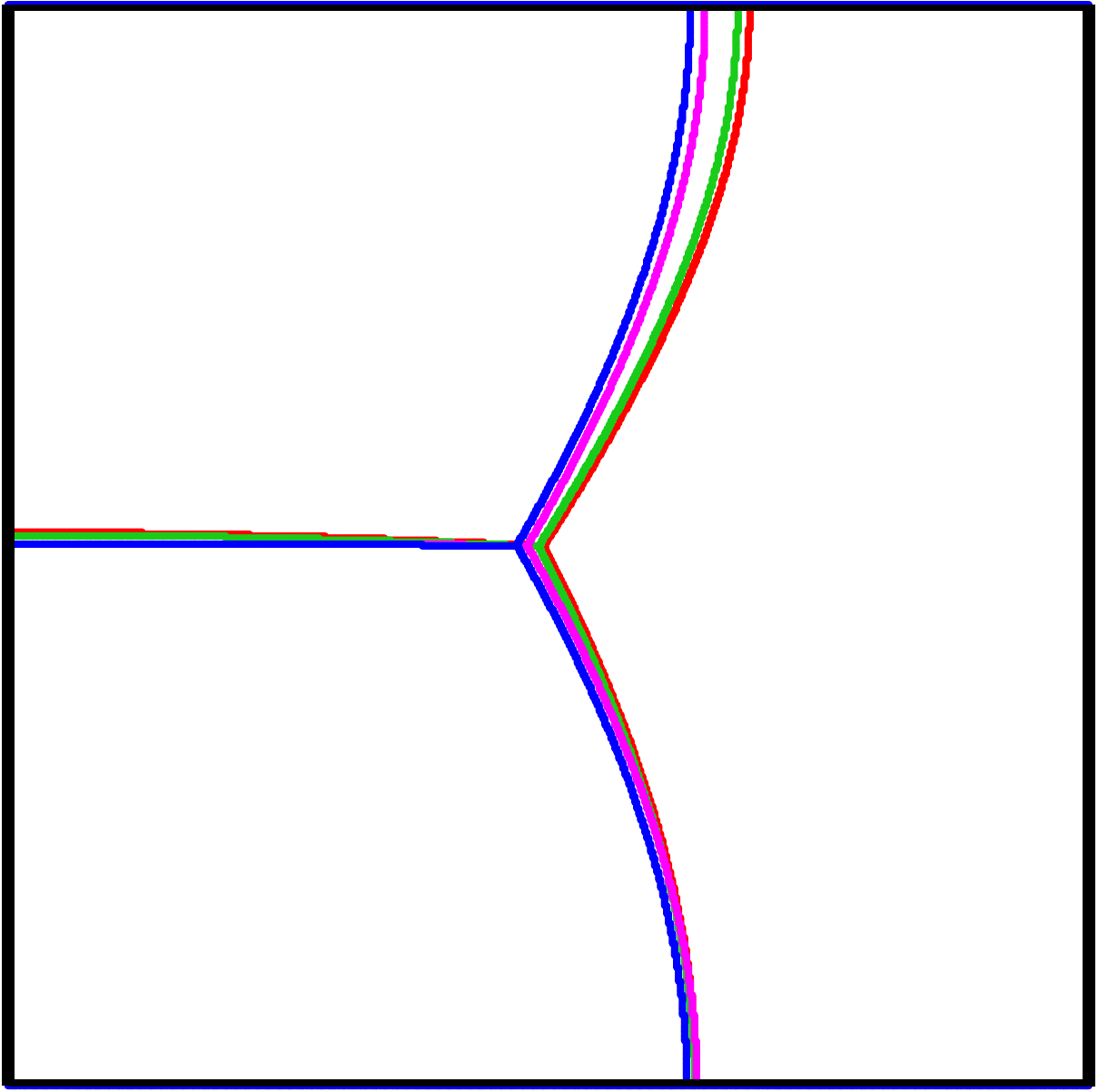}\qquad
\includegraphics[height=3cm]{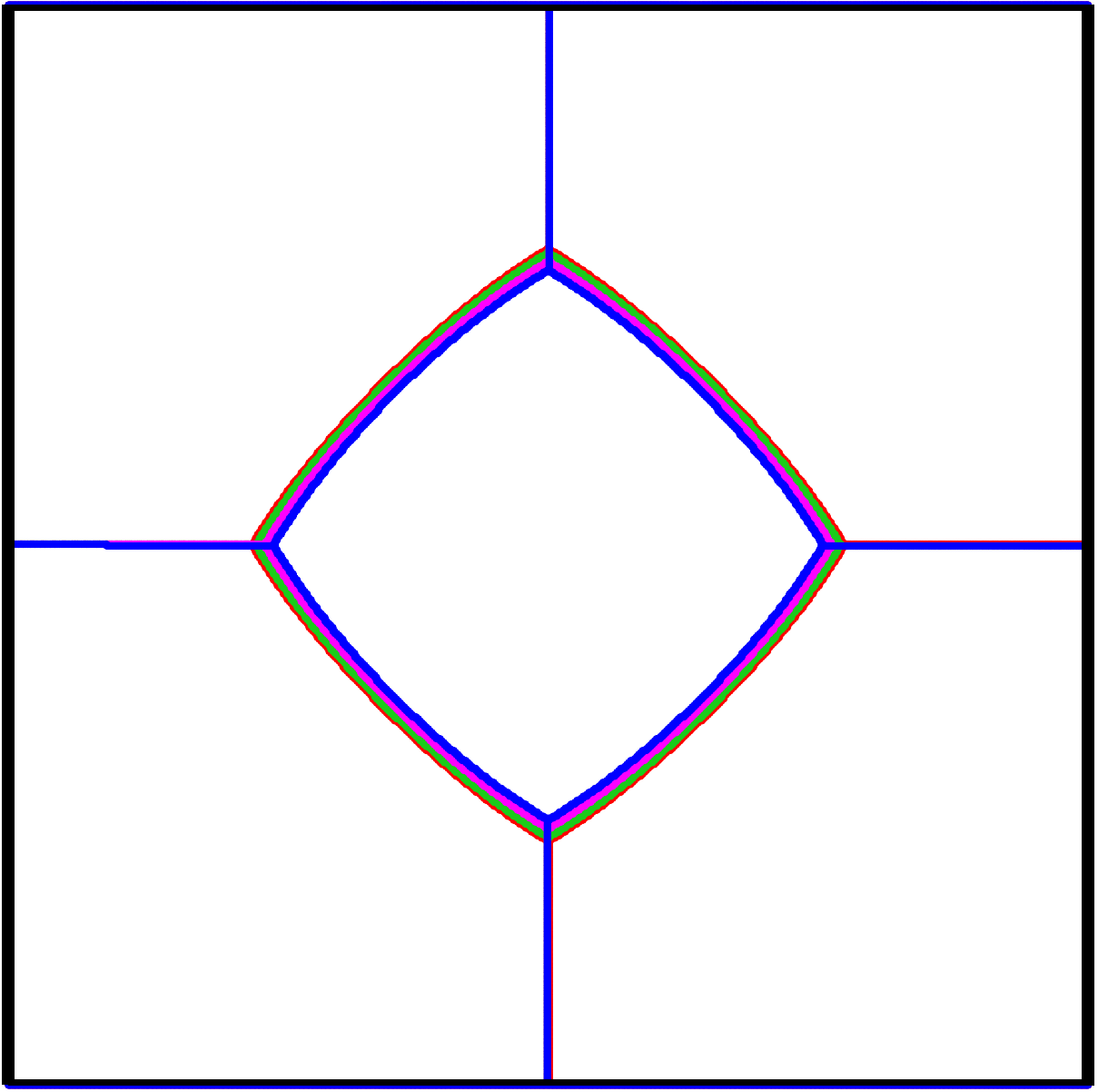}\qquad
\includegraphics[height=3cm]{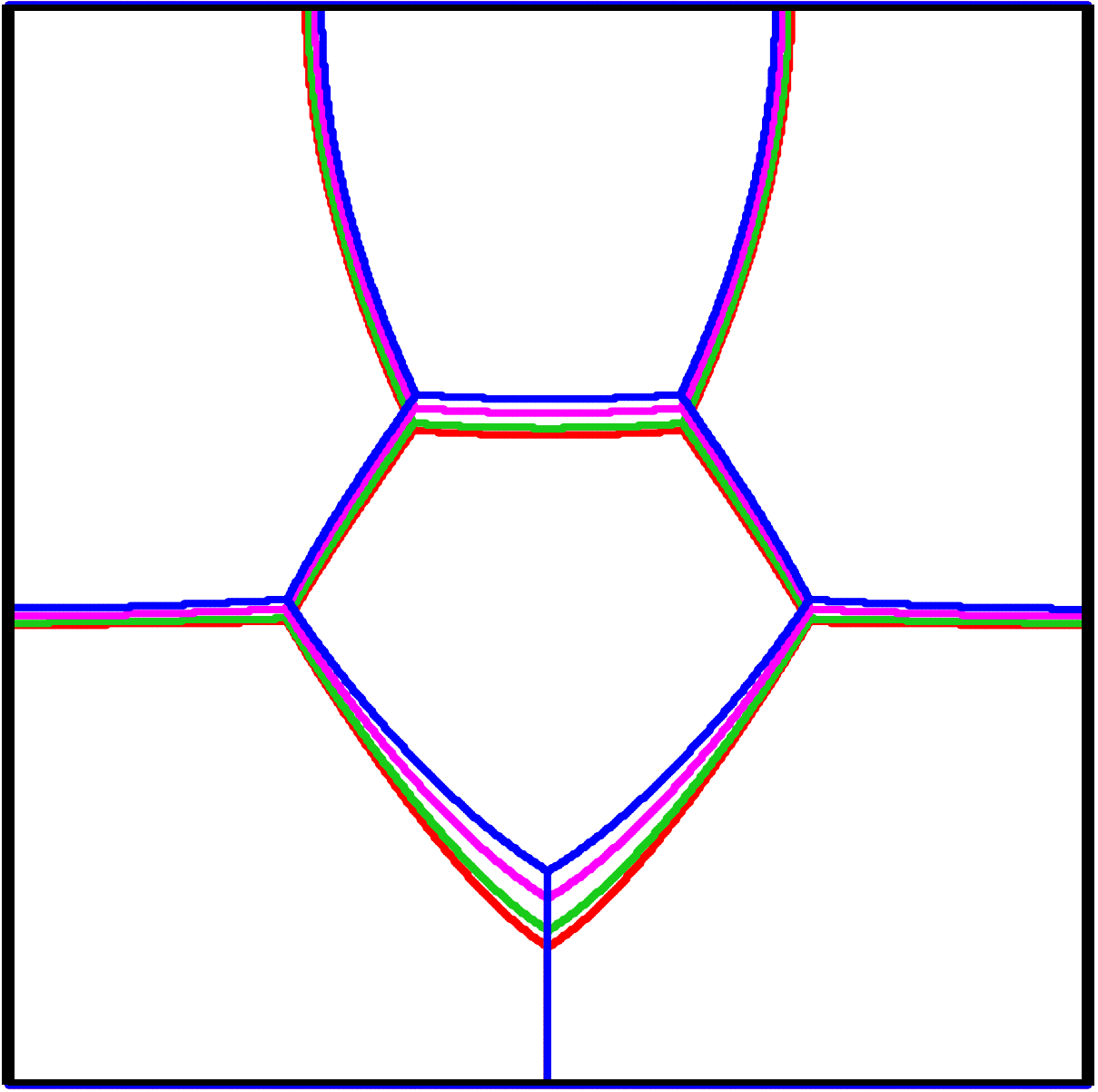}
\caption{Candidates $\underline{\mathcal D}^{k,p}$ for the $p$-minimal $k$-partition on the square, $p=1,2,5,10$ (in {\Bl blue}, {\Mg magenta}, {\Gr green} and {\Rd red} respectively), $k=3,5,6$.\label{fig.partkpcarre}}
\end{center}
\end{figure}
\begin{figure}[h!t]
\begin{center}
\begin{tabular}{ccc}
$k=3$ & $k=5$ & $k=6$\\
\includegraphics[height=4cm]{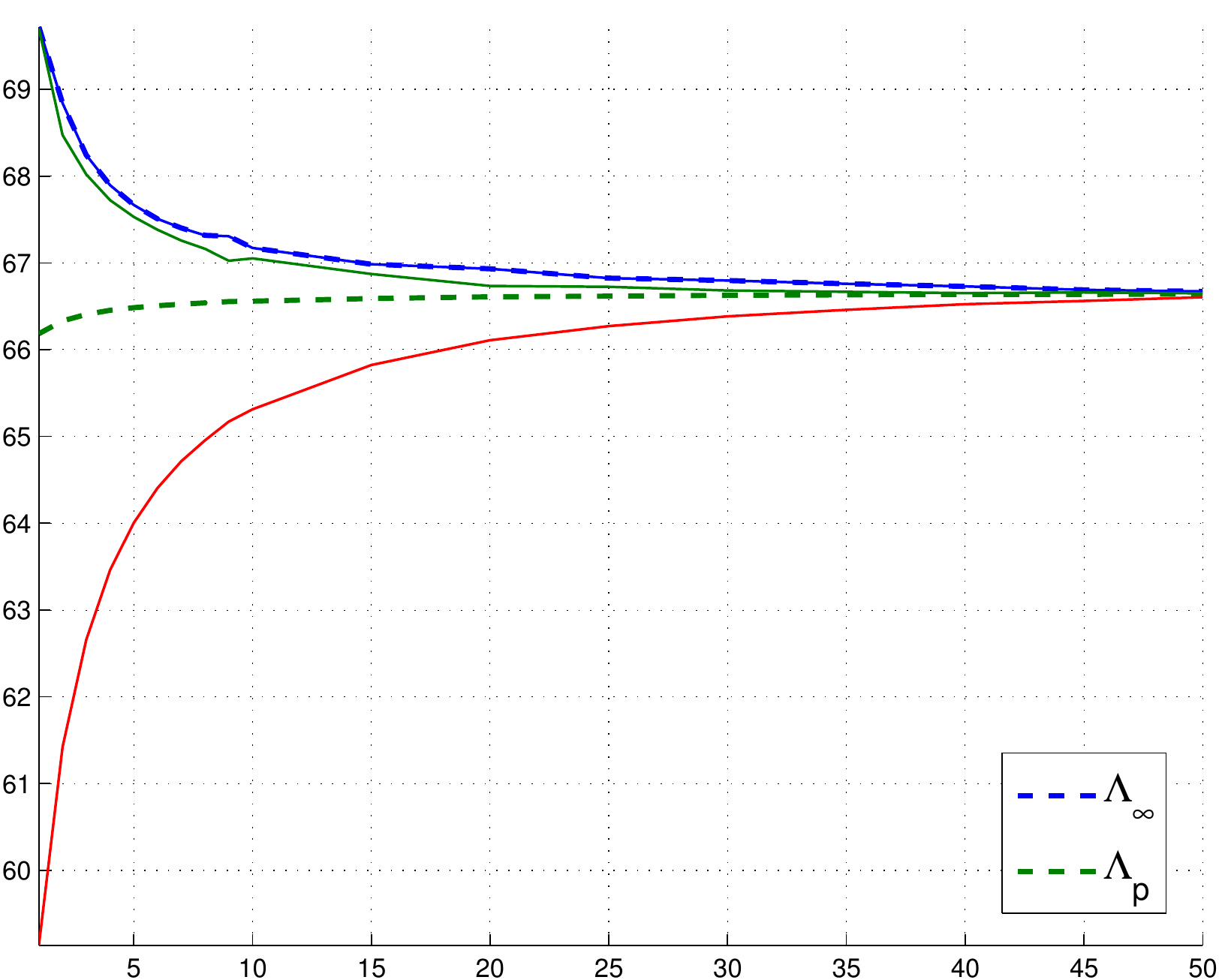}
&\includegraphics[height=4cm]{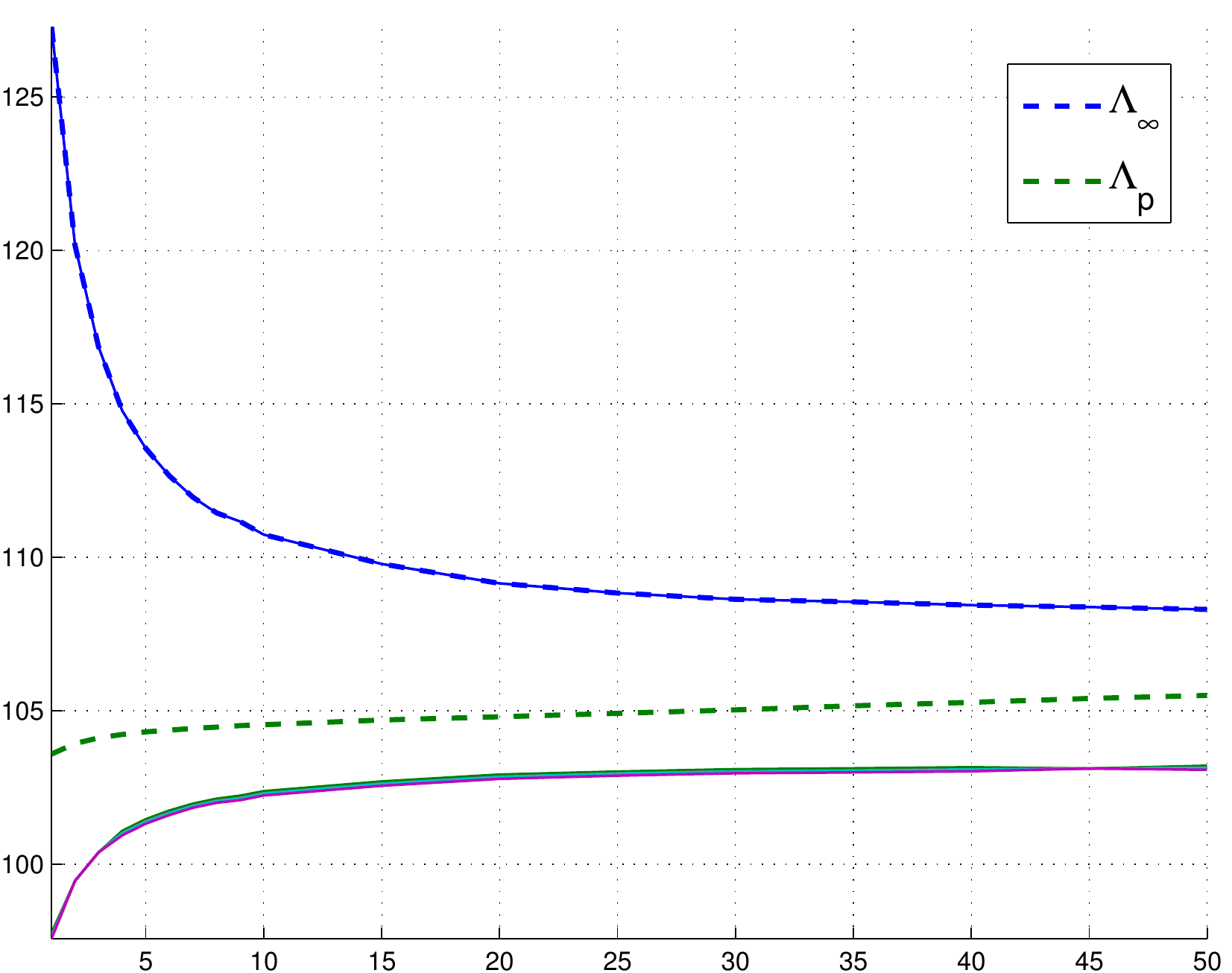}
&\includegraphics[height=4cm]{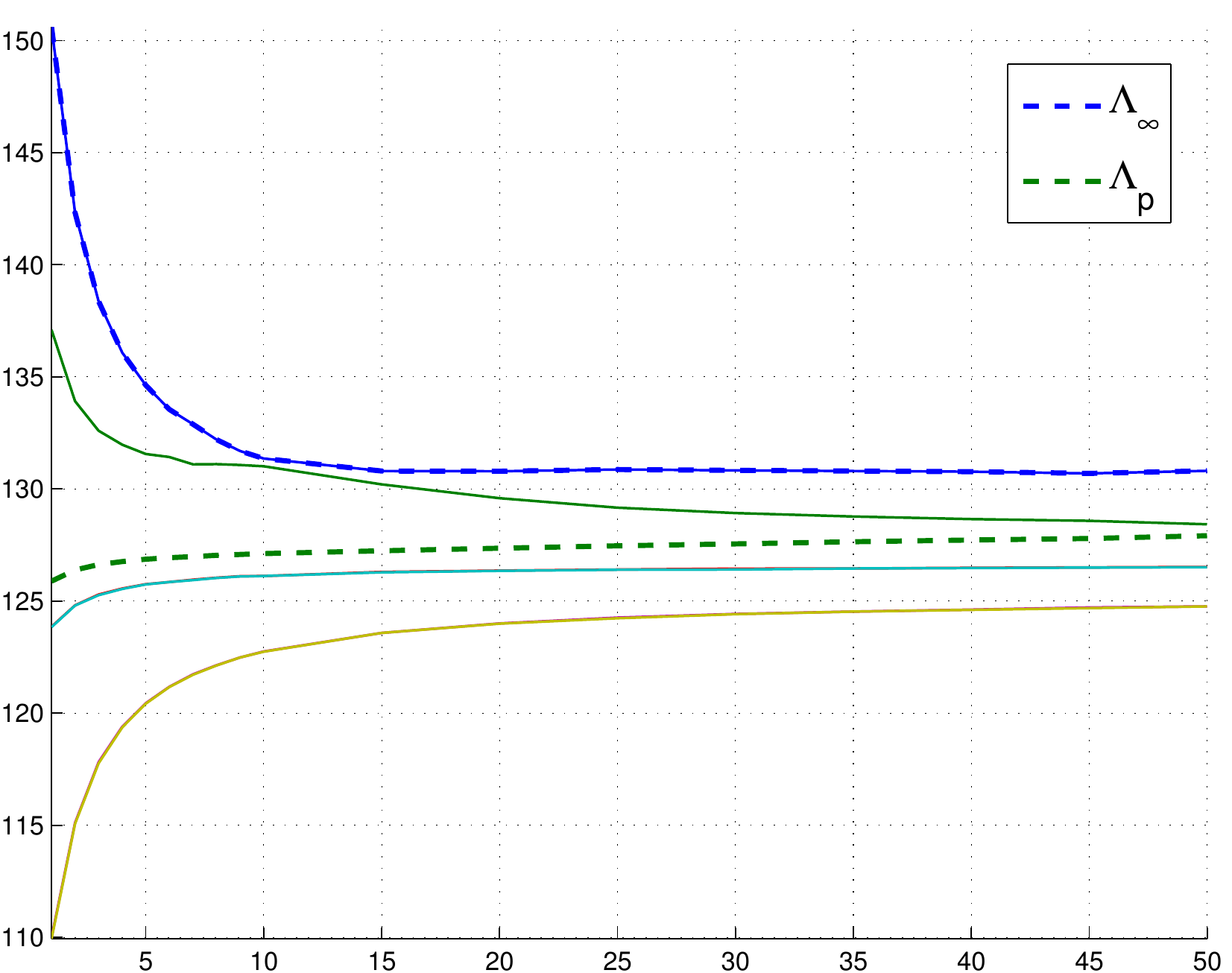}
\end{tabular}
\caption{Energies of $\underline{\mathcal D}^{k,p}$ according to $p$, $k=3,5,6$ on the square.\label{fig.Lkpcarre}}
\end{center}
\end{figure}
 
In the case of the isotropic torus $(\mathbb R/\mathbb Z)^2$, numerical simulations in \cite{Len3} for $3\leq k\leq 6$ suggest that the $p$-minimal $k$-partition is a spectral equipartition for any $p\geq1$ and thus 
$$p(k,(\mathbb R/\mathbb Z)^2)=1\quad\mbox{ for }k=3,4,5,6.$$
Candidates are given in Figure~\ref{fig.torepart}  where we color two neighbors with two different colors, using the minimal number of colors.
\begin{figure}[h!bt]
\begin{center}
\includegraphics[height=3cm]{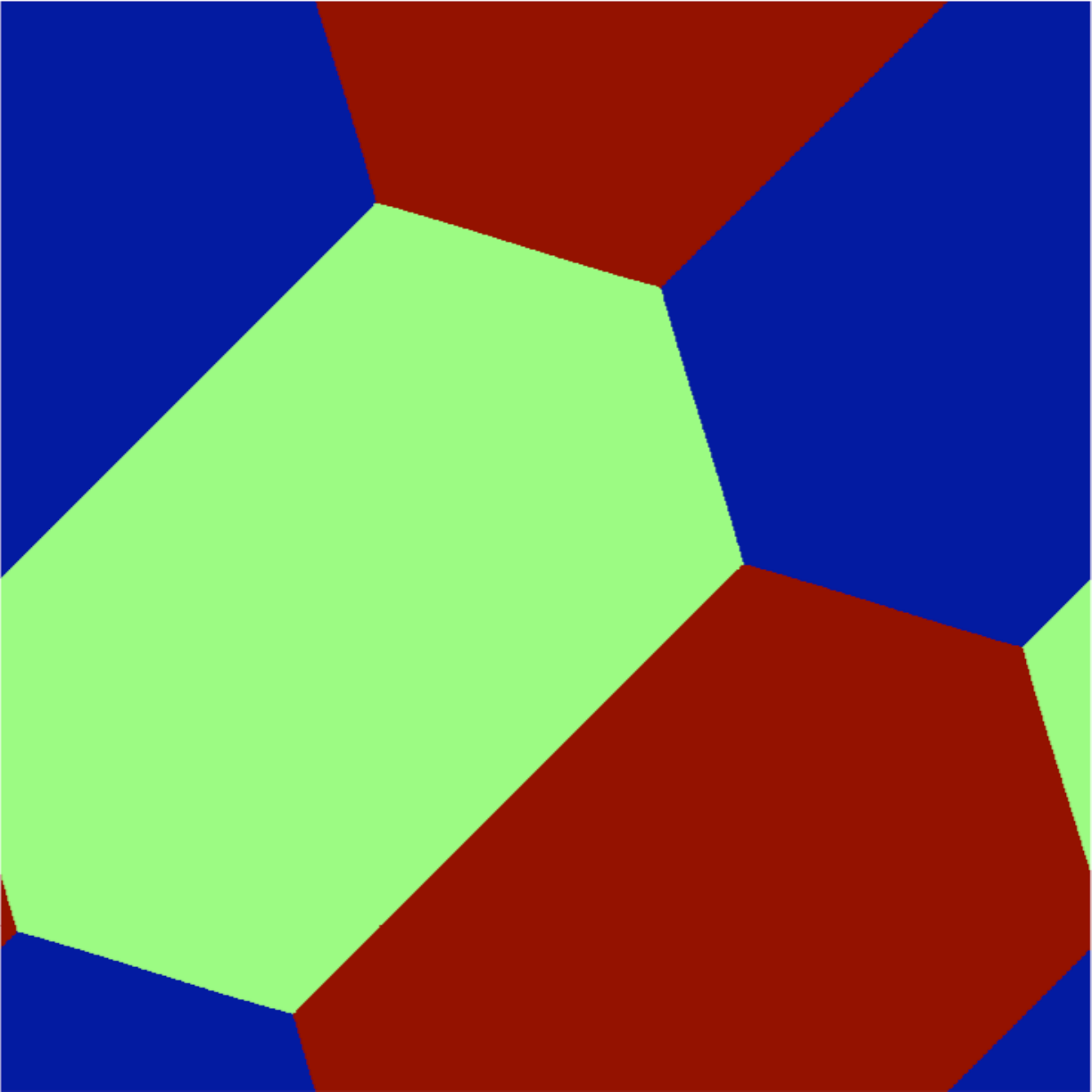}
$\qquad$\includegraphics[height=3cm]{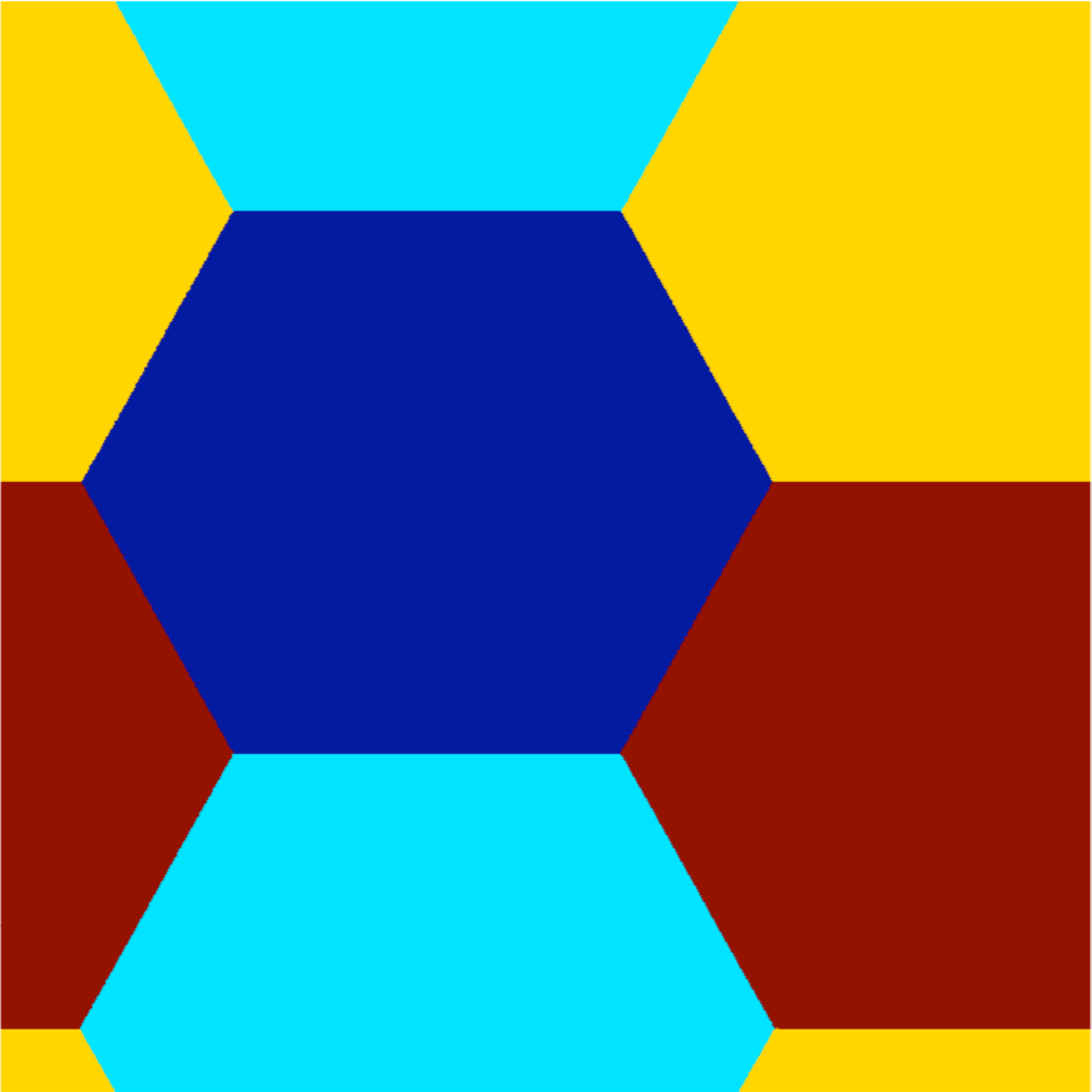}
$\qquad$\includegraphics[height=3cm]{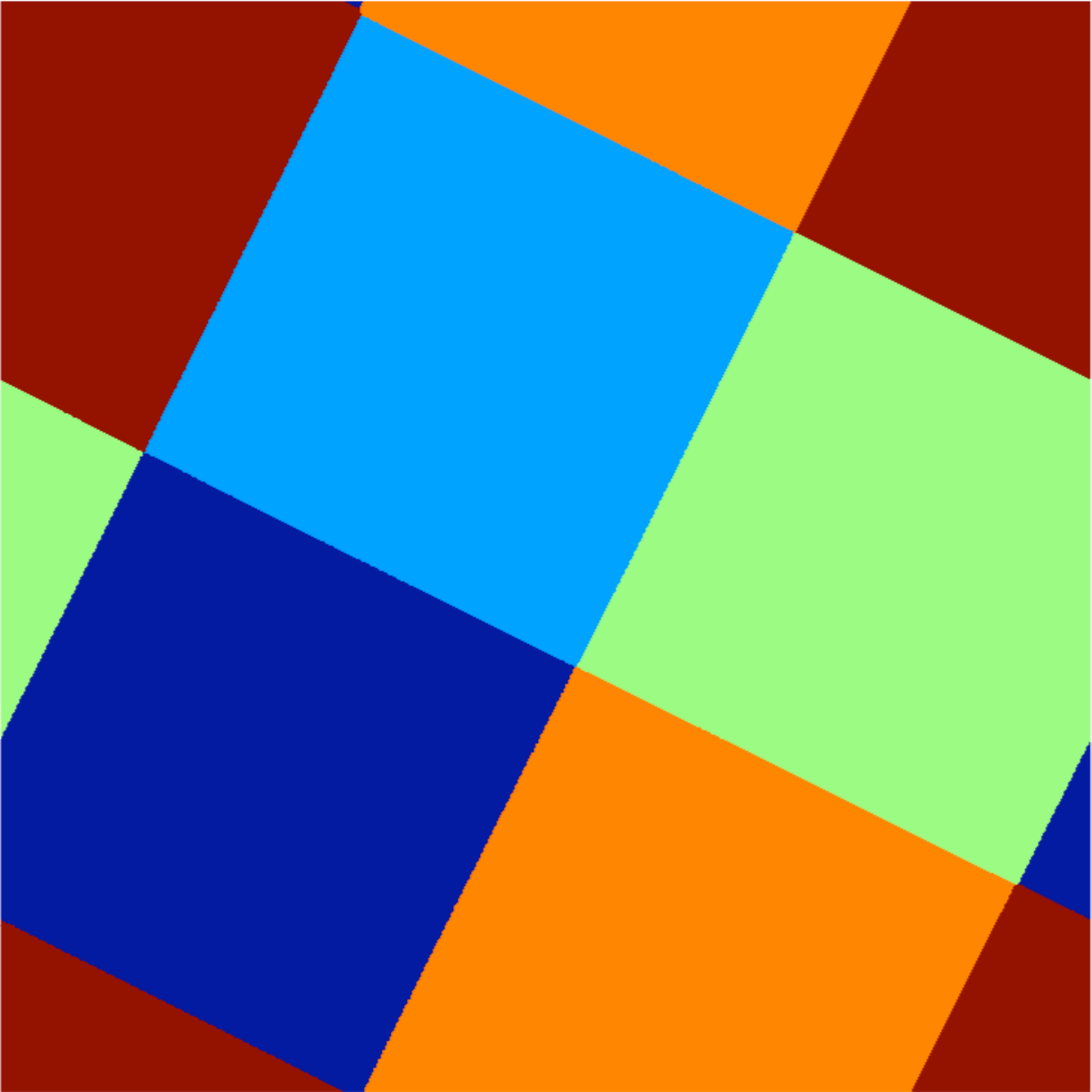}
$\qquad$\includegraphics[height=3cm]{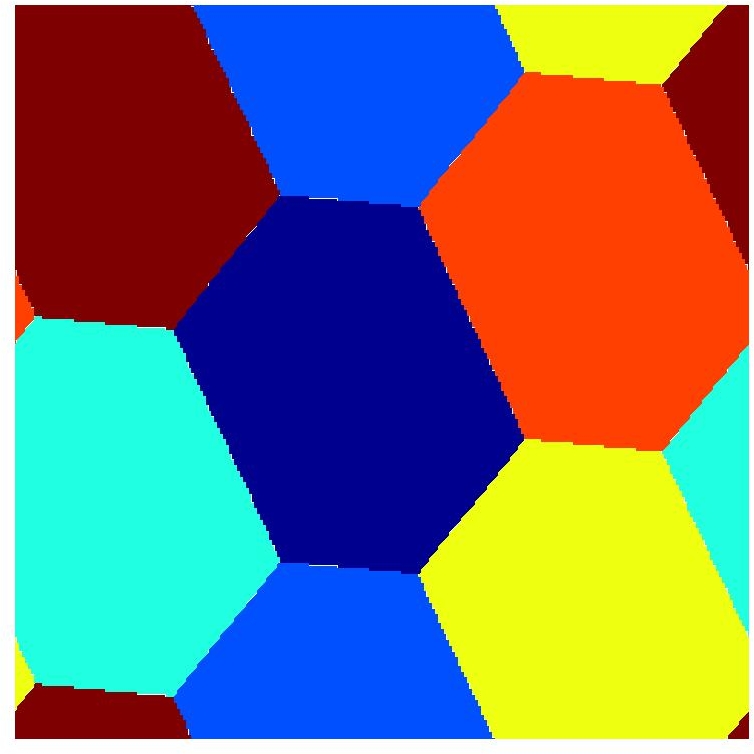}
\caption{Candidates for isotropic torus, $k=3,4,5,6$.\label{fig.torepart}}
\end{center}
\end{figure}

\subsection{Notes}\label{ss5.3}
In the case of the sphere $\mathbb S^2$, it was proved that \eqref{ineq.2part} is an equality (see \cite{Bis, FrHa} and the references in \cite{HHOT2}). But this is the only known case for which the equality is proved. For $k=3$, it is a conjecture reinforced by the numerical simulations of Elliott-Ranner \cite{ER15} or more recently of B. Bogosel \cite{Bogosel}. For $k=4$, simulations produced by Elliott-Ranner \cite{ER15}  for $\mathfrak L_{4,1}$ suggest that the spherical tetrahedron is a good candidate for a $1$-minimal $4$-partition. For $k >6$, it seems that the candidates for the $1$-minimal $k$-partitions are not spectral equipartitions.

\section{Topology of regular partitions}\label{s6}
\subsection{Euler's formula for regular partitions}\label{ss6.1}
In the case of planar domains (see \cite{HH:2005a}), we will use the following result. 
\begin{proposition}\label{Euler}
Let $\Omega$ be an open set in $\mathbb R^2$ with piecewise $C^{1,+}$ boundary and $\mathcal D$ be a $k$-partition with $\partial \mathcal D $ the boundary set (see Definition~\ref{Definition4} and notation therein). Let $b_0$ be the number of components of $\partial \Omega$ and $b_1$ be the number of components of $\partial \mathcal D \cup\partial \Omega$. Denote by $\nu(\xb_i)$ and $\rho(\yb_i)$ the numbers of curves ending at $\xb_i\in X(\partial \mathcal D)$, respectively $\yb_i\in Y(\partial \mathcal D)$. Then
\begin{equation}\label{Emu}
k= 1 + b_1-b_0+\sum_{\xb_i\in X(\partial \mathcal D )}\Big(\frac{\nu(\xb_i)}{2}-1\Big)
+\frac{1}{2}\sum_{\yb_i\in Y(\partial \mathcal D )}\rho(\yb_i)\,.
\end{equation}
\end{proposition} 

This can be applied, together with other arguments to determine upper bounds for the number of singular points of minimal partitions. This version of the Euler's formula appears in \cite{HOMN} and can be recovered by the Gauss-Bonnet formula (see for example \cite{BeHe0}). There is a corresponding result for compact manifolds involving the Euler characteristics.

\begin{proposition}
 Let $M$ be a flat compact surface $M$ without boundary. Then the Euler's formula for a partition $\mathcal D=\{D_{i}\}_{1\leq i\leq k}$ reads
$$
\sum_{i=1}^k \chi(D_{i})= \chi(M)+\sum_{\xb _i\in X(\partial \mathcal D )}\Big(\frac{\nu(\xb_i)}{2}-1\Big)\,,
$$
where $\chi$ denotes the Euler characteristics.
\end{proposition}
It is well known that $\chi (\mathbb S^2)=2$, $\chi (\mathbb T^2)=0$ and that for open sets in $\mathbb R^2$ the Euler characteristic is $1$ for the disk and $0$ for the annulus.

\subsection{Application to regular $3$-partitions}\label{ss6.2}
Following \cite{HH:2006}, we describe here the possible ``topological'' types of non bipartite minimal $3$-partitions for a general domain $\Omega$ in $\mathbb R^2$.
\begin{proposition}\label{PropositionA}
Let $\Omega$ be a simply-connected domain in $\mathbb R^2$ and consider a minimal $3$-partition $\mathcal D=\{D_1,D_2,D_3\}$ associated with $\mathfrak L_3(\Omega)$ and suppose that it is not bipartite. Then the boundary set $\partial \mathcal D $ has one of the following properties:
\begin{enumerate}[label={\rm[\alph*]},itemsep=0pt]
\item one interior singular point $\xb_0 \in\Omega$ with $\nu(\xb_0)=3$, three points $\{\yb_i\}_{1\leq i\leq3}$ on the boundary $\partial\Omega$ with $\rho(\yb_i)=1$;
\item two interior singular points $\xb_0,\ \xb_1\in\Omega$ with $\nu(\xb_0)=\nu(\xb_1)=3$ and two boundary singular points $\yb_1,\ \yb_2\in\partial\Omega$ with $\rho(\yb_1)=1=\rho(\yb_2)$;
\item two interior singular points $\xb_0,\ \xb_1\in\Omega$ with $\nu(\xb_0)=\nu(\xb_1)=3$ and no singular point on the boundary.
\end{enumerate}
\end{proposition}
The three types are described in Figure \ref{fig.configEuler}.
\begin{figure}[h!bt]
\begin{center}
\includegraphics[height=2cm]{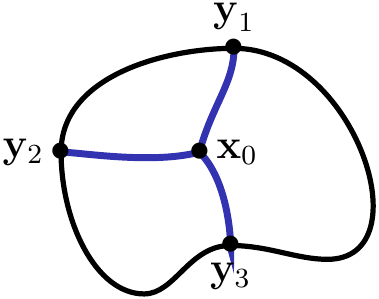}
\qquad 
\includegraphics[height=2cm]{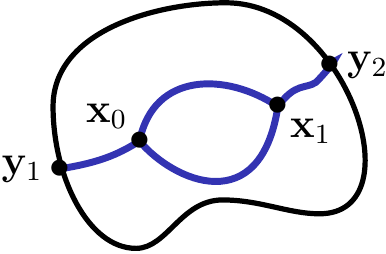}
\qquad 
\includegraphics[height=2cm]{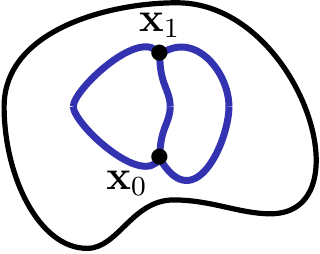}
\caption{Three topological types : [a], [b] and [c].\label{fig.configEuler}}
\end{center}
 \end{figure}

The proof of Proposition \ref{PropositionA} relies essentially on the Euler formula. This leads (with some success) to analyze the minimal $3$-partitions with some topological type. We actually do not know any example where the minimal $3$-partitions are of type [b] and [c]. Numerical computations never produce candidates of type [b] or [c] (see \cite{BHV} for the square and the disk, \cite{BL} for angular sectors and \cite{BoH2} for complements for the disk).\\
Note also that we do not know about results claiming that the minimal $3$-partition of a domain with symmetry should keep some of these symmetries. We actually know in the case of the disk (see \cite[Proposition 1.6]{HH:2006}) that a minimal $3$-partition cannot keep all the symmetries.\\
In the case of angular sectors, it has been proved in \cite{BL} that a minimal $3$-partition can not be symmetric for some range of $\omega$'s.\\

\subsection{Upper bound for the number of singular points \label{ss6.3}}
\begin{proposition}\label{prop.uppk}
Let $\mathcal D$ be a minimal $k$-partition of a simply connected domain $\Omega$ with $k\geq 2$. Let $X^{\sf odd}(\partial \mathcal D )$ be the subset among the interior singular points $X(\partial\mathcal D)$ for which $\nu(\xb_{i})$ is odd (see Definition~\ref{Definition4}). Then the cardinal of $X^{\sf odd}(\partial \mathcal D)$ satisfies
\begin{equation}\label{eulergrandk}
\sharp X^{\sf odd}(\partial \mathcal D ) \leq 2k -4 \,. 
\end{equation}
\end{proposition}
\begin{proof}
Euler's formula implies that for a minimal $k$-partition $\mathcal D$ of a simply connected domain $\Omega$ the cardinal of $X^{\sf odd}(\partial \mathcal D )$ satisfies
\begin{equation}
\sharp X^{\sf odd}(\partial\mathcal D) \leq 2k -2\,. 
\end{equation}
Note that if $b_1=b_0$, we necessarily have a singular point in the boundary. If we implement the property that the open sets of the partitions are nice, we can exclude the case when there is only one point on the boundary. Hence, we obtain 
$$b_1-b_0 + \frac 12 \sum_{i} \rho(\yb_i) \geq 1\,, $$
which implies \eqref{eulergrandk}.
\end{proof}

\subsection{Notes}\label{ss6.4}
In the case of $\mathbb S^2$ one can prove that a minimal $3$-partition is not nodal (the second eigenvalue has multiplicity $3$), and as a step to a characterization, one can show that non-nodal minimal partitions have necessarily two singular \emph{triple points} (i.e. with $\nu(\xb)=3$).\\
If we assume, for some $k\geq 12$, that a minimal $k$-partition has only singular triple points and consists only of (spherical) pentagons and hexagons, then Euler's formula in  its historical version for convex polyedra $V-E+F= \chi(\mathbb S^2) =2$ (where $F$ is the number of faces, $E$ the number of edges and $V$ the number of vertices) implies that the number of pentagons is $12$. This is what is used for example for the soccer ball ($12$ pentagons and $20$ hexagons). We refer to \cite{ER15} for enlightening pictures. \\
More recently, it has been proved by Soave-Terracini \cite[Theorem 1.12]{ST14} that 
$$\mathfrak L_{3}(\mathbb S^n) = \frac 32\left(n+\frac 12\right).$$

%%%%%%%%%%%%%%%%%%%%%%%%%%%%
\section{Examples of minimal $k$-partitions}\label{sexamples}
\subsection{The disk}
In the case of the disk, Proposition~\ref{prop.nodalDisk} tells us that the minimal $k$-partition are nodal only for $k=1,2,4$. Illustrations are given in Figure~\ref{fig.partminDisk}.
\begin{figure}[h!bt]
\begin{center}
\subfigure[{Minimal $k$-partitions, $k=1,2,4$.\label{fig.partminDisk}}]{\begin{tabular}{ccc}\ \\
\includegraphics[height=1.7cm]{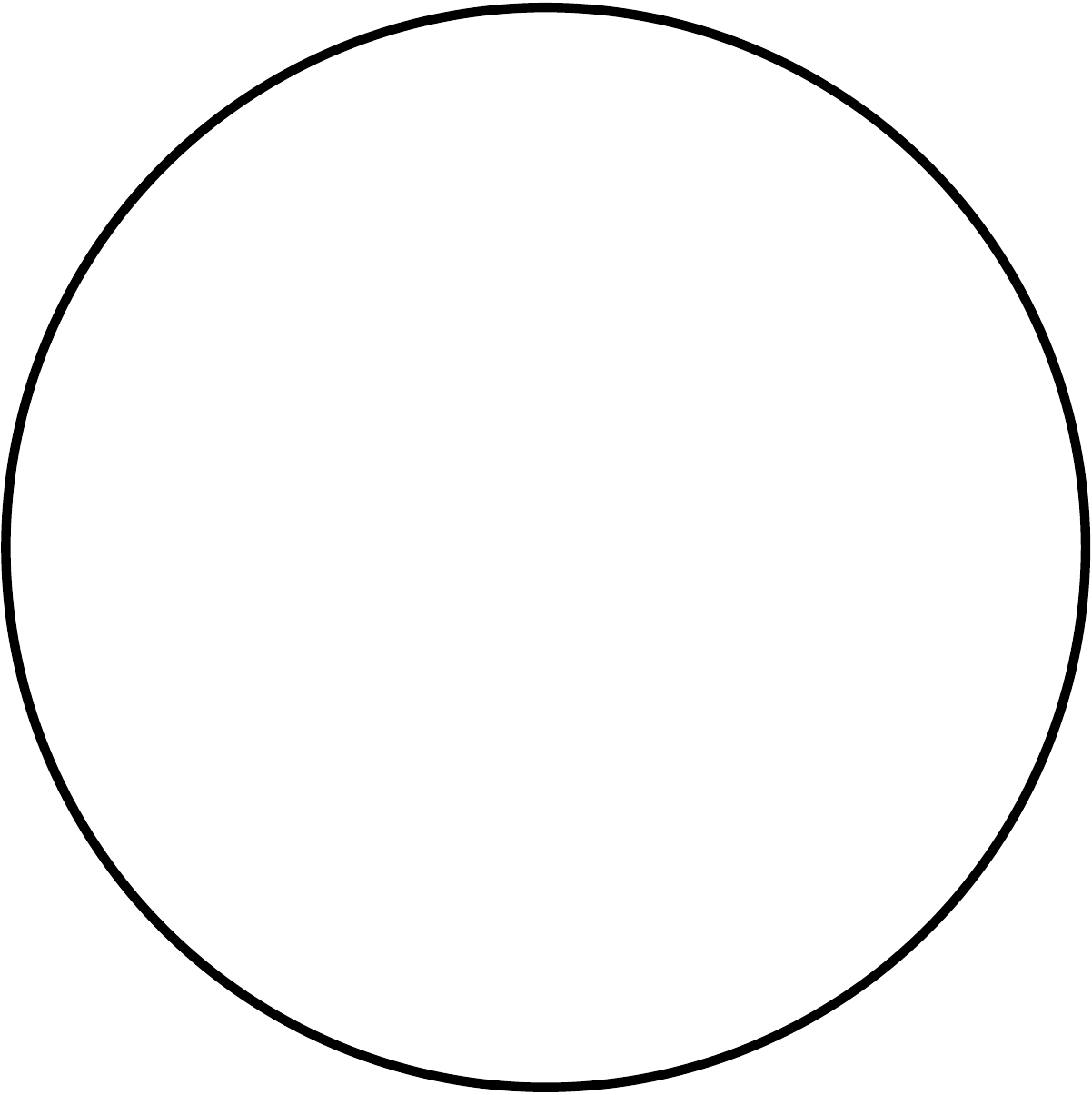}
&\includegraphics[height=1.7cm]{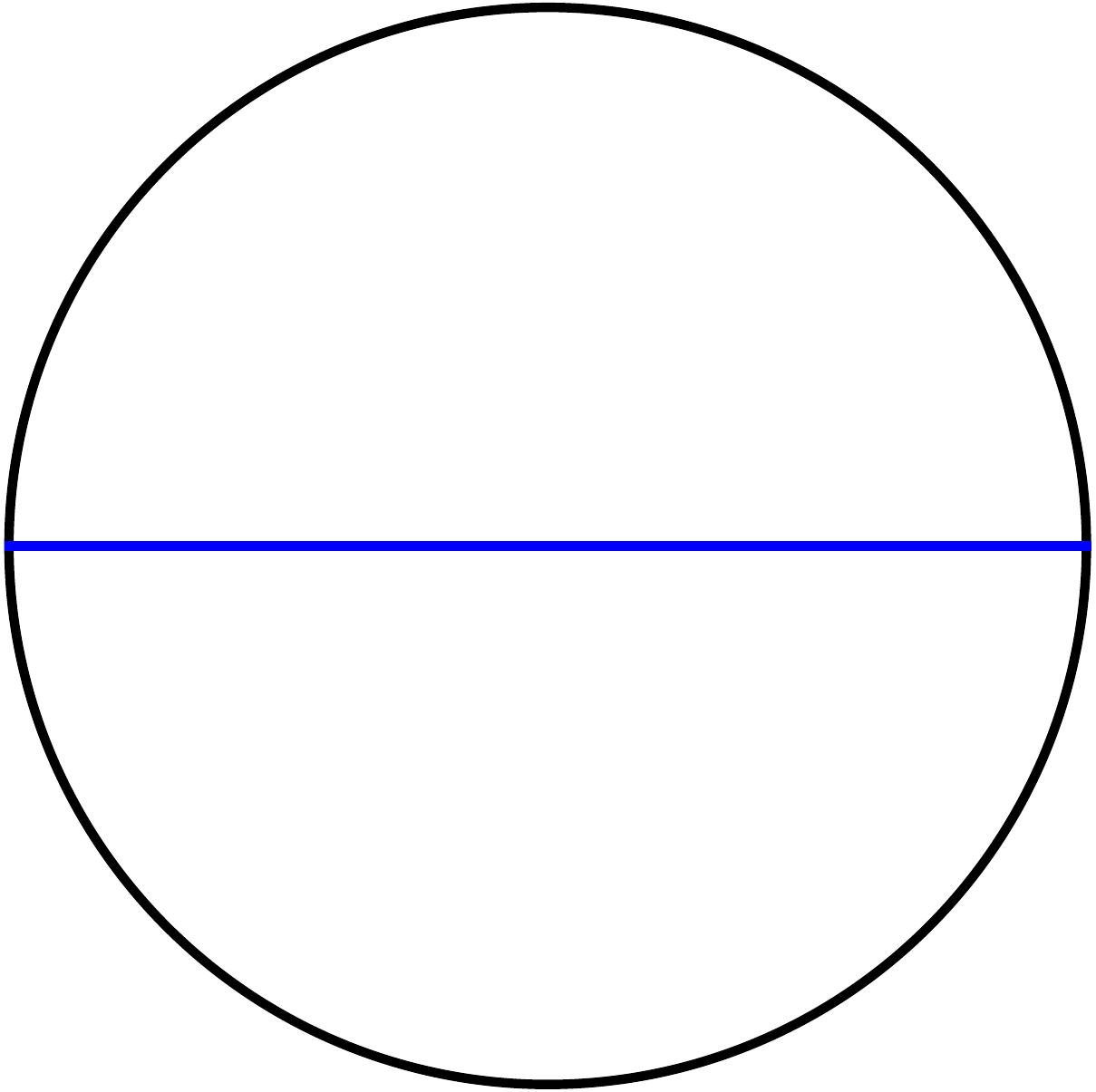}
&\includegraphics[height=1.7cm]{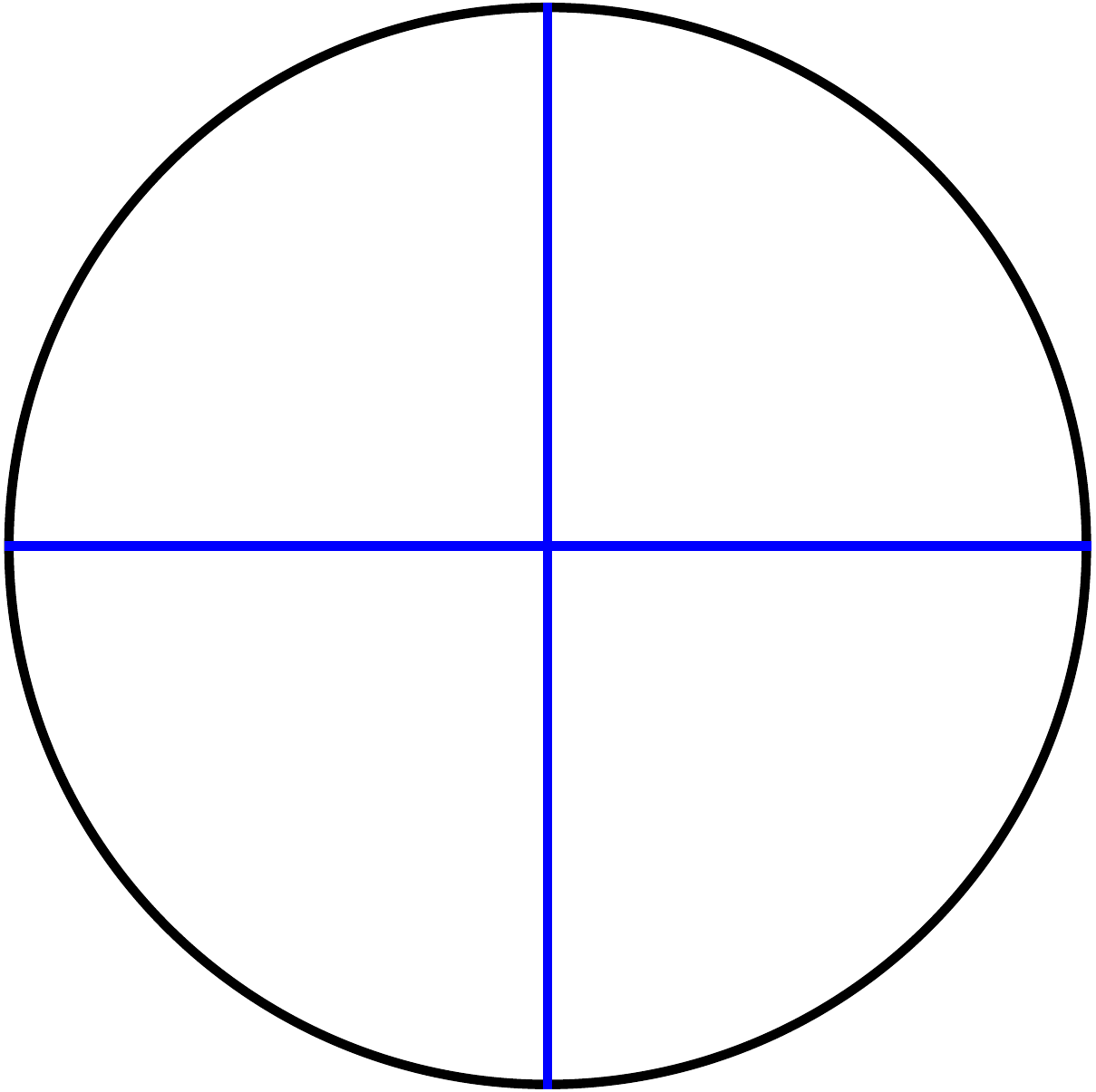}\end{tabular}}
\subfigure[$3$-partition.\label{fig.MS}]{\quad\begin{tabular}{c}\ \\ \includegraphics[height=1.7cm,angle=90]{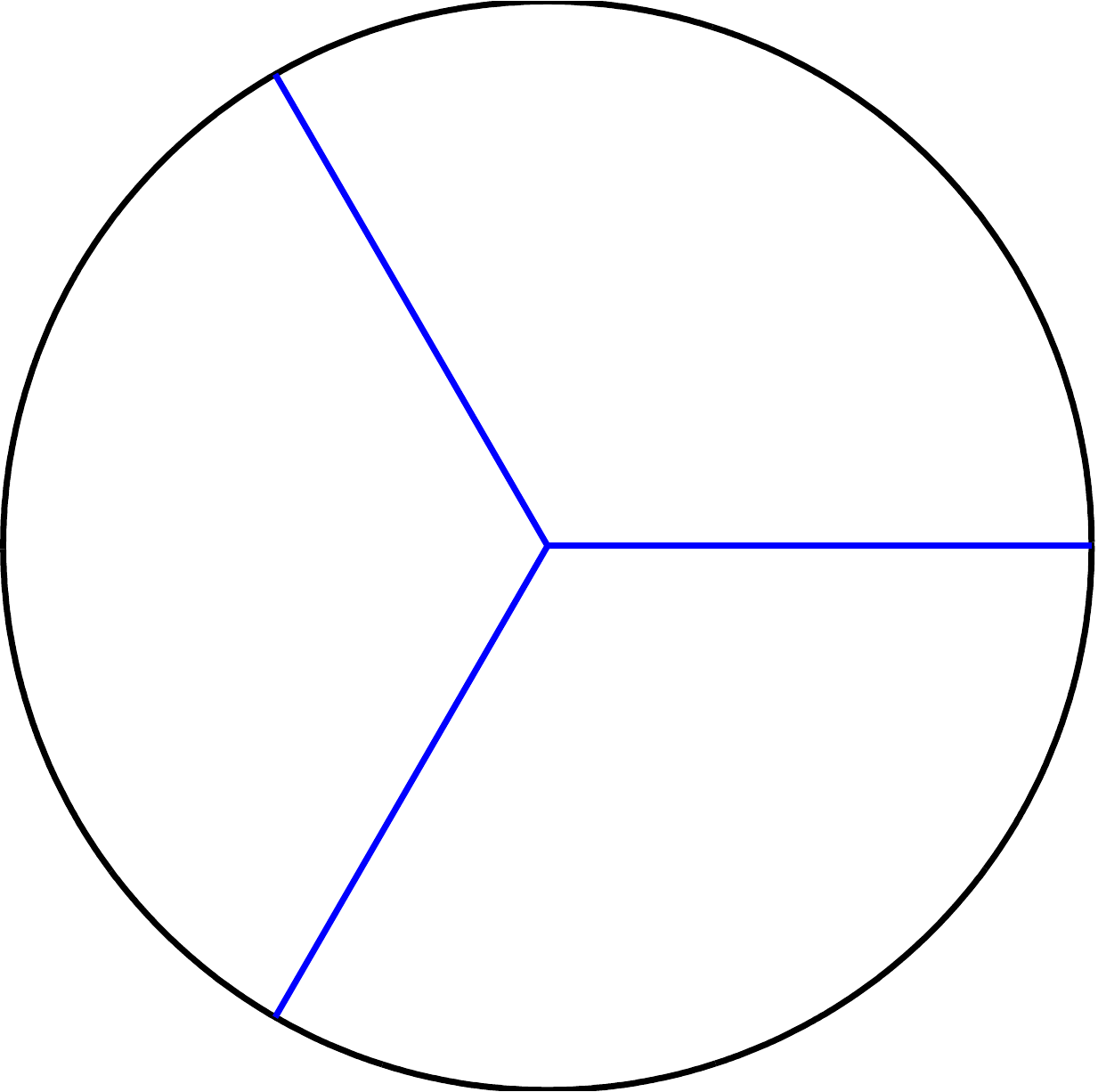}\end{tabular}}
\subfigure[Energy for two $5$-partitions.]{\begin{tabular}{cc}
$\small\Lambda(\mathcal D) \simeq 104.37$ & $\Lambda(\mathcal D) \simeq 110.83$\\
\includegraphics[height=1.7cm]{disque1.pdf}
&\includegraphics[height=1.7cm]{disque2.pdf}
\end{tabular}}
\caption{Candidates for the disk.\label{fig.partdisk}}
\end{center}
\end{figure}
For other $k$'s, the question is open. Numerical simulations in \cite{BH,BoH2} permit to exhibit candidates for the $\mathfrak L_{k}({\mathcal B}_{1})$ for $k=3,5$ (see Figure~\ref{fig.partdisk}). Nevertheless we have no proof that the minimal $ 3$-partition is the ``Mercedes star'' (see Figure~\ref{fig.MS}), except if we assume that the center belongs to the boundary set of the minimal partition \cite{HH:2006} or if we assume that the minimal $3$-partition is of type [a] (see \cite[Proposition 1.4]{BoH2}).

\subsection{The square}
 When $\Omega$ is a square, the only cases which are completely solved are $k=1,2,4$ as mentioned in Theorem~\ref{thm.partnodalcarre} and the minimal $k$-partitions for $k=2,4$ are presented in the Figure~\ref{fig.bip}. Let us now discuss the $3$-partitions. It is not too difficult to see that $\mathfrak L_3$ is strictly less than $ L_3$. We observe indeed that $\lambda_4$ is Courant sharp, so $\mathfrak L_4 =\lambda_4$, and there is no eigenfunction corresponding to $\lambda_2=\lambda_3$ with three nodal domains (by Courant's Theorem). Restricting to the half-square and assuming that there is a minimal partition which is symmetric with one of the perpendicular bisectors of one side of the square or with one diagonal line, one is reduced to analyze a family of problems with mixed conditions on the symmetry axis (Dirichlet-Neumann, Dirichlet-Neumann-Dirichlet or Neumann-Dirichlet-Neumann according to the type of the configuration ([a], [b] or [c] respectively). Numerical computations\footnote{see \href{http://w3.bretagne.ens-cachan.fr/math/simulations/MinimalPartitions/}{\sf http://w3.bretagne.ens-cachan.fr/math/simulations/MinimalPartitions/}} in \cite{BHV} produce natural candidates for a symmetric minimal $3$-partition.
\begin{figure}[h!bt]
 \begin{center}
\includegraphics[width=3cm]{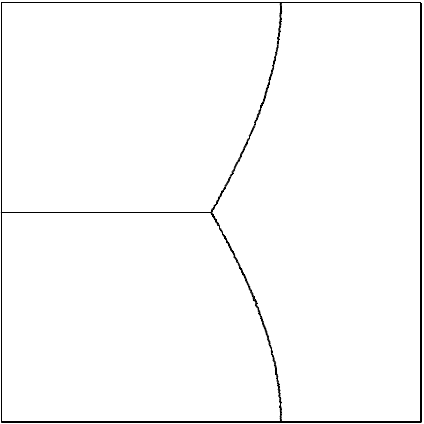}$\qquad$
\includegraphics[width=3cm]{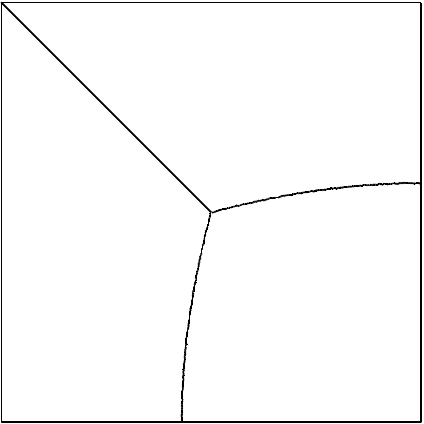} 
\caption{Candidates $\mathcal D^{\sf perp}$ and $\mathcal D^{\sf diag}$ for the square.\label{fig.carrecand1}}
\end{center}
\end{figure}
Two candidates $\mathcal D^{\sf perp}$ and $\mathcal D^{\sf diag}$ are obtained numerically by choosing the symmetry axis (perpendicular bisector or diagonal line) and represented in Figure \ref{fig.carrecand1}. Numerics suggests that there is no candidate of type [b] or [c], that the two candidates $\mathcal D^{\sf perp}$ and $\mathcal D^{\sf diag}$ have the same energy $\Lambda(\mathcal D^{\sf perp})\simeq\Lambda(\mathcal D^{\sf diag})$ and that the center is the unique singular point of the partition inside the square. Once this last property is accepted, one can perform the spectral analysis of an Aharonov-Bohm operator (see Section~\ref{s8}) with a pole at the center. This point of view is explored numerically in a rather systematic way by Bonnaillie-No\"el--Helffer \cite{BH} and theoretically by Noris-Terracini \cite{NT} (see also \cite{BNNT}). In particular, it was proved that, if the singular point is at the center, the mixed Dirichlet-Neumann problems on the two half-squares (a rectangle and a right angled isosceles triangle depending on the considered symmetry) are isospectral with the Aharonov-Bohm operator. This explains why the two partitions $\mathcal D^{\sf perp}$ and $\mathcal D^{\sf diag}$ have the same energy.
\begin{figure}[h!bt]
\begin{center}
\includegraphics[height=1.5cm]{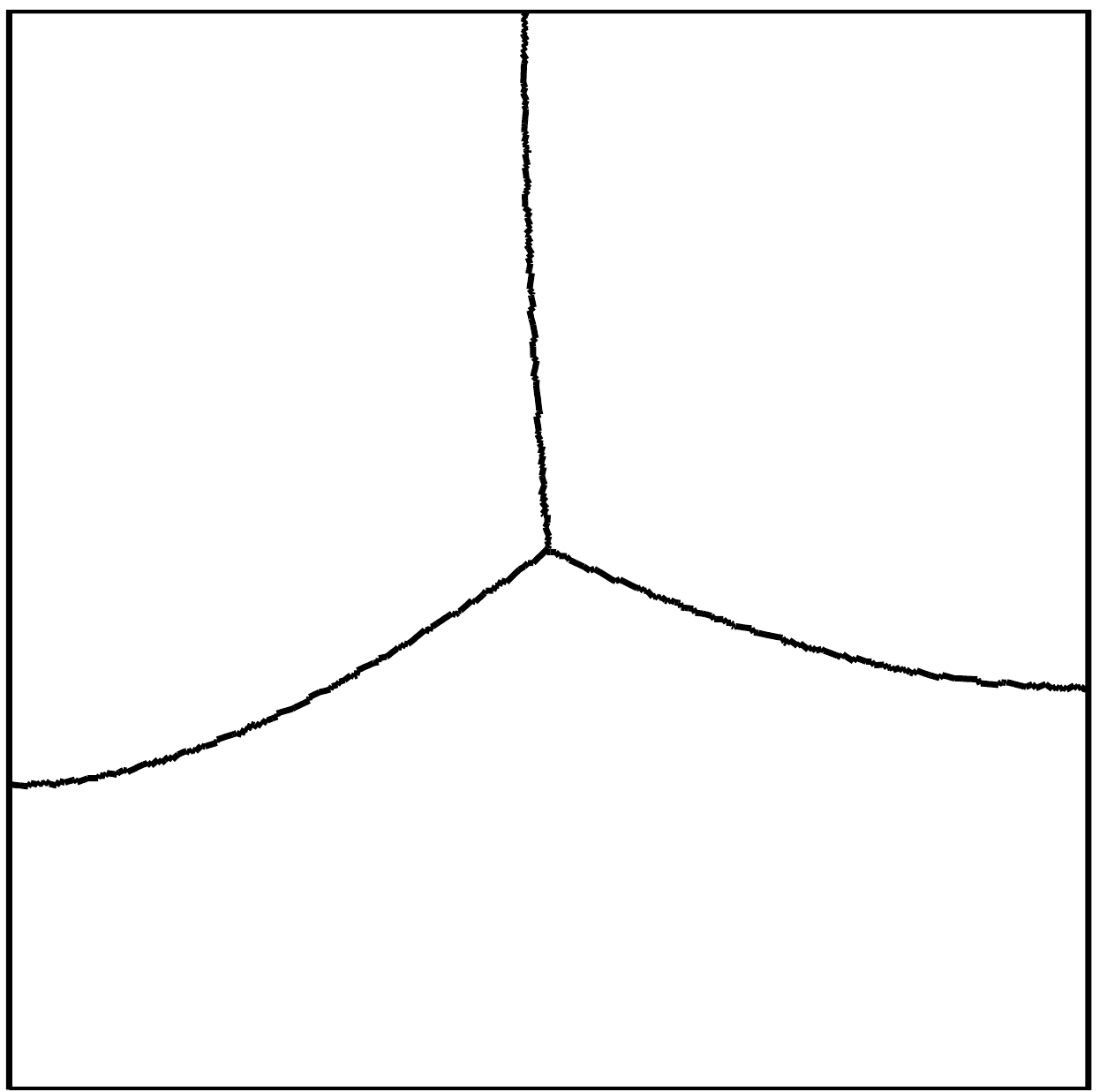}
\includegraphics[height=1.5cm]{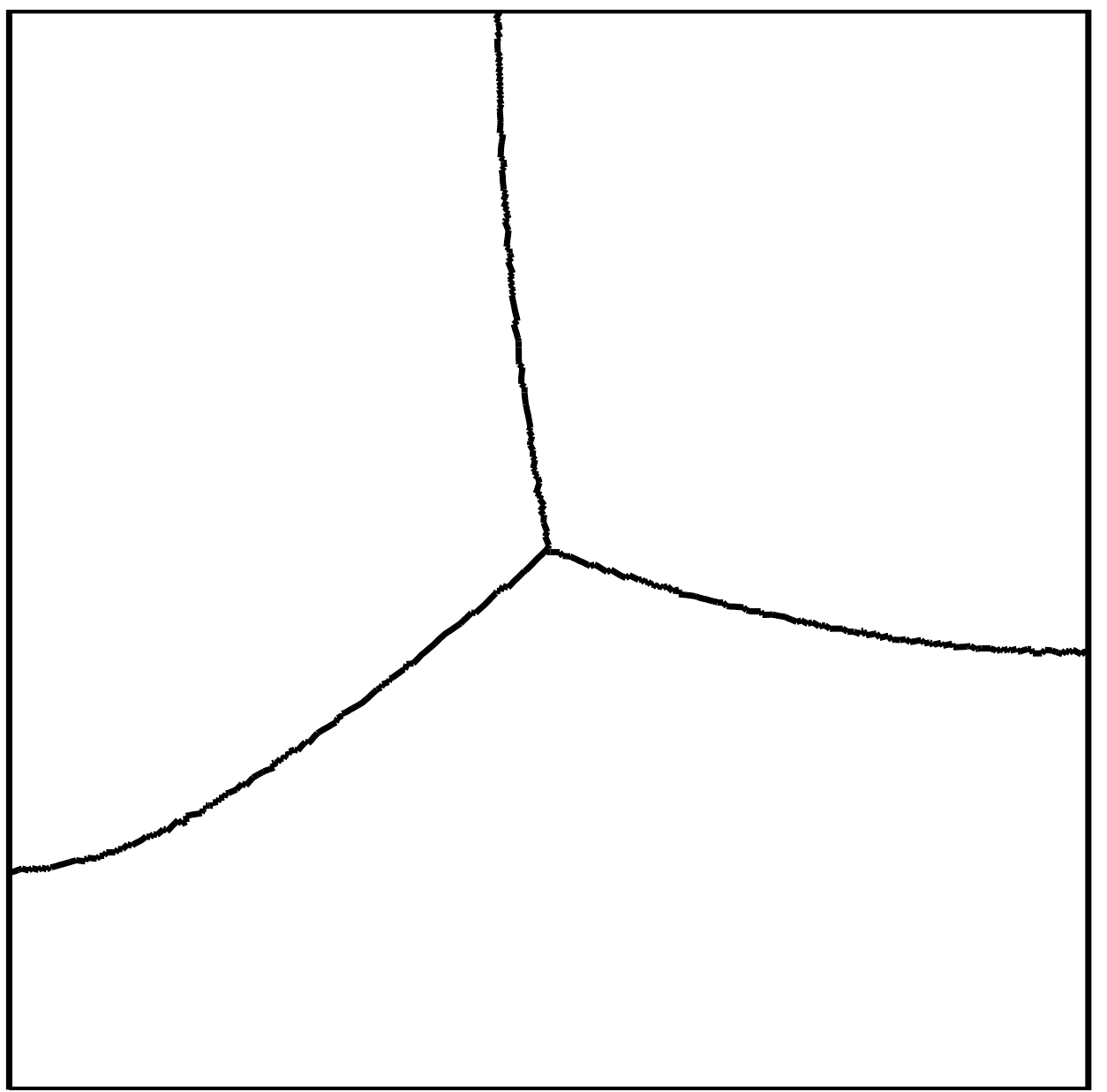}
\includegraphics[height=1.5cm]{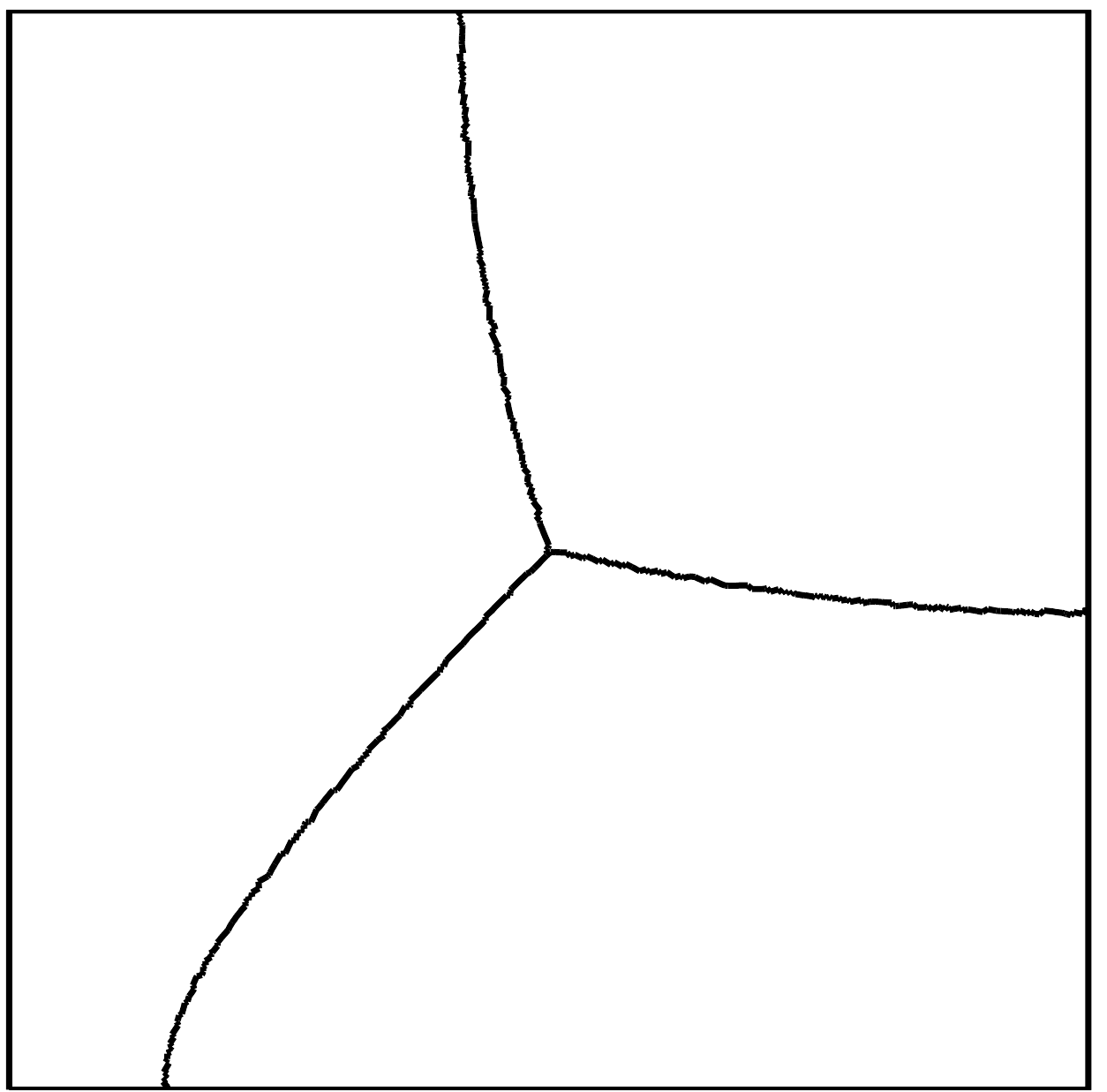}
\includegraphics[height=1.5cm]{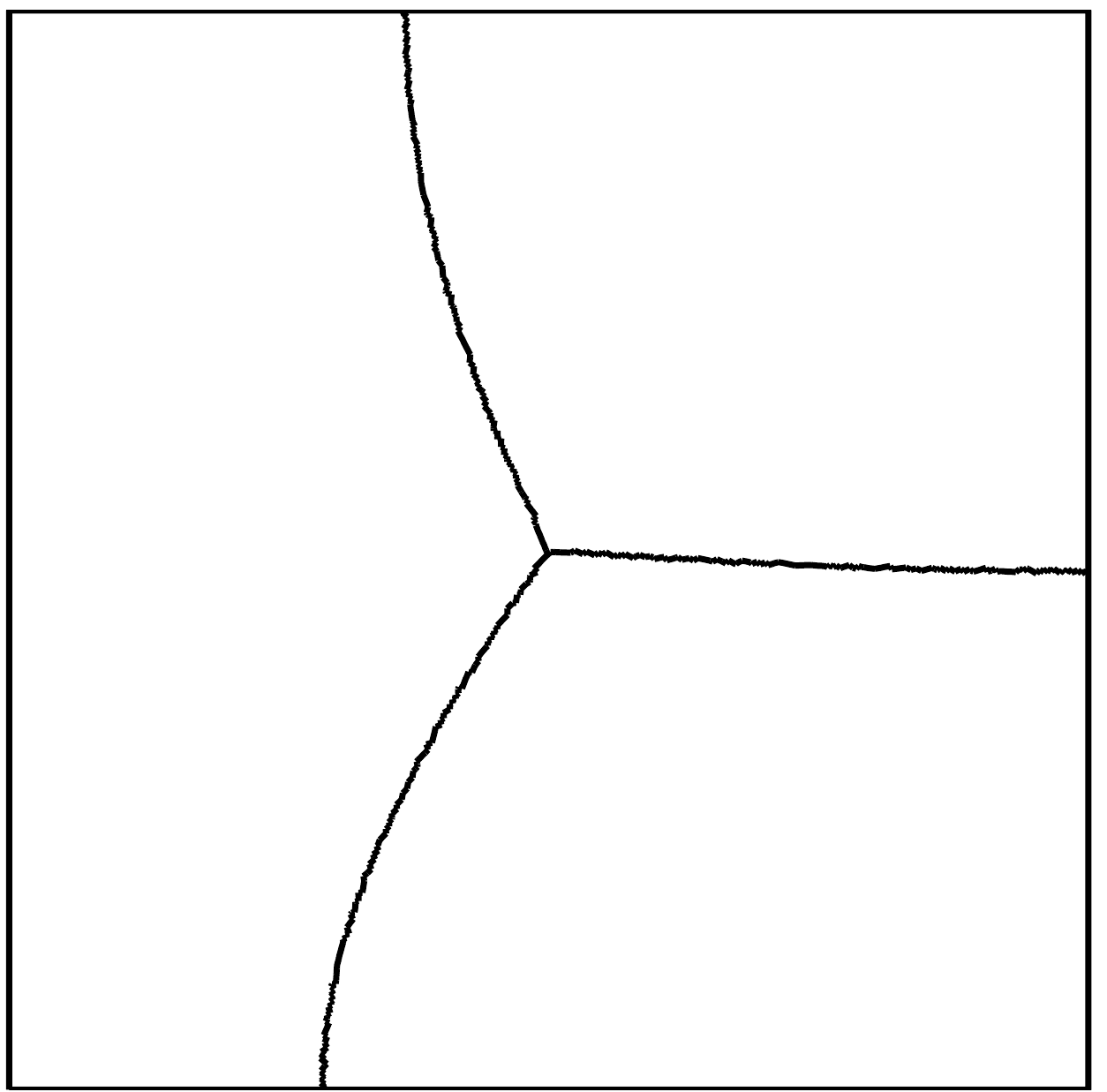}
\includegraphics[height=1.5cm]{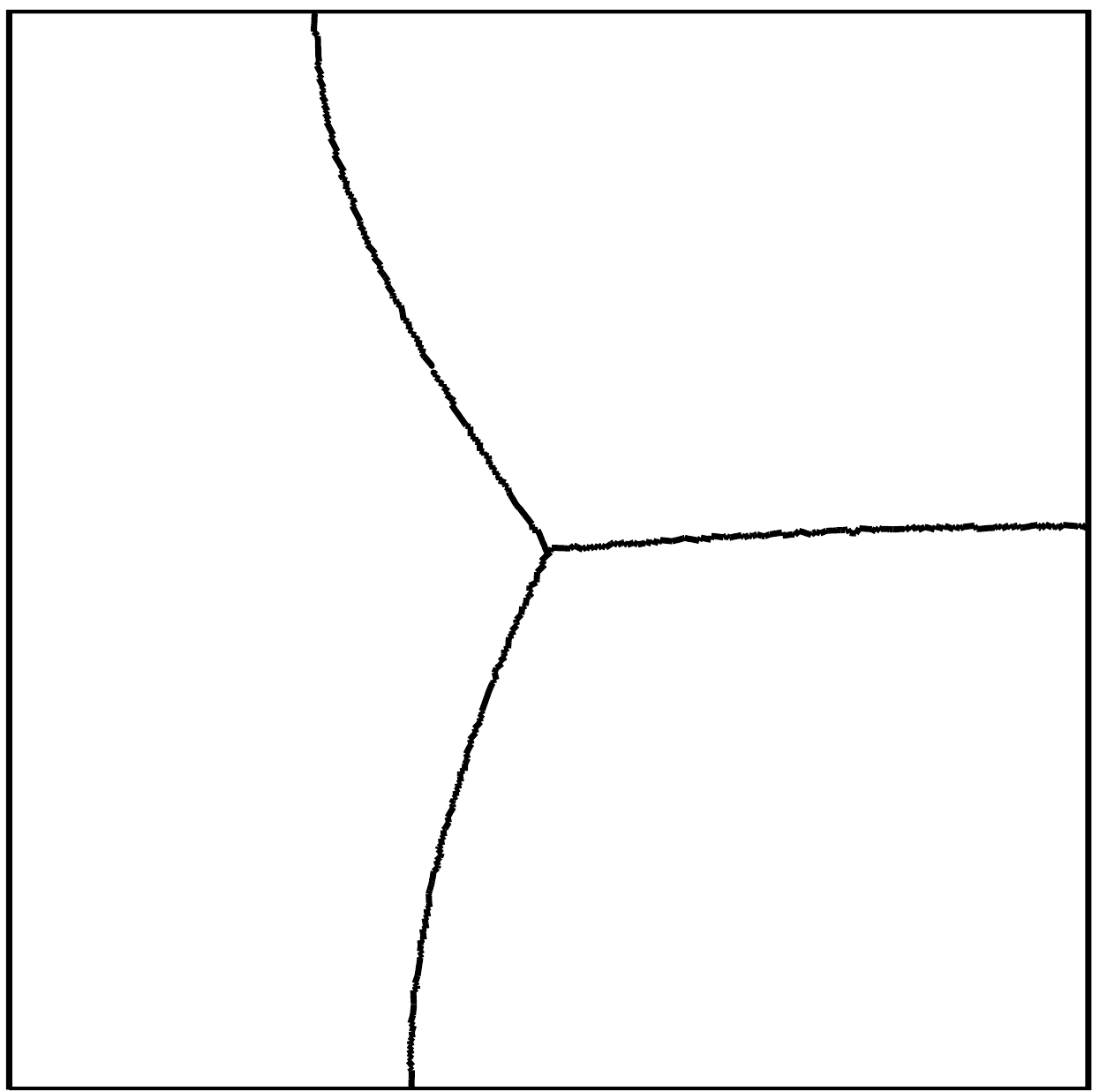}
\includegraphics[height=1.5cm]{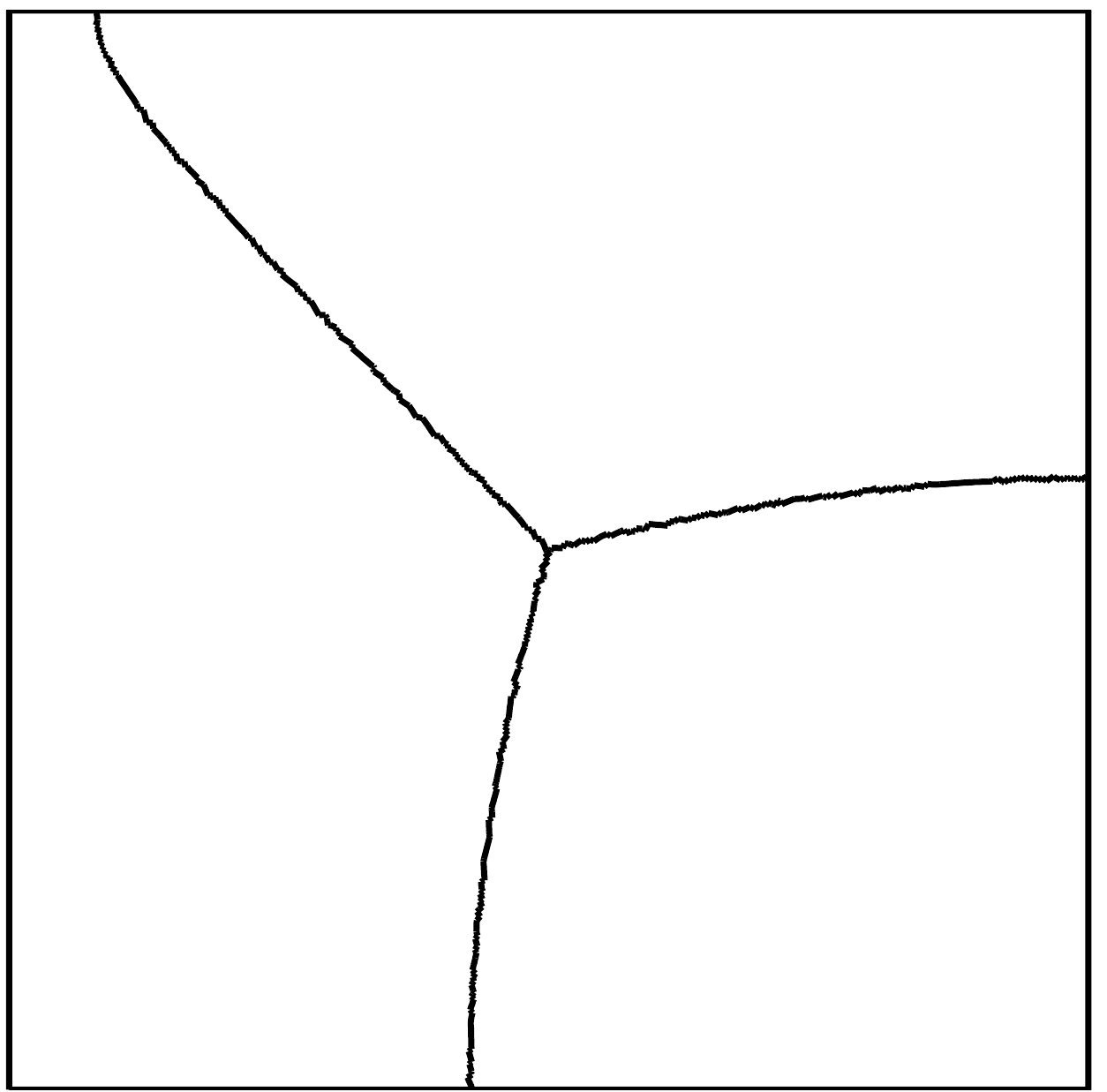}
\includegraphics[height=1.5cm]{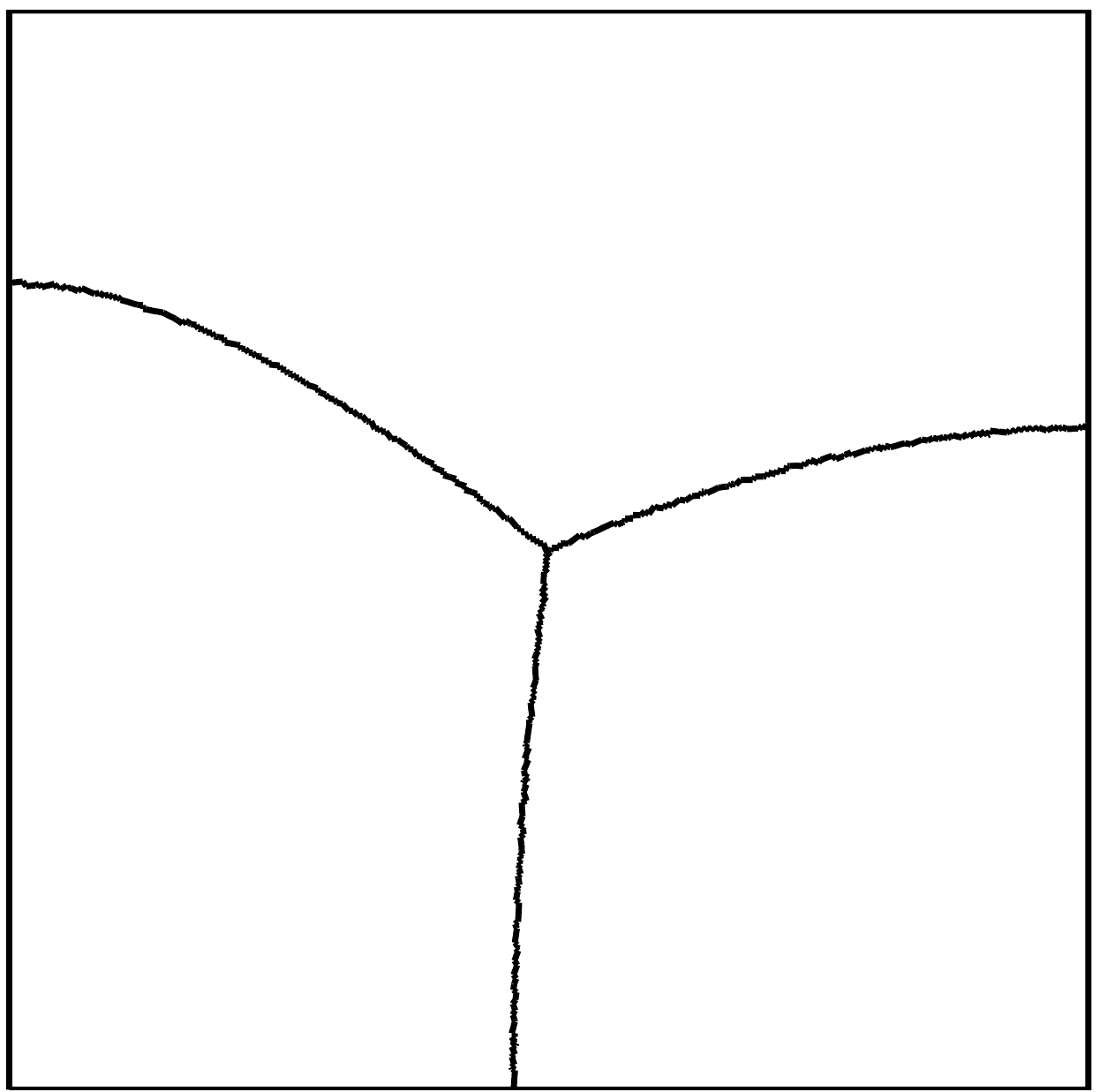}
\includegraphics[height=1.5cm]{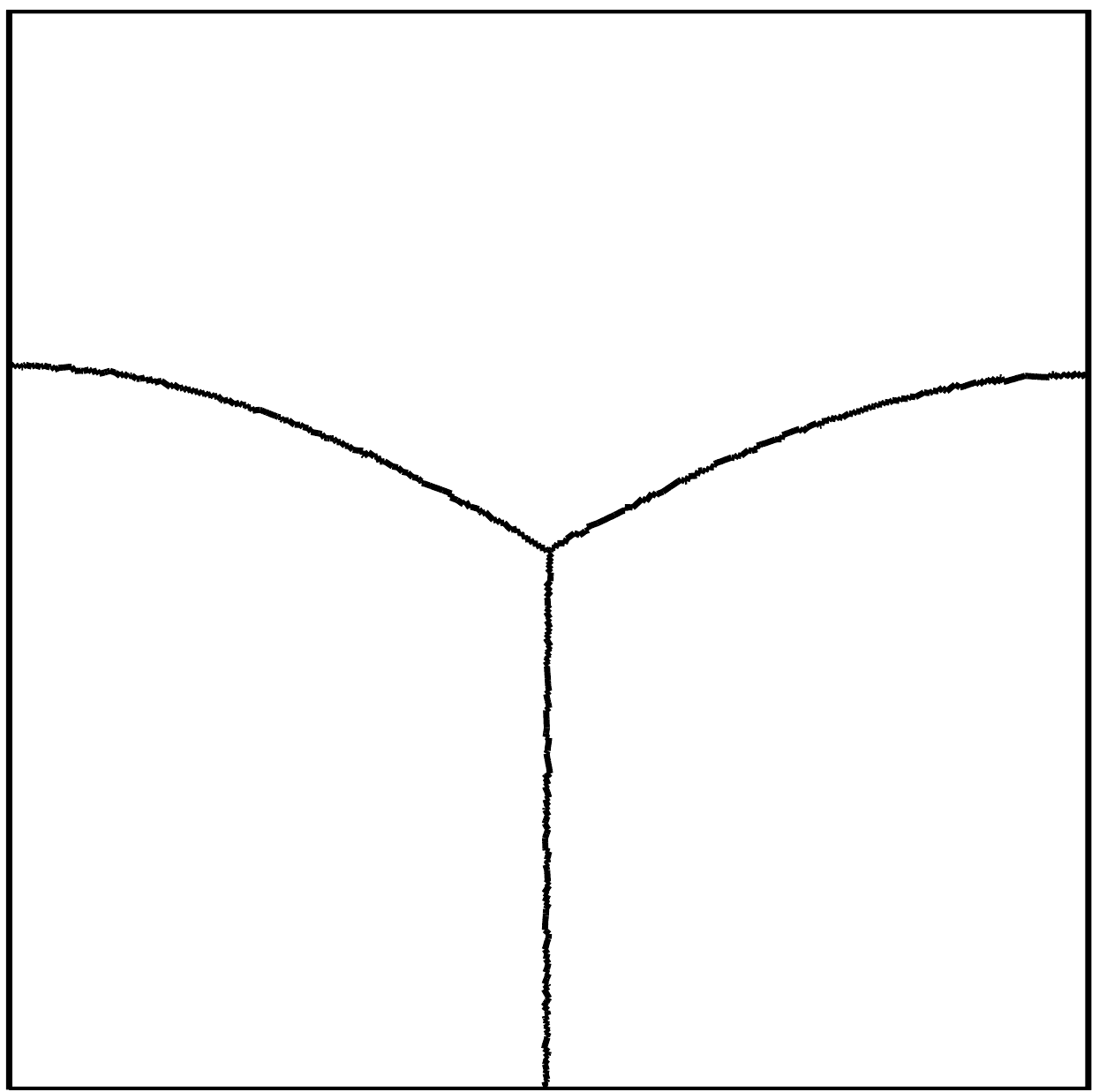}
\includegraphics[height=1.5cm]{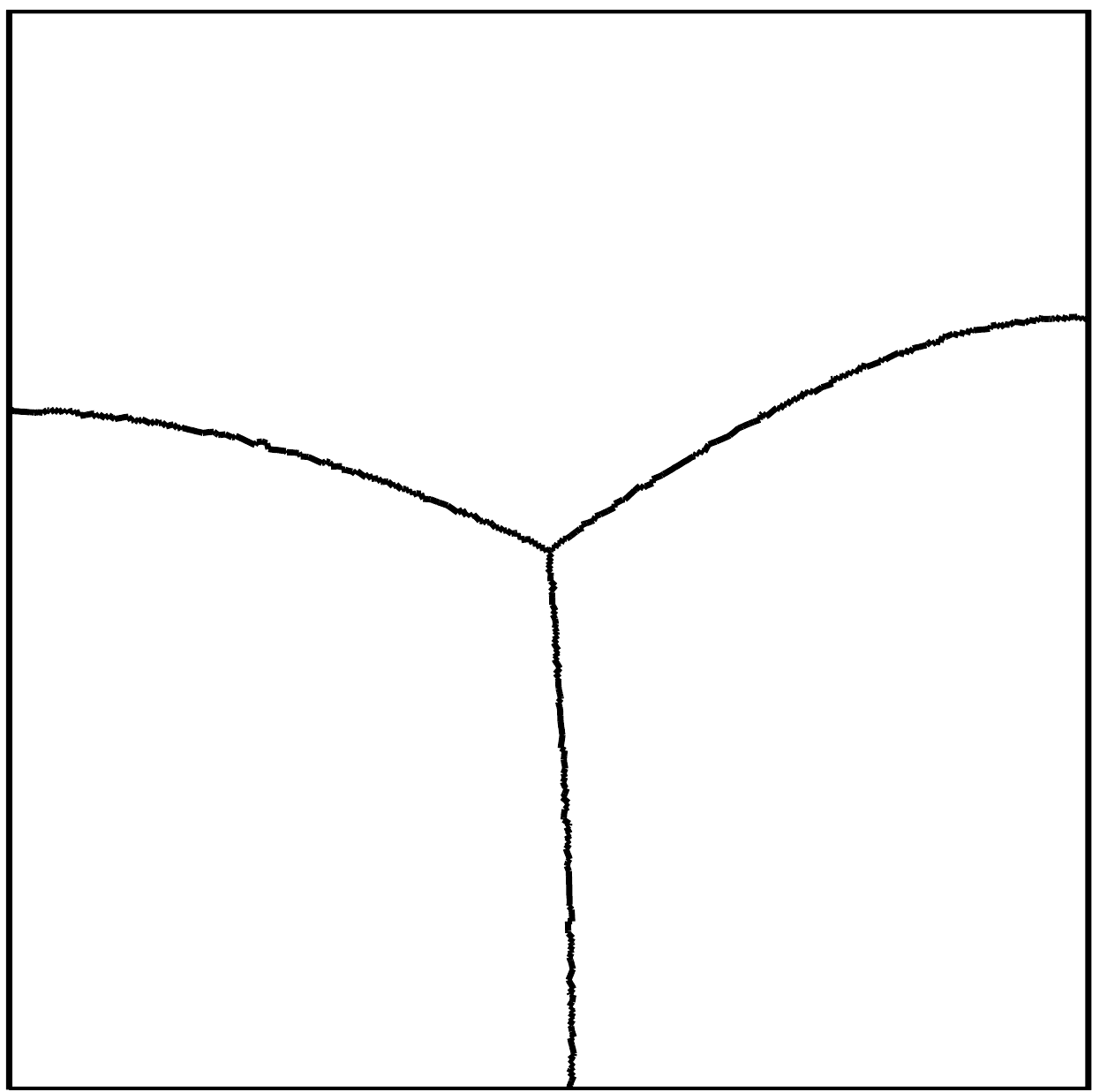}
\caption{A continuous family of $3$-partitions with the same energy.\label{fig.famille3part}}
\end{center}
\end{figure}
So this strongly suggests that there is a continuous family of minimal $3$-partitions of the square. This is done indeed numerically in \cite{BH} and illustrated in Figure~\ref{fig.famille3part}. This can be explained in the formalism of the Aharonov-Bohm operator presented in Section~\ref{s8}, observing that this operator has an eigenvalue of multiplicity $2$ when the pole is at the center. This is reminiscent of the argument of isospectrality of Jakobson-Levitin-Nadirashvili-Polterovich \cite{JLNP} and Levitin-Parnovski-Polterovich \cite{LPP}. We refer to \cite{BHHO, BH} for this discussion and more references therein.

Figure~\ref{fig.5part} gives some $5$-partitions obtained with several approaches: Aharonov-Bohm approach (see Section \ref{s8}), mixed conditions on one eighth of the square (with Dirichlet condition on the boundary of the square, Neumann condition on one of the other part and mixed Dirichlet-Neumann condition on the last boundary). The first $5$-partition corresponds with what we got by minimizing over configurations with one  interior singular point. The second $5$-partition ${\mathcal D}^{\sf perp}$ (which has four  interior singular points) gives the best known candidate to be minimal.
\begin{figure}[h!bt]
\begin{center}
\begin{tabular}{ccccc}
$\Lambda({\mathcal D}^{\sf AB})=   111.910$ && $\Lambda({\mathcal D}^{\sf perp})=104.294$ && $\Lambda({\mathcal D}^{\sf diag})=131.666$\\
\includegraphics[height=2.6cm]{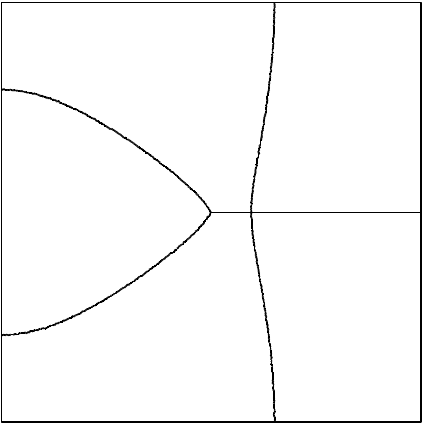}\
&\qquad& \includegraphics[height=2.6cm]{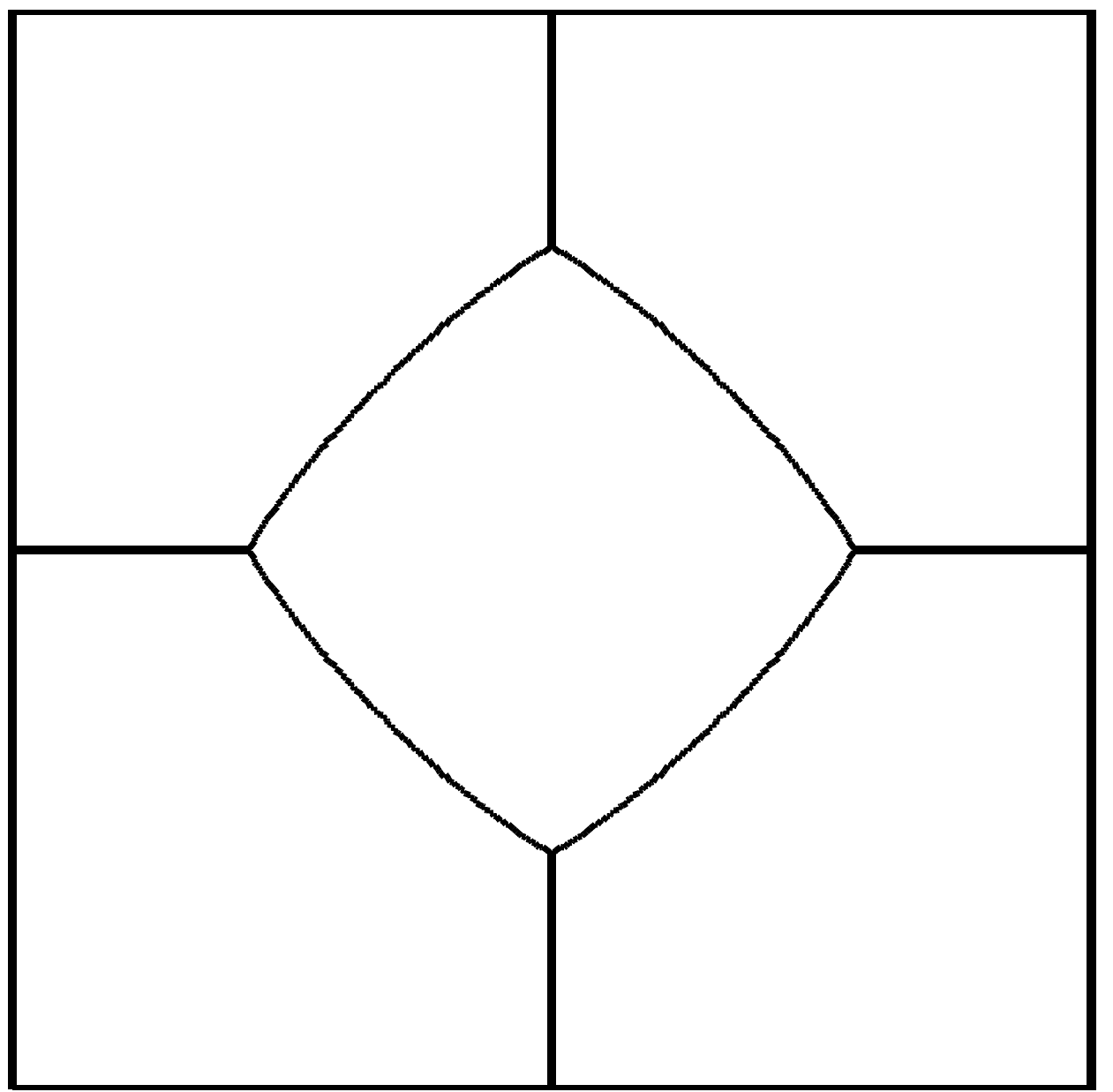}
&\qquad& \includegraphics[height=2.6cm]{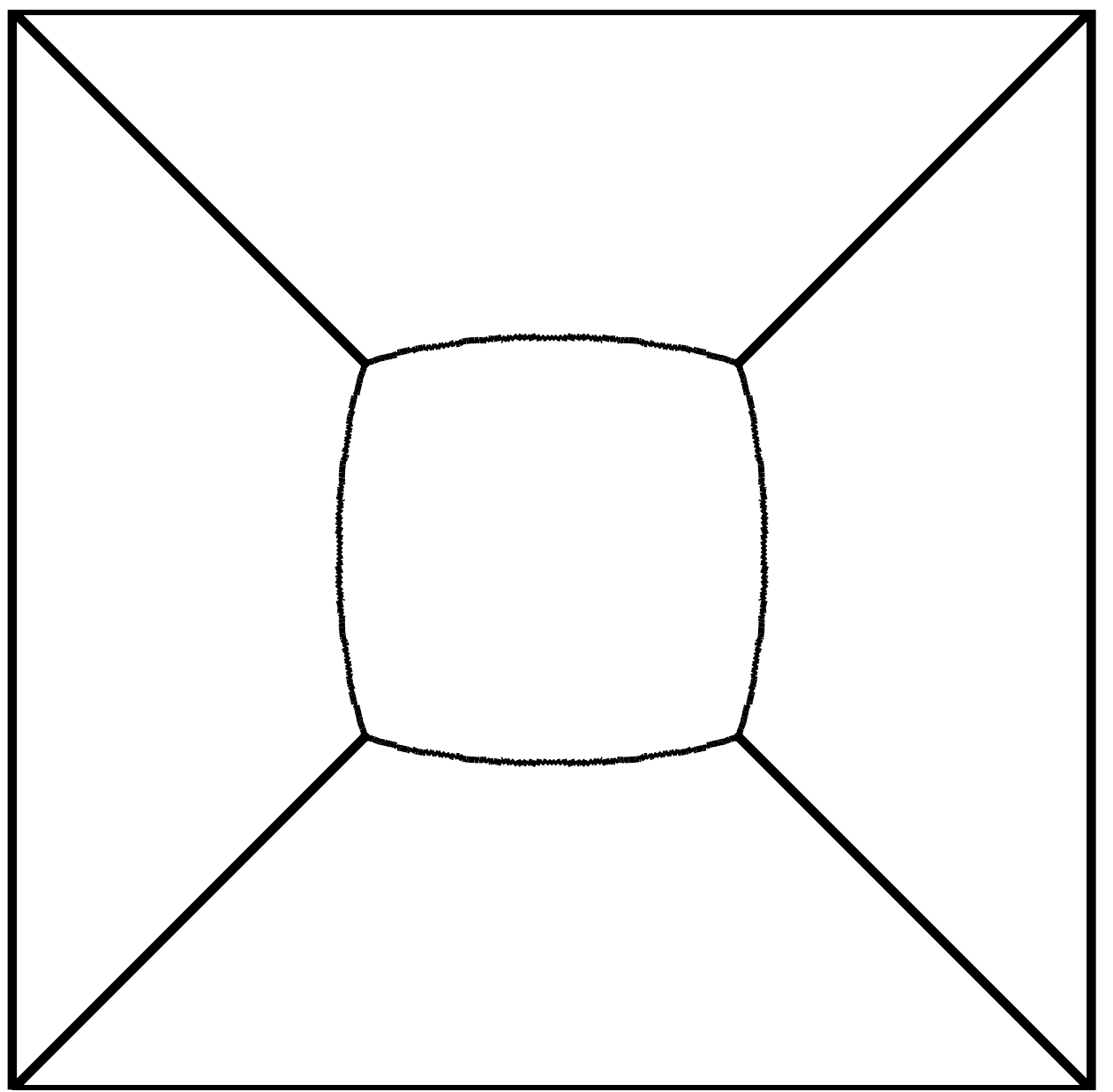}
\end{tabular}
\caption{Three  candidates for the $5$-partition of the square.\label{fig.5part}}
\end{center}
\end{figure}

\subsection{Flat tori}\label{ss7.3}
In the case of thin tori, we have a similar result to Subsection \ref{ss3.1} for minimal partitions.
\begin{theorem}\label{torus}
There exists $b_k >0$ such that if $b<b_k$, then $\mathfrak L_k(T(1,b))=k^2 \pi^2$ and the corresponding minimal $k$-partition $\mathcal D_k=\{D_i\}_{1\leq i\leq k}$ is represented in $\overline{\mathcal R(1,b)}$ by 
\begin{equation}\label{minD}
D_i= \Big(\, \frac{i-1}k\,,\,\frac ik\,\Big)\times [0\,,\, b\,)\,,\quad\mbox{ for } i=1,\dots, k\,.
\end{equation}
Moreover $b_k \geq \frac{1}{ k}$ for $k$ even and $b_k\geq \min ( \frac 1k, \frac{j^2}{k^2\pi })$ for $k$ odd.
\end{theorem}
This result extends Remark~\ref{rem.recttore} to odd $k$'s, for which the minimal $k$-partitions are not nodal. We can also notice that the boundaries of the $D_i$ in $T(1,b)$ are just $k$ circles.

In the case of isotropic flat tori, we have seen in Subsection \ref{ss3.5}, that minimal partitions are not nodal for $k>2$. Following C. L\'ena \cite{Len2}, some candidates are given in Figure \ref{fig.torepart} for $k=3,4,5,6$.

\subsection{Angular sectors}
Figure~\ref{fig.sectpart} gives some symmetric and non symmetric examples for angular sectors. Note that the energy of the first partition in the second line is lower than any symmetric $3$-partition. This proves that the minimal $k$-partition of a symmetric domain is not necessarily symmetric.
\begin{figure}[h!bt]
\begin{center}
\includegraphics[height=2.5cm]{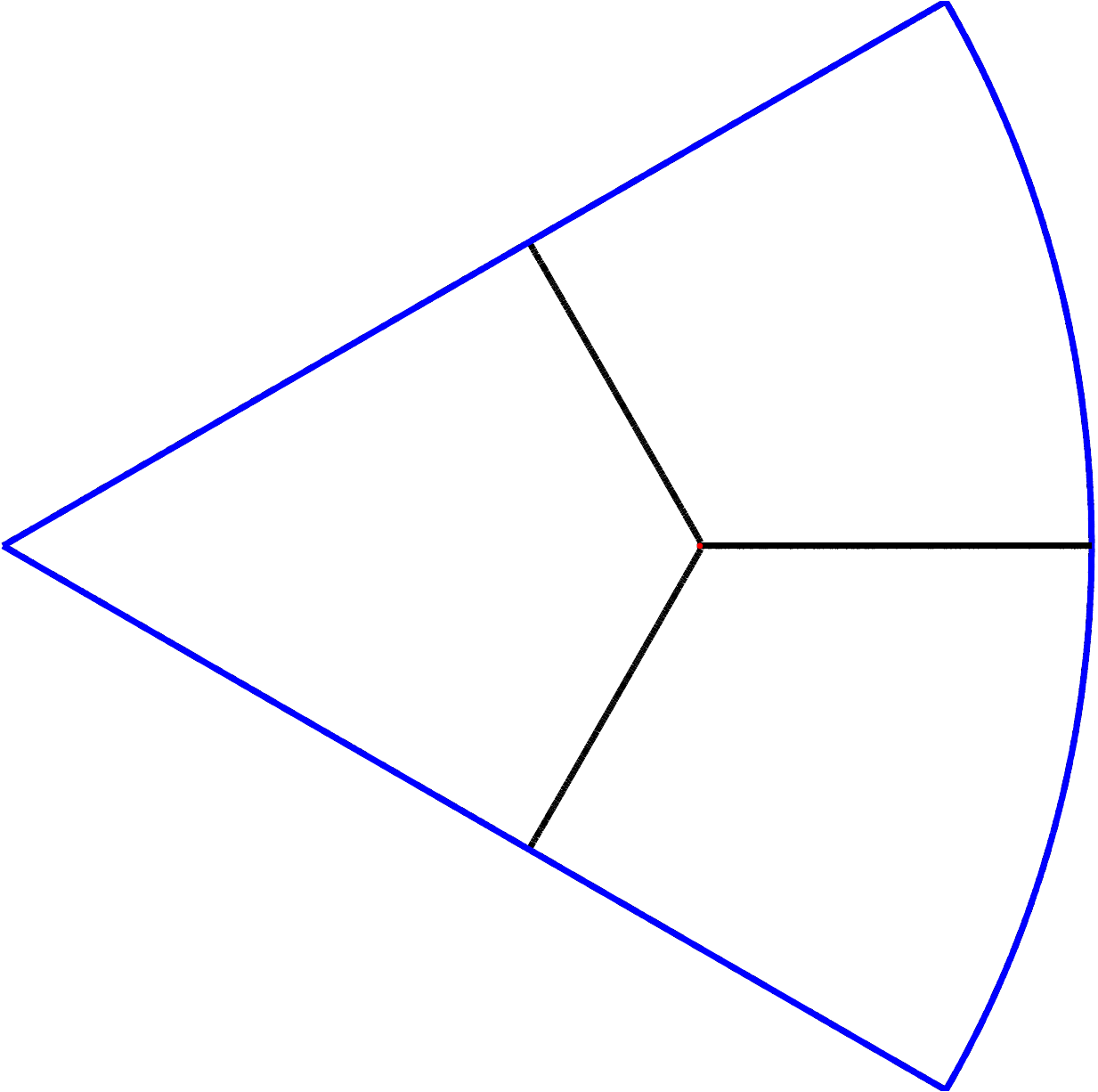}
$\qquad$\includegraphics[height=3cm]{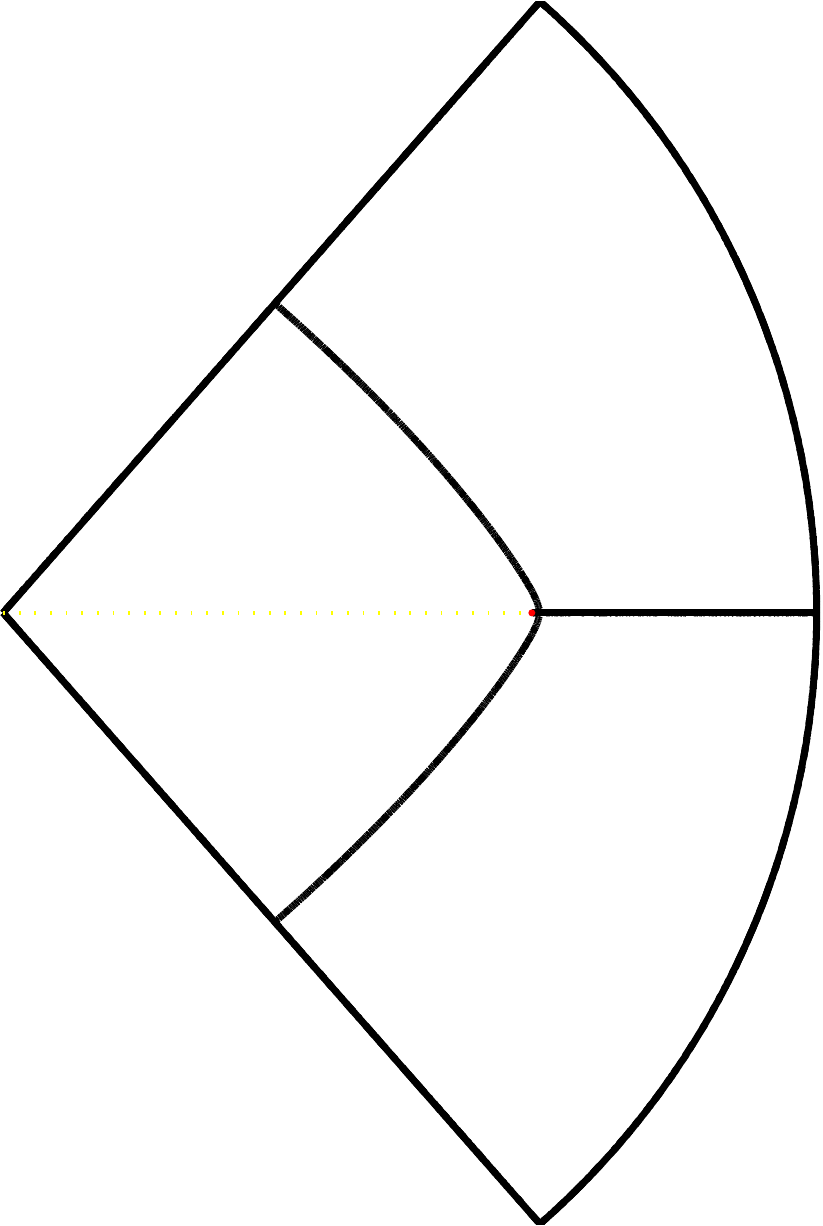}
$\qquad$\includegraphics[height=3cm]{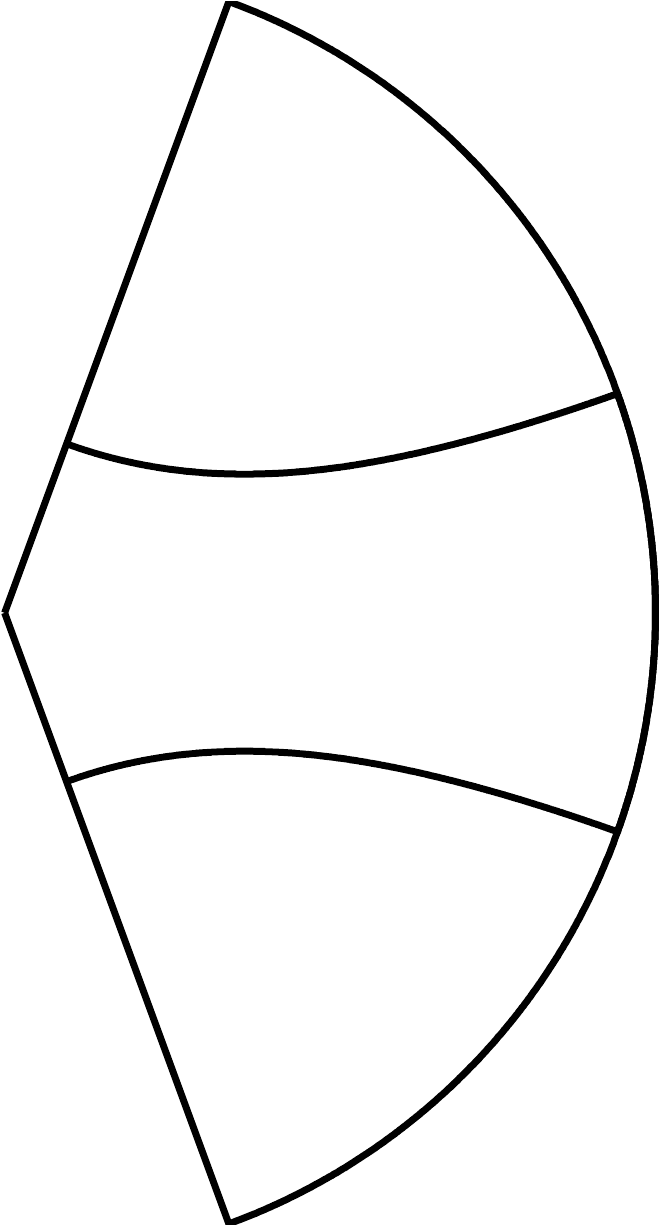}
$\qquad$\includegraphics[height=3cm]{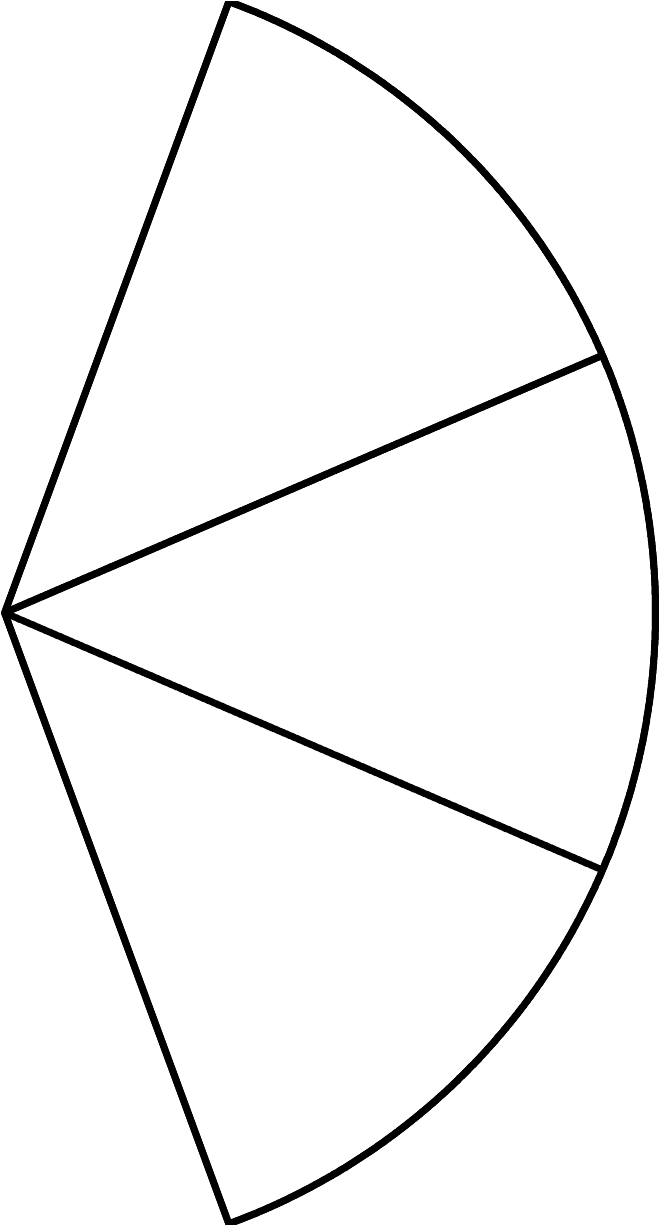}
$\qquad$\includegraphics[height=3cm]{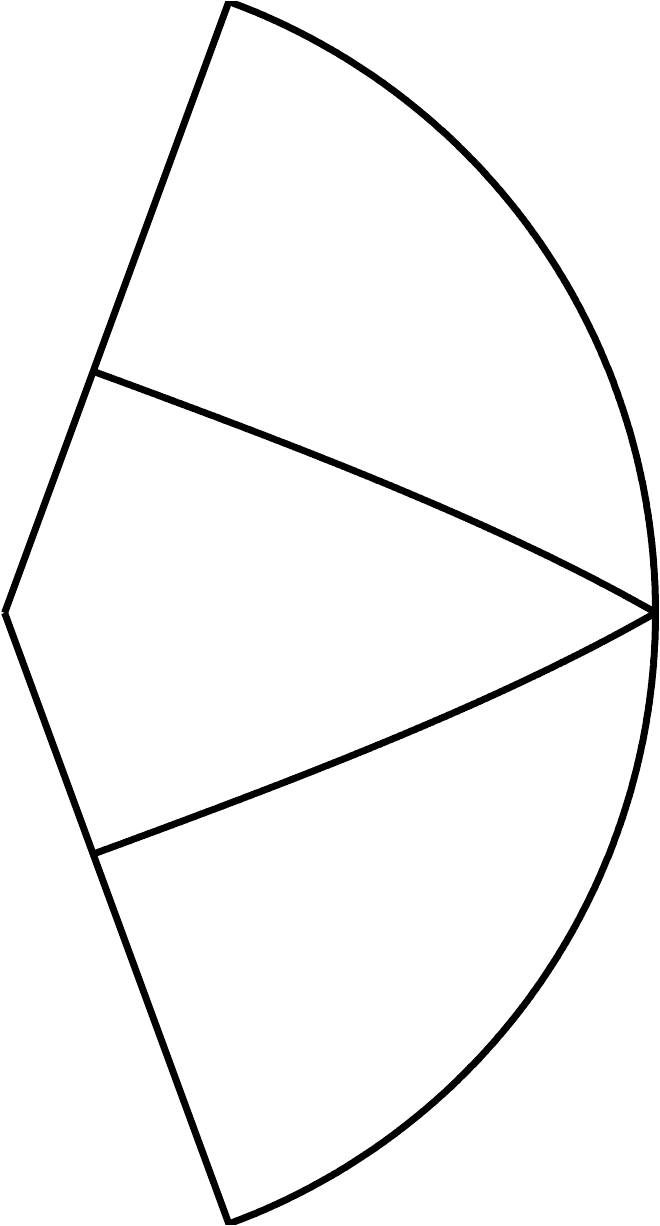}\\
\includegraphics[width=3cm]{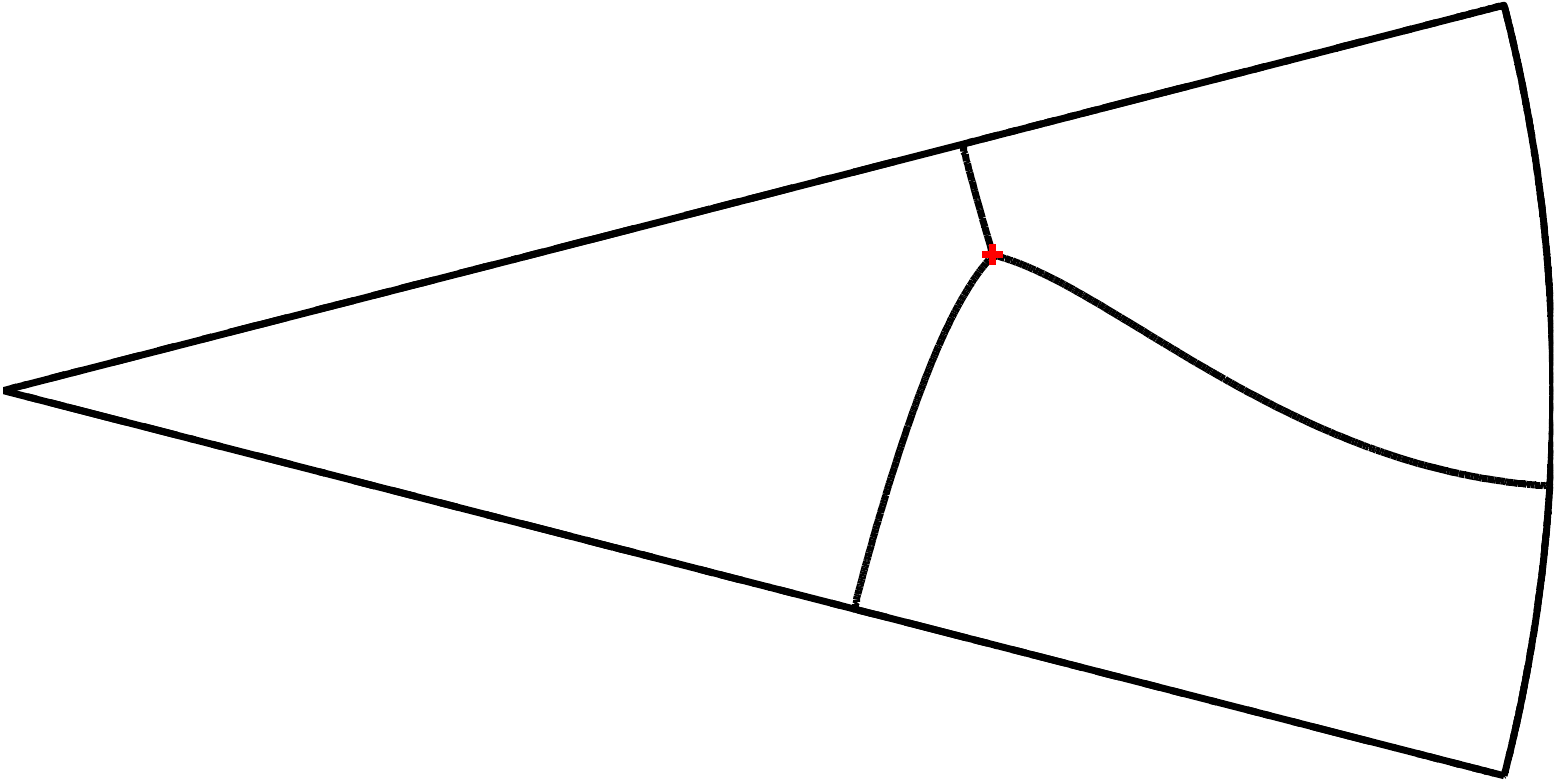}
$\quad$\includegraphics[width=3cm]{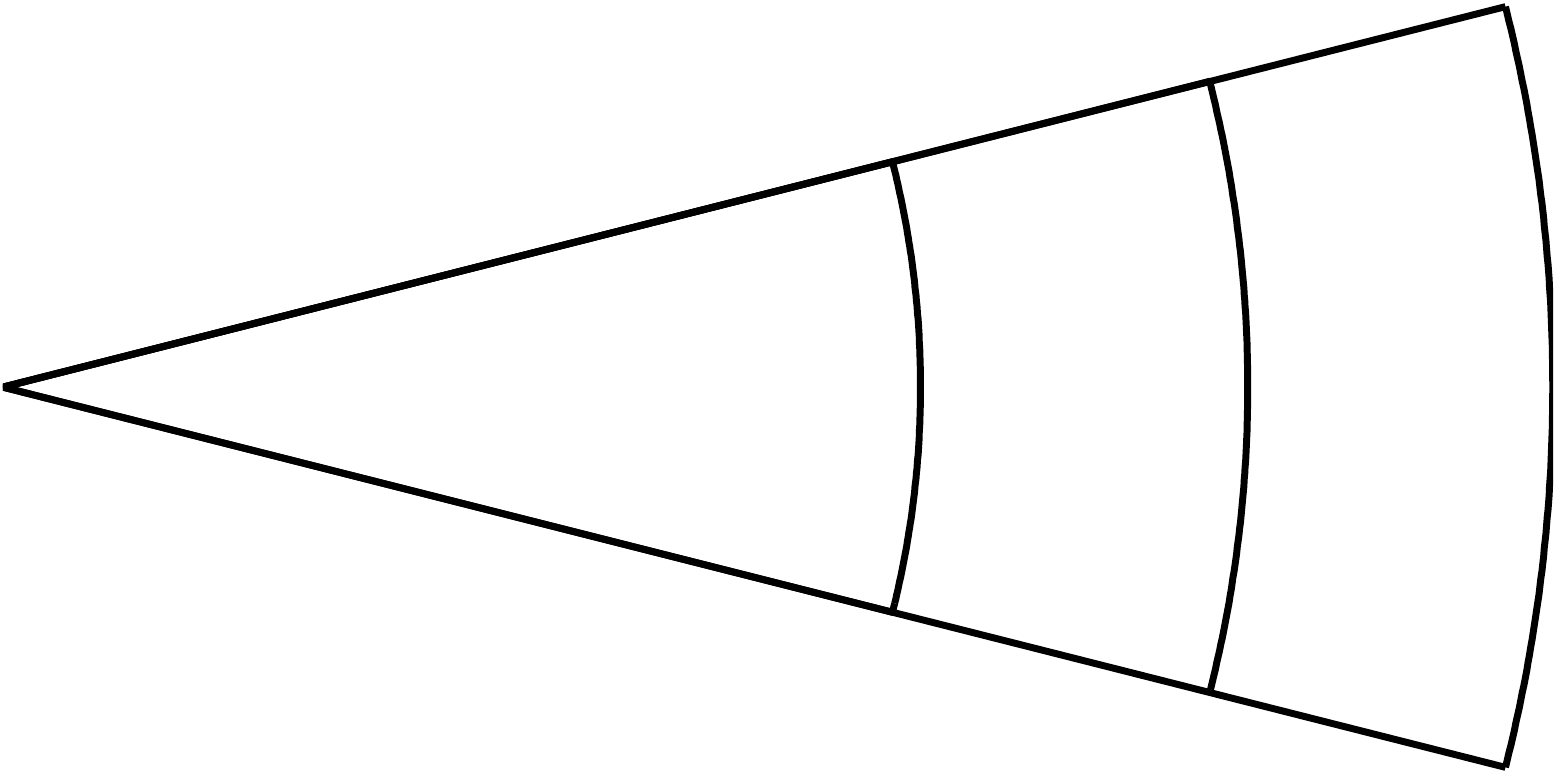}
$\quad$\includegraphics[width=3cm]{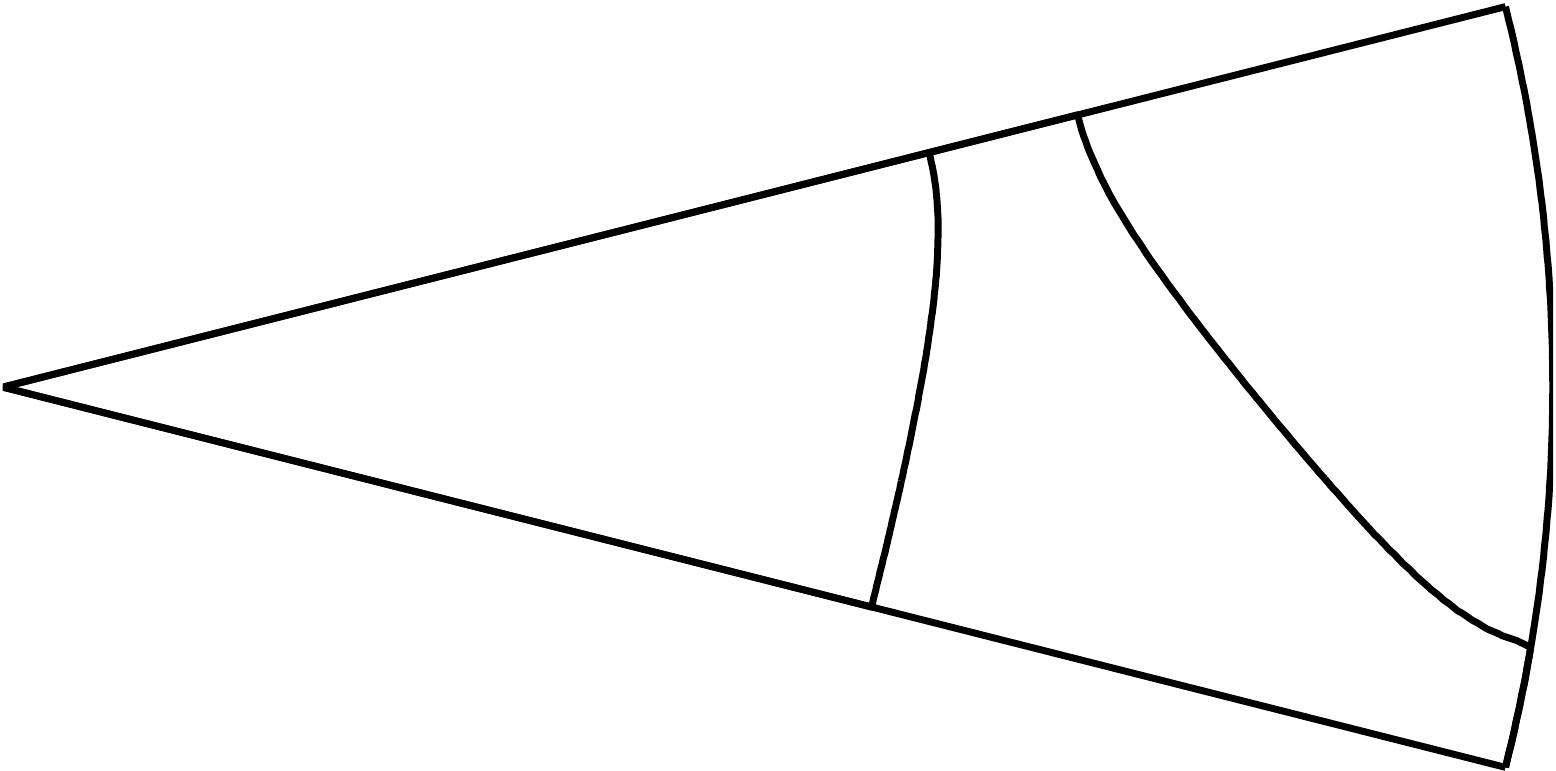}
$\quad$\includegraphics[width=3cm]{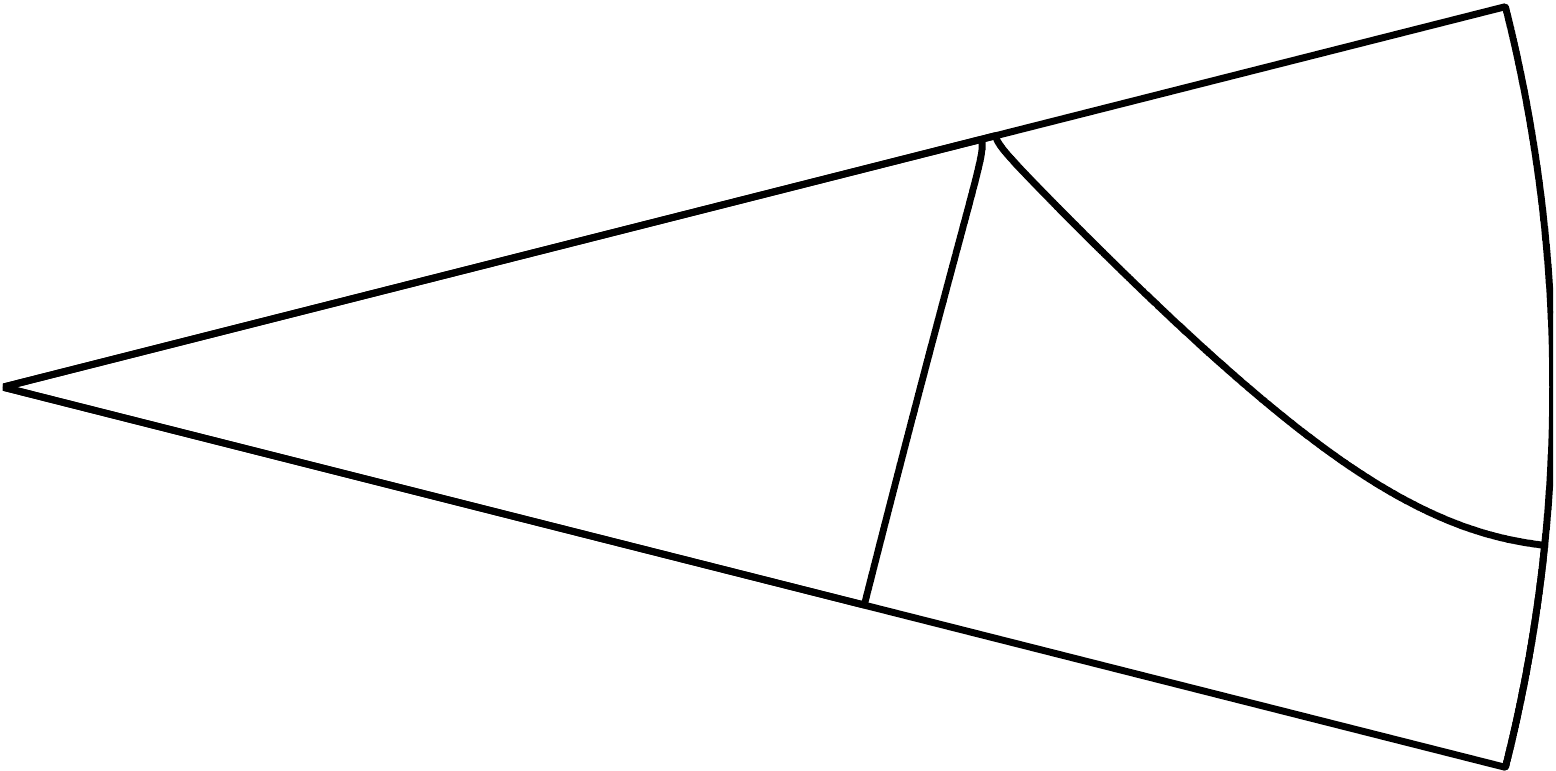}
\caption{Candidates for angular sectors.\label{fig.sectpart}}
\end{center}
\end{figure}

\subsection{Notes}
\begin{figure}[h!t]
\begin{center}
\includegraphics[width=2.5cm]{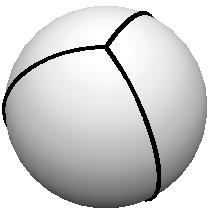}
\caption{Minimal $3$-partition of the sphere.\label{fig.3partsphere}}
\end{center}
\end{figure}
The minimal $3$-partitions for the sphere $\mathbb S^2$ have been determined mathematically in \cite{HHOT2} (see Figure~\ref{fig.3partsphere}). This is an open question known as the Bishop conjecture \cite{Bis} that the same partition is a $1$-minimal $3$-partition.
The case of a thin annulus is treated in \cite{HH5} for Neumann conditions.  The case of Dirichlet is still open.

\section{Aharonov-Bohm approach }\label{s8}
The introduction of Aharonov-Bohm operators in this context is an example of ``physical mathematics''. There is no magnetic fied in our problem and it is introduced artificially. But the idea comes from \cite{HHOO}, which was motivated by a problem in superconductivity in non simply connected domains.
\subsection{Aharonov-Bohm operators}
Let $\Omega$ be a planar domain and $\cb=(p_{1},p_{2})\in\Omega$. Let us consider the Aharonov-Bohm Laplacian in a punctured domain $\dot\Omega_{\cb}:=\Omega\setminus\{\cb\}$ with a singular magnetic potential and normalized flux $\alpha$. We first introduce 
$${\mathbf A}^\cb(\xb)=(A_{1}^\cb(\xb),A_{2}^\cb(\xb)) = \frac{(\xb-\cb)^\perp}{|\xb-\cb|^2},\qquad\mbox{ with }\quad \yb^\perp=(-y_{2},y_{1})\,.$$
This magnetic potential satisfies 
$$\curl {\mathbf A}^\cb(\xb)=0\quad\mbox{ in } \dot\Omega_{\cb}.$$
If $\cb \in \Omega$, its circulation along a path of index $1$ around $\cb$ is $2\pi$ (or the flux created by $\cb$). If $\cb \not\in \Omega$, $\mathbf A^\cb $ is a gradient and the circulation along any path in $\Omega$ is zero. From now on, we renormalize the flux by dividing the flux by $2\pi$.\\
The Aharonov-Bohm Hamiltonian with singularity $\cb$ and flux $\alpha$ (written for shortness $H^{AB}(\dot \Omega_{\cb},\alpha)$) is defined by considering the Friedrichs extension starting from $ C_0^\infty(\dot \Omega_{\cb})$ and the associated differential operator is
\begin{equation}
-\Delta_{\agb {\bf A}^\cb} := (D_{x_{1}} - \alpha A_{1}^\cb)^2 + (D_{x_{2}}-\alpha A_{2}^\cb)^2\,\qquad\mbox{with }\qquad D_{x_{j}} =-i\partial_{x_{j}}.
\end{equation}
This construction can be extended to the case of a configuration with $\ell$ distinct points $\cb_1,\dots, \cb_\ell$ (putting a flux $\alpha_{j}$ at each of these points). We just take as magnetic potential 
$$ {\bf A}_{\boldsymbol{\alpha}}^\Cb = \sum_{j=1}^\ell \alpha_j {\mathbf A}^{\cb_j}\,, \qquad\mbox{ where }\quad \Cb=(\cb_1,\dots,\cb_\ell)\quad\mbox{ and }\quad{\boldsymbol{\alpha}}=(\alpha_{1},\ldots,\alpha_{\ell}).$$
Let us point out that the $\cb_j$'s can be in $\mathbb R^2\setminus \Omega$, and in particular in $\partial \Omega$. It is important to observe the following
\begin{proposition}\label{Prop8.1}
If $\agb =\agb'$ modulo $\mathbb Z^\ell $, then $H^{AB}(\dot \Omega_{\cb},\agb)$ and $H^{AB}(\dot \Omega_{\cb},\agb')$ are unitary equivalent.
\end{proposition}
 
\subsection{The case when the fluxes are $1/2$.}
Let us assume for the moment that there is a unique pole $\ell=1$ and suppose that the flux $\alpha$ is $\frac 12$. For shortness, we omit $\alpha$ in the notation when it equals $1/2$. Let $ K_{\cb}$ be the antilinear operator $ K_{\cb} = {\rm e}^{i \theta_{\cb}} \; \Gamma\,$, where $\Gamma$ is the complex conjugation operator $\Gamma u = \bar u\,$ and 
$$ (x_{1}-p_1)+ i (x_{2}-p_2) = \sqrt{|x_{1}-p_{1}|^2+|x_{2}-p_{2}|^2}\, {\rm e}^{i\theta_{\cb}}\,,$$
$\theta_\cb$ such that 
$$d\theta_\cb= 2 {\bf A} ^\cb\,.$$
Here we note that because the (normalized) flux of $2 {\bf A} ^\cb$ belongs to $\mathbb Z$ for any path in $\dot \Omega_{\pb}$, then $\xb \mapsto \exp i \theta_\cb(\xb)$ is a $C^\infty$ function (this is indeed the variable $\theta$ in polar coordinates centered at $\cb$).\\
A function $ u$ is called $K_{\cb}$-real, if $K_{\cb} u =u\,.$ The operator $H^{AB}(\dot \Omega_{\cb})=H^{AB}(\dot \Omega_{\cb},\frac 12)$ is preserving the $ K_{\cb}$-real functions. Therefore we can consider a basis of $K_{\cb}$-real eigenfunctions. Hence we only analyze the restriction of $H^{AB}(\dot \Omega_{\cb},\frac 12)$ to the $ K_{\cb}$-real space $ L^2_{K_{\cb}}$ where
$$
 L^2_{K_{\cb}}(\dot{\Omega}_{\cb})=\{u\in L^2(\dot{\Omega}_{\cb}) \;:\; K_{\cb}\,u =u\,\}\,.
$$
If there are several poles ($\ell>1$) and $\agb = (\frac 12,\dots, \frac 12)$, we can also construct the antilinear operator $ K_\Cb$, where $\theta_\cb$ is replaced by 
\begin{equation}\label{defTheta} 
\Theta_{\Cb} = \sum_{j=1}^\ell \theta_{\cb_j}\,.
\end{equation} 

\subsection{Nodal sets of $K$-real eigenfunctions}
As mentioned previously, we can find a basis of $K_\Cb$-real eigenfunctions. It was shown in \cite{HHOO} and \cite{AFT} that the $ K_{\Cb}$-real eigenfunctions have a regular nodal set (like the eigenfunctions of the Dirichlet Laplacian) with the exception that at each singular point $\cb_j$ ($j=1,\dots,\ell$) an odd number $\nu(\pb_j)$ of half-lines meet. So the only difference with the notion of regularity introduced in Subsection \ref{ss4.2} is that some $\nu(p_j)$ can be equal to $1$.
\begin{proposition}
The zero set of a $K_{\Cb}$-real eigenfunction of $H^{AB}(\dot \Omega^{\Cb})$ is the boundary set of a regular partition if and only if $\nu(\cb_j) \geq 2$ for $j=1,\dots, \ell$.
\end{proposition}
Let us illustrate the case of the square with one singular point. Figure~\ref{fig.exvecpABcarre} gives the nodal lines of some eigenfunctions of the Aharonov-Bohm operator. In these examples, there are always one or three lines ending at the singular point (represented by a red point). Note that only the fourth picture gives a regular and nice partition.

\begin{figure}[h!t]
\begin{center}\begin{tabular}{cccc}
\includegraphics[height=2cm]{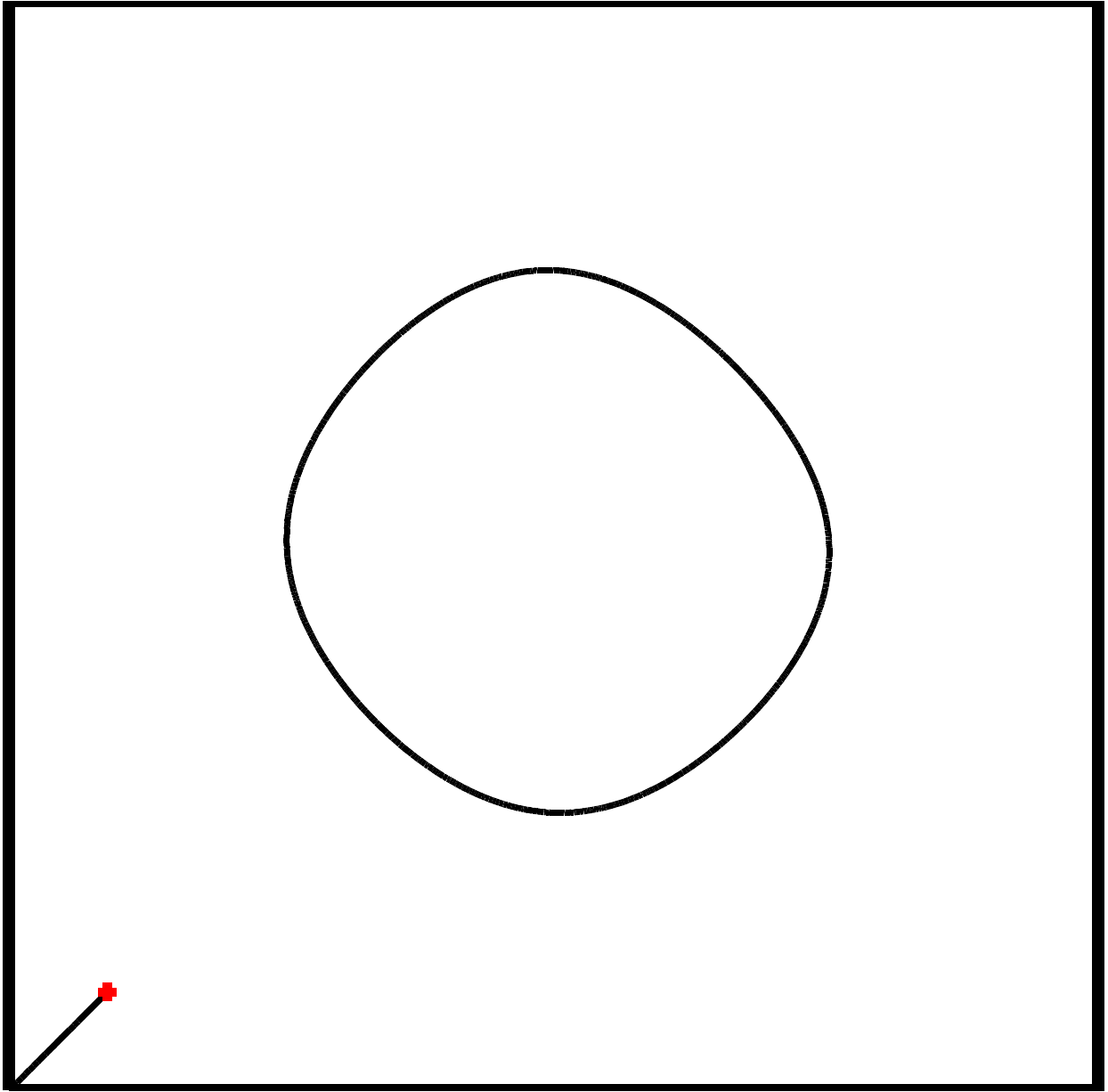}
&\includegraphics[height=2cm]{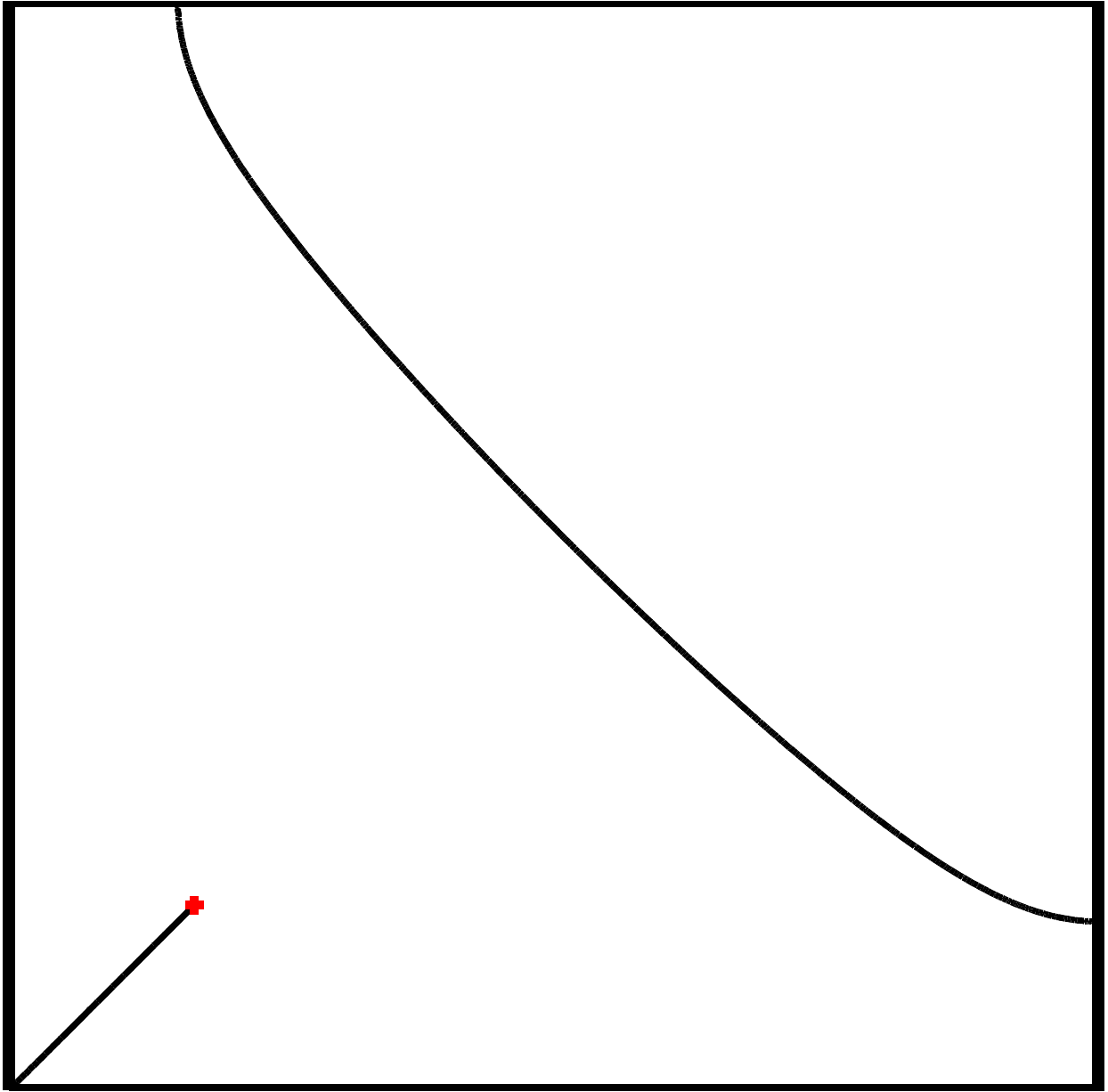}
&\includegraphics[height=2cm]{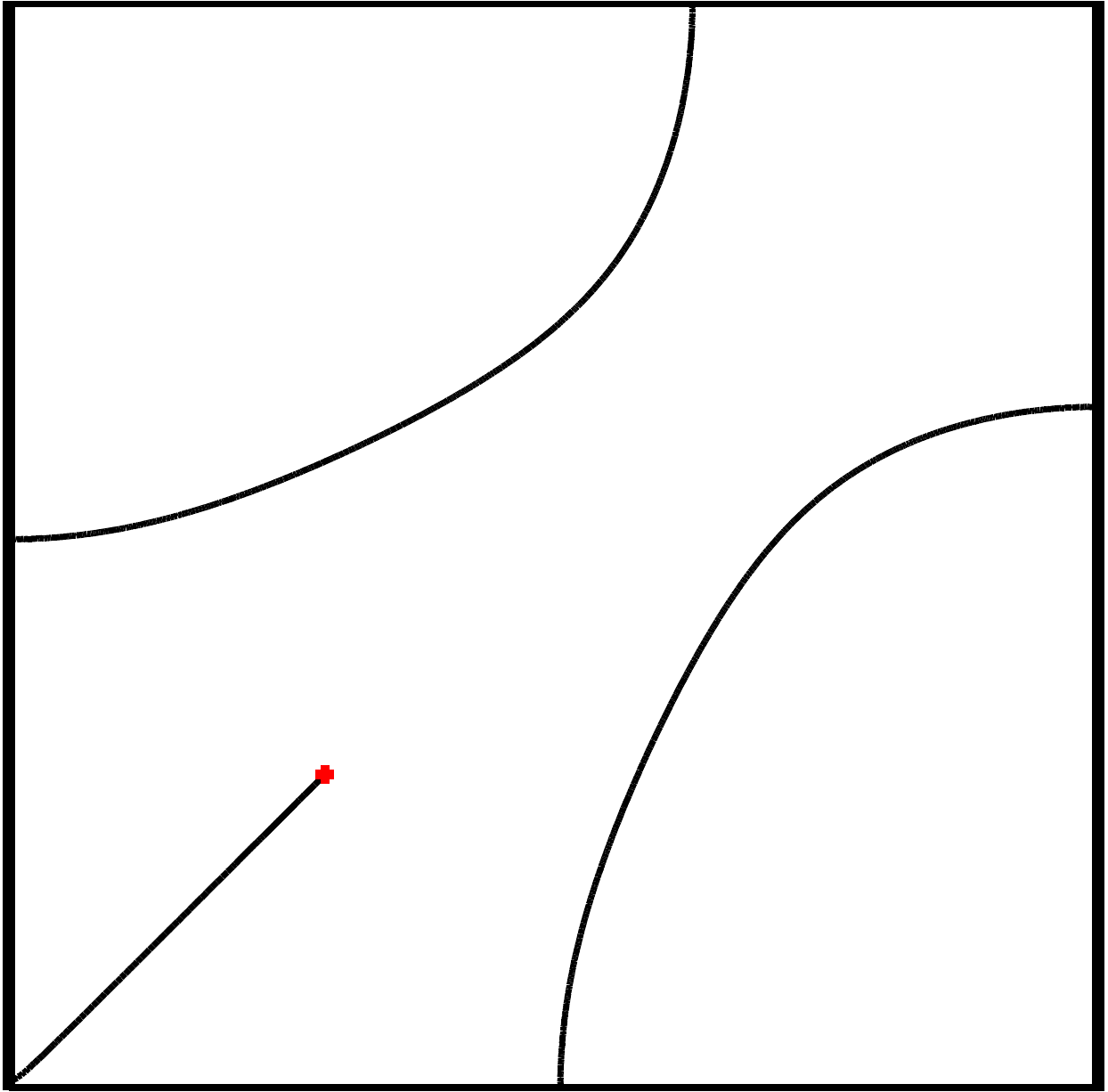}
&\includegraphics[height=2cm]{carre_t25.pdf} \\
\includegraphics[height=2cm]{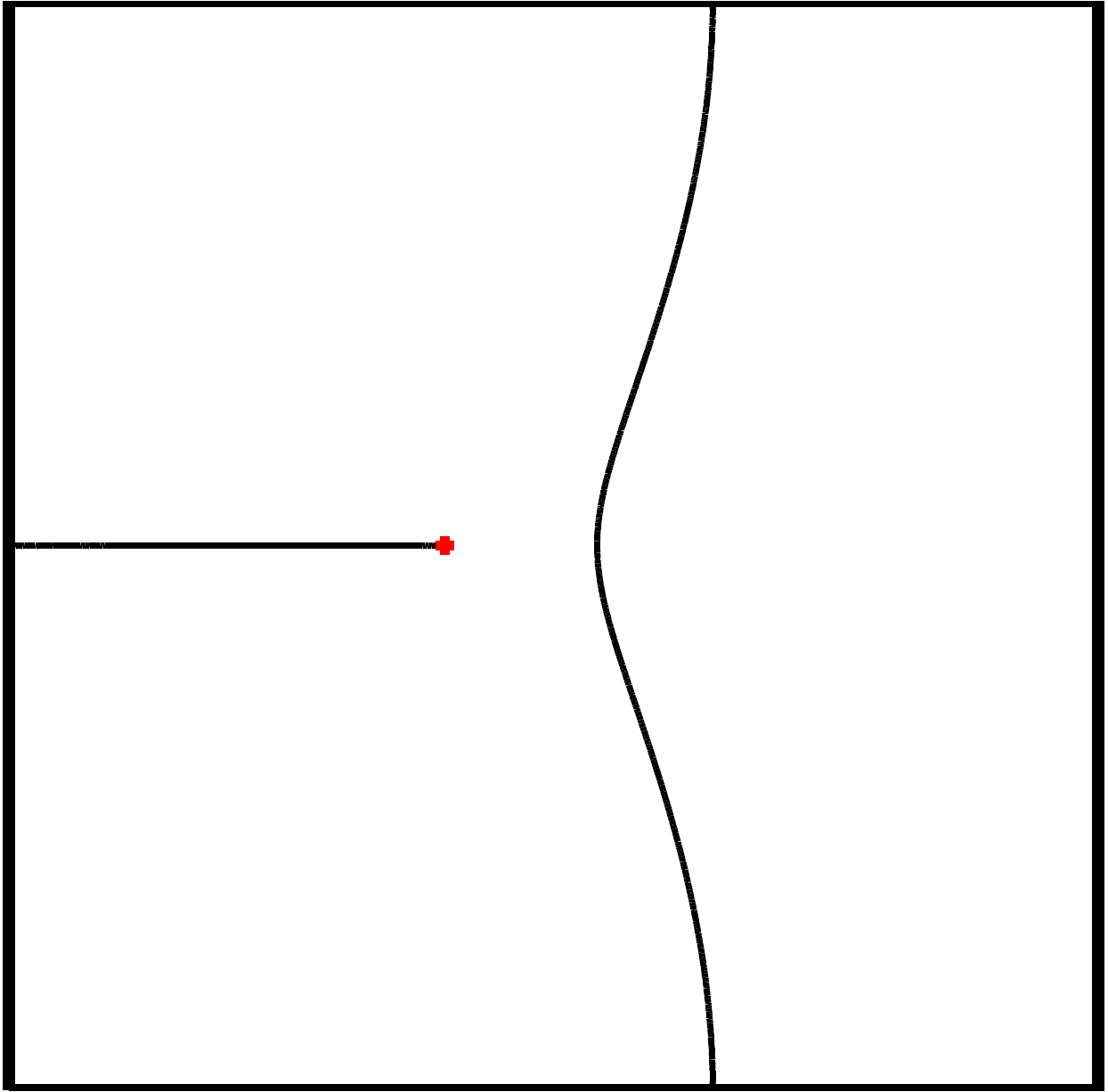}
&\includegraphics[height=2cm]{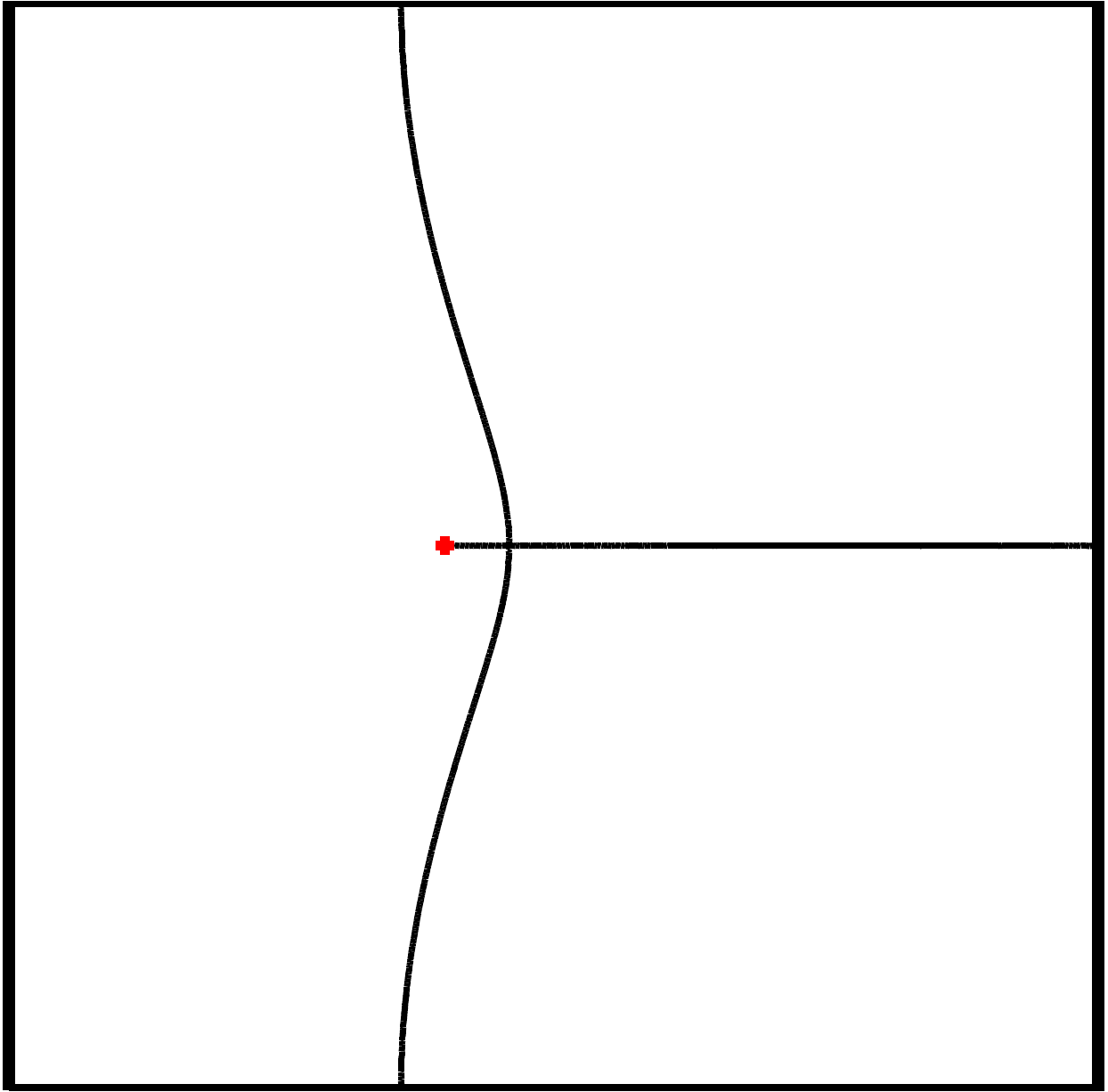}
&\includegraphics[height=2cm]{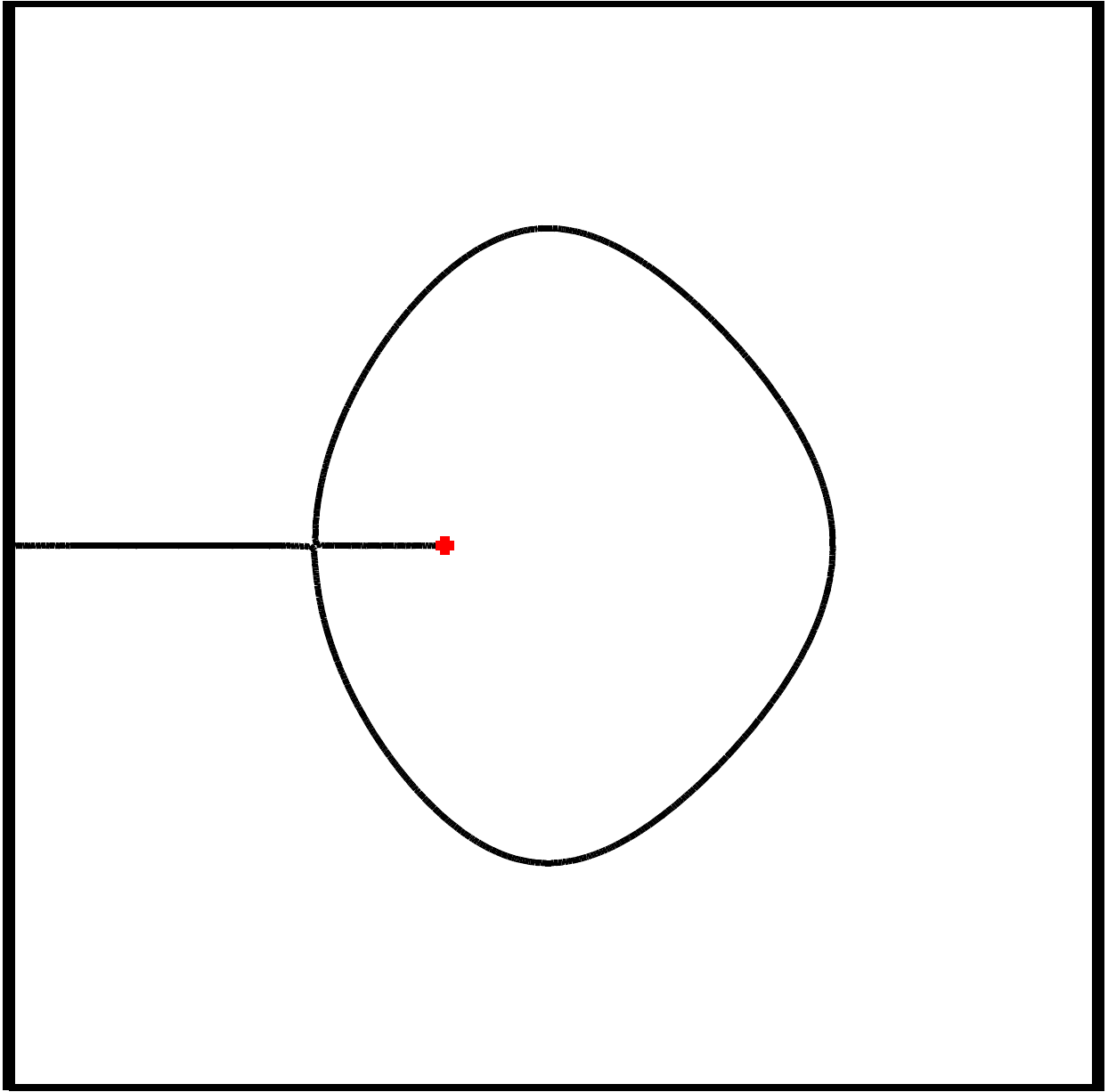}
&\includegraphics[height=2cm]{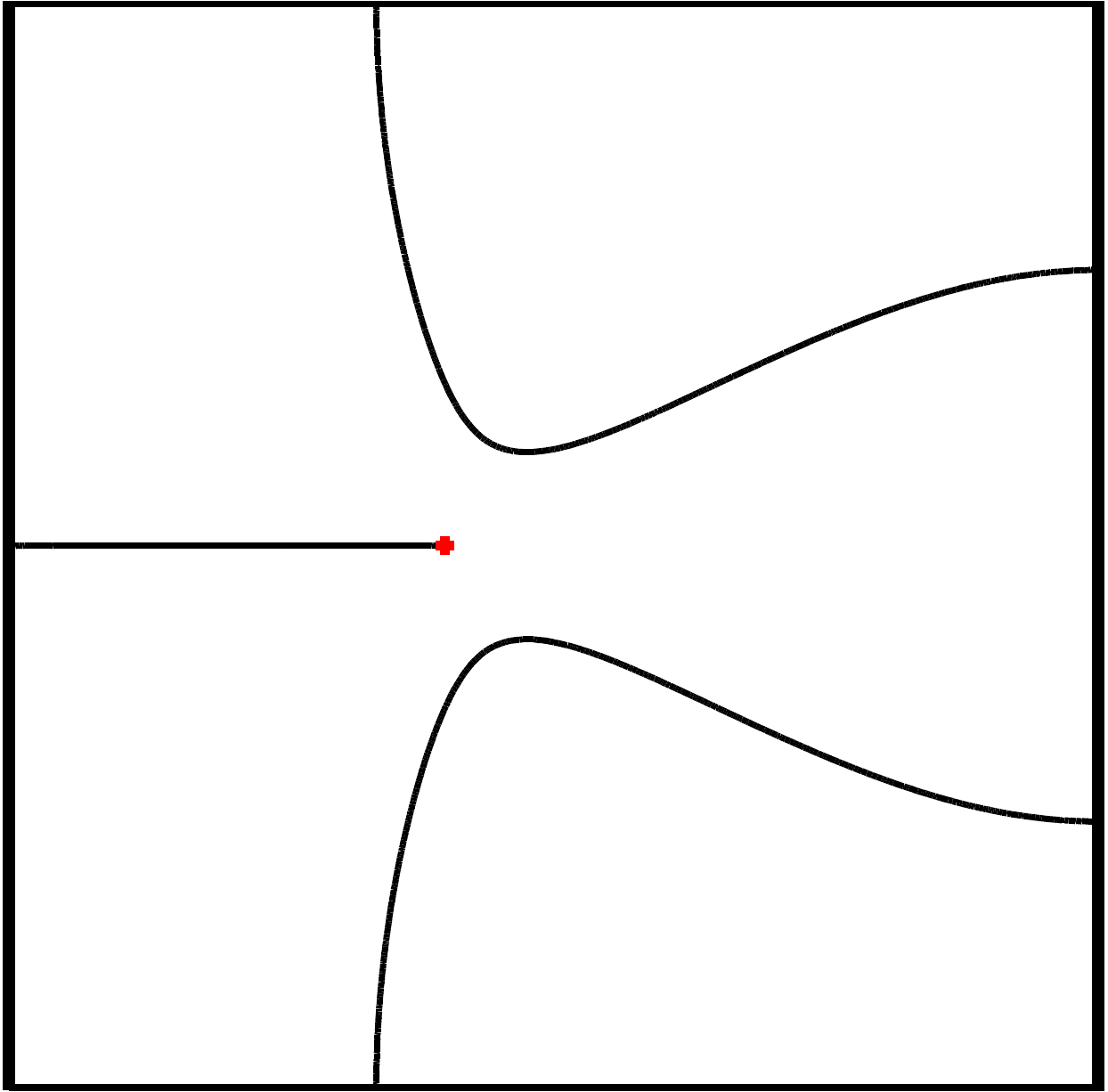}
\end{tabular}
\caption{Nodal lines of some Aharonov-Bohm eigenfunctions on the square.\label{fig.exvecpABcarre}}
\end{center}
\end{figure}
\begin{figure}[h!bt]
\begin{center}
\begin{tabular}{ccccc}
\includegraphics[height=2cm]{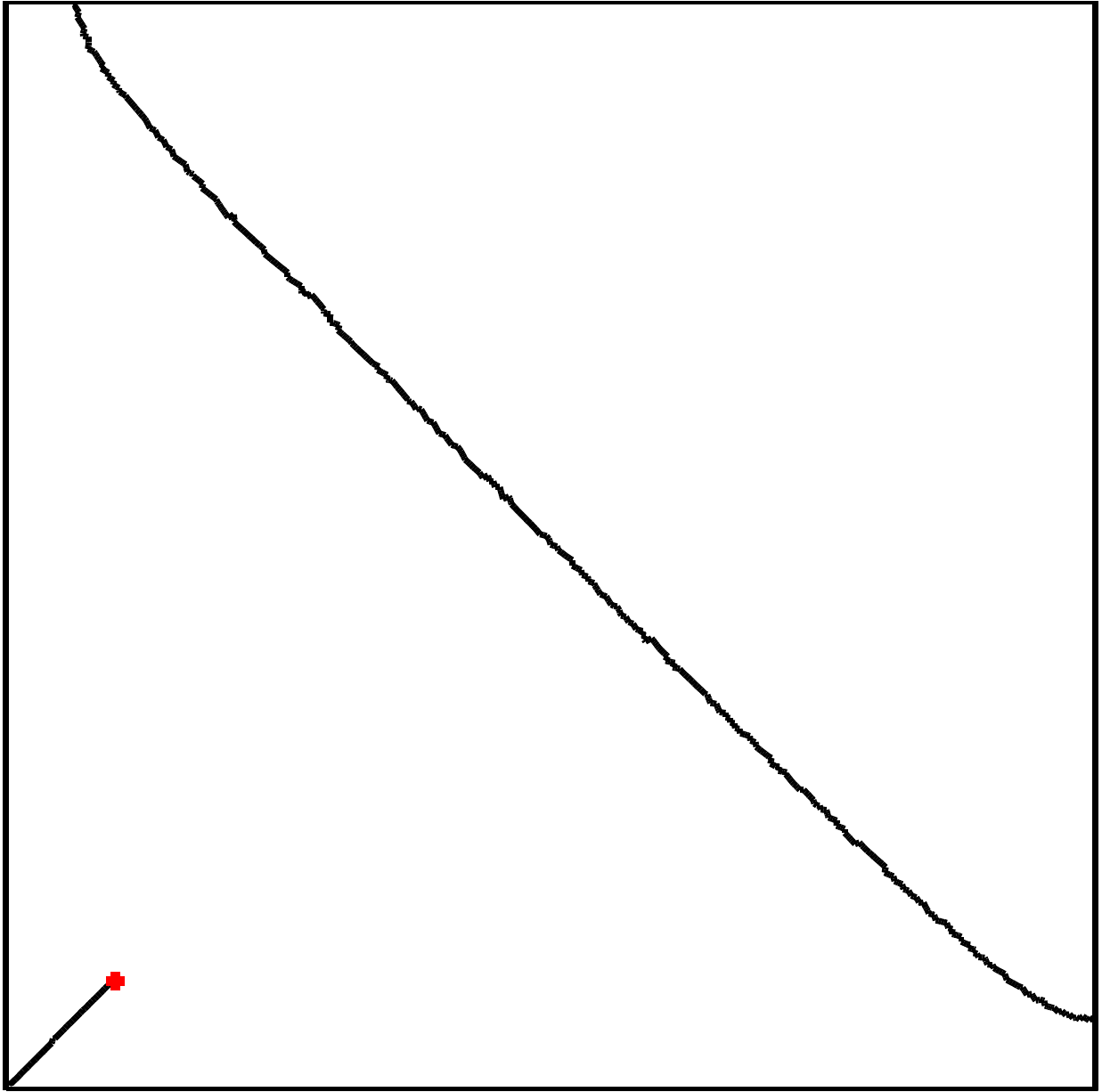}&
\includegraphics[height=2cm]{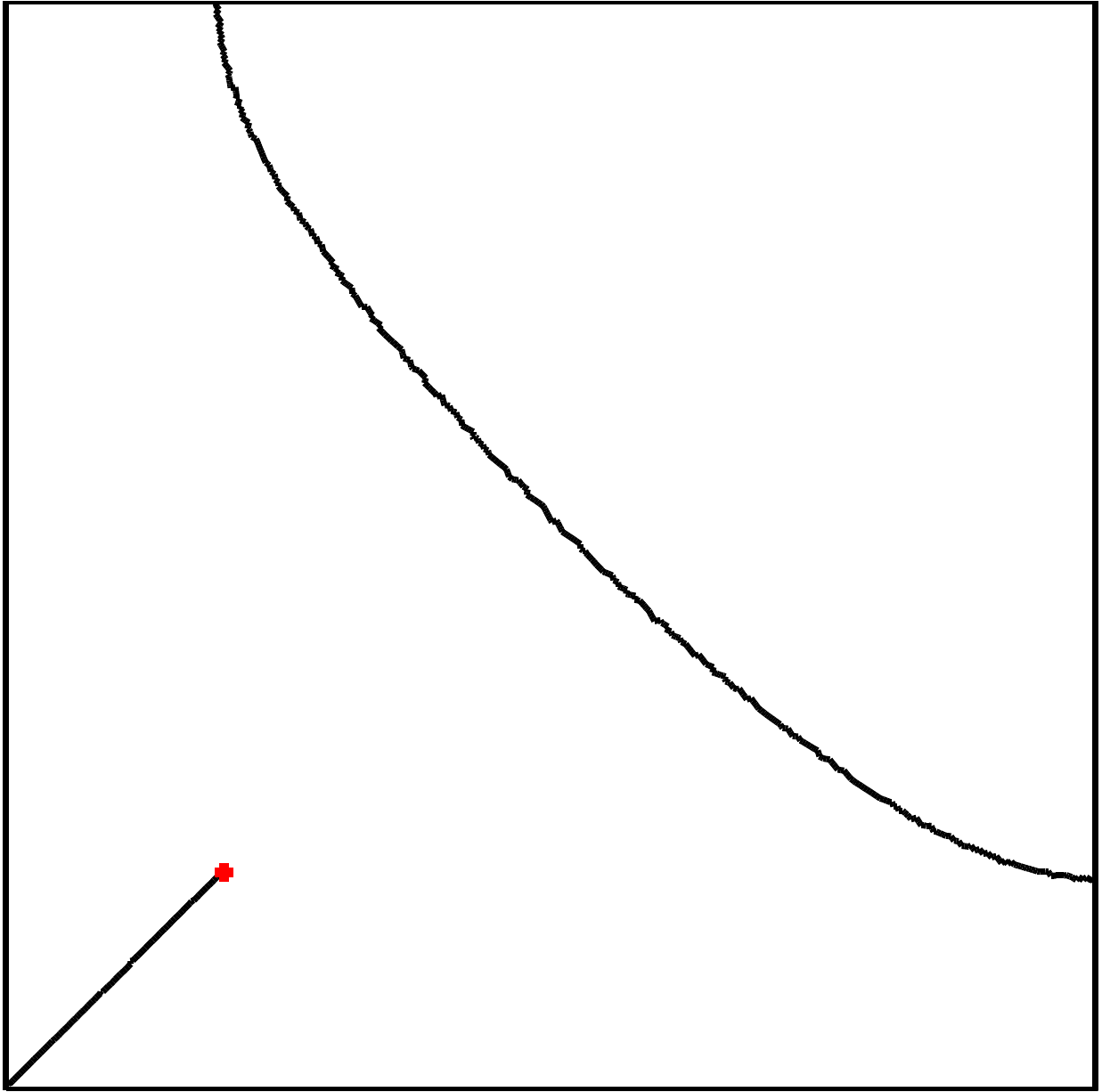}&
\includegraphics[height=2cm]{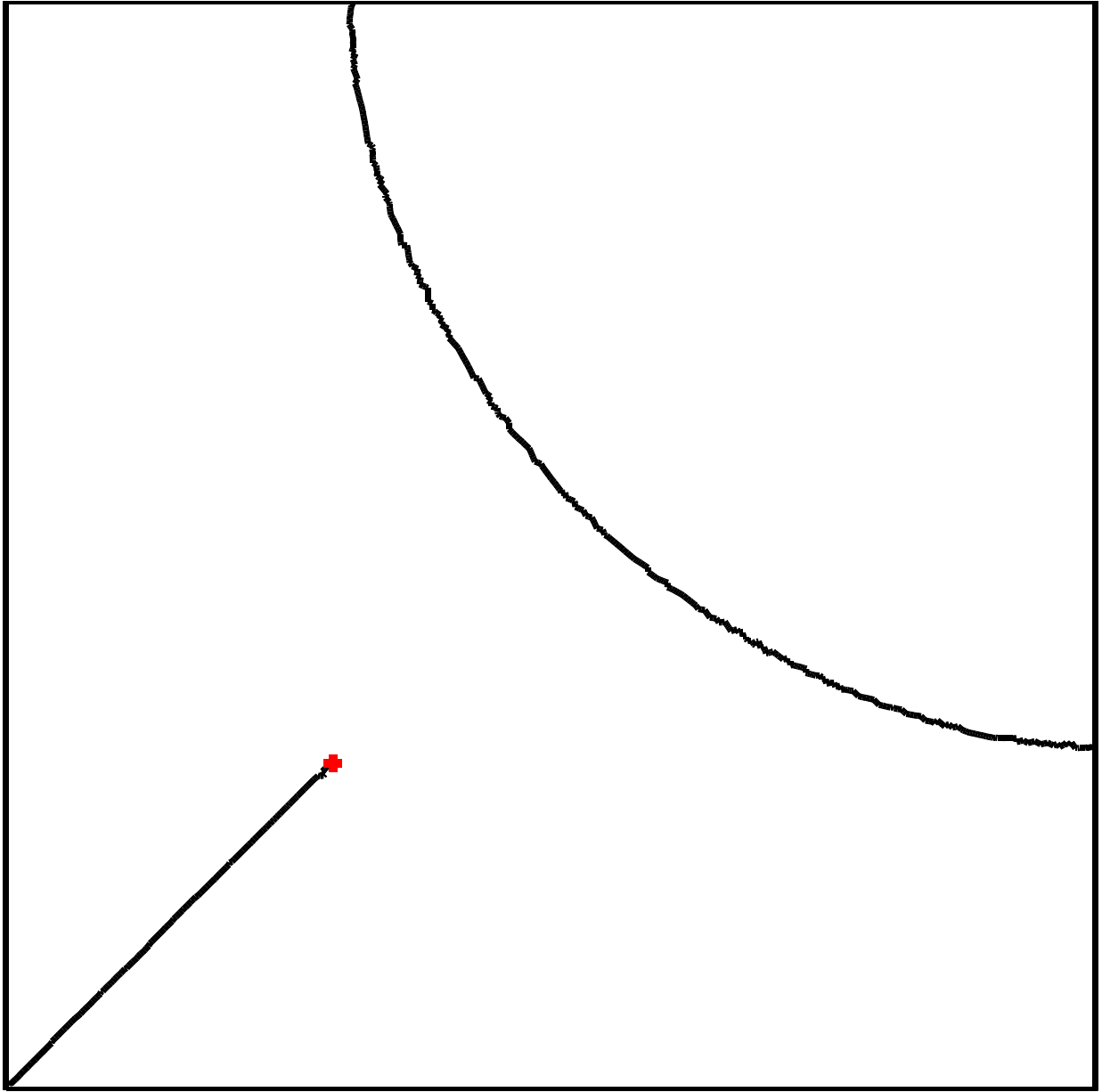}&
\includegraphics[height=2cm]{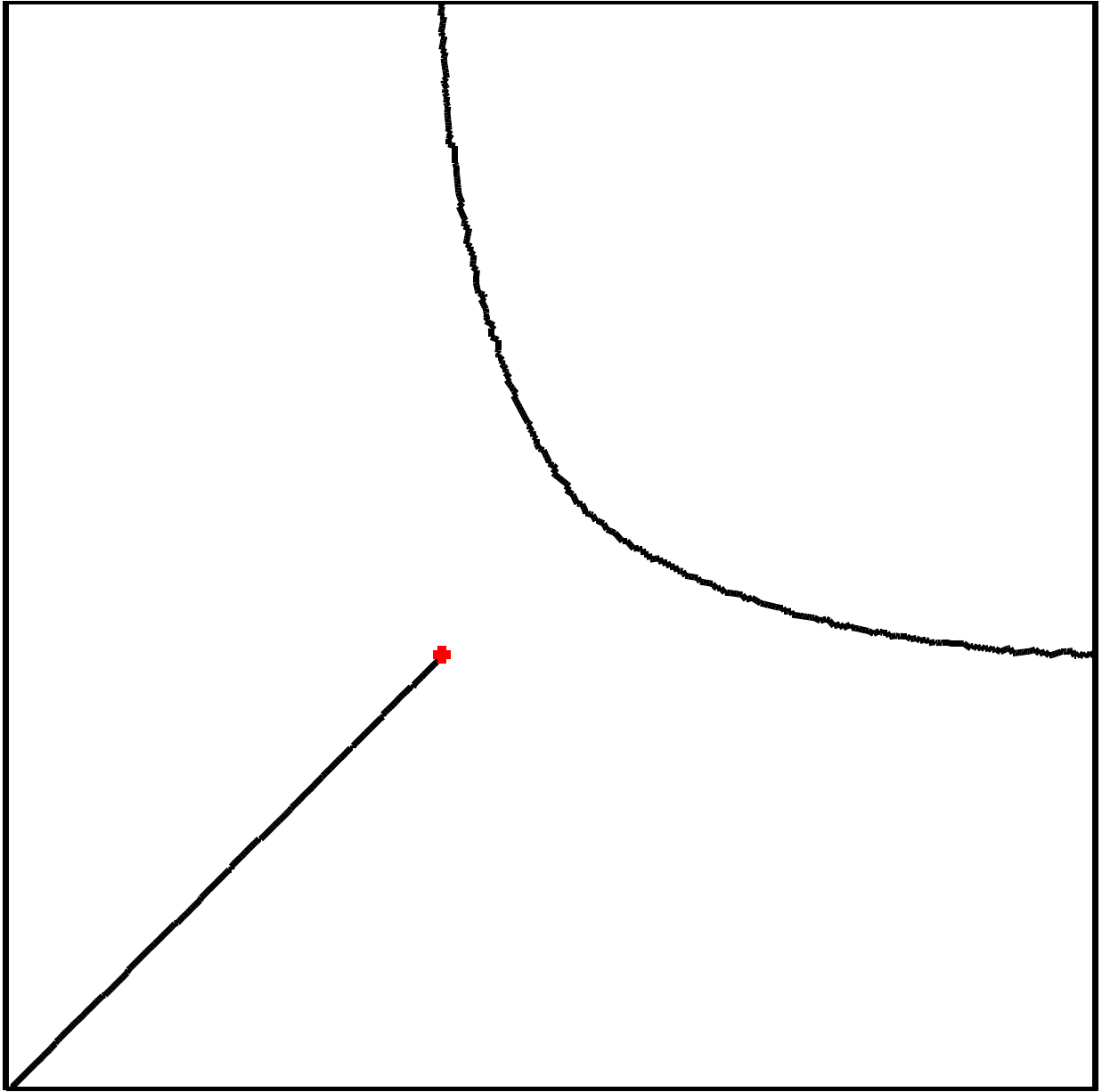}&
\includegraphics[height=2cm]{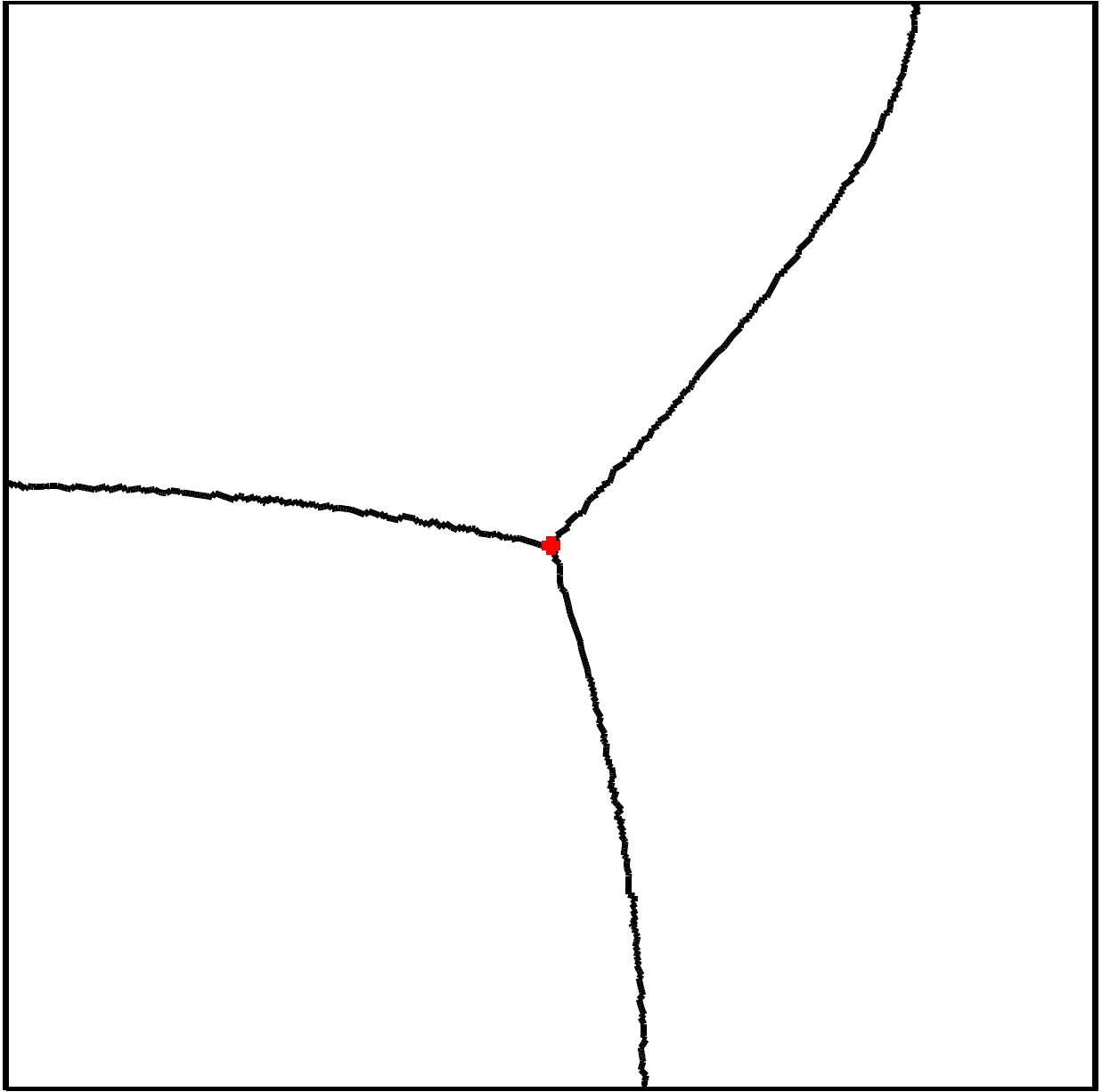}
\end{tabular}
\caption{Nodal lines for the third Aharonov-Bohm eigenfunction in function of $\cb$.\label{fig.VP3AB} on the diagonal.}
\end{center}
\end{figure}
The guess for the punctured square ($\pb$ at the center) is that any nodal partition of a third $K_{\cb}$-real eigenfunction gives a minimal $ 3$-partition.
Numerics shows that this is only true if the square is punctured at the center (see Figure~\ref{fig.VP3AB} and \cite{BH} for a systematic study). Moreover the third eigenvalue is maximal there and has multiplicity two (see Figure~\ref{fig.ABsquare}).

\subsection{Continuity with respect to the poles}
In the case of a unique singular point, \cite{NT}, \cite[Theorem 1.1]{BNNT} establish the continuity with respect to the singular point till the boundary.
\begin{theorem}\label{thm.BNNT} 
Let $\alpha\in [0,1)$ and $\lambda_{k}^{AB}(\cb,\alpha)$ be the $k$-th eigenvalue of $H^{AB}(\dot\Omega_{\cb},\alpha)$. Then the function $\cb\in\Omega\mapsto \lambda_{k}^{AB}(\cb,\alpha)$ admits a continuous extension on $\overline\Omega$ and 
\begin{equation}\label{contpole}
\lim_{\cb\to\partial\Omega}\lambda_{k}^{AB}(\cb,\alpha) = \lambda_{k},\qquad\forall k\geq1,
\end{equation}
where $\lambda_k$ is the $k$-th eigenvalue of $H(\Omega)$. 
\end{theorem}
The theorem implies that the function $\cb\mapsto \lambda_{k}^{AB}(\cb,\alpha)$ has an extremal point in $\overline \Omega$. Note also that $\lambda_k^{AB}(\cb,\alpha)$ is well defined for $\cb\not\in \Omega$ and is equal to $\lambda_k(\Omega)$. One can indeed find a solution $\phi$ in $\Omega$ satisfying $d\phi =\Ab_{\cb}$, and $u \mapsto \exp(i \alpha \phi) \, u$ defines the unitary transform intertwining $H(\Omega)$ and $H^{AB}(\dot \Omega_{\cb}, \alpha)$.
\begin{figure}[h!bt]
\begin{center}
\subfigure[$\cb\mapsto \lambda^{AB}_{k}(\cb)$, $\cb\in\Omega$, $1\leq k\leq 5$]{\begin{tabular}{cccccc}
$\lambda_{1}^{AB}$ & $\lambda_{2}^{AB}$ & $\lambda_{3}^{AB}$ & $\lambda_{4}^{AB}$ & $\lambda_{5}^{AB}$\\
\includegraphics[width=2.8cm]{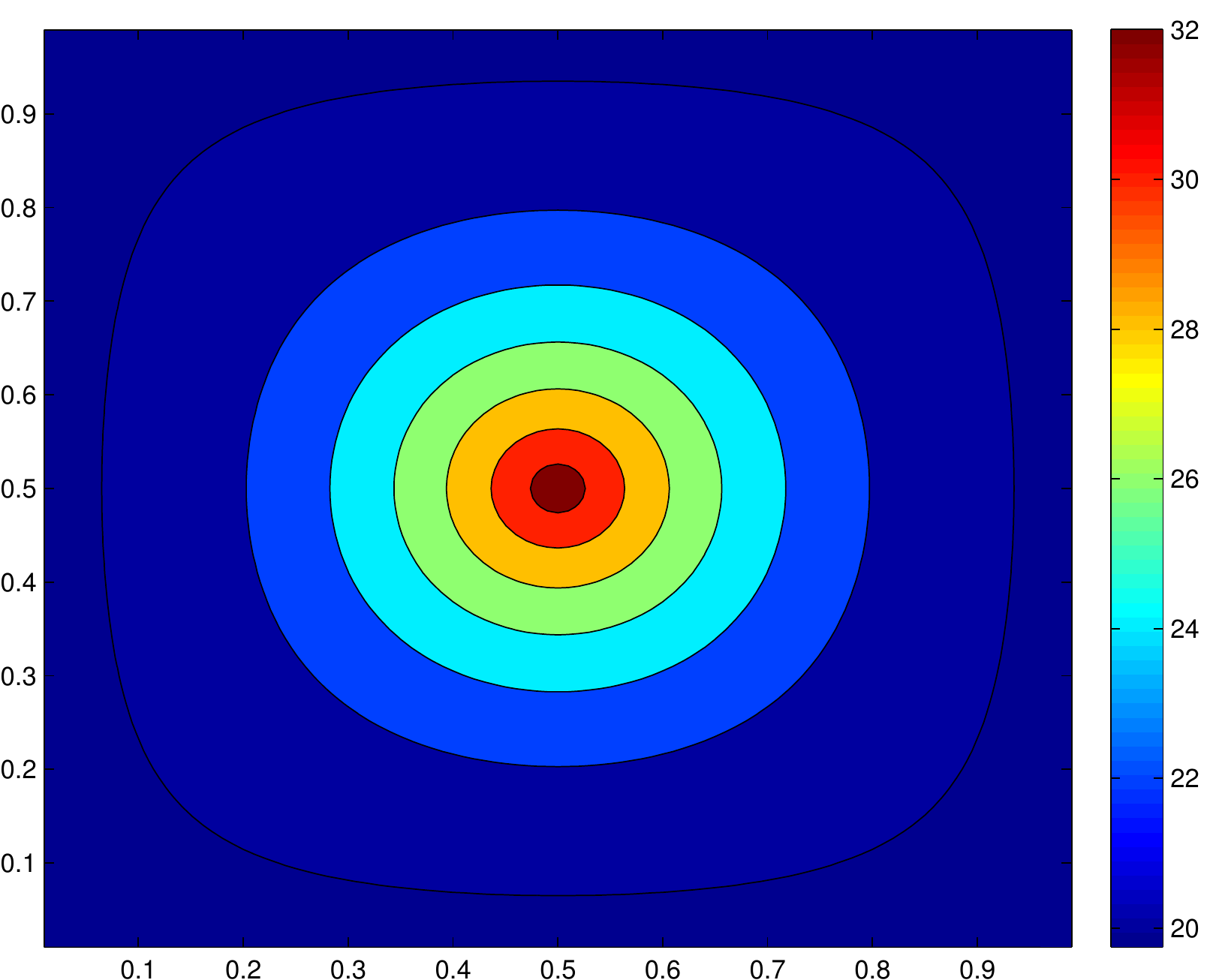}
&\includegraphics[width=2.8cm]{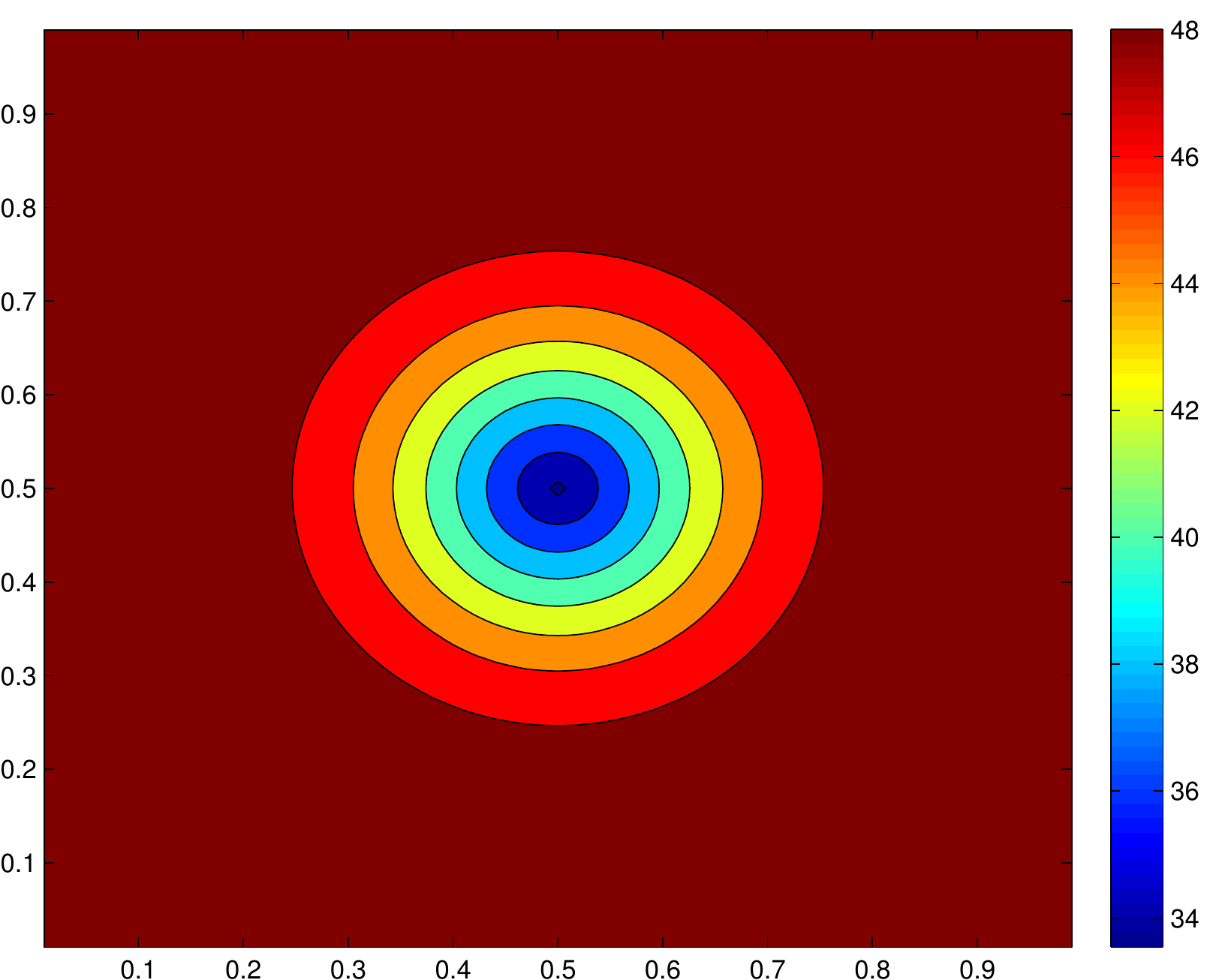}
&\includegraphics[width=2.8cm]{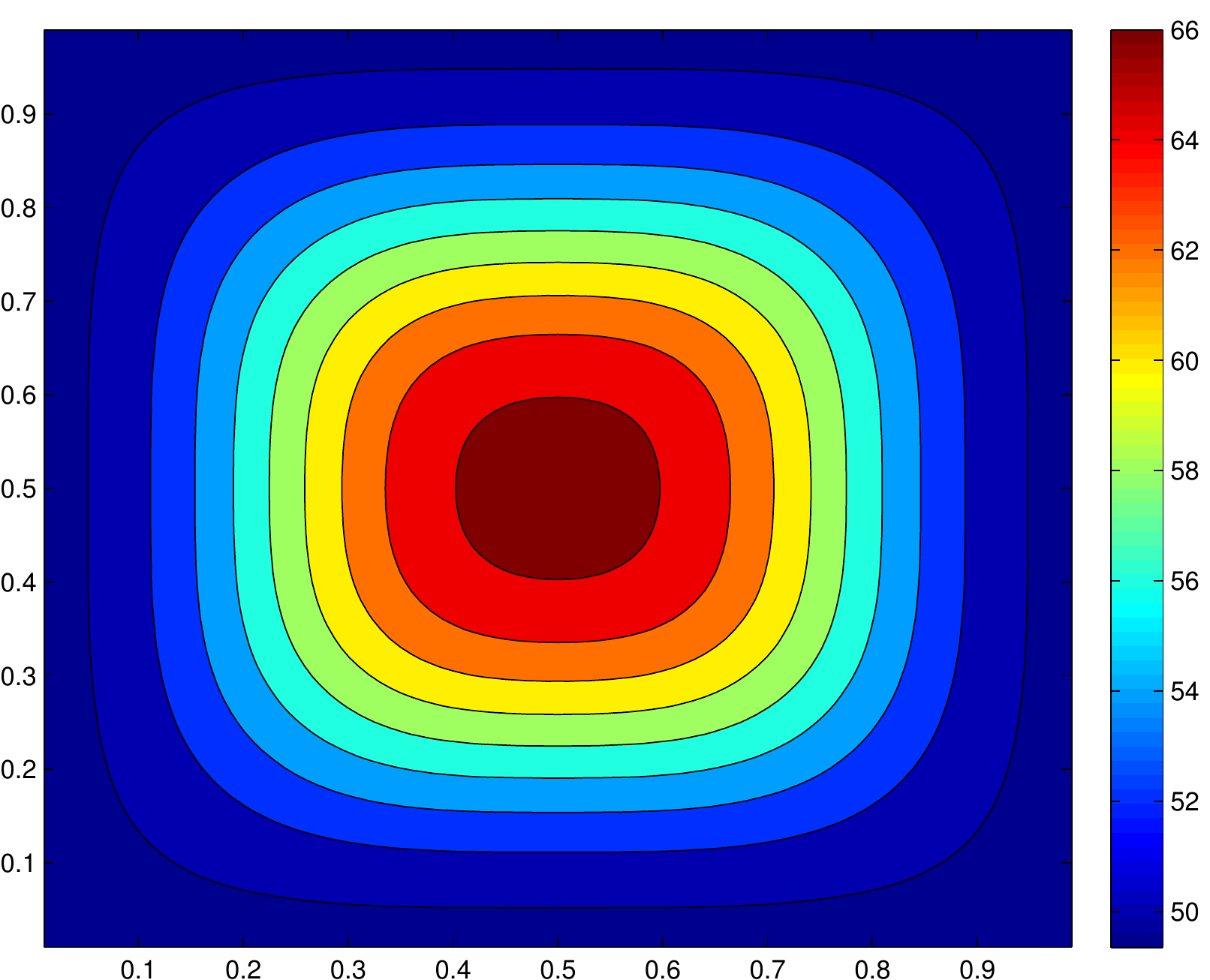}
&\includegraphics[width=2.8cm]{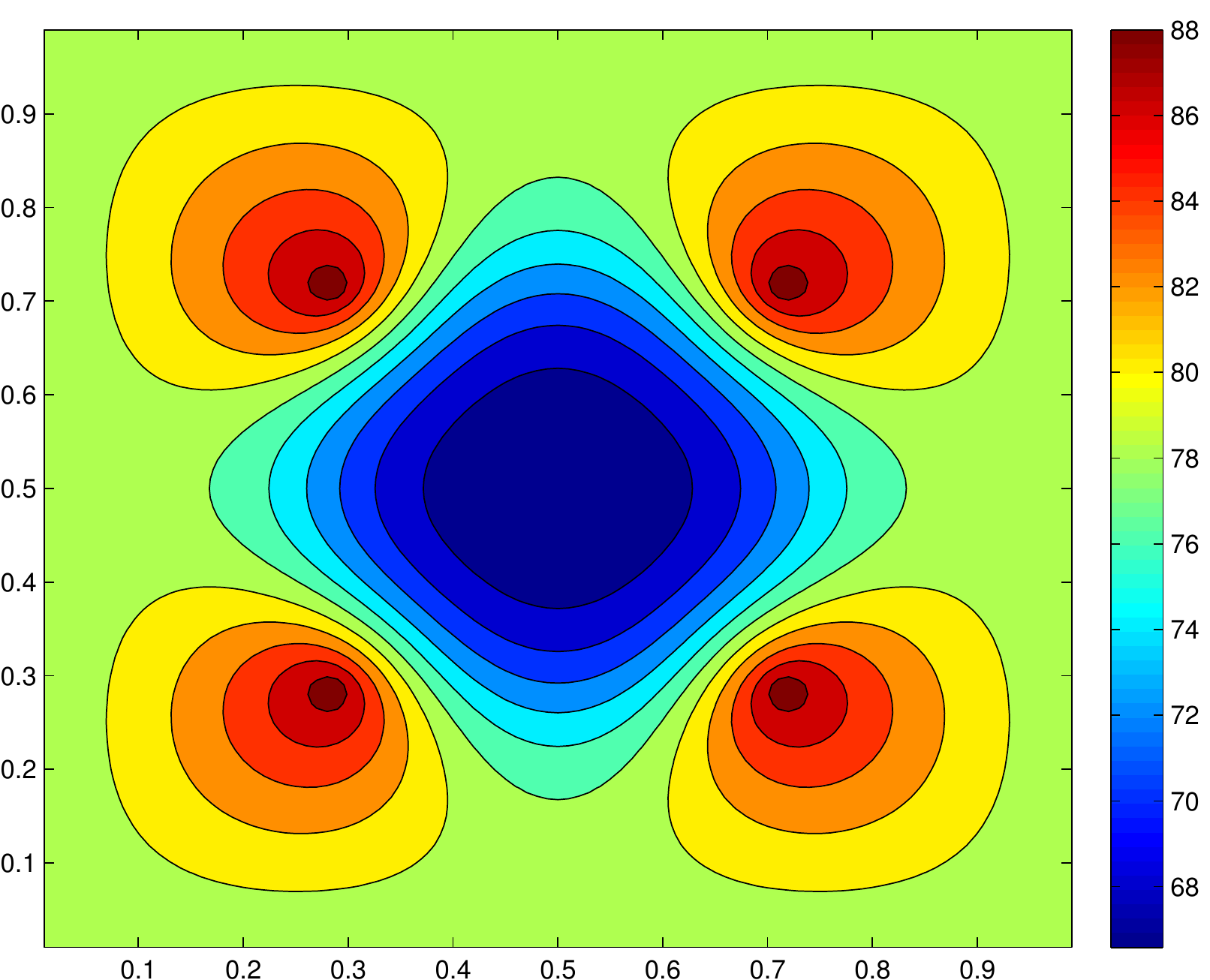}
&\includegraphics[width=2.8cm]{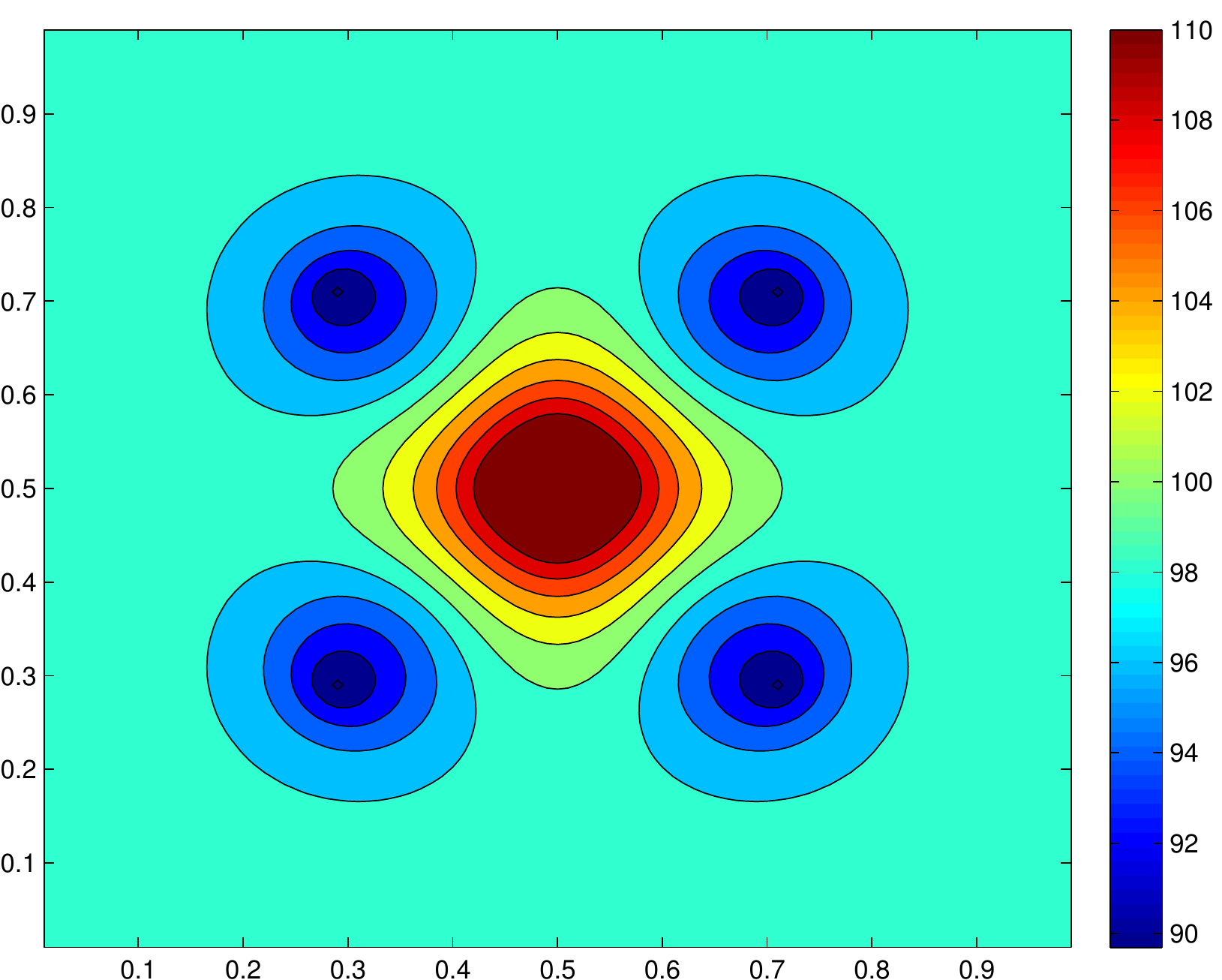}
\end{tabular}}
\subfigure[$\cb\mapsto \lambda^{AB}_{k}(\cb)$, $\cb=(p,p)$, $1\leq k\leq 5$.\label{fig.diagcarre}]{ \quad\includegraphics[height=4cm]{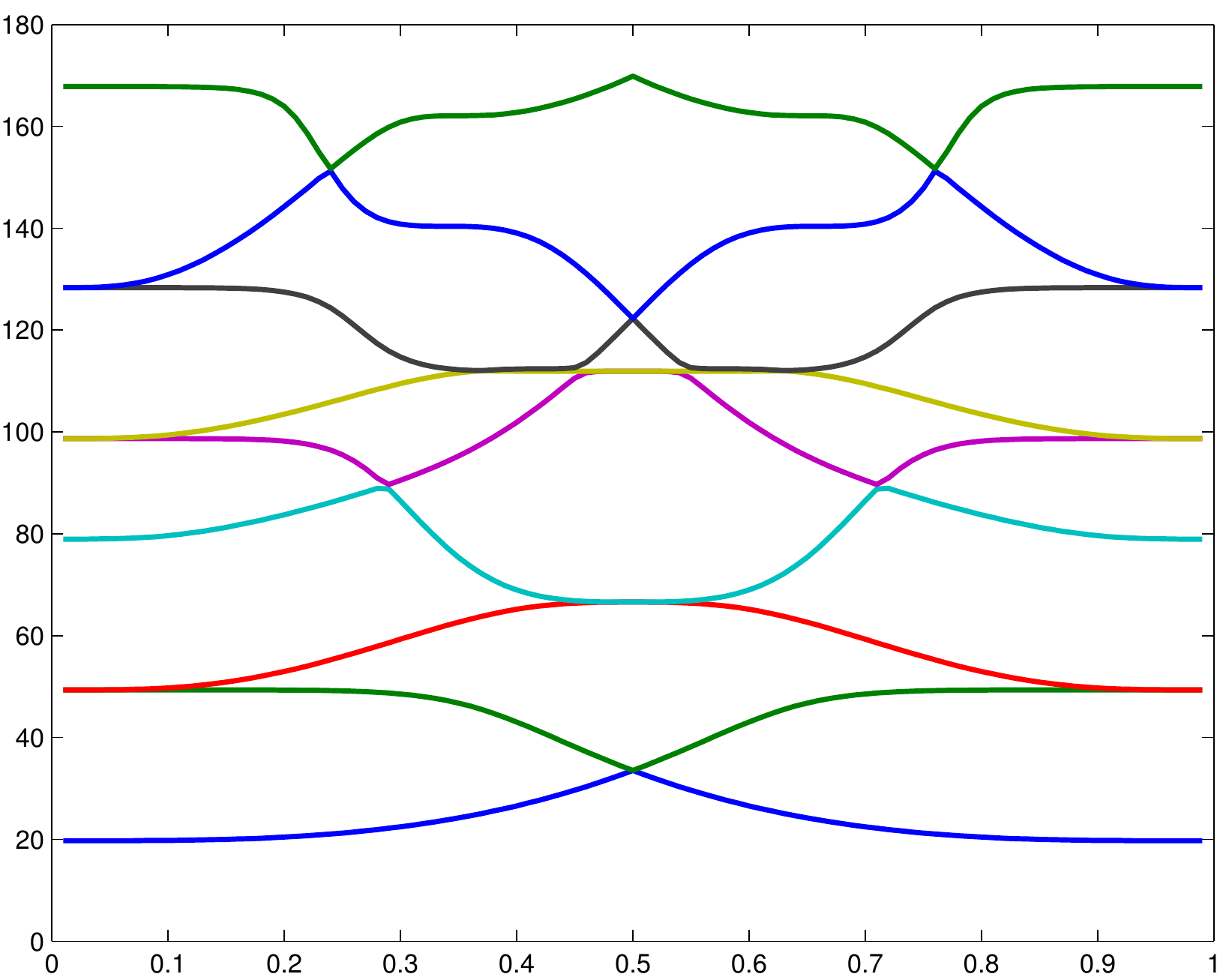}\quad}
\subfigure[$\cb\mapsto \lambda^{AB}_{k}(\cb)$, $\cb=(p,\frac12)$, $1\leq k\leq 5$.\label{fig.medcarre}]{ \quad\includegraphics[height=4cm]{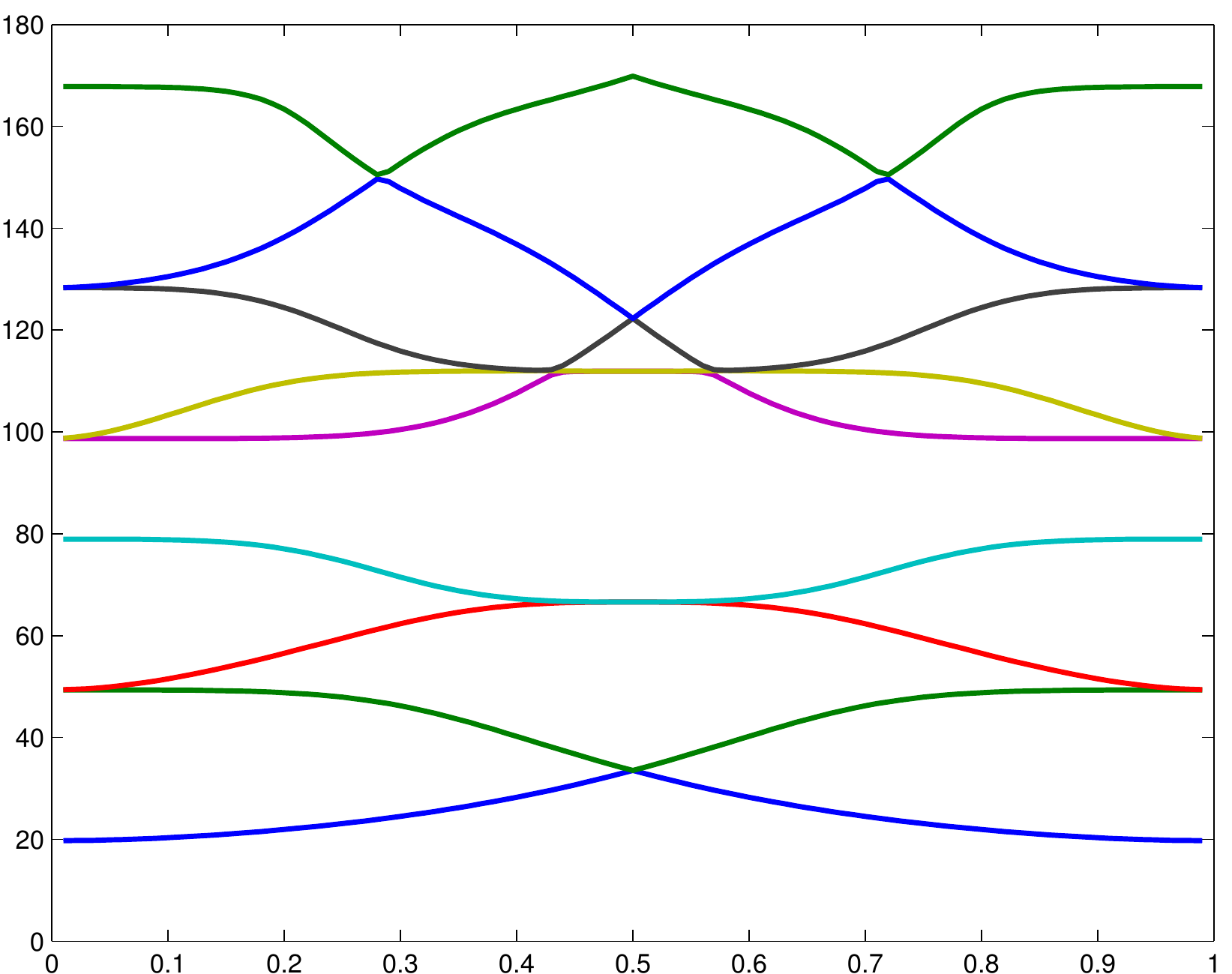}\quad}
\caption{Aharonov-Bohm eigenvalues on the square as functions of the pole.\label{fig.ABsquare}}
\end{center}
\end{figure}

Figures~\ref{fig.ABsquare}--\ref{fig.VPABy0} give some illustrations (see also \cite{BH,BNNT}) in the case of a square or of an angular sector with opening $\pi/3$ or $\pi/4$ and with a flux $\alpha=1/2$. \\
Figure~\ref{fig.ABsquare} gives the first eigenvalues of $H^{AB}(\dot\Omega_{\cb})$ in function of $\cb$ in the square $\Omega=[0,1]^2$ and demonstrates \eqref{contpole}. When $\pb = (1/2,1/2)$, the eigenvalue is extremal and always double (see in particular Figures~\ref{fig.diagcarre} and \ref{fig.medcarre} which represent the first eigenvalues when the pole is either on a diagonal line or on a bisector line). 

\begin{figure}[h!tb]
\begin{center}
\begin{tabular}{cccccc}
$\lambda_{1}^{AB}$ & $\lambda_{2}^{AB}$ & $\lambda_{3}^{AB}$ & $\lambda_{4}^{AB}$ & $\lambda_{5}^{AB}$\\
\includegraphics[width=2.8cm]{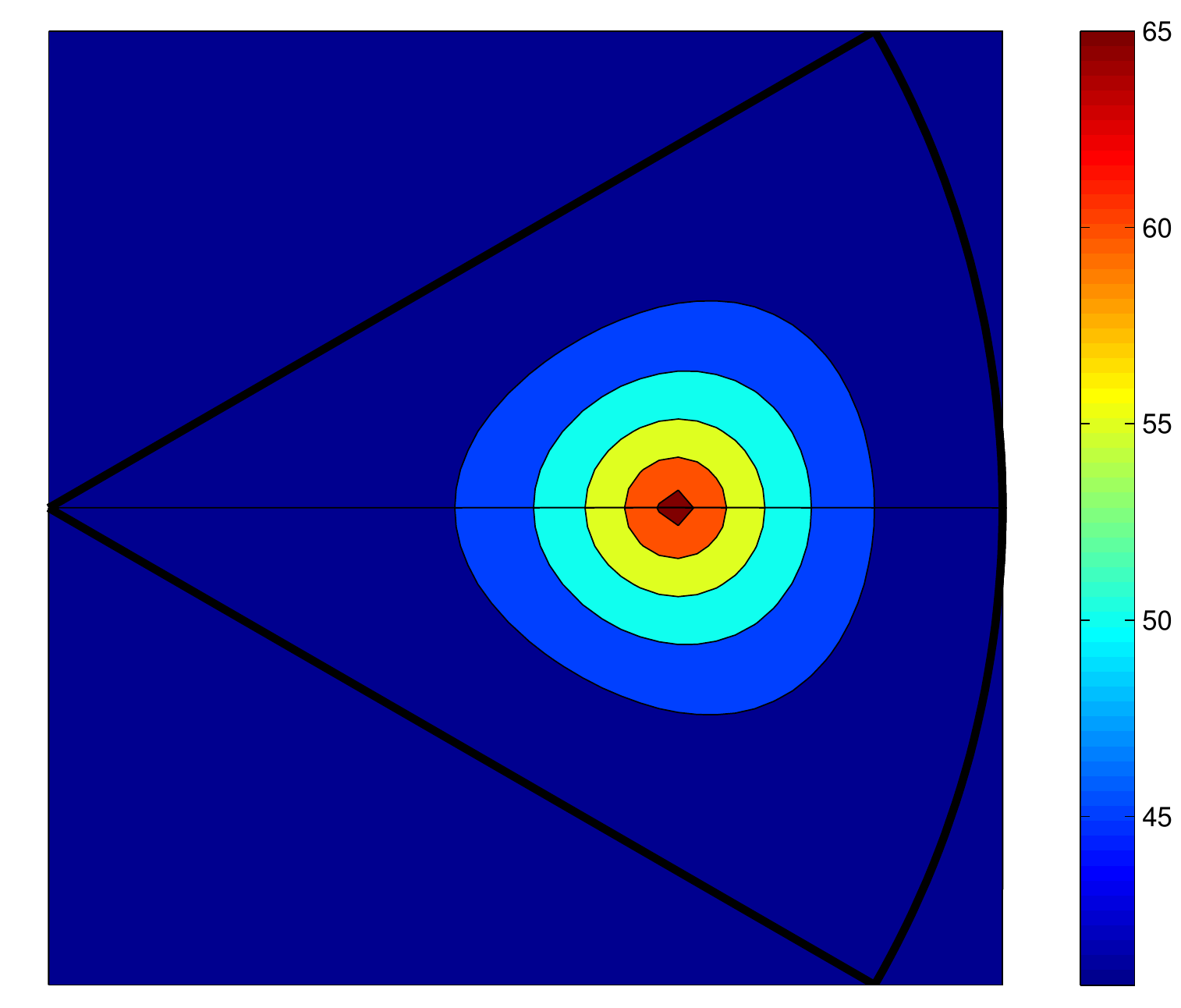}
&\includegraphics[width=2.8cm]{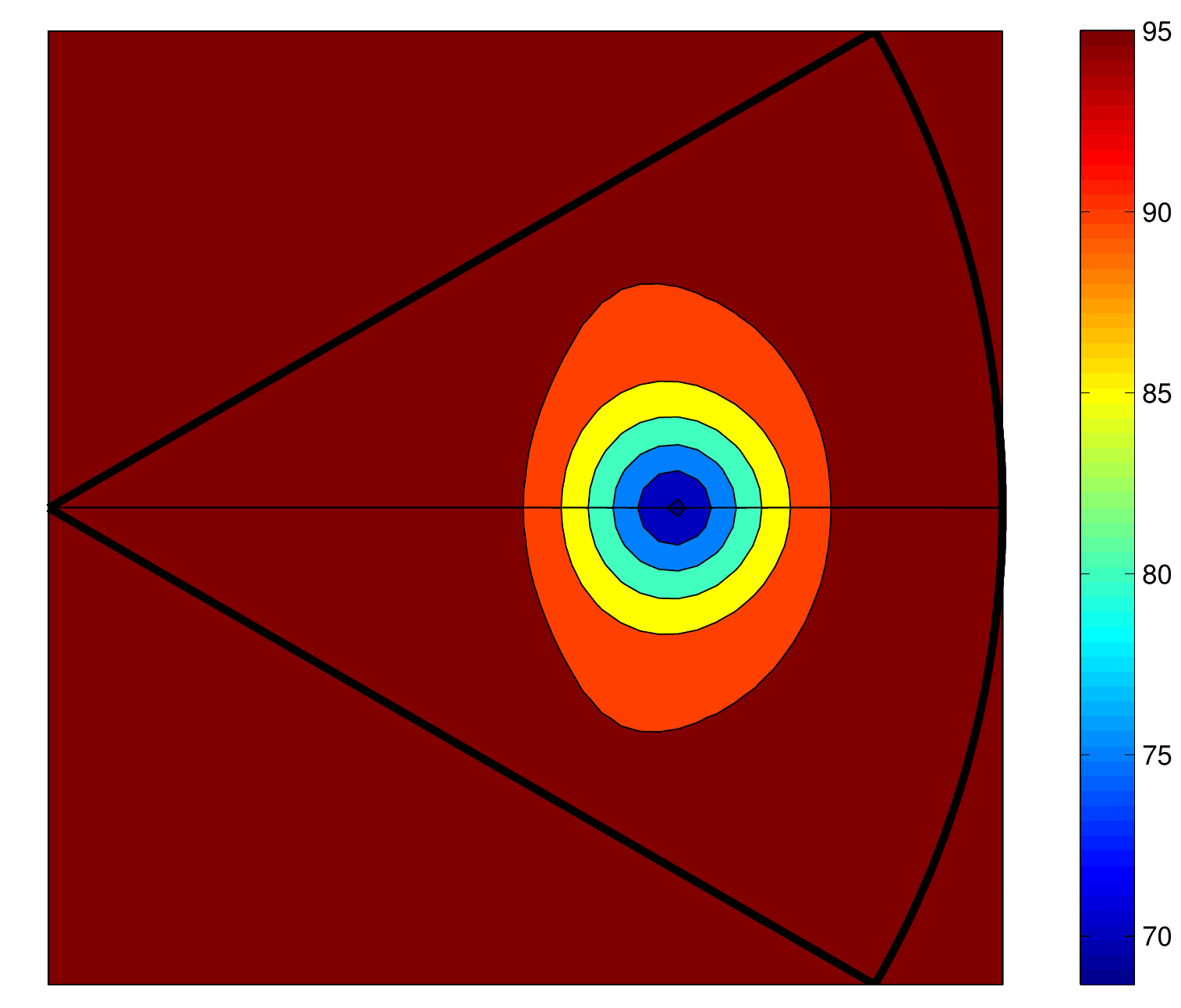}
&\includegraphics[width=2.8cm]{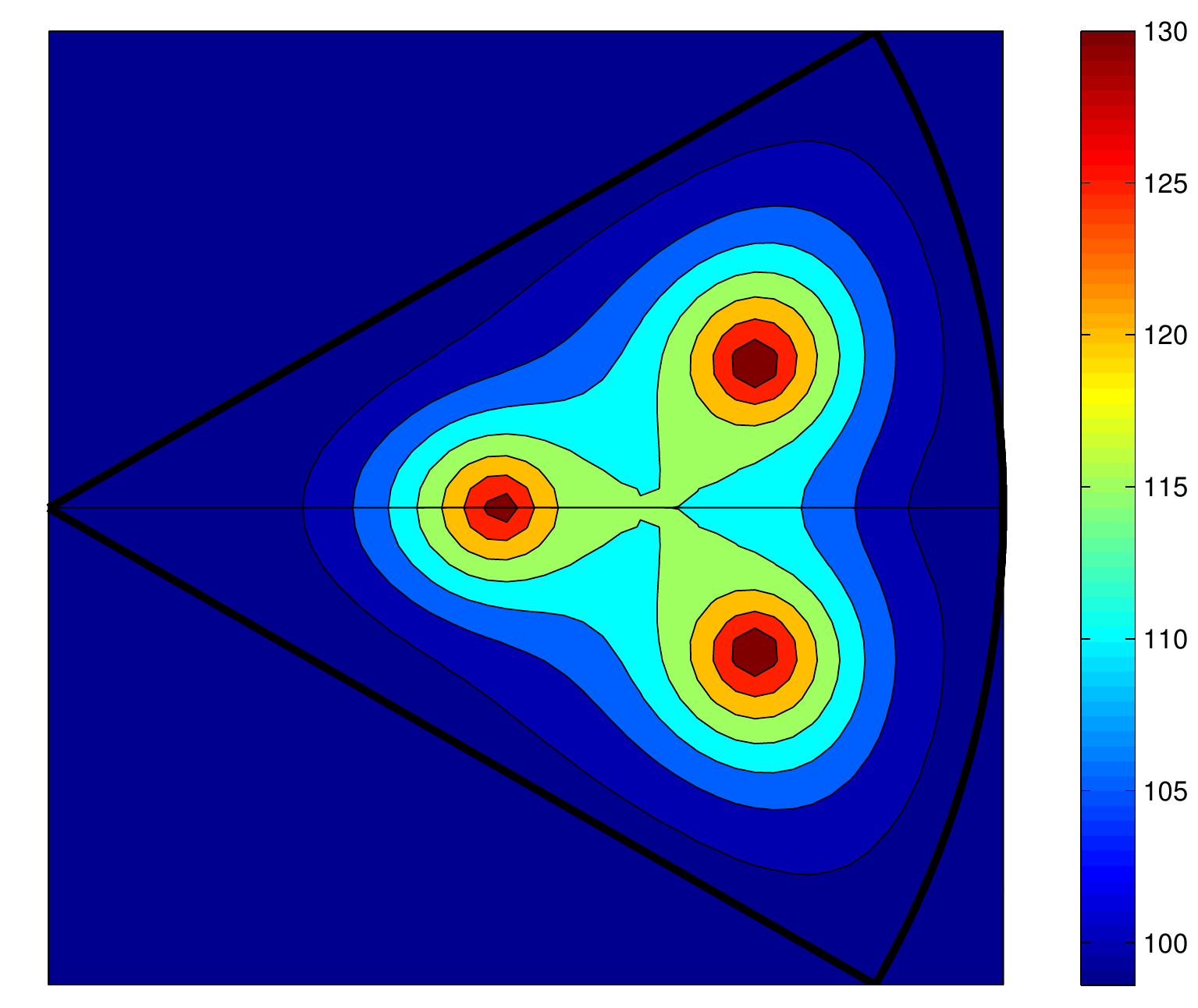}
&\includegraphics[width=2.8cm]{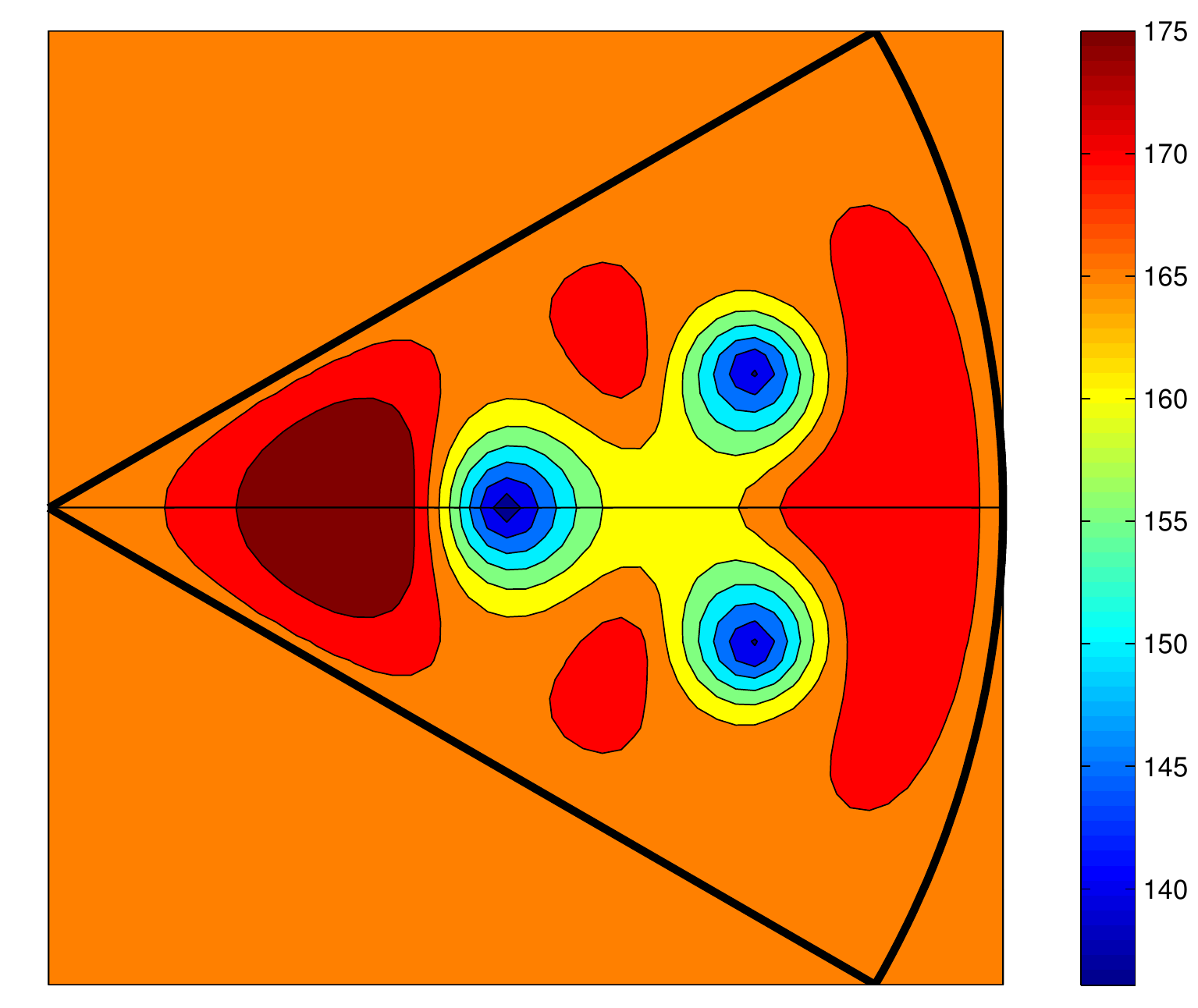}
&\includegraphics[width=2.8cm]{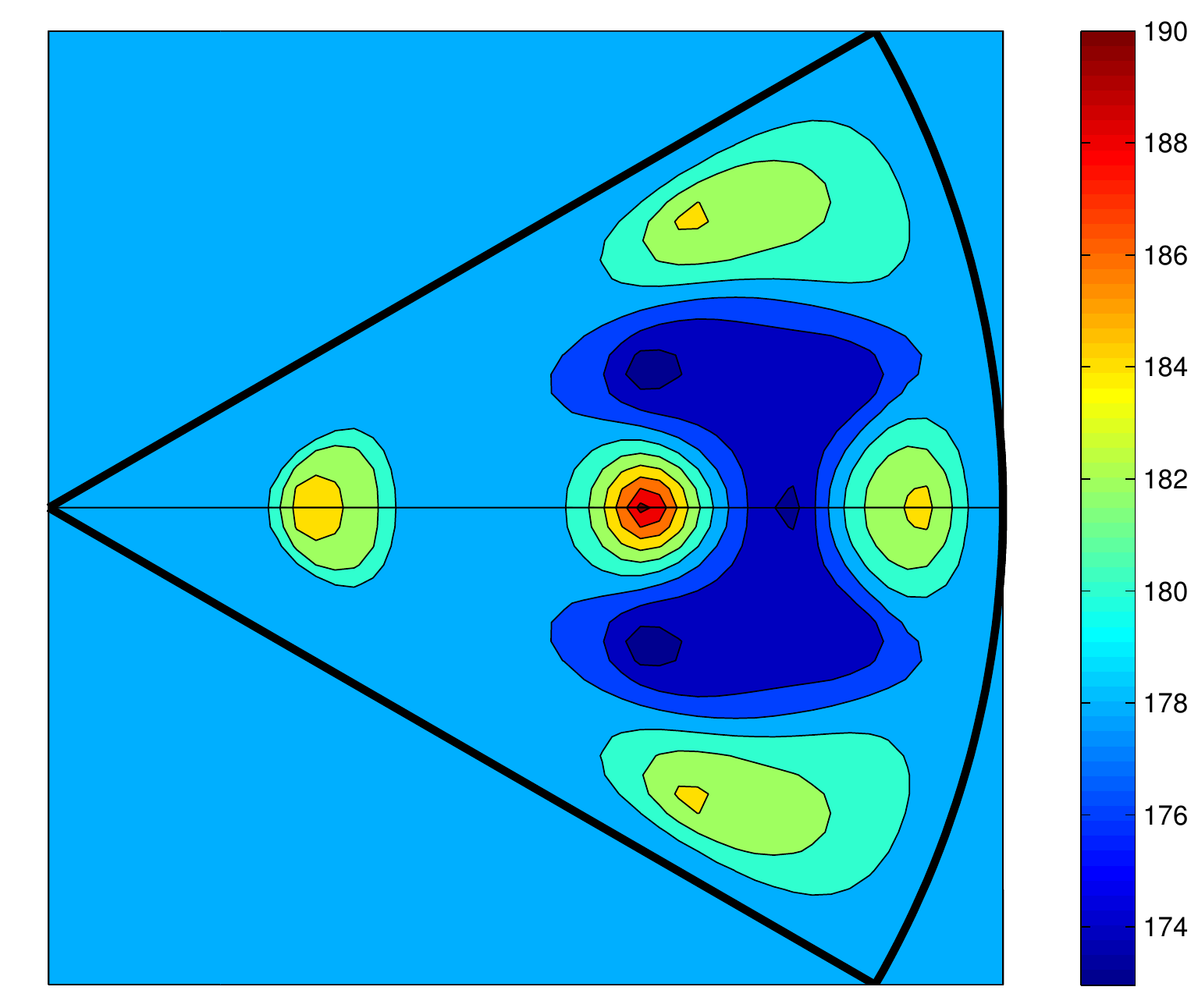}\\
\includegraphics[width=2.8cm]{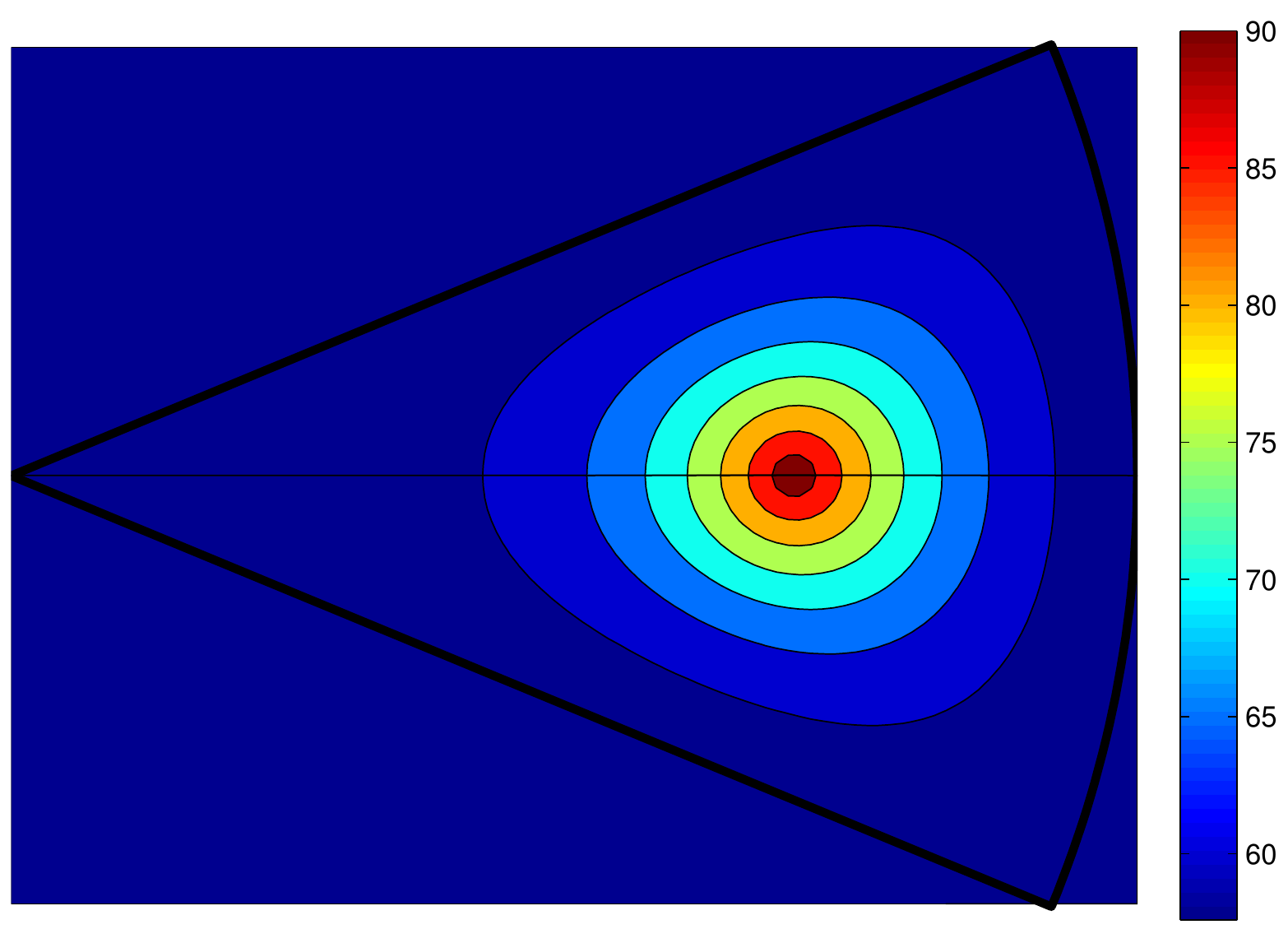}
&\includegraphics[width=2.8cm]{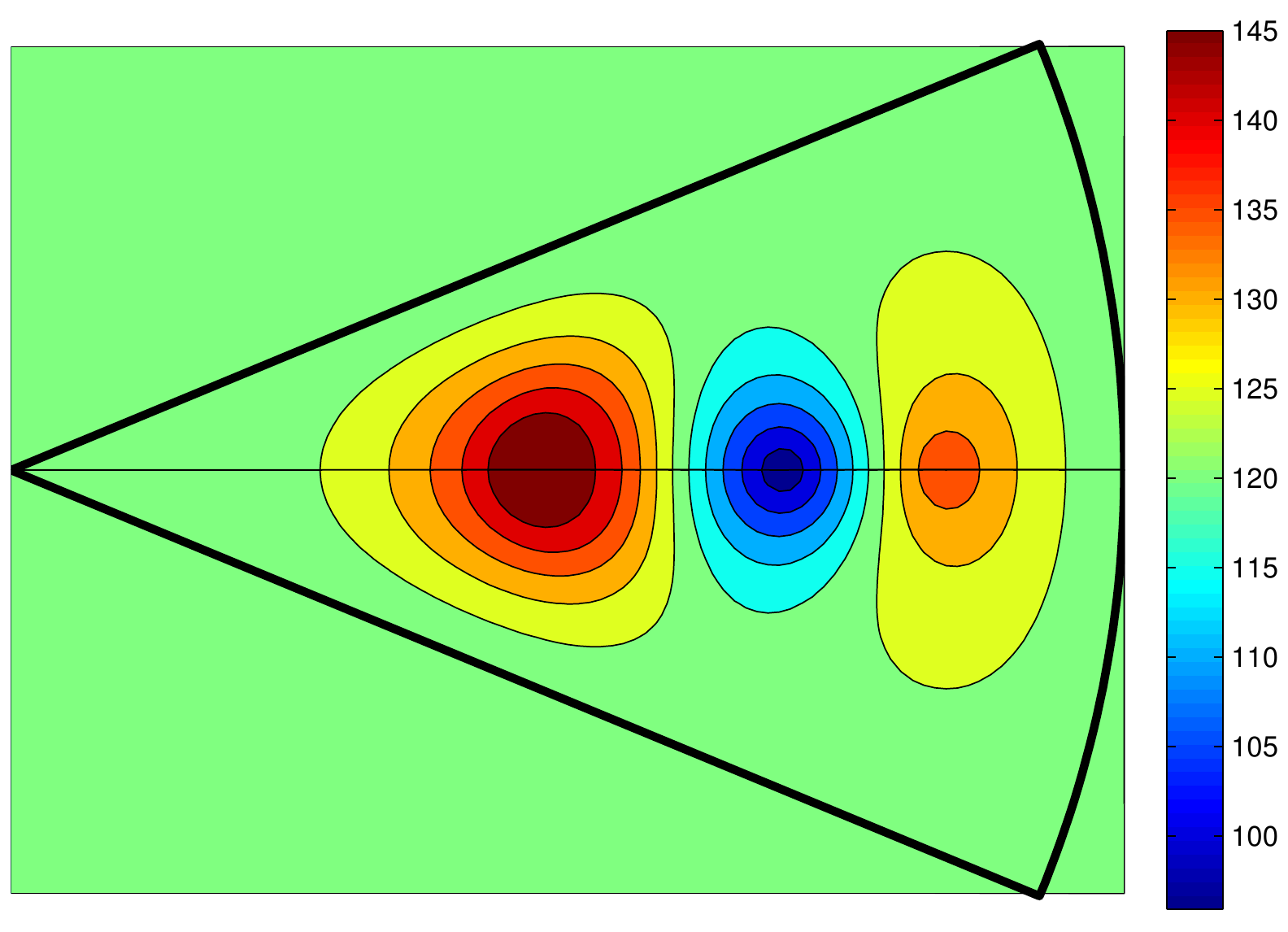}
&\includegraphics[width=2.8cm]{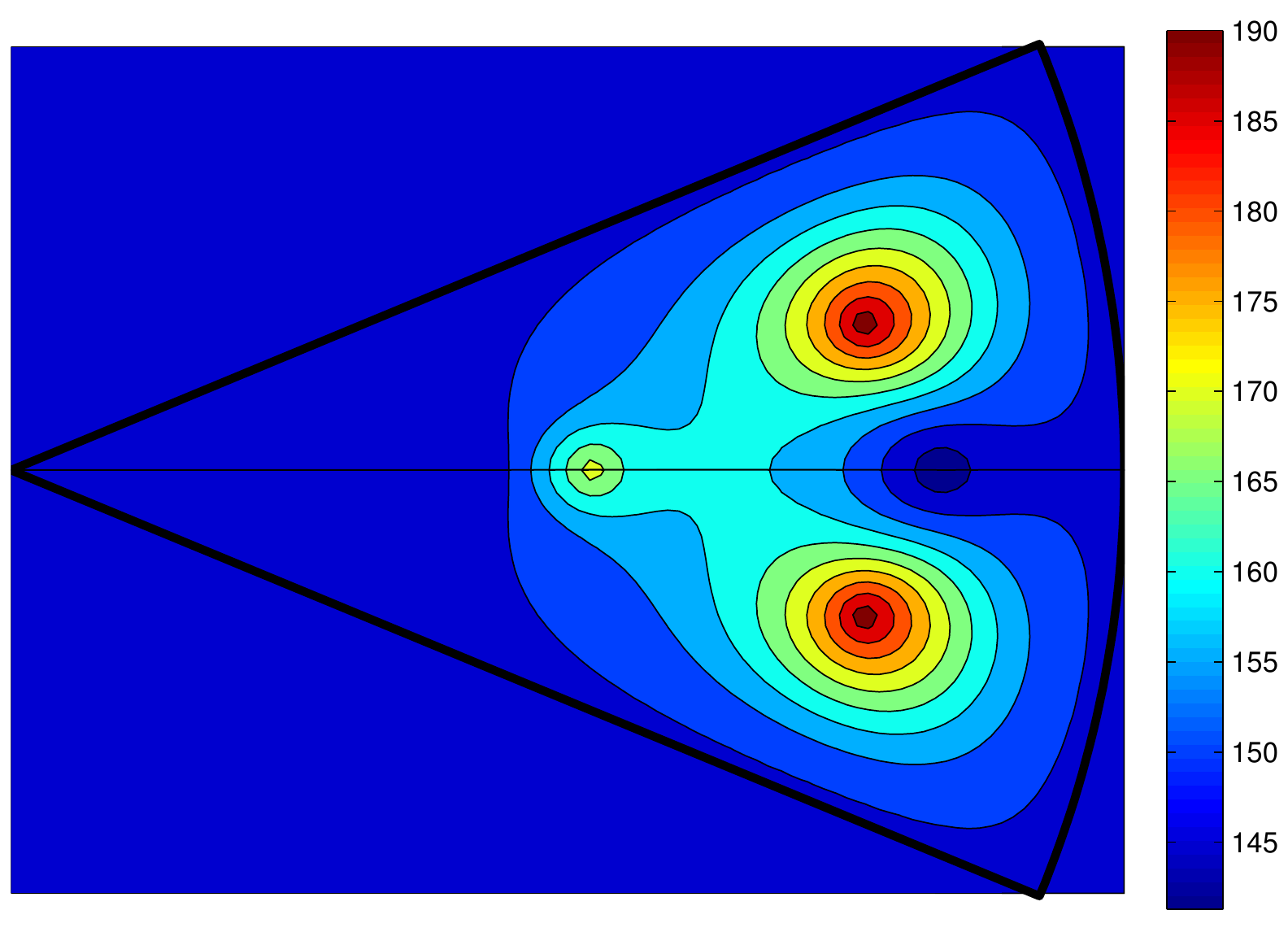}
&\includegraphics[width=2.8cm]{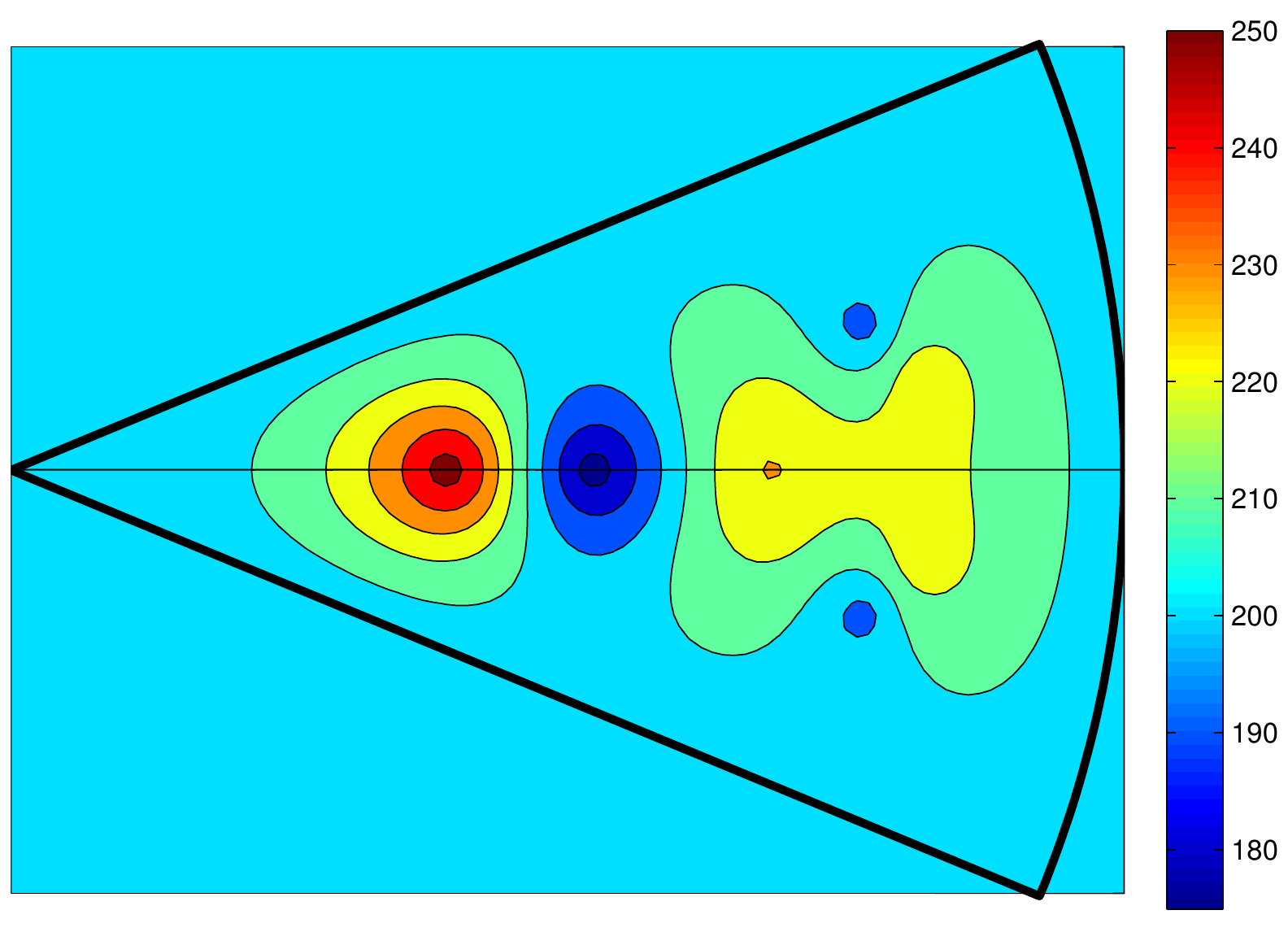}
&\includegraphics[width=2.8cm]{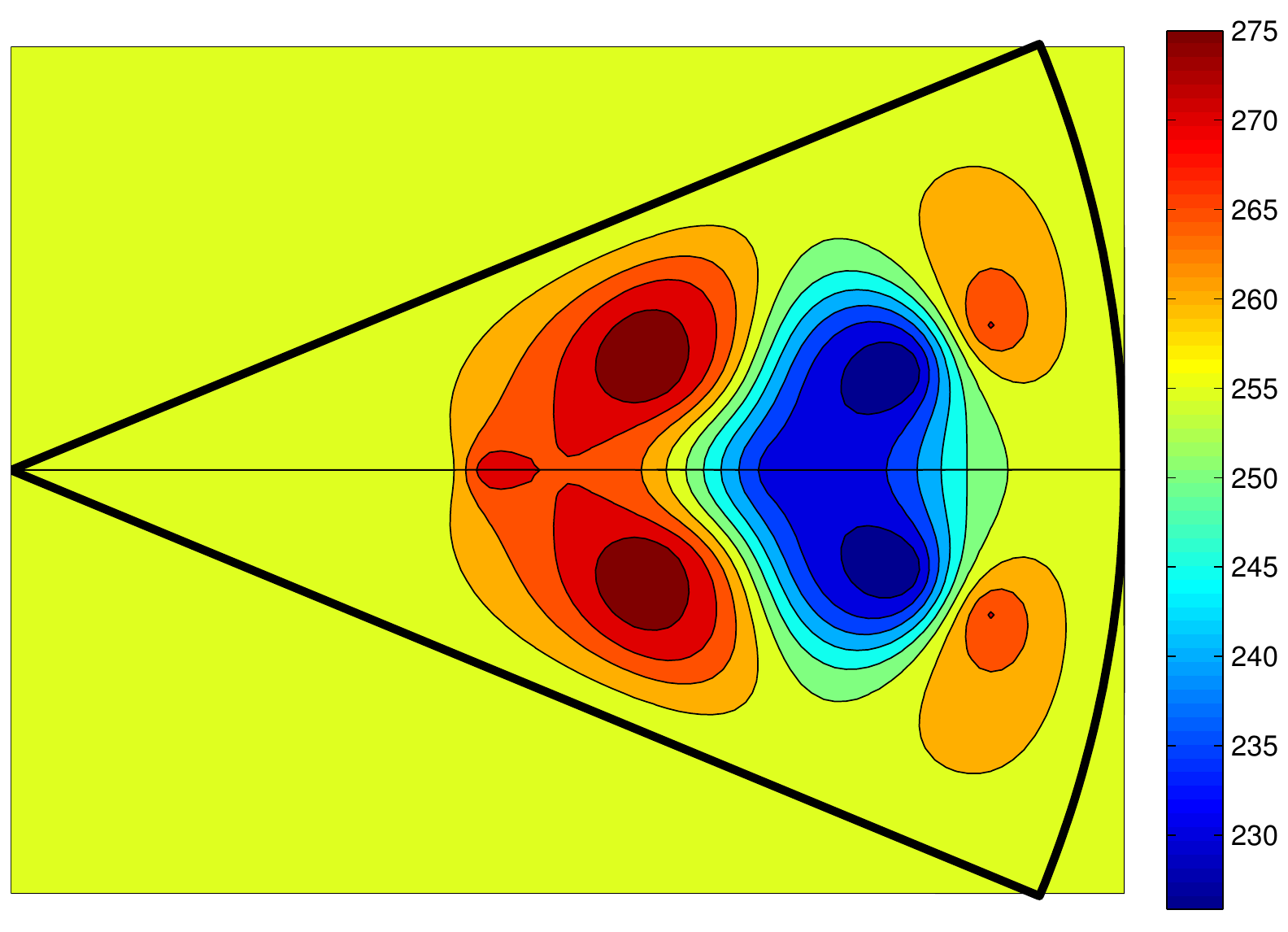}
\end{tabular}
\caption{Aharonov-Bohm eigenvalues for an angular sector $\Sigma_{\omega}$ of opening $\omega=\frac\pi3,\ \frac\pi4$.\label{fig.sectpi34}}
\end{center}
\end{figure}
 
Figures~\ref{fig.sectpi34} give the first five eigenvalues of $H^{AB}(\dot\Sigma_{\omega,\cb})$ when $\Omega$ is an angular sector $\Sigma_{\omega}$ of opening $\omega=\pi/3, \pi/4$ in function of $\cb$. The $k$-th line of each figure gives $\lambda_{k}^{AB}(\cb)$ at the point $\cb\in\Sigma_{\omega}$ and $\lambda_{k}( \Sigma_{\omega})$ outside $\Sigma_{\omega}$. We recover \eqref{contpole} and observe that there exists always an extremal point on the symmetry axis. Figure~\ref{fig.VPABy0} gives the eigenvalues of $H^{AB}(\dot\Sigma_{\frac\pi4,\cb})$ when $\cb$ belongs to the bisector line of $\Sigma_{\pi/4}$. \\
\begin{figure}[h!t]
\begin{center}
\includegraphics[height=6cm]{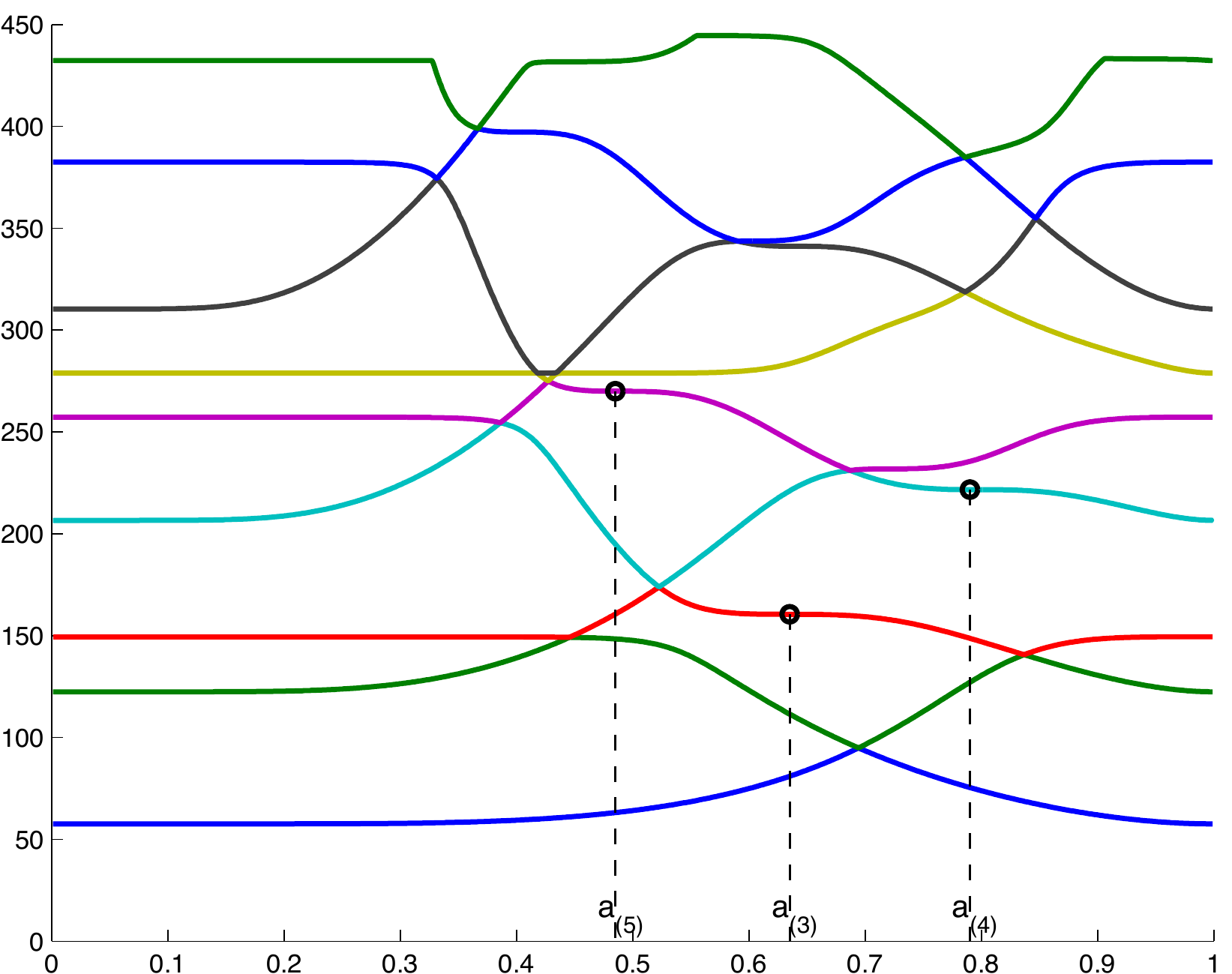}
\caption{$\cb\mapsto \lambda^{AB}_{k}(\cb)$, $\cb\in (0,1)\times \{0\}$, $1\leq k\leq 9$, on $\Sigma_{\pi/4}$.\label{fig.VPABy0}}
\end{center}
\end{figure} 
Let us analyze what can happen at an extremal point (see \cite[Theorem 1.1]{NT}, \cite[Theorem 1.5]{BNNT}).
\begin{theorem}\label{thm.BNNT2}
Suppose $\alpha=1/2$. For any $k\geq1$ and $\cb\in\Omega$, we denote by $\varphi_{k}^{AB,\cb}$ an eigenfunction associated with $\lambda_{k}^{AB}(\cb)$. 
\begin{itemize}
\item If $\varphi_{k}^{AB,\cb}$ has a zero of order $1/2$ at $\cb\in\Omega$, then either $\lambda_{k}^{AB}(\cb)$ has multiplicity at least $2$, or $\cb$ is not an extremal point of the map $\xb\mapsto \lambda_{k}^{AB}(\xb)$. 
\item If $\cb\in\Omega$ is an extremal point of $\xb\mapsto \lambda_{k}^{AB}(\xb)$, then either $\lambda_{k}^{AB}(\cb)$ has multiplicity at least $2$, or $\varphi_{k}^{AB,\cb}$ has a zero of order $m/2$ at $\cb$, $m\geq3$ odd.
\end{itemize}
\end{theorem}
This theorem gives an interesting necessary condition for candidates to be minimal partitions. Indeed, knowing the behavior of the eigenvalues of Aharonov-Bohm operator, we can localize the position of the critical point for which the associated eigenfunction can produce a nice partition (with singular point where an odd number of lines end).\\
For the case of the square, we observe in Figure~\ref{fig.ABsquare} that the eigenvalue is never simple at an extremal point. When $\Omega$ is the angular sector $\Sigma_{\pi/4}$ (see Figures~\ref{fig.sectpi34} and \ref{fig.VPABy0}), the only critical points of $\xb\mapsto \lambda_{k}^{AB}(\xb)$ which correspond to simple eigenvalues are inflexion points located on the bisector line. Their abscissa are denoted $\rm a_{(k)}, k=3,4,5$ in Figure~\ref{fig.VPABy0}. Let $\cb_{(k)}=(\rm a_{(k)},0)$. Figure~\ref{fig.vecpAB3-5} gives the nodal partitions associated with $\lambda_{k}^{AB}(\cb_{(k)})$. We observe that there are always three lines ending at the singular point $\cb_{(k)}$. In Figure~\ref{fig.vecpAB3nodal} are represented the nodal partitions for singular points near $\cb_{(3)}$. When $\cb\neq \cb_{(3)}$, there is just one line ending at $\cb$. 
\begin{figure}[h!t]
\begin{center}
\subfigure[$\lambda^{AB}_{3}(\cb_{(3)})$]{\quad\includegraphics[height=3cm]{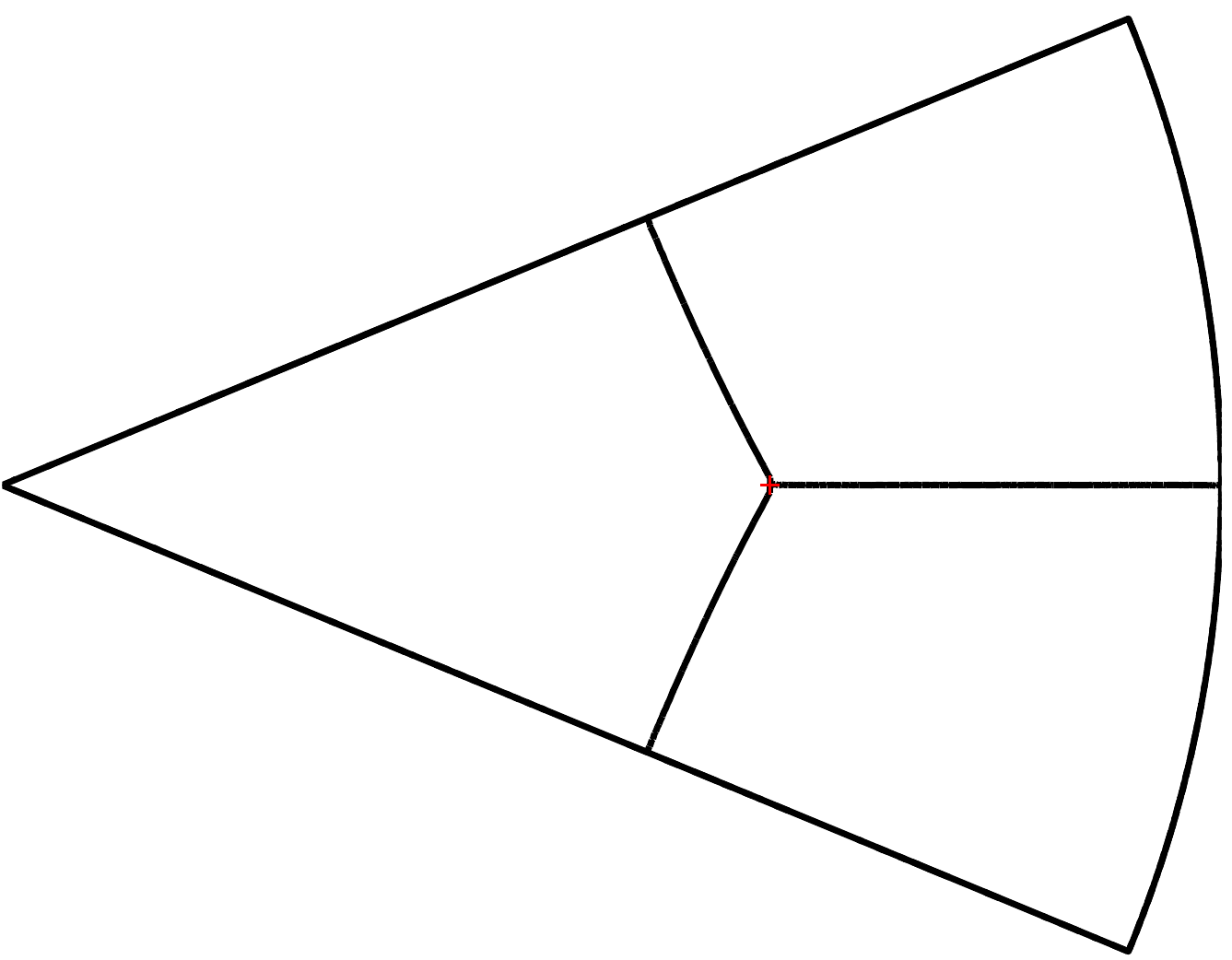}\quad}
\subfigure[$\lambda^{AB}_{4}(\cb_{(4)})$]{\quad\includegraphics[height=3cm]{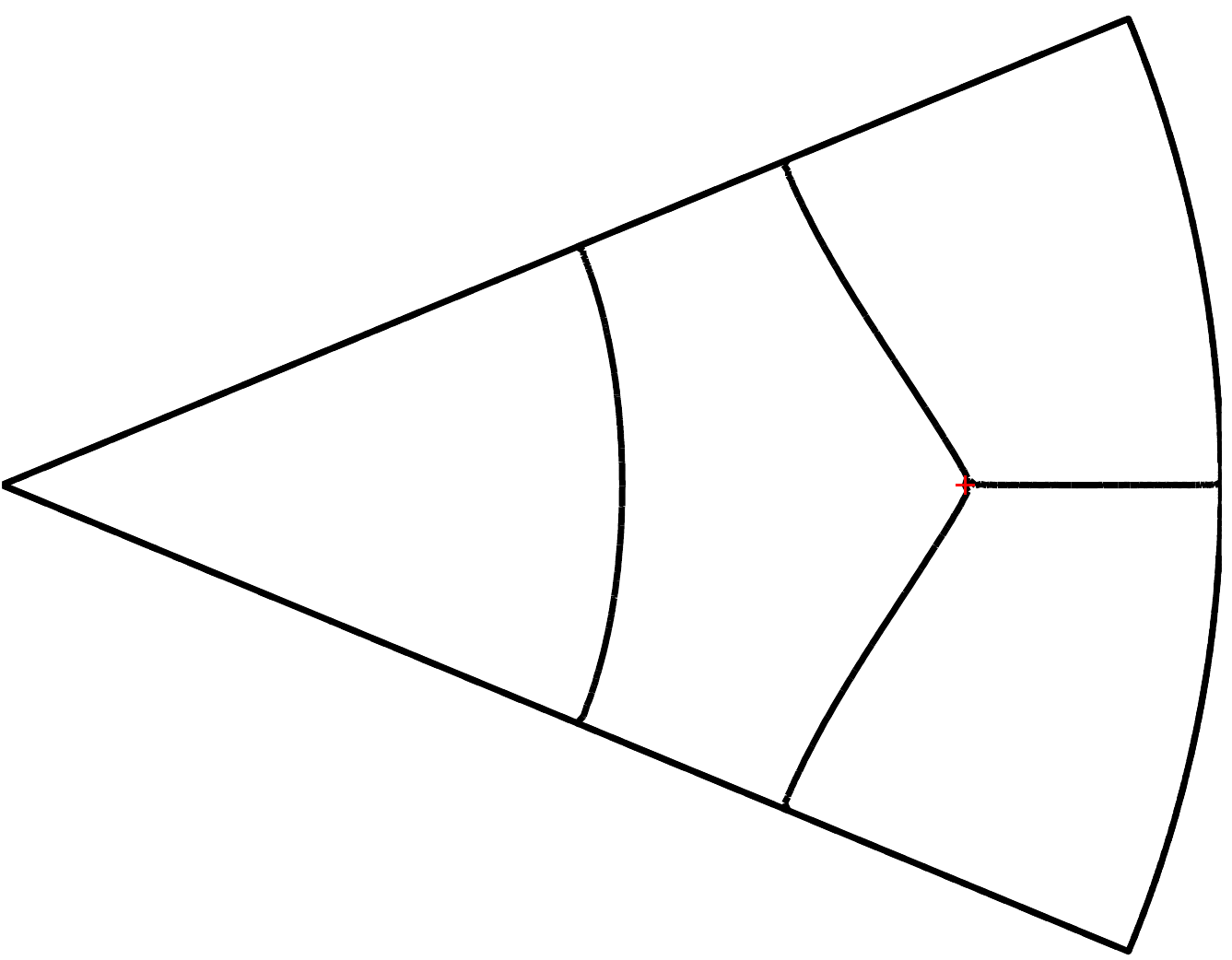}\quad}
\subfigure[$\lambda^{AB}_{5}(\cb_{(5)})$]{\quad\includegraphics[height=3cm]{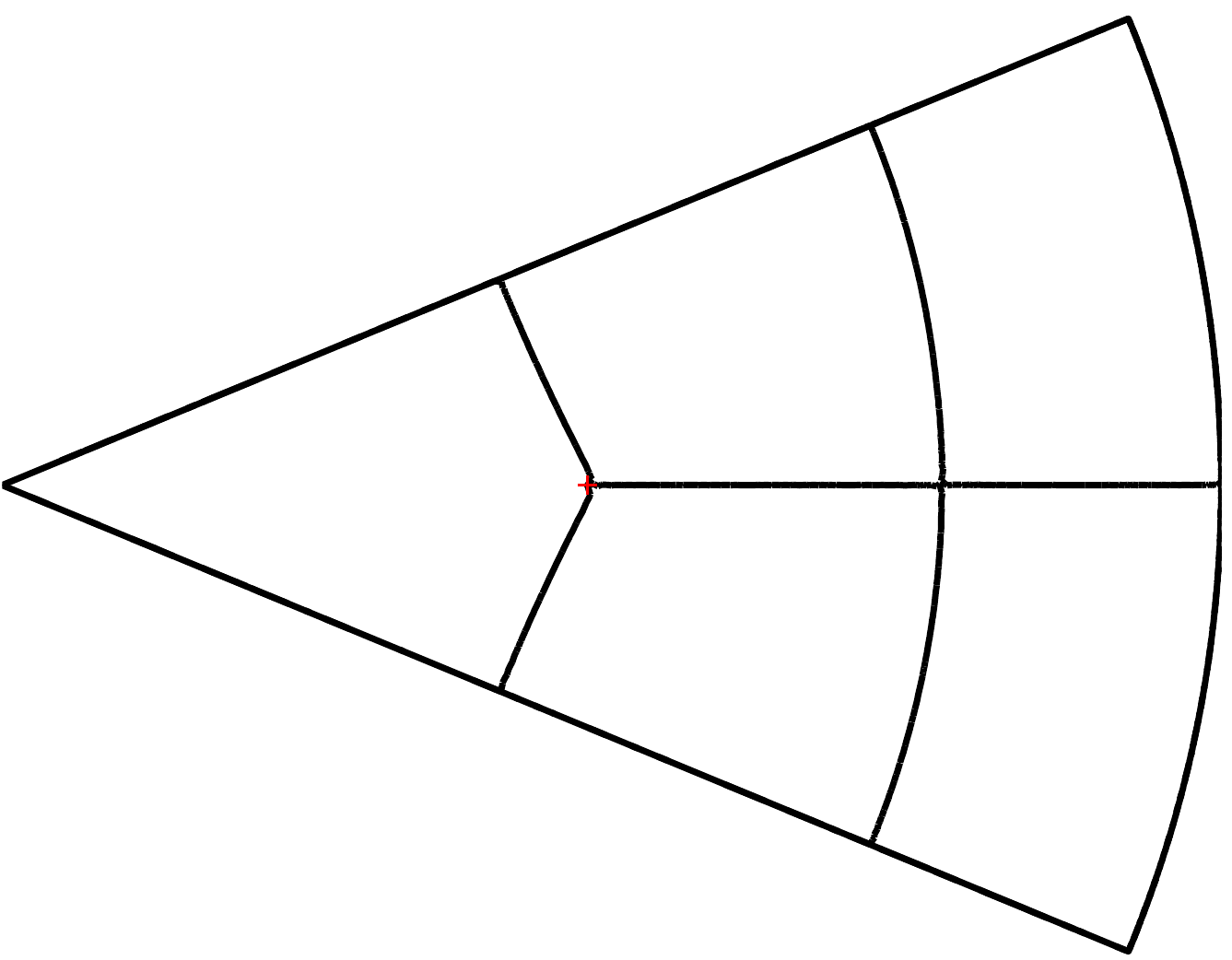}\quad}
\caption{Nodal lines of an eigenfunction associated with $\lambda_{k}(\cb_{(k)})$, $k=3,4,5$.\label{fig.vecpAB3-5}}
\end{center}
\end{figure}
\begin{figure}[h!t]
\begin{center}
\subfigure[$\lambda^{AB}_{3}(\cb)$, $\cb=(0.60,0)$]{\quad\includegraphics[height=3cm]{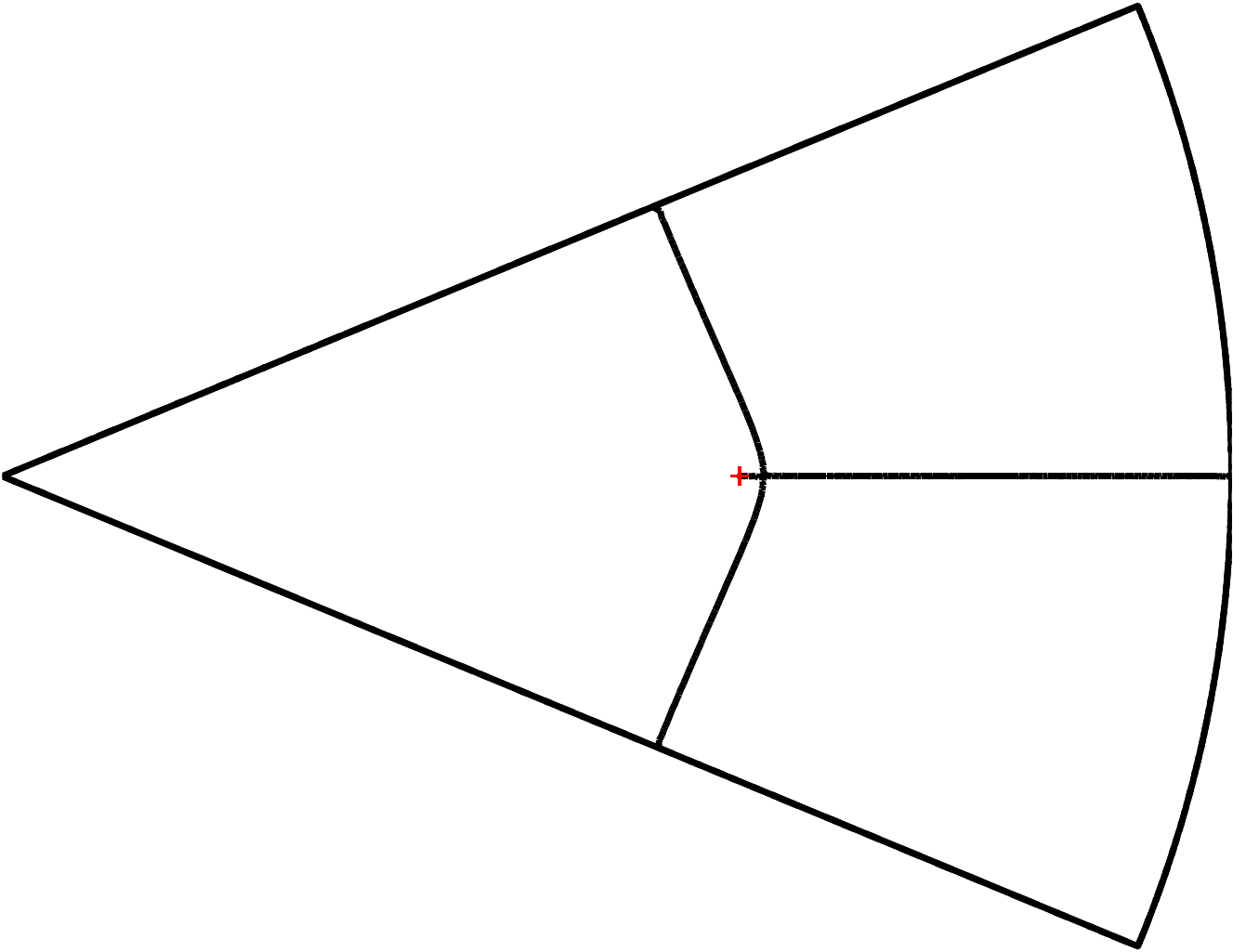}\quad}
\subfigure[$\lambda^{AB}_{3}(\cb)$, $\cb\simeq (\rm a_{(3)},0)$]{\quad\includegraphics[height=3cm]{SectABbissec25pisur100Xc63sur100P6f1VP6.pdf}\quad}
\subfigure[$\lambda^{AB}_{3}(\cb$), $\cb=(0.65,0)$]{\quad\includegraphics[height=3cm]{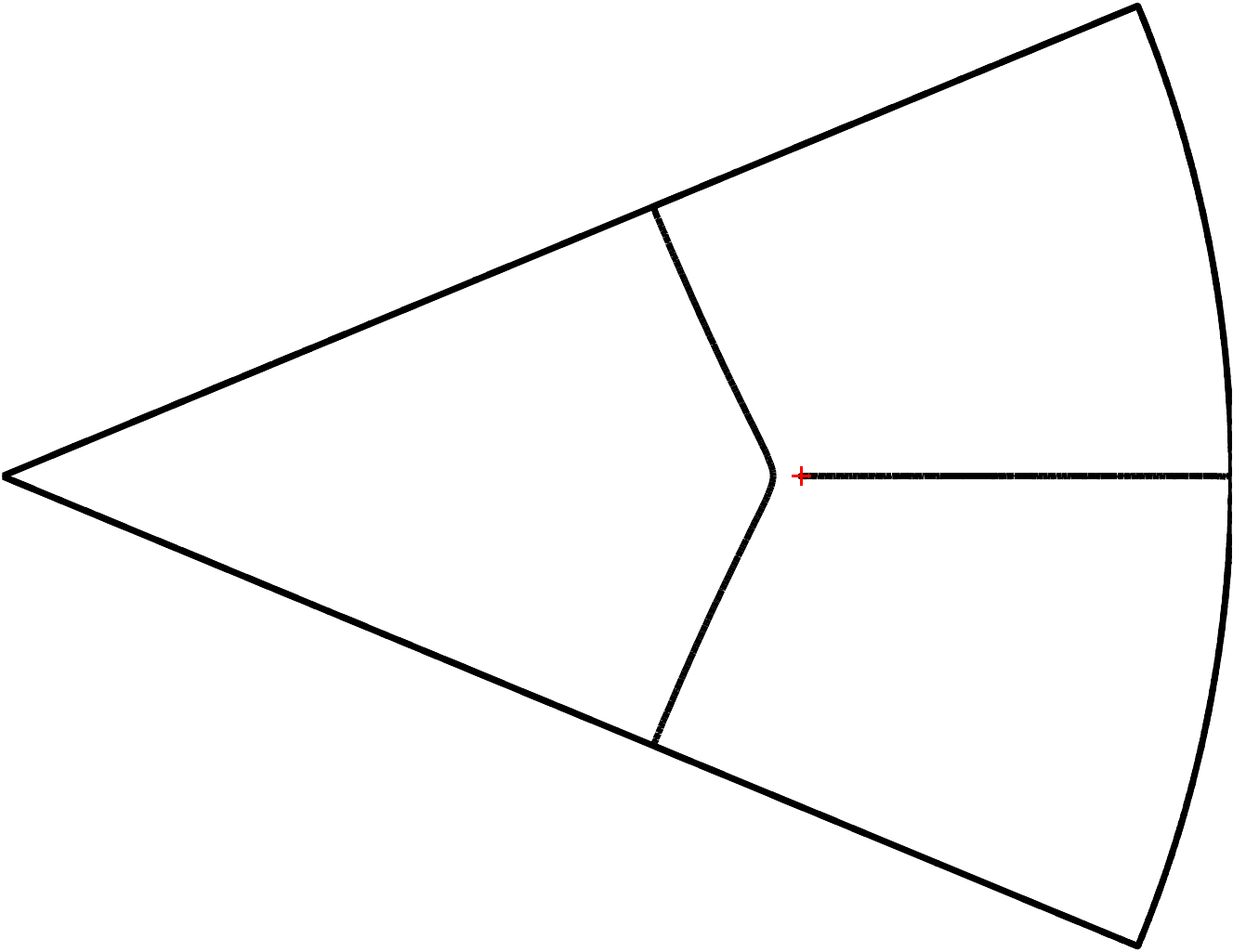}\quad}
\caption{Nodal lines of an eigenfunction associated with $\lambda^{AB}_{3}(\cb)$.\label{fig.vecpAB3nodal}}
\end{center}
\end{figure}

When there are several poles, the continuity result of Theorem~\ref{thm.BNNT} still holds. Let us explain shortly this result (see \cite{Len1} for the proof and more details). This is rather clear in $\Omega^\ell \setminus \mathcal C$, where $\mathcal C$ denotes the ${\Cb}$'s such that $\pb_i \neq \cb_j$ when $i\neq j$. It is then convenient to extend the function $\Cb \mapsto \lambda^{AB}_k(\Cb, \agb )$ to $(\mathbb R^2)^\ell$. We define $\lambda^{AB}_k(\Cb, \agb)$ as the $k$-th eigenvalue of $H^{AB}(\dot\Omega_{\tilde\Cb}, \tilde{\boldsymbol{\alpha}})$, where the $m$-uple $\tilde\Cb=(\tilde\cb_1,\dots,\tilde\cb_{m})$ contains once, and only once, each point appearing in $\Cb=(\cb_1,\dots,\cb_{\ell})$ and where $\tilde{\boldsymbol{\alpha}}=(\tilde\alpha_1,\dots,\tilde\alpha_M)$ with
$\tilde\alpha_k=\sum_{j,\,\cb_j=\tilde\cb_k}\alpha_j,$ for $1\leq k\leq m\,.$
\begin{theorem} \label{thm.EigCont}
If $k\ge 1$ and $\boldsymbol{\alpha} \in \mathbb R^\ell$, then the function $\Cb\mapsto \lambda^{AB}_k (\Cb,\agb )$ is continuous in $\mathbb R^{2\ell}$.
\end{theorem}
This result generalizes Theorems~\ref{thm.BNNT} and \ref{thm.BNNT2}. It implies in particular continuity of the eigenvalues when one point tends to $\partial\Omega$, or in the case of coalescing points. For example, take $\ell=2$, $\alpha_1=\alpha_2 =\frac 12$, $\Cb=(\cb_1,\cb_2) $ and suppose that $\cb_1$ and $\cb_2$ tends to some $\cb$ in $\Omega$. Together with Proposition \ref{Prop8.1}, we obtain in this case that $\lambda^{AB}_k(\Cb, \agb )$ tends to $\lambda_k(\Omega)$.

\subsection{Notes}\label{ss8.5}
More results on the Aharonov-Bohm eigenvalues as function of the poles can be found in \cite{BH,NT,BNNT, AF15, Len1}. We have only emphasized in this section on the results which have direct applications to the research of candidates for minimal partitions.\\
In many of the papers analyzing minimal partitions, the authors refer to a double covering argument. Although this point of view (which appears first in \cite{HHOO} in the case of domains with holes) is essentially equivalent to the Aharonov approach, it has a more geometrical flavor. One can in an abstract way construct a double covering manifold $\dot \Omega^\Cb_{\mathcal R}$ above $\dot \Omega^\Cb$. This permits to lift the problem on this new (singular) manifold but the $K^{\Cb}$-real eigenfunctions can be lifted into real eigenfunctions of the Laplace operator on $\dot \Omega^\Cb_{\mathcal R}$ which are antisymmetric with respect to the deck map (echanging two points having the same  projection on $\dot \Omega^\Cb$). It appears that nodal sets of antisymmetric Courant sharp eigenfunctions on $\dot \Omega^\Cb_{\mathcal R}$ (say with $2k$ nodal domains) give good candidates (by projection) for minimal $k$-partitions. The difficulty is of course with the choice of $\Cb$.\\ 
In the case of the disk, the construction is equivalent to consider $\theta \in (0,4\pi)$, the deck map corresponding to the translation by $2\pi$. The nodal set of the $6$-th eigenfunction gives by projection the Mercedes star and the $11$-th eigenvalue (which is the $5$-th in the space of antiperiodic functions) gives by projection the candidate presented in Figure \ref{fig.MS}.

%%%%%%%%%%%%%%%%%%%
\section{On the asymptotic behavior of minimal $k$-partitions.}\label{sec.klarge}
The hexagon has fascinating properties and appears naturally in many contexts (for example the honeycomb). If we consider polygons generating a tiling, the ground state energy $\lambda(\hexagon)$ gives the smallest value (at least in comparison with the square, the rectangle and the equilateral triangle). We analyze in this section, the asymptotic behavior of minimal $k$-partitions as $k \rightarrow +\infty$.

\subsection{The hexagonal conjecture}\label{ss9.1}
\begin{conjecture}\label{ConjAs1} 
The limit of ${ {\mathfrak L_k(\Omega)}/{k}}$ as 
${ k\rightarrow +\infty}$ exists and $$ 
 |\Omega|\lim_{k\rightarrow +\infty} \frac{\mathfrak L_k(\Omega)}{k}=\lambda(\hexagon)\;.
$$
\end{conjecture}
Similarly, one has
\begin{conjecture}\label{ConjAs2} 
The limit of ${ {\mathfrak L_{k,1}(\Omega)}/{k}}$ as 
${ k\rightarrow +\infty}$ exists and 
\begin{equation}
 |\Omega|\lim_{k\rightarrow +\infty} \frac{\mathfrak L_{k,1}(\Omega)}{k}=\lambda(\hexagon)\;.
\end{equation}
\end{conjecture}
These conjectures, that we learn from M. Van den Berg in 2006 and are also mentioned in Caffarelli-Lin \cite{CL1} for $\mathfrak L_{k,1}$, imply in particular that the limit is independent of $\Omega$.\\
Of course the optimality of the regular hexagonal tiling appears in various contexts in Physics. It is easy to show, by keeping the hexagons belonging to the intersection of $\Omega$ with the hexagonal tiling, the upper bound in Conjecture \ref{ConjAs1},
\begin{equation}\label{upperboundhexa} 
 |\Omega|\lim\sup_{k\rightarrow +\infty} \frac{\mathfrak L_k(\Omega)}{k} \leq \lambda(\hexagon)\;.
\end{equation}
We recall that the Faber-Krahn inequality \eqref{eq.FK} gives a weaker lower bound
\begin{equation}\label{lowerboundhexa}
|\Omega| \frac{\mathfrak L_k(\Omega)}{k} \geq |\Omega| \frac{\mathfrak L_{k,1}(\Omega)}{k} \geq \lambda(\Circle) \,.
\end{equation}
Note that Bourgain \cite{Bo} and Steinerberger \cite{St} have recently improved the lower bound by using an improved Faber-Krahn inequality together with considerations on packing property by disks (see Remark~\ref{rem.FK}).\\
The inequality $\mathfrak L_{k,1}(\Omega) \leq \mathfrak L_k(\Omega)$ together with the upper bound \eqref{upperboundhexa} shows that the second conjecture implies the first one.\\
Conjecture \ref{ConjAs1} has been explored in \cite{BHV} by checking numerically non trivial consequences of this conjecture (see Corollary \ref{corext}). Other recent numerical computations devoted to $\lim_{k\ar +\infty} \frac{1}{k} \mathfrak L_{k,1}(\Omega)$ and to the asymptotic structure of the minimal partitions by Bourdin-Bucur-Oudet \cite{BBO} are very enlightening.
 
\subsection{Universal and asymptotic lower bounds for the length}\label{ss9.2}
We refer to \cite{BeHe0} and references therein for proof and more results. Let $\mathcal{D} = \{D_i\}_{1\leq i\leq k}$ be a regular spectral equipartition with energy $\Lambda = \Lambda(\mathcal{D})$. We define the length of the boundary set $\partial \mathcal D$ by the formula,
\begin{equation}\label{E-lbs}
|\partial\mathcal{D}| := \frac 12 \sum_{i=1}^k |\partial D_i| \,.
\end{equation}

\begin{proposition}\label{P-bg}
Let $\Omega$ be a bounded open set in $\mathbb R^2$, and let $\mathcal{D}$ be a regular spectral equipartition of $\Omega$. The length $|\partial\mathcal{D}|$ of the boundary set of $\mathcal{D}$ is bounded from below in terms of the energy $\Lambda(\mathcal{D})$. More precisely,
\begin{equation}\label{bg}
\frac{|\Omega| }{2{\bf j} } \sqrt{\Lambda(\mathcal{D})} +\frac{\pi{\bf j}}{2\sqrt{\Lambda(\mathcal{D})}} \left(\chi(\Omega ) +\frac{1}{2} \sigma(\mathcal{D})\right) 
\leq |\partial\mathcal{D}|\,.
\end{equation}
Here 
$$
\sigma (\mathcal D):=\sum_{\xb_i\in X(\partial \mathcal D )}\Big(\frac{\nu(\xb_i)}{2}-1\Big) +\frac{1}{2}\sum_{\yb_i\in Y(\partial \mathcal D )}\rho(\yb_i)\,,
$$
which is the quantity appearing in Euler's formula \eqref{Emu}.
\end{proposition}
The proof of \cite{BeHe0} is obtained by combining techniques developed by Br\"uning-Gromes \cite{BrGr} together with ideas of A. Savo \cite{Sav}.

The hexagonal conjecture leads to a natural corresponding hexagonal conjecture for the length of the boundary set, namely
\begin{conjecture}\label{ConjAs3}
\begin{equation}\label{hcl}
\lim_{k\rightarrow +\infty} \frac{ |\partial\mathcal{D}_k|}{\sqrt{k}} = \frac{1}{2} \ell (\hexagon) \sqrt{ |\Omega|}\,,
\end{equation}
where $\ell (\hexagon)= 2 \sqrt{2\sqrt{3}} $ is the length of the boundary of \hexagon.
\end{conjecture}
For regular spectral equipartitions $\mathcal{D}$ of the domain $\Omega$, inequaly \eqref{bg} and Faber-Krahn's inequality yield,
\begin{equation}\label{bga}
\liminf_{\sharp(\mathcal{D}) \to \infty} \frac{|\partial\mathcal{D}|}{\sqrt{\sharp(\mathcal{D})}}
\ge \frac{\sqrt{\pi}}{2} \sqrt{|\Omega|}.
\end{equation}
Assuming that $\chi(\Omega) \ge 0$, we have the uniform lower bound,
\begin{equation}\label{bga-1}
\frac{|\partial\mathcal{D}|}{\sqrt{\sharp(\mathcal{D})}} \ge
\frac{\sqrt{\pi}}{2} \sqrt{|\Omega|}.
\end{equation}
The following statement can be deduced from particular case of Theorem~1-B established by T.C.~Hales \cite{Ha} in his proof of Lord Kelvin's honeycomb conjecture (see also \cite{BeHe0}) which states than in $\mathbb R^2$ regular hexagons provide a perimeter-minimizing partition of the plane into unit areas.

\begin{theorem} \label{C-hales}
For any regular partition $\mathcal{D}$ of a bounded open subset $\Omega$ of $\mathbb R^2$,
\begin{equation}\label{honey1}
|\partial\mathcal{D}| + \frac{1}{2} |\partial\Omega| \geq (12)^{\frac{1}{4}} \, (\min_i |D_i| )^{\frac{1}{2}} \, \sharp(\mathcal{D})\,.
\end{equation}
\end{theorem}

\begin{theorem}\label{P-ale}
Let $\Omega$ be a regular bounded domain in $\mathbb R^2$. For $k\ge 1$, let $\mathcal{D}_k$ be a minimal regular $k$-partition of $\Omega$. Then,
\begin{equation}\label{asym}
\liminf_{k\ar +\infty} \frac { |\partial\mathcal{D}_k|}{\sqrt{k}} \geq
 \frac{1}{2} \ell (\hexagon) \left( \frac{\pi{\bf j}^2}{\lambda(\hexagon)} \right)^{\frac{1}{2}}\, |\Omega|^{\frac{1}{2}}\,.
\end{equation}
\end{theorem} 
\begin{proof}
Let $\mathcal{D}=\{D_i\}_{1\leq i\leq k}$ be a regular equipartition of $\Omega$. Combining Faber-Krahn's inequality \eqref{eq.FK} for some $D_i$ of minimal area with \eqref{honey1}, we obtain
\begin{equation}\label{haa}
|\partial\mathcal{D}|+ \frac{1}{2} |\partial \Omega| \geq (12)^{\frac{1}{4}} \, (\pi {\bf j}^2)^{\frac{1}{2}} \frac{\sharp(\mathcal{D})}{\sqrt{\Lambda(\mathcal{D})}}\,.
\end{equation}
Let $\mathcal{D}_k$ be a minimal regular $k$-partition of $\Omega$. Using \eqref{upperboundhexa} in \eqref{haa} gives \eqref{asym}.
\end{proof} 
To see the efficiency of each approach, we give the approximate value of the different constants:
$$
 \frac{1}{2} \ell (\hexagon)\simeq 1.8612\,,\qquad 
 \frac{1}{2} \ell (\hexagon) \left( \frac{\pi{\bf j}^2}{\lambda(\hexagon)} \right)^{\frac{1}{2}}\simeq 1.8407\,,\qquad
 \frac{\sqrt{\pi}}2 \simeq 0.8862\,.
$$ 
Assume now that $\widehat {\mathcal D_k} :=\mathcal{D}(u_k)$ is the nodal partition of some $k$-th eigenfunction $u_k$ of $H(\Omega)$. Assume furthermore that $\chi(\Omega) \ge 0$. Combining \eqref{bg} with Weyl's theorem leads to:
\begin{equation}
\liminf_{k\to +\infty} \frac{|\partial \widehat{\mathcal D_k}| }{\sqrt k} \ge \frac{\sqrt{\pi}}{\mathbf{j}} \, \sqrt{|\Omega|}\,.
\end{equation}
In the case of a compact manifold this kind of lower bound appears first in \cite{Br}, see also the celebrated work by Donnelly-Feffermann \cite{D-F,D-F1} around a conjecture by Yau.

\subsection{Magnetic characterization and lower bounds for the number of singular points}\label{ss9.3}
 Helffer--Hoffmann-Ostenhof prove a magnetic characterization of minimal $k$-partitions (see \cite[Theorem 5.1]{HHmag}):
\begin{theorem}\label{thchar}
Let $\Omega$ be simply connected and $\mathcal D$ be a minimal $k$-partition of $\Omega$. Then $\mathcal D$ is the nodal partition of some $k$-th $ K_{{\bf \Pb}}$-real eigenfunction of $H^{AB} ( \dot{\Omega}_{\bf P})$ with $\{\pb_1,\ldots, \pb_\ell\}= X^{\sf odd}(\partial \mathcal D )$.
\end{theorem}{
\begin{proof}
We come back to the proof that a bipartite minimal partition is nodal for the Laplacian. Using the $u_j$ whose existence was recalled for minimal partitions, we can find a sequence $\varepsilon_j =\pm 1$ such that $\sum_j \varepsilon_j \exp ( \frac i 2 \Theta_{\Cb} (\xb) )\, u_j(\xb)$ is an eigenfunction of $H^{AB} ( \dot{\Omega}_{\bf P})$, where $\Theta_{\Cb}$ was defined in \eqref{defTheta}.
\end{proof}
The next theorem of \cite{Hel?} improves a weaker version proved in \cite{HH7}.

 \begin{theorem}
 Let $(\mathcal D_k)_{k\in \mathbb N} $ be a sequence of regular minimal $k$-partitions.
 Then there exist $c_0 >0$ and $k_0$ such that for $k\geq k_0$, 
 $$\nu_k:= \sharp X^{\sf odd}(\partial \mathcal D_k )\geq c_0 k\,.
 $$
 \end{theorem}
\begin{proof}
The idea is to get a contradiction if $\nu_k/k$ or a subsequence tends to $0$, with what we get from a Pleijel's like proof. This involves this time for any $k$, a lower bound in the Weyl's formula (for the eigenvalue $\mathfrak L_k$) for the Aharonov-Bohm operator $ H^{AB}(\dot{\Omega}_{\bf P}) $ associated with the odd singular points of $\mathcal D_k$. The proof gives an explicit but very small $c_0$. This is to compare with the upper bound proven in Subsection \ref{ss6.3}.
\end{proof}

\subsection{Notes}
The hexagonal conjecture in the case of a compact Riemannian manifold is the same. We refer to \cite{BeHe0} for the details, the idea being that for $k$ large this is the local structure of the manifold which plays the main role, like for Pleijel's formula (see \cite{BeMe}).
In \cite{ER15} the authors analyze numerically the validity of the hexagonal conjecture in the case of the sphere (for $\mathfrak L_{k,1}$). As mentioned in Subsection \ref{ss6.4}, one can add in the hexagonal conjecture that there are $(k-12)$ hexagons and $12$ pentagons for $k$ large enough. In the case of a planar domain one expects hexagons inside $\Omega$ and around $\sqrt{k}$ pentagons close to the boundary (see \cite{BBO}).
\small
%\bibliographystyle{mnachrn}
%alpha
%\bibliography{ref}
\def\cprime{$'$}

\end{document}